\documentclass[12pt]{article}

\usepackage
{
makeidx,epsfig, amstext, setspace, amsmath, amssymb,amsthm,amsfonts,times,latexsym,mathptmx,algorithm,algorithmic,
wrapfig,bbm,graphicx,color,enumerate,endnotes
}
\usepackage{tabularx}
\usepackage[dvipsnames]{xcolor}

\usepackage[top=2.5cm, bottom= 2.5cm, left= 2.5cm, right= 2.5cm]{geometry}
\newtheorem{theorem}{Theorem}[section]
\newtheorem{CO}[theorem]{Corollary}
\newtheorem{LE}[theorem]{Lemma}

\newtheorem{CN}[theorem]{Conjecture}

\theoremstyle{definition}
  \newtheorem{DEF}[theorem]{Definition}

\newcounter{claim_nb}[theorem]
\newtheorem{claim}[claim_nb]{Claim}
\newtheorem*{claim*}{Claim}

\newcommand{\zC}{\mathcal C}
\newcommand{\zD}{\mathcal D}

\newcommand{\zH}{\mathcal H}

\newcommand{\zL}{\mathcal L}
\newcommand{\zR}{\mathcal R}
\newcommand{\zP}{\mathcal P}
\newcommand{\zQ}{\mathcal Q}
\newcommand{\zS}{\mathcal S}
\newcommand{\zT}{\mathcal T}
\newcommand{\zX}{\mathcal X}
\newcommand{\zY}{\mathcal Y}

\newcommand{\ignore}[1]{}

\newenvironment{cproof}
{\begin{proof}
 [Proof.]
 \vspace{-1.5\parsep}
}
{ \end{proof}}

\newcommand{\?}[1]{%
  \marginpar{%
    \begin{minipage}{2cm}
      \begin{flushleft}%
      \end{flushleft}%
   \end{minipage}%
  }%
}%
\font\smallrm=cmr8

\def\myclaim#1#2\par{{\medbreak\noindent\rlap{\rm(#1)}\ignorespaces
 \rightskip20pt
 \hangindent=20pt\hskip20pt{\ignorespaces\sl#2}\smallskip}}

\def\rt#1{#1}

\def\cc#1{{\color{Magenta}#1}}

\begin{document}
{ 
\baselineskip=12pt
\phantom{a}\vskip .25in
\centerline{\bf  QUICKLY EXCLUDING A NON-PLANAR GRAPH}\vskip.4in
\centerline{{\bf Ken-ichi Kawarabayashi}
\footnote{Supported by JST ERATO Kawarabayashi Large Graph Project JPMJER1201 and by JSPS Kakenhi JP18H05291.}}
\centerline{National Institute of Informatics}
\centerline{2-1-2 Hitotsubashi, Chiyoda-ku, Tokyo 101-8430, Japan}
\medskip
\centerline{{\bf Robin Thomas}%
\footnote{Partially supported by NSF under
Grant DMS-1202640. }}
\centerline{School of Mathematics}
\centerline{Georgia Institute of Technology}
\centerline{Atlanta, Georgia  30332-0160, USA}
\medskip
\centerline{and}
\medskip
\centerline{{\bf Paul Wollan}
\footnote{Supported by the European Research Council under the European Unions Seventh Framework Programme (FP7/2007-2013)/ERC Grant Agreement no. 279558}}
\centerline{Department of Computer Science}
\centerline{University of Rome}
\centerline{00198 Rome, Italy}

\vskip 1in \centerline{\bf ABSTRACT}
\bigskip
\parshape=1.5truein 5.5truein
A cornerstone theorem in the Graph Minors series of Robertson and Seymour
is the result that every graph $G$ with
no minor isomorphic to a fixed graph $H$ has a certain structure.
The structure can then be exploited to deduce far-reaching consequences.
The exact statement requires some explanation, but roughly it says that there
exist  integers $k,n$ depending on $H$ only such that $0<k<n$ and
for every $n\times n$ grid minor $J$ of $G$ the graph $G$ has a
a $k$-near embedding in a surface $\Sigma$ that does
not embed $H$ in such a way that a substantial part of $J$ is embedded in $\Sigma$.
Here a $k$-near embedding means that after deleting at most $k$
vertices the graph can be drawn in $\Sigma$ without crossings,
except for  local
areas of non-planarity, where crossings are permitted, but
at most $k$ of these areas are attached to the rest of the graph by four or more vertices
and inside those the graph is
constrained in a different way, again depending on the parameter $k$.

\parshape=1.5truein 5.5truein
The original and only proof so far is quite long and uses many results
developed in the Graph Minors series.
We give a proof that uses only our earlier paper
[A new proof of the flat wall theorem,
{\it J.~Combin.\ Theory Ser.\ B \bf 129} (2018), 158--203]
and results from graduate textbooks.

\parshape=1.5truein 5.5truein
Our proof is constructive and yields a polynomial time algorithm to construct such a structure.  We
also give explicit constants for the structure theorem, whereas the original proof only guarantees the
existence of such constants.

\vfil\eject
} 

\section{Introduction}

This paper is a continuation of~\cite{flatwall}.
Together these two papers give a self-contained  proof of the excluded clique  theorem~\cite{RS16}
of Robertson and Seymour.

All graphs in this paper are finite, and may have loops and parallel edges.
A graph is a \emph{minor}\?{minor} of another if the first can be obtained
from a subgraph of the second by contracting edges.
An \emph{$H$ minor}\?{$H$ minor} is a minor isomorphic to $H$.
The  question we address is: for a fixed graph $H$,
what can we say about  the structure of graphs that  have no $H$ minor?

There are many so-called excluded minor theorems
in graph theory that answer this question in the form of a necessary and sufficient condition
for various small graphs $H$ (or sets of graphs).
The best known such theorem is perhaps Wagner's reformulation
of Kuratowski's theorem~\cite{Wag37Kur},
which says that a graph has no $K_5$ or
$K_{3,3}$ minor if and only if it is planar.
One can also characterize graphs that exclude only one of those
minors.
To state such a characterization for excluded $K_5$ we need the following
definition.
Let $H_1$ and $H_2$ be graphs, and let $J_1$ and $J_2$ be complete subgraphs
of $H_i$ and $H_2$, respectively, with the same number of vertices.
Let $G$ be obtained from the disjoint union of $H_1-E(J_1)$ and
$H_2-E(J_2)$ by choosing a bijection between $V(J_1)$ and $V(J_2)$
and identifying the corresponding pairs of vertices.
We say that $G$ is a \emph{clique-sum}\?{clique-sum} of $H_1$ and $H_2$.
Since we allow parallel edges, the set that results from the identification
of $V(J_1)$ and $V(J_2)$ may include edges of the clique-sum.
For instance, the graph obtained from $K_4$ by deleting an edge can be
expressed as a clique-sum of two smaller graphs, where one is a triangle and
the other is a triangle with a parallel edge added.
By $V_8$ we mean the graph obtained from a cycle of length eight by
adding an edge joining every pair of vertices at distance four in the cycle.
The characterization of graphs with no $K_5$ minor, due to
Wagner~\cite{Wag37}, reads as follows.

\begin{theorem}
\label{thm:wagner}
A graph has no $K_5$ minor if and only if it can be obtained by
repeated clique-sums, starting from planar graphs and $V_8$.
\end{theorem}

We refer to~\cite{DieGrDec} for a survey of excluded minor theorems.
One might expect that  Theorem~\ref{thm:wagner}  could be extended
to $K_p$ minors for $p\ge6$, but  no such extension is known. Even the next case
of excluding a $K_6$ minor appears to be quite complicated.
For instance, there is a class of graphs that came up in the work of
Robertson,  Seymour and the second author.
In~\cite{WhaPhD} it is shown that no member of that class  has a $K_6$ minor,
and it seems that a structural characterization of graphs with no $K_6$ minor
will have to include this class. On the other hand, the fairly complicated nature of
this class of graphs indicates that it will not be easy to prove such a characterization.
That is the situation for excluding a $K_6$ minor, whereas characterizing graphs
with no $K_p$ minor for $p\ge7$ exactly seems hopeless at the moment.

Thus when excluding an $H$ minor for a general graph $H$ we need to settle
for a less ambitious goal---a theorem that gives a necessary condition
for excluding an $H$ minor, but not necessarily a sufficient one.
However, for such a theorem to be meaningful, the structure it describes
must be sufficient to exclude some other, possibly larger graph $H'$.
For planar graphs $H$ this has been done by
Robertson and Seymour~\cite{RS5}.
To state their theorem we need to recall that the
\emph{tree-width}\?{tree-width} of a graph $G$ is the least integer
$k$ such that $G$ can be obtained by repeated clique-sums,
starting from graphs on at most $k+1$ vertices.

\begin{theorem}
\label{thm:grid}
For every planar graph $H$ there exists an integer $k$ such that
every graph with no $H$ minor has tree-width at most $k$.
If $H$ is not planar, then no such integer exists.
\end{theorem}

This is a very satisfying theorem, because it is best possible in
at least two respects. Not only is there no such integer when $H$ is not planar,
but no graph of tree-width $k$ has a minor isomorphic to the
$(k+1)\times(k+1)$-grid.

\subsection{The Excluded Clique Minor Theorem}

We now turn to excluding a non-planar graph. First a definition.
An $r$-wall is obtained from a $2r\times r$-grid by deleting
every odd vertical edge in every odd row and every even vertical
edge in every even row, then deleting the two resulting vertices
of degree one, and finally subdividing edges arbitrarily.
(We define this again more carefully in Section~\ref{sec:weak}.)
Our objective is to  prove the excluded clique minor theorem of Robertson and Seymour~\cite{RS16}.
Here is a preliminary version
(definitions to follow).

\begin{theorem}
\label{thm:main0}
For every graph $H$ there exist integers $k,n$ with $0<k<n$ such  that the following holds.
Let $G$ be a graph with no $H$ minor.
Then for every $n$-wall $W$ in $G$ there exist
a set $A\subseteq V(G)$ of size at most $k$,
a surface $\Sigma$ of Euler genus at most $k$ and
a $W$-central $\Sigma$-decomposition
of $G-A$
of depth at most $k$ and  breadth at most $k$.
\end{theorem}

To understand the theorem we need to define $\Sigma$-decompositions, their
depth and breadth, and  what it means for them to be $W$-central. We first do so informally, leaving the precise definitions for later.

A {\em $\Sigma$-decomposition} $\delta$ of a graph $G$ consists of a drawing of $G$ in $\Sigma$ with crossings
and a family of closed disks $\Delta_1,\Delta_2,\ldots,\Delta_m$ such that
\begin{itemize}
\item the disks have pairwise disjoint interiors,
\item the boundary of each disk intersects (the drawing of) $G$ in vertices only,
\item $\Delta_i\cap\Delta_j\subseteq V(G)$ for distinct $i,j\in\{1,2,\ldots ,m\}$, and
\item every crossing of $G$ belongs to one of the disks.
\end{itemize}

The {\em breadth} of $\delta$ is the number of integers $i\in\{1, 2, \ldots,m\}$ such  that the boundary of
$\Delta_i$ intersects  $G$ at least four times.
We will refer to such disks as {\em vortices}.
The {\em depth} of a vortex $\Delta_i$ is the  maximum integer $d$ such that the boundary of
$\Delta_i$  can be written as a disjoint union of two connected sets $A$ and $B$,
and there exist $d$ disjoint paths in $G\cap\Delta$,
the subgraph of $G$ drawn in $\Delta$, from $A\cap V(G)$ to $B\cap V(G)$.
The {\em depth} of $\delta$ is the maximum depth of a vortex.
Finally, $\delta$ is {\em $W$-central} if (roughly) for every $i\in\{1, 2, \ldots,m\}$, most of $W$ is
drawn in the complement of $\Delta_i$.

Before presenting the full technical definitions and the statement of the excluded minor theorem, we conclude this section presenting several additional results.  We give the full statements, but note that they require the following definitions all of which are defined rigorously in the next section: model of a minor grasped by a wall $W$, a subwall of a wall, a flat wall in a $\Sigma$-decomposition, and a $W$-central $\Sigma$-decomposition.

\subsection{In the context of the Graph Minors series}

Robertson and Seymour, in their original Graph Minors series of papers which develops the theory of excluded minors, proved several versions of the characterization of graphs excluding a fixed minor.  Specifically, \cite{RS16} presents a \emph{global} excluded minor structure  theorem and an excluded minor structure theorem \emph{with respect to a fixed tangle}.  The global excluded minor structure theorem, Theorem 1.3 of \cite{RS16}, states that given graph $G$ with no $H$ minor can be decomposed into a tree-structure of graphs each admitting a $\Sigma$-decomposition of bounded breadth and depth in a surface $\Sigma$ of bounded genus.  The second variant, Theorem 3.1 of \cite{RS16}, gives the surface decomposition with respect to a fixed object called a tangle, much in the same way that Theorem \ref{thm:main0} gives the structure with respect to a fixed wall.

Of the two variants, perhaps the global decomposition theorem is better known (see for example \cite{Diestel5}) and it has a number of algorithmic applications.  However, Robertson and Seymour's  original application in their proof of Wagner's conjecture that graphs are well-quasi ordered under graph minors relies on the structure with respect to a tangle, and Robertson and Seymour referred to the global structure theorem as a ``red herring" in their search for the proof of Wagner's conjecture \cite{RS16}.  Moreover, starting from the structure theorem for excluded minors with respect to a tangle (Theorem~3.1 of \cite{RS16}), it is possible to quickly derive the global decomposition theorem (Theorem~1.3 of \cite{RS16}).  After giving the appropriate definitions, the proof is only two pages. See, for example, Theorem 4 of \cite{DKMW}.  For these reasons, in this paper we  will focus on Theorem 3.1 of \cite{RS16} which gives the structure with respect to a tangle.

Our Theorem~\ref{thm:main0} differs from Theorem 3.1 of \cite{RS16} in two ways.  First, Theorem \ref{thm:main0} is weaker than the original theorem of Robertson and Seymour
in the sense that in~\cite[Theorem~3.1]{RS16} the surface $\Sigma$ is further restricted by saying that the excluded graph does not embed in $\Sigma$.  In Section~\ref{sec:bdgenus} we deduce such a version, but not with explicit bounds, as follows.

\begin{theorem}
\label{thm:main2}
For every graph $L$, there exists an integer $\mu = \mu(L)$ such that the following holds.  Let $r \ge 1$ be a positive integer and let $R \ge 49152|V(L)|^{24}r + \mu$.
Let $G$ be a graph, and let $W$ be an $R$-wall in $G$.
Then either $G$ has a model of an $L$ minor grasped by $W$,
or there exist a set $A\subseteq V(G)$ of size at most $\mu$,
a surface $\Sigma$ such that $L$ cannot be drawn in $\Sigma$,
a $\Sigma$-decomposition $\delta$ of $G-A$
of depth and breadth at most $\mu$ and an $r$-subwall
$W'$ of $W$ such that $V(W')\cap A=\emptyset$ and $W'$ is a flat wall in $\delta$.
\end{theorem}

Note that if $L=K_p$, then the Euler genus of $\Sigma$ is at most $\lceil \frac{(p-3)(p-4)}{12} \rceil$.

The second notable difference between Theorem \ref{thm:main0} and Theorem 3.1 of \cite{RS16} is that Theorem \ref{thm:main0} gives the surface decomposition which incorporates a large portion of a fixed wall, whereas Theorem 3.1 of \cite{RS16} gives the structure with respect to a large \emph{tangle}.

Tangles, originally introduced in \cite{RS10}, are defined by a consistent orientation of all small order separations in a graph.  A large order tangle is a certificate that a graph has large tree-width.  Theorem \ref{thm:grid} implies therefore that a large order tangle ensures the existence of a large wall subgraph.  Conversely, a wall subgraph naturally defines a tangle in a graph, as shown in \cite{RS10}.

The relationship between walls and tangles, made explicit through Theorem \ref{thm:grid}, has received significant attention to date, yielding both simplified proofs and improved the bounds.  The original proof of \cite{RS5} did not give an explicit bound on the value $k$.  Robertson, Seymour, and Thomas gave a simpler proof with explicit bounds in \cite{RST}.  Further simplifications in \cite{DGJ} yielded the proof of Theorem \ref{thm:grid} in Diestel's graduate text \cite{Diestel}.  Most recently, Chekuri and Chuzoy proved a polynomial bound on the relationship between the maximum size of a wall subgraph and the order of a tangle in \cite{CC} with further improvements to the polynomial bound in \cite{ChuImproved}.

Given that the relationship between tangles and walls has been well established elsewhere, and in the interest of smoothing the presentation, we have chosen to avoid tangles when possible and present Theorem \ref{thm:main0} with respect to a large wall.  In Section \ref{sec:tangles}, we give the necessary background and definitions to derive the structure with respect to a tangle as well.

To summarize, the primary results of this paper are the following.
\begin{itemize}
\item Our main result, Theorem \ref{thm:main} below, gives the structure of a graph $G$ excluding a fixed clique of order $t$ as a minor with respect to a fixed wall $W$ of $G$.  The proof is self-contained (using the results of \cite{flatwall}) and gives explicit values for all the parameters.  After presenting the necessary definitions and stating the theorem in Section 2, the proof occupies Sections 3 to 12.

\item Theorem \ref{thm:main2} gives the structure with respect to a fixed wall and improves the bound on the Euler genus of the surface $\Sigma$ in the $\Sigma$-decomposition, giving the optimal bound.  The proof is self-contained, but does not give explicit bounds on the other parameters in the statement. The proof is given in Section \ref{sec:bdgenus}.

\item We derive Robertson and Seymour's structure theorem with respect to a tangle (Theorem 3.1 of \cite{RS16}), stated using our notation as in Theorem \ref{thm:RS_struct} below.  The proof relies on Theorem \ref{thm:main2} as well as the version of Theorem \ref{thm:grid} found in \cite{Diestel}.  The necessary definitions and lemmas on tangles are presented in Section \ref{sec:tangles} and the proof is given in Section \ref{sec:RSproof}.  This completes our goal of giving a self contained proof of the Robertson and Seymour structure theorem using only results from graduate texts.

\item By substituting a version of Theorem \ref{thm:grid} with explicit bounds into the proof, we also derive a version of Theorem \ref{thm:RS_struct} with explicit bounds on the parameters.  Following the proof found in \cite{DKMW}, we derive a version the global decomposition theorem as well with explicit bounds on the parameters.  The theorem is stated as Theorem \ref{thm:global} and the proof is given in Section \ref{sec:global}.
\end{itemize}

\subsection{Sufficiency}

Theorem~\ref{thm:main0} a gives a necessary condition for excluding a $K_p$ minor grasping a fixed wall $W$ in a graph $G$.  However, it is not a sufficient condition: there exist choices for the graph $G$ and wall $W$ which admit the $W$-central $\Sigma$-decomposition with the parameters given in Theorem \ref{thm:main0} but which still have a $K_p$ minor grasped by $W$.  Theorem \ref{thm:main0} is \emph{approximately} sufficient though: if $G$ admits a $W$-central $\Sigma$-decomposition then the graph excludes a $K_{p'}$ minor grasped by $W$ for a larger value of $p'$.

\begin{theorem}
\label{thm:mainconverse}
Let $R,\alpha,g,b,d\ge 0$ be integers.
Let $G$ be a graph and $W$ an $R$-wall in $G$.  Assume that there exists a set $A\subseteq V(G)$ of size at most $\alpha$,
a surface $\Sigma$ of Euler genus at most $g$, and
a $W$-central $\Sigma$-decomposition of $G-A$
of depth at most $d$ and  breadth at most $b$.
If $G$ has a $K_{p}$ minor grasped by $W$, then $$p\le \alpha+ 4 + \sqrt{6g} + 202(2g + b)^3d.$$
\end{theorem}

As a corollary, if every sufficiently large wall in a graph admits a $W$-central $\Sigma$-decomposition, then the graph does not contain a large clique minor.

\begin{CO}\label{cor:mainconverse}
Let $\alpha,g,b,d\ge 0$ be integers.
Let $R = \alpha + 5 + \lceil \sqrt{6g} \rceil + 202(2g + b)^3d$.
Let $G$ be a graph such that for every
$2R$-wall $W$ in $G$
there exist
a set $A\subseteq V(G)$ of size at most $\alpha$,
a surface $\Sigma$ of Euler genus at most $g$, and
a $W$-central $\Sigma$-decomposition of $G-A$
of depth at most $d$ and  breadth at most $b$.
If $G$ has a $K_{p}$ minor, then $p \le 8R^2$.
\end{CO}

We prove Theorem \ref{thm:mainconverse} and its corollary in Section~\ref{sec:suffic}.  A more careful argument by Joret and Wood \cite{JW} finds the asymptotically optimal upper bound on the size of the clique minor in terms of $\alpha, b, g, d$ for an alternate variant of the excluded minor theorem.


\subsection{Algorithmic applications}

Minor-closed familes of graphs have received significant attention in the graph algorithms community.
Specifically, for $H$-minor-free graphs for a fixed graph $H$,
one can obtain good approximation algorithms for NP-hard problems such as graph coloring, independent set, and others. In general graphs, these problems are known to be hard to even approximate. Such algorithms typically use the topological structure of Theorem \ref{thm:main0} to generalize algorithmic results for planar graphs to $H$-minor-free graphs.

In order to obtain such results, one needs a polynomial time algorithm to obtain the global decomposition theorem, Theorem~1.3 of \cite{RS16}.  Each step of our proof of Theorem \ref{thm:main} is constructive and as such, yields an efficient algorithm for the structure in Theorem \ref{thm:main}.  Along with the algorithmic variants of the previously known results necessary to derive the global decomposition theorem, we get an efficient algorithm to find the structure.
More precisely, in the proofs leading to Theorem \ref{thm:main}, we only use the following algorithmic ingredients:
\begin{itemize}
\item[(1)]
defining an appropriate potential function which we maximize or minimize;
\item[(2)]
finding a maximum set of pairwise disjoint vertex disjoint paths or finding a minimum vertex cut;
\item[(3)]
routing sets of paths through grids;
\item[(4)]
testing ``flatness'' (which reduces essentially to a polynomial time algorithm for finding two disjoint paths with specific terminals - see \cite{klr});
\item[(5)]
recursion.
\end{itemize}
Considering (1), when using such a potential function in the proof, the potential function improves in each step with an upper bound of at most $n^2$ possible improvements, yielding a polynomial bound.
For applications of the ideas in (2)--(4), are all known to be computable in polynomial time.  Finally, the starting point for building the structure in the proof, namely a large wall, can be found in polynomial time.
See \cite{CC, ChuImproved}.
Thus it can be shown that there exists an $f(H)n^{O(1)}$ time algorithm which finds the decomposition for $H$-minor-free graphs for some function $f$.  We defer the details to a future paper fully exploring the algorithmic aspects of the proof of Theorem \ref{thm:main}.




\section{Notation and the full statement of the Excluded Clique Theorem}

\subsection{The Two Disjoint Paths Problem}

Let $G$ be a graph, and let $s_1,s_2,t_1,t_2\in V(G)$.
The TWO DISJOINT PATHS PROBLEM asks whether there exist two disjoint
paths $P_1,P_2$ in $G$ such that $P_i$ has ends $s_i$ and $t_i$.
There is a beautiful characterization of the feasible instances,
which we now describe.
First of all, it is profitable to formulate this problem more generally.

Let $\Omega$ be a cyclic permutation of the elements of some set;
we denote this set by $V(\Omega)$.
A {\em society} is a pair $(G,\Omega)$, where $G$ is a graph, and $\Omega$
is a cyclic permutation with $V(\Omega)\subseteq V(G)$.
A {\em cross} in a society $(G,\Omega)$ is a pair $(P_1,P_2)$ of disjoint
paths in $G$ such as that $P_i$ has endpoints $u_i,v_i\in V(\Omega)$ and is
otherwise disjoint from $V(\Omega)$, and the vertices $u_1,u_2,v_1,v_2$
occur in $\Omega$ in the order listed.
Thus if $G,s_1,s_2,t_1,t_2$ are as in the previous sentence, and we define
$\Omega$ by saying that $V(\Omega)=\{s_1,s_2,t_1,t_2\}$ and that
 $s_1,s_2,t_1,t_2$ occur  in $\Omega$ in the order listed, then the
TWO DISJOINT PATHS PROBLEM is feasible if and only if the society
$(G,\Omega)$ has a cross.

A \emph{separation} in a graph $G$ is a pair $(A,B)$ of subsets of vertices such that
$A\cup B=V(G)$, and there is no edge of $G$ with one endpoint in $A \setminus B$ and
the other in $B \setminus A$.
The \emph{order} of the separation $(A,B)$ is $|A\cap B|$.
We  say that a society $(G,\Omega)$ is {\em internally $4$-connected}
if $|V(\Omega)|\ge4$ and there exists no separation $(A,B)$ of order at most three such  that
$V(\Omega)\subseteq A$, $B \setminus A\ne\emptyset$
 and if the order of $(A,B)$ is three, then $G[B]$,
the subgraph of $G$ induced by $B$, has at least four edges.
Internally $4$-connected societies with no cross have a fairly simple characterization:

\begin{theorem}
\label{thm:i4ccross}
An  internally $4$-connected society $(G,\Omega)$ has no cross if and only if
$G$ can be drawn in a disk with the vertices of $V(\Omega)$ drawn on the boundary
of the disk in the order given by $\Omega$.
\end{theorem}

To state the corresponding theorem for societies that are not necessarily internally $4$-connected
we need to introduce some terminology. We do so in greater generality than immediately necessary.

\begin{DEF}
By a {\em drawing (with crossings) in a surface $\Sigma$} we mean a triple $(U,V,E)$ such that
\begin{itemize}
\item $V\subseteq U\subseteq \Sigma $ and $V$ is finite,
\item $E$ is finite,
\item $V\cup\bigcup_{e\in E}e=U$ and $V$ is disjoint from every $e\in E$,
\item for every $e\in E$, either $e=h((0,1))$, where $h:[0,1]\to U $ is a homeomorphism onto
its image  with $h(0),h(1)\in V$, or $e=h({\mathbb S}^2-(1,0))$, where $h:{\mathbb S}^2\to U$
is a homeomorphism onto its image  with $h((1,0))\in V$, and
\item if $e,e'\in E$ are distinct, then $|e\cap e'|$ is finite.
\end{itemize}
Thus if $\Gamma=(U,V,E)$ is a drawing, then there is a graph $G$ with vertex-set $V$, edge-set $E$ and the
 obvious incidences. We say that $\Gamma $ is a {\em drawing of} every graph isomorphic to $G$.  Two distinct elements $e, e' \in E$ \emph{cross} if $e \cap e' \neq \emptyset$.  We say the drawing is \emph{cross-free} if no pair of elements $e, e'$ cross.
\end{DEF}

\begin{DEF}
By a {\em surface} we mean a compact $2$-dimensional manifold with or without boundary.
By the classification theorem of surfaces every surface is homeomorphic to the sphere
with $h$ handles and $c$ cross-caps added, and the interiors of $d$ disjoint closed disks removed,
 in which case the {\em Euler genus}
of the surface is defined to be $2h+c$.
\end{DEF}

\begin{DEF}
\label{def:sigdec}
Let $\Sigma$ be a surface.
A {\em $\Sigma$-decomposition} of a graph $G$ is
a pair $\delta=(\Gamma,\zD)$, where
\begin{itemize}
\item $\Gamma$ is a drawing of $G$ in $\Sigma$ with crossings, and
\item $\zD$ is a collection of closed disks, each a subset of $\Sigma$
\end{itemize}
such that
\begin{itemize}
\item[(D1)] the disks in $\zD$ have pairwise disjoint interiors,
\item[(D2)] the boundary of each disk in $\zD$ intersects  $\Gamma$ in vertices only,
\item[(D3)] if $\Delta_1,\Delta_2\in \zD$ are distinct, then  $\Delta_1\cap\Delta_2\subseteq V(\Gamma)$, and
\item[(D4)] every edge of $\Gamma$
belongs to the
interior of one of the disks in~$\zD$.
\end{itemize}
Let $N$ be the set of all vertices of $\Gamma$ that do not belong to the interior of any of the disks in~$\zD$.
We will refer to the elements of $N$ has the {\em nodes} of $\delta$.
If $\Delta\in\zD$, then we refer to the set $\Delta-N$ as a {\em cell} of $\delta$.
We denote the set of nodes of $\delta$ by $N(\delta)$ and the set of cells by $C(\delta)$.
Thus the nodes and cells form a partition of $\bigcup(\Delta\,:\,\Delta\in\zD)$.
For a cell $c\in C(\delta)$ we denote by $\widetilde c$  the set of nodes that belong to the closure of $c$.
Thus the cells $c$ of $\delta$ that satisfy $\widetilde c\ne\emptyset$
form the hyperedges of a hypergraph with vertex-set $N(\delta)$, where
 $\widetilde c$ is the set of vertices incident with $c\in C(\delta)$.
For a cell $c\in C(\delta)$ we define $\sigma_\delta(c)$ (or $\sigma(c)$ when $\delta$ is understood)
 to be the subgraph of $G$ consisting of all vertices and edges drawn in the closure of $c$.
We define $\pi_\delta:N(\delta)\to V(G)$  to be the mapping that assigns to every node
in $N(\delta)$ the corresponding vertex of $G$.

A cell $c\in C(\Gamma)$ is called a {\em vortex} if $|\widetilde c|\ge4$.
The $\Sigma$-decomposition $\delta$ is {\em vortex-free}
if no cell in $C(\delta)$ is a vortex.
\end{DEF}

\begin{DEF}
Let $(G,\Omega)$ be a society, and let $\Sigma $ be a surface with one boundary component $B$.
A {\em rendition in $\Sigma $} of $(G,\Omega)$ is a $\Sigma$-decomposition
$\delta$ of $G$
such that
the image under $\pi_\delta$ of $N(\Gamma)\cap B$
is $V(\Omega)$, mapping one of the cyclic
orders of $B$ to the order of $\Omega$.
\end{DEF}

The characterization of cross-free societies, obtained in various forms by
Jung~\cite{Jung},
Robertson and Seymour~\cite{RS9},
Seymour~\cite{SeyDisj}, Shiloach~\cite{Shi}, and Thomassen~\cite{Tho2link}
reads as follows.

\begin{theorem}
\label{thm:crossreduct}
A society $(G,\Omega)$ has no cross if and only if
it has a vortex-free rendition in a disk.
\end{theorem}

A society which has a vortex-free rendition in a disk
 is also called \emph{rural} and in the interest of brevity, we will often refer to it as such.  An expository proof of
Theorem~\ref{thm:crossreduct},
originally due to Robertson and Seymour~\cite{RS9}, is given in~\cite{flatwall}.


\subsection{The Flat Wall Theorem}
\label{sec:weak}

In this section we formulate the weaker version of the excluded
$K_p$ minor theorem of Robertson and Seymour~\cite[Theorem~9.8]{RS13},
for which we have given a simplified proof and a numerical improvement in~\cite{flatwall}.
A further numerical improvement was given by Chuzhoy~\cite{ChuFlat}.

We use $[r]$ to denote $\{1,2,\ldots,r\}$.
Let  $r\ge2$ be an  integer.
An $r\times r$-grid is the graph with vertex-set $[r]\times[r]$ in which
$(i,j)$ is adjacent to $(i',j')$ if and only if $|i-i'|+|j-j'|=1$.
An \emph{elementary $r$-wall} is obtained from the
$2r\times r$-grid
by deleting all edges with endpoints $(2i-1,2j-1)$ and $(2i-1,2j)$
for all $i=1,2,\ldots,r$ and $j=1,2,\ldots,\lfloor r/2\rfloor$
and all edges with endpoints  $(2i,2j)$ and $(2i,2j+1)$
for all $i=1,2,\ldots,r$ and $j=1,2,\ldots,\lfloor (r-1)/2\rfloor$
and then deleting the two resulting vertices of degree one.
An \emph{$r$-wall}
is any graph obtained from an elementary $r$-wall by subdividing
edges.
In other words each edge of the elementary $r$-wall is replaced by
a path.

Let $W$ be an $r$-wall, where $W$ is a subdivision of an elementary
wall $Z$. Let $X$ be the set of vertices of $W$ that
correspond to vertices $(i,j)$ of $Z$ with $j=1$, and let $Y$ be the
set of vertices of $W$ that correspond to vertices $(i,j)$ of $Z$ with
$j=r$.
There is a unique set of
$r$ disjoint paths $Q_1,Q_2,\ldots,Q_r$ in $W$, such that each has
one endpoint in $X$ and the other endpoint in $Y$, and no other vertex in $X\cup Y$.
 We may assume that the
paths are numbered so that the first coordinates of their vertices
are increasing.
We say that $Q_1,Q_2,\ldots,Q_r$ are the
\emph{vertical paths} of $W$.
There is a unique set of $r$ disjoint paths with one endpoint in $Q_1$,
the other endpoint in $Q_r$, and otherwise disjoint from $Q_1\cup Q_r$.
Those will be called the \emph{horizontal paths} of $W$.
Let $P_1,P_2,\ldots,P_r$ be the horizontal paths numbered in the
order of increasing second coordinates.
Then $P_1\cup Q_1\cup P_r\cup Q_r$ is a cycle, and we will call it the
\emph{outer cycle} of $W$.
If $W$ is drawn as a plane graph in the obvious way, then this is
indeed the cycle bounding the outer face.
The sets $V(P_1\cap Q_1)$, $V(P_1\cap Q_r)$, $V(P_r\cap Q_1)$, and
$V(P_r\cap Q_r)$ each include exactly one vertex of $W$;
those vertices will be called the \emph{corners} of $W$.
The vertices of $W$ that correspond to vertices of $Z$ of degree two
will be called the \emph{pegs} of $W$.
Thus given $W$ as a graph the corners and pegs are not necessarily uniquely
determined.
Finally let $W,W'$ be walls such that $W'$ is a subgraph of $W$.
We say that $W'$ is a \emph{subwall} of $W$ if every
horizontal path of $W'$ is a subpath of a horizontal path of $W$,
and every vertical path of $W'$ is a subpath of a vertical path of $W$.

Now let $W$ be a large wall in a graph $G$ with no $K_t$ minor.
The Flat Wall Theorem asserts that there exist a set of vertices
$A\subseteq V(G)$ of bounded size and a reasonably big subwall $W'$
of $W$ which is disjoint from $A$ and has the following property.
Let $C'$ be the outer cycle of $W'$.
The property we want is that $C'$ separates the graph $G-A$ into two graphs,
and the one containing $W'$, say $H$, can be drawn in the plane with
$C'$ bounding a face.
However, as the discussion of the previous subsection attempted to
explain, the latter condition is too strong.
The most we can hope for is for the graph $H$ to be $C'$-flat.
That is, in spirit, what the theorem will guarantee,
except that we cannot guarantee that all of $C'$ be part of a planar
$C'$-reduction of $H$.
The correct compromise is that some subset of $V(C')$ separates off
the wall $W'$, and it is that subset that is required to be incident with
one face of the planar drawing.
Here is the formal definition.

\begin{DEF}
Let $G$ be a graph,
and let $W$ be a wall in
$G$ with outer cycle $D$.
Let us assume that there exists a separation $(A,B)$ such that
$A\cap B\subseteq V(D)$,  $V(W)\subseteq B$,
and there is a choice of pegs of $W$ such that every peg belongs to $A$.
Let $\Omega$ be the cyclic permutation with $V(\Omega)=A\cap B$ and
the cyclic ordering determined by the order on $D$.
If  $(G[B],\Omega)$ is rural,
then we say that
the wall $W$ is \emph{flat in $G$}\?{flat in $G$}.
It follows that it is possible to choose the corners of $W$ is such
a way that every corner belongs to $A$.
\end{DEF}

We need one more definition. Given a wall $W$ in a graph $G$ we will
(sometimes) produce a $K_t$ minor in $G$.
However, this $K_t$ will not be arbitrary; it will be
very closely related to the wall $W$.
To make this notion precise we first notice that a $K_t$ minor in $G$
is determined by $t$ pairwise disjoint sets $X_1,X_2,\ldots,X_t$
such that each induces a connected subgraph and every two of the sets
are connected by an edge of $G$.
We say that $X_1,X_2,\ldots,X_t$ form a \emph{model}\?{model} of a $K_t$
minor and we will
refer to the sets $X_i$ as the \emph{branch-sets} of the model.
Often we will shorten this to a model of $K_t$.
Let $P_1,P_2,\ldots,P_r$ be the horizontal paths and
$Q_1,Q_2,\ldots,Q_r$  the vertical paths of $W$.
We say that a model of a $K_t$ minor in $G$ is \emph{grasped}\?{grasped}
by the  wall $W$ if for every branch-set  $X_k$ of the model
there exist distinct indices $i_1,i_2,\ldots,i_t\in\{1,2,\ldots,r\}$ and distinct
indices $j_1,j_2,\ldots,j_t\in\{1,2,\ldots,r\}$ such that $V(P_{i_l}\cap Q_{j_l})\subseteq X_k$
for all $l=1,2,\ldots,t$.
Let us remark, for those familiar with the literature, that if
a wall grasps a model of $K_t$, then
the tangle determined by $W$
controls it in the sense
of~\cite{RS16}.  We can similarly extend the definition to an arbitrary $H$ minor.  Let $H$ be a graph.  The wall $W$ \emph{grasps an $H$ minor} if there exist pairwise disjoint subsets $\{X_v \subseteq V(G): v \in V(H)\}$ each inducing a connected subgraph and with the property that for all $uv \in E(H)$ there exists an edge of $G$ with one endpoint in $X_u$ and the other endpoint in $X_v$ and such that for all $v \in V(H)$, there exist distinct indices $i_1,i_2,\ldots,i_{|V(H)|}\in\{1,2,\ldots,r\}$ and distinct
indices $j_1,j_2,\ldots,j_{|V(H)|}\in\{1,2,\ldots,r\}$ such that $V(P_{i_l}\cap Q_{j_l})\subseteq X_v$

We can now formulate the Flat Wall Theorem.
It first appeared in a slightly weaker form in~\cite[Theorem~9.8]{RS13}
with an unspecified bound
on $R$ in terms of $t$ and $r$.

\begin{theorem}
\label{thm:flatwall}
Let $r,t \ge 1$ be integers,
let $R=49152t^{24}(\rt{40}t^2+r)$, let $G$ be a graph, and
let $W$ be an $R$-wall in $G$.
Then either $G$ has a model of a $K_t$ minor grasped by $W$, or there exist
a set $A\subseteq V(G)$ of size at most $12288t^{24}$ and an $r$-subwall
$W'$ of $W$ such that $V(W')\cap A=\emptyset$ and $W'$ is a flat wall
in $G-A$.
\end{theorem}

An earlier version of~\cite{flatwall} as well as other articles refer to Theorem~\ref{thm:flatwall}
as the Weak Structure Theorem. However, we prefer the current name, because it gives a more
accurate description of the result.

By  Theorem~\ref{thm:grid}
every graph of sufficiently large tree-width has an $R$-wall.
It follows from~\cite{kk} that in Theorem~\ref{thm:flatwall}
the hypothesis that $G$ has an $R$-wall can be replaced by the assumption
that $G$ has tree-width at least $t^{\Omega(t^2\log t)}R$,
and by~\cite{ChuImproved} it can be replaced by the assumption
that $G$ have tree-width at least $\Omega(R^{19}\hbox{ poly }\log R)$.

\subsection{Linkages and the Excluded Clique Minor Theorem}

Let $G$ be a graph.  A {\em linkage} in $G$ is a set of pairwise vertex disjoint paths.  In a slight abuse of notation, for a linkage $\zP$, we will use $V(\zP)$ and $E(\zP)$ to denote $\bigcup_{P \in \zP} V(P)$ and $\bigcup_{P \in \zP} E(P)$, respectively.  If $P$ is a path and $x$ and $y$ are vertices on $P$, we will denote by $xPy$ the subpath of $P$ with endpoints $x$ and $y$.  If $H$ is a subgraph of a graph $G$, an \emph{$H$-path} is a path of length at least one with both endpoints in $V(H)$, no internal vertex in $H$, and no edge in $H$.  For a subset of vertices $X\subseteq V(G)$, an $X$-path denotes an $H$-path where $H$ is the subgraph with vertex set $X$ and no edges.  In a society $(G, \Omega)$, an \emph{$\Omega$-path} will denote a $V(\Omega)$-path.  An \emph{$H$-bridge} is either an edge $e$ with both endpoints in $V(H)$ such that $e \notin E(H)$, or a subgraph of $G$ formed by a component $C$ of $G-V(H)$ along with all the edges of $G$ with one endpoint in $V(C)$ and the other endpoint in $V(H)$.

Let $(G, \Omega)$ be a society.
A \emph{segment} of $\Omega$ is a subset $X \subseteq V(\Omega)$ such that there do not exist
$x_1, x_2 \in X$ and $y_1, y_2 \in V(\Omega) \setminus X$ all distinct such that $x_1, y_1, x_2, y_x$ occur in $\Omega$ in that order. A vertex $x \in X$ is an \emph{endpoint} of the segment $X$ if there is a vertex $y \notin X$ such that $y$ immediately precedes or immediately succeeds $x$ in the order $\Omega$.  Note that every segment $X$ which is a \cc{non-empty} proper subset of $V(\Omega)$ has endpoints, one of which is \emph{first} and one which is \emph{last} in the order given by $\Omega$.  For vertices $x, y \in V(\Omega)$, we will denote by $x\Omega y$ the uniquely determined segment with first vertex $x$ and final vertex $y$.  In the case that $y$ immediately precedes $x$, we define $x \Omega y$ to be the trivial segment $V(\Omega)$.
A \emph{transaction} in $(G, \Omega)$ is a linkage  $\zP$  of $\Omega$-paths in $G$ such that there exist  disjoint segments $A, B$
of  $\Omega$ such that every element of $\zP$ has one endpoint in $A$ and one endpoint in $B$.
We define  the {\em depth} of $(G, \Omega)$ as the maximum cardinality of  a transaction in $(G, \Omega)$.

We define the {\em breath} of the $\Sigma$-decomposition $\delta$
to be the number
of cells $c\in C(\delta)$ which are a vortex.
Every vortex $c$ defines a society $(\sigma(c),\Omega)$,
called the {\em vortex society of $c$}, by saying that
$\Omega$ consists of the vertices $\pi_\delta(n)$ for
$n\in\widetilde c$ in the order given by bd$(c)$.
(There are two possible choices of $\Omega$, namely $\Omega$ and its
reversal. Either choice gives a valid vortex society.) The \emph{complimentary society} of the vortex society is $(H, \Omega)$ where $H$ is the union of $\bigcup_{c' \in C(\delta), c' \neq c} \sigma(c')$ and the subgraph with vertices $V(\Omega)$ and no edges.
We define the {\em depth} of the $\Sigma$-decomposition as the maximum of
the depths of the vortex societies $(\sigma(c),\Omega)$ over all
vortex cells $c\in C(\delta)$.

\begin{DEF}
Let $G$ be a graph  and $W$  a wall in $G$ with outer cycle $C$.
We say that $W$ is {\em flat in a $\Sigma$-decomposition $\delta$ of $G$}
if there exists a closed disk $\Delta \subseteq\Sigma$   such that
\begin{itemize}
\item the boundary of $\Delta$  does not intersect any cell of $\delta$,
\item $\pi(N(\delta)\cap\hbox{\rm bd}(\Delta))\subseteq V(C)$,
\item every peg of $W$ is the image under $\pi_\delta$ of some element of $N(\delta)\cap\hbox{\rm bd}(\Delta)$,
\item no cell $c\in C(\delta)$ with $c\subseteq\Delta$ is a vortex, and
\item $W-V(C)$ is a subgraph of $\bigcup(\sigma(c)\,:\, c\subseteq\Delta)$.
\end{itemize}
It follows from the proof of \cite[Lemma~5.1]{flatwall} if a wall is flat in $\delta$, then it is flat.
\end{DEF}

\begin{DEF}
Let $G$ be a graph, let $r\ge1$ be an integer and let $W$ be an $r$-wall in $G$.
If $(A,B)$ is a separation of $G$ of order at most $r-1$, then exactly one of the sets $A\setminus B,B \setminus A$ includes the vertex set
of a horizontal and a vertical path of $W$.
If it is the set $A$, then we say that $A$ is the {\em $W$-majority side of the separation $(A,B)$};
otherwise, we say that $B$ is the $W$-majority side.
Now let $\Sigma$ be a surface and let $\delta$ be a $\Sigma$-decomposition of $G$.
We say that $\delta$ is {\em $W$-central} if there is no cell $c\in C(\delta)$ such that $V(\sigma(c))$
includes the $W$-majority side of a separation of $G$ of order at most $r-1$.
Similarly, let $Z \subseteq V(G)$, $|Z| \le r-1$, let $\Sigma'$ be a surface and $\delta'$ be a $\Sigma'$-decomposition of $G-Z$.  The surface decomposition $\delta'$ is $W$-central decomposition of $G-Z$ if for all separations $(A, B)$ of order at most $r - |Z| -1$ such that $B \cup Z$ is the majority side of the separation $(A \cup Z, B \cup Z)$ of $G$, there is no cell $c \in C(\delta')$ such that $V(\sigma_{\delta'}(c))$ includes $B$.  For those readers familiar with the original notation of Robertson and Seymour, if $\zT$ is the tangle defined by the wall $W$, then the $\Sigma$-decomposition is $W$-central if and only the natural segregation defined by the decomposition is $\zT$-central.
\end{DEF}

The following is the full statement of our main theorem.

\begin{theorem}
\label{thm:main}
Let $r,p \ge 0$ be integers,
and let $$R:=49152p^{24}r + p^{10^7p^{26}}.$$
Let $G$ be a graph, and
let $W$ be an $R$-wall in $G$.
Then either $G$ has a model of a $K_p$ minor grasped  by $W$,
or there exist
a set $A\subseteq V(G)$ of size at most $p^{10^7p^{26}}$,
a surface $\Sigma$ of Euler genus at most $p(p+1)$,
a $W$-central $\Sigma$-decomposition $\delta$ of $G-A$
of depth at most $p^{10^7p^{26}}$ and  breadth at most $2p^2$
and an $r$-subwall
$W'$ of $W$ such that $V(W')\cap A=\emptyset$ and $W'$ is a flat wall
in $\delta$.
\end{theorem}

The function $p^{10^7p^26}$ could probably be tightened through more careful analysis in the proof.  
But the bound on the Euler genus will be tightened as mentioned in the introduction so that $K_p$ cannot be drawn in $\Sigma$. Note that the Euler genus of $\Sigma$ is at most $\lceil \frac{(p-3)(p-4)}{12} \rceil$. So our tightened bound only improves the constant. 

Our goal has been to simplify the presentation whenever possible while maintaining a bound which is exponential in a polynomial of $p$.  Our proof cannot fundamentally yield a better bound than exponential, but we conjecture that the correct value should by polynomial.

\begin{CN}
There exist polynomials $f, g$ which satisfy the following.  Let $r, p \ge 0$ be integers.  Let $R = f(r, p) + g(p)$.  Let $G$ be a graph and $W$ an $R$-wall in $G$.  Then either $G$ has a model of a $K_p$ grasped by $W$, or there exists a subset $A$ of at most $g(p)$ vertices, a surface $\Sigma$ of Euler genus at most $g(p)$ and a $W$-central $\Sigma$-decomposition $\delta$ of $G-A$ of depth and breadth at most $g(p)$ and an $r$-subwall $W'$ of $W$ such that $V(W') \cap A = \emptyset$ and $W'$ is flat in $\delta$.
\end{CN}

The proof of Theorem \ref{thm:main} uses some of the ideas presented in the extended abstract \cite{KW}; however, the proof presented here differs in a significant way.  The proof in \cite{KW} required the use of the Unique Linkage theorem, a technical result proven in \cite{RS21}.  Thus, the techniques in \cite{KW} yield neither a self-contained proof nor explicit bounds on the parameters in contrast to the present article.
\subsection{Societies}

The proof of Theorem \ref{thm:main} is based on the systematic study of societies.  We conclude this section with some of the notation and definitions we will need in discussing societies.

\begin{DEF}
Let $(G, \Omega)$ be a society, let $\delta=(\Gamma,\zD)$ be a rendition of $(G,\Omega)$ in
a disk, and let $c_0\in C(\delta)$ be such that no cell
in $C(\delta)\setminus \{c_0\}$ is a vortex.
In those circumstances we say that
the triple  $(\Gamma,\zD,c_0)$ is
a {\em cylindrical rendition of $(G,\Omega)$}.
\end{DEF}

\begin{DEF}
Let $\rho=(\Gamma,\zD)$ be a  rendition of a society $(G,\Omega)$
in a surface $\Sigma$.
For every cell $c\in C(\rho)$ with $|\widetilde c|=2$
we select one of the components of bd$(c)-\widetilde c$.
We will refer to this selection as a {\em tie-breaker in $\rho$},
and we will assume, without any further mention, that every rendition comes equipped
with a tie-breaker.
Now let $P$ be either a cycle or a path in $G$
that uses no edge of $\sigma(c)$ for every vortex $c\in C (\rho)$.
We say that $P$ is {\em grounded in $\rho$}
if either $P$ is a non-zero length path with both endpoints in $\pi_\rho(N(\rho))$, or $P$ is a cycle and it uses
edges of $\sigma(c_1)$ and $\sigma(c_2)$ for two distinct cells
$c_1,c_2\in C(\rho)$.
If $P$ is grounded we define the track of $P$ as follows.  Let $Q_1, \dots, Q_k$ be the distinct maximal subpaths of $P$ such that $Q_i$ is a subgraph of $\sigma(c)$ for some cell $c$.  Fix an index $i$.  The maximality of $Q_i$ implies that its endpoints are $\pi(n_1)$ and $\pi(n_2)$ for distinct nodes $n_1, n_2 \in N(\Gamma)$.
If $|\widetilde c|=2$, define $L_i$ to be the component of bd$(c)-\{n_1,n_2\}$ selected by the tie-breaker,
and if $|\widetilde c|=3$, define $L_i$ to be the component of bd$(c)-\{n_1,n_2\}$ that is disjoint from
$\widetilde c$.
Finally, we define $L_i'$ by pushing $L_i$ slightly so that it is disjoint from all cells in $C(\rho)$.
We define such a curve $L_i'$ for all $i$, maintaining that the curves intersect only at a common endpoint.  The \emph{track of $P$} is defined to be $\bigcup_1^k L_i'$.
Since a track is unique up to homeomorphic shifts, we can speak of the track of $P$
without serious ambiguity.
Thus the track of a cycle is the homeomorphic image of the unit circle,
and the track of a path is an arc in $\Delta$ with both endpoints in $N(\Gamma)$.
We will use the notation $tr(P)$ to denote the track of the subgraph $P$.
\end{DEF}

\begin{DEF}
Let $\rho=(\Gamma,\zD)$ be a  rendition of a society $(G,\Omega)$
in a surface $\Sigma$ and let $\Delta\subseteq \Sigma$ be an arcwise connected set.
A {\em nest in $\rho$ around $\Delta$} is a sequence $\zC=(C_1,C_2,\ldots,C_s)$ of disjoint cycles
in $G$ such that each of them is grounded in $\rho$,
and the track of $C_i$ bounds a closed disk $\Delta_i$ in such a way that
$\Delta \subseteq\Delta_1\subseteq \Delta_2\subseteq\cdots\subseteq
\Delta_s\subseteq\Sigma$.
We will often use the letter $\zC$ to denote the graph $\bigcup_{i=1}^s C_i$.
If $\rho=(\Gamma, \zD, c_0)$ is a cylindrical rendition, then we say that $\zC$
is a {\em nest in $\rho$} if it is a nest around $c_0$.
\end{DEF}

\begin{DEF}
\label{def:innerouter}
Let $(G, \Omega)$ be a society and $\rho = (\Gamma, \zD)$ a  rendition of $(G, \Omega)$ in a surface $\Sigma$.
Let $C$ be a  cycle in $G$ that is grounded in $\rho$ such that track $T$ of $C$ bounds a closed disk $D$.
Let $L$ be the subgraph of $G$ with vertex set $\pi (N(\rho) \cap T)$ and no edges.
We define the {\em outer graph of $C$ in $\rho$} as the graph
$$L \cup \bigcup(\sigma(c)\,:\,c\in C(\rho)\hbox{ and }c\not\subseteq D).$$
We define the \emph{inner graph of $C$ in $\rho$} as the graph
$$L \cup \bigcup(\sigma(c)\,:\,c\in C(\rho)\hbox{ and }c\subseteq D).$$
Finally, we define the inner society of $C$ in $\rho$ as follows.  Let $\Omega'$ be the cyclically ordered set of vertices with $V(\Omega') = V(L)$ and the cyclic order taken from traversing $T$ clockwise.  If $G'$ is the inner graph of $C$ in $\rho$, then the \emph{inner society of $C$ in $\rho$} is the society $(G', \Omega')$.  If $G''$ is the outer graph of $C$, note that $G'$ and $G''$ are edge disjoint, $G' \cup G'' = G$, and $G' \cap G'' = L$.  Moreover, the cycle $C$ need not be a subgraph of either the inner graph $G'$ or the outer graph $G''$.

We say that a linkage $\zQ$ is {\em coterminal with a linkage
 $\zP$ up to level $C$} if there exists  a subset $\zP'$ of $\zP$ such that
$H \cap \zQ =H \cap \zP'$, where $H$ denotes the outer  graph of $C$.
When $\zC=(C_1,C_2,\ldots,C_s)$ will be a nest in $\rho$ and it will be clear from the context which nest
we are referring to, then we will abbreviate ``coterminal with $\zP$ up to level $C_i$" by
``coterminal with $\zP$ up to level $i$".
\end{DEF}

\begin{DEF}
\label{def:radialorthogonal}
Let $(G,\Omega)$ be a society,
let $\zC = (C_1, \dots, C_s)$ be a nest in a rendition
$\rho=(\Gamma, \zD)$ of $(G,\Omega)$ in a surface $\Sigma$.
We say that a linkage $\zP$ is a \emph{radial linkage in $\rho$} if every component has one endpoint in $V(\Omega)$, the other endpoint in the inner graph of $C_1$, and no internal vertex in $V(\Omega)$.
We say that $\zP$ is
 {\em orthogonal to $\zC$} if for all $C\in\zC$ and all $P \in \zP$ the graph $C\cap P$
has exactly one component.
\end{DEF}

\begin{DEF}
Let $(G, \Omega)$ be a society and let $\zP$ be a transaction in $(G, \Omega)$.
We say that $\zP$ is a {\em crosscap transaction} if the members of $\zP$ can be numbered
$P_1,P_2,\ldots,P_n$ and the endpoints of $P_i$ denoted by $u_i$ and $v_i$ in such a way that
$u_1,u_2,\ldots,u_n,v_1,v_2,\ldots,v_n$ appear in $\Omega$ in the order listed.
We say that  $\zP$ is a  crosscap transaction of {\em thickness} $n$.

We say that $\zP$ is a {\em handle transaction} if the members of $\zP$ can be numbered
$P_1,P_2,\ldots,P_{2n}$ and the endpoints of $P_i$ denoted by $u_i$ and $v_i$ in such a way that
$$u_1,u_2,\ldots,u_{2n},v_n,v_{n-1},\ldots,v_1,v_{2n},v_{2n-1},\ldots,v_{n+1}$$
appear in $\Omega$ in the order listed.
We say that  $\zP$ is a  handle transaction of {\em thickness} $n$.  In this case, the linkages $\zP_1 = \{P_1, \dots, P_n\}$ and $\zP_2 = \{P_{n+1}, \dots, P_{2n}\}$ are the \emph{constituent} transactions of the handle transaction.
\end{DEF}

\begin{DEF}
\label{def:grasped}
Let $(G,\Omega)$ be a society,
let $\zC=(C_1,C_2,\ldots,C_s)$ be a nest in a cylindrical rendition
$\rho=(\Gamma, \zD, c_0)$ of $(G,\Omega)$,
and  let ${\cal P}$ be radial a linkage in $\rho$.
We say that a model $(X_1,X_2,\ldots ,X_p)$ of a $K_p$ minor is grasped by $\zC\cup\zP$
if for each $l=1,2,\ldots ,p$ there exist distinct indices $i_1,i_2,\ldots ,i_p\in\{1,2,\ldots ,s\}$
and distinct paths $P_1,P_2,\ldots ,P_p\in\zP$ such that $V(C_{i_j})\cap V(P_j)\subseteq X_l$
for all $j=1,2,\ldots ,p$.
\end{DEF}

\section{Excluding a crooked transaction}\label{sec:GM9}

Let $\zP$ be a transaction in a society $(G, \Omega)$.
We say that $\zP$ is {\em planar} if no two members of $\zP$ form a cross in $(G, \Omega)$.
An element $P \in \zP$ is \emph{peripheral} if there exists a segment $X$
of $\Omega$ containing both the endpoints of $P$
and no endpoint of another path in $\zP$.
A transaction is \emph{crooked} if it has no peripheral element.

In this section, we describe the basic structure of a society with no large crooked transaction. Our main result in this section, Lemma \ref{lem:GM9}, is a slight weakening
of~\cite[Theorem 6.1]{RS9} which gives a better constant in the bound on the depth of the rendition.  However, our proof is slightly shorter. 

We first need several lemmas.

\begin{LE}\label{lem:planarorcrooked}
Let $(G, \Omega)$ be a society and $p\ge 1, q \ge 2$ positive integers.  Let $\zP$ be a transaction in  $(G, \Omega)$
of order $p+q - 2$.  Then there exists $\zP' \subseteq \zP$ such that $\zP'$ is either a planar transaction of order $p$ or a crooked transaction of order $q$.
\end{LE}

\begin{proof}
We proceed by induction on $p$.  If $p = 1$, any element of $\zP$ forms a planar transaction of order one, and so the claim trivially follows.  Assume $p \ge 2$, and thus $|\zP| \ge q$.  If $\zP$ is crooked, the claim again holds.  Thus, we may assume that there exists an element $P \in \zP$ which is peripheral.  By induction $\zP - P$ has a transaction $\zP'$ which is either a crooked transaction of order $q$ or a planar transaction of order $p-1$.  If $\zP'$ is crooked, the claim holds, and otherwise, $\zP' \cup \{P\}$ is a planar transaction of order $p$, as desired.
\end{proof}


Note that the bound in Lemma \ref{lem:planarorcrooked} is best possible.  Let $p$ and $q$ be positive integers, and let $\zP$ be a transaction in a society $(G, \Omega)$ such that $\zP = \zP_1 \cup \zP_2$ where $\zP_1$ and $\zP_2$ are disjoint transactions which satisfy the following.  The transaction $\zP_1$ is a planar transaction of order $p-2$, $\zP_2$ is a transaction of order $q-1$ such that the elements pairwise cross (i.e., a crosscap), and for every $P \in \zP_1$, $Q \in \zP_2$, $P$ and $Q$ do not cross.  It follows that $\zP$ is a transaction of order $p+q - 3$ which does not contain either a planar transaction of order $p$ or a crooked transaction of order $q$.

\begin{LE}\label{lem:crookeddeletions}
Let $(G, \Omega)$ be a society.  Let $\zP$ be a crooked transaction in $(G, \Omega)$ of order at least five.  Then there exists an element $P \in \zP$ such that $\zP - P$ is a crooked transaction.
\end{LE}

\begin{proof}
Assume the claim is false.  Let $P_1$ be an element of $\zP$.  There exists a path $Q_1$ which is peripheral in $\zP - P_1$.  Let $a_1$ and $b_1$ be the endpoints of $Q_1$ such that the $a_1 \Omega b_1$ does not contain the endpoint of any element of $\zP - P_1$ other than $Q_1$.  As $Q_1$ is not peripheral in $\zP$, the path $P_1$ must have at least one endpoint in $a_1\Omega b_1$, and since $P_1$ is not peripheral, it cannot have both endpoints in $a_1 \Omega b_1$.  Thus $P_1$ and $Q_1$ cross.

Let $P_2$ be a path in $\zP \setminus \{P_1, Q_1\}$.  Again, $\zP - P_2$ has a peripheral element $Q_2$ with endpoints $a_2$ and $b_2$ such that the segment $b_2 \Omega a_2$ is disjoint from $\zP - P_2$.  As in the paragraph above, it holds that $P_2$ and $Q_2$ cross.  The vertices $a_1, b_1, b_2,a_2$ occur on $\Omega$ in that order and the two segments $a_1 \Omega b_1$ and $b_2 \Omega a_2$ are disjoint.  It now follows that every element  of $\zP \setminus \{P_1, Q_1, P_2, Q_2\}$ must have one endpoint in $a_2 \Omega a_1$ and the other endpoint in $b_1 \Omega b_2$.  By our assumptions on the order of $\zP$, this set is non-empty and any such path can be deleted from $\zP$ and the resulting linkage will still be crooked, a contradiction.
\end{proof}

There is a natural exchange property for crosses which given a cross $P_1$ and $P_2$ in a society $(G, \Omega)$ along with a path $R$ from $V(\Omega)$ to $P_1 \cup P_2$ allows us to find a new cross containing $R$.

\begin{LE}\label{lem:crossreroute}
Let $(G, \Omega)$ be a society and $P_1, P_2$ a cross in $(G, \Omega)$.  Let $R$ be a path with one endpoint in $V(\Omega)$, the other endpoint in $V(P_1) \cup V(P_2)$ which is internally disjoint from $V(\Omega) \cup V(P_1) \cup V(P_2)$.  Then there exists an index $i \in \{1, 2\}$ and a subpath $Q$ contained in $P_i$ such that $R \cup Q$ and $P_{3-i}$ form a cross.
\end{LE}

The analogous statement for crooked transactions does not hold, but a similar type of result can be shown.
The next two lemmas are restated results of \cite{RS9}.  The first lemma is~\cite[(3.2)]{RS9}
and is in preparation for the following lemma which shows a variant of Lemma \ref{lem:crossreroute} for crooked transactions.  We include the proofs for completeness.

\begin{LE}
\label{lem:mcdonaldsprep}
Let $(G, \Omega)$ be a society and let $\zP$ be a crooked transaction in $(G, \Omega)$.
Let $Q$  be a path in $G$ with one end  $w\in V(\Omega)$, the other endpoint $z\in V(P)$, where $P\in\zP$,
and no vertex besides $z$ of $P$ is in $V(\zP)$.
Let $u,v$ be the endpoints of $P$, and let $R_1:=uPz$ and $R_2:=vPz$.
If neither $(\zP - P) \cup \{R_1 \cup Q\}$ nor $(\zP - P) \cup \{R_2 \cup Q\}$ is a crooked transaction in $(G, \Omega)$,
then there exist distinct paths $S_1,S_2\in\zP-P$ with endpoints $x_1,y_1$ and $x_2,y_2$, respectively, such that
\begin{itemize}
\item[{\rm(i)}] the vertices $u,x_1,x_2,v,y_2,y_1$ occur in $\Omega$ in the order listed,
\item[{\rm(ii)}] every path in $\zP-P$  has one endpoint in $x_1\Omega x_2$ and the other endpoint in $y_2\Omega y_1$, and
\item[{\rm(iii)}] $w\in x_1\Omega x_2\cup y_2\Omega y_1$.
\end{itemize}
\end{LE}

\begin{proof}
Assume that neither $(\zP - P) \cup (R_1 \cup Q)$ nor $(\zP - P) \cup (R_2 \cup Q)$ is a crooked transaction.  Fix $\zP_1 = (\zP - P) \cup (Q \cup R_1)$ and $\zP_2 = (\zP - P) \cup (Q \cup R_2)$.

We begin by noting that $w$ is not an endpoint of any member of $\zP$.
Otherwise, since $z$ is the only vertex of $Q$ in $V(\zP)$, we have that $Q$ is a trivial path and $w=z\in\{u,v\}$.  If say $w=z=v$, then $R_1\cup Q=P$, and hence $\zP_1 = \zP$ which is crooked, a contradiction. Thus $w$ is not an end of any member of $\zP$.
Let $A$ and $B$ be disjoint, complimentary segments of $\Omega$ such that  $u\in A$ and every element of $\zP$ has one end in $A$ and one endpoint in $B$.  Let $a_1$ and $a_2$ be the ends of $A$ such that $A = a_1 \Omega a_2$.  Similarly, let $b_1$ and $b_2$ be the endpoints of $B$ such that $B = b_1 \Omega b_2$.
It follows that $v\in B$.
We may assume from the symmetry that $w\not\in A$ and that $u,v,w$ occur in $\Omega$ in the order listed.
Then  $\zP_1$ is a transaction in $(G, \Omega)$, and since it is not crooked there exists a
peripheral path $S_1\in \zP_1$.

We claim that $S_1\in\zP-P$. To prove this claim, suppose for a contradiction that $S_1=R_1 \cup Q$.
Thus either $w\Omega u$ or $u\Omega w$ includes no endpoint of a member of $\zP-P$.  If $u\Omega w$ does not contain the endpoint of an element of $\zP -P$, then as $u \Omega v \subseteq u \Omega w$, we have that $P$ is also peripheral in $\zP$, a contradiction.  Thus, no element of $\zP - P$ has an endpoint in $w \Omega u$.  As $P$ is not peripheral in $\zP$, some element of $\zP - P$ has an endpoint in $v \Omega u$.  As no such element has an endpoint in $w \Omega u \subseteq v \Omega u$, we conclude that some element of $\zP - P$ has an endpoint in $v \Omega w$.  Fix $P_1$ to be such a path.  Note that as $P_1$ has no endpoint in $a_1 \Omega u$, and thus $P_1$ crosses $P$.  The linkage $\zP_2$ is a transaction from $B \setminus w\Omega b_2$ to $A \cup w\Omega b_2$.  Moreover, neither $Q \cup R_2$ nor $P_1$ is peripheral in $\zP_2$ as $Q \cup R_2$ crosses $P_1$.  By assumption, $\zP_2$ has a peripheral element $P_2 \notin\{ P_1, Q \cup R_2\}$.  Note that $P_2 \in \zP - P$.  Let the endpoints of $P_2$ be $c_1 \in A$ and $c_2 \in B$.  Thus, $c_1 \in u \Omega a_2$.  As $P_2$ does not cross $Q \cup R_2$, $c_2 \in b_1 \Omega v$.  We conclude that $c_1 \Omega c_2$ does not contain the endpoint of any element of $(\zP - P) \cup (Q \cup R_2)$ nor does it contain $v$.  Thus, $P_2$ is peripheral also in $\zP$, a contradiction.  This proves the claim that $S_1 \in \zP - P$.

Fix the endpoints of $S_1$ to be $x_1 \in A$ and $y_1 \in B$.  There exists a segment $Z = x_1 \Omega y_1$ or $Z = y_1 \Omega x_1$ such that $Z$ does not have an endpoint of an element of $\zP_1$ other than the vertices $x_1$ and $y_1$.  As $S_1$ is not peripheral in $\zP$, $Z$ contains an endpoint of an element of $\zP$ as an internal vertex.  The only possibility for such a vertex is $v$ and, thus, $v \in Z$, $u,w \notin Z$.  We conclude that $u, x_1, v, y_1, w$ occur in $\Omega$ in that order with $Z = x_1 \Omega y_1$.  Specifically, $x_1 \Omega y_1$ does not contain the endpoint of any element of $\zP - P$.

We conclude that $\zP_2$ is a transaction from $A \cup b_1 \Omega v$ to $B \setminus b_1 \Omega v$ and by assumption, it has a peripheral element $S_2$.  Note that as $Q \cup R_2$ crosses $S_1$, the path $S_2$ is distinct from $Q \cup R_2$ and $S_1$.  Let the endpoints of $S_2$ be $x_2 \in A$ and $y_2 \in B$.  We have already shown that $x_2$ and $y_2$ are not in $x_1 \Omega y_1$, and so the vertices $x_2 \in a_1 \Omega x_1$ and $y_2 \in y_1 \Omega b_2$.  As $S_2$ is peripheral in $\zP_2$, but not in $\zP$, we have as above that $u \in y_2 \Omega x_2$ but neither $Q \cup R_2$ nor an element of $\zP-P$ has an endpoint as an internal vertex of $y_2 \Omega x_2$.  We conclude that $x_2 \in u \Omega x_1$ and $y_2 \in w \Omega b_2$.  Thus, $u, x_2, x_1, v, y_1, y_2$ occur in $\Omega$ in that order, completing the proof of the lemma.
\end{proof}

The final lemma we need is a restatement of~\cite[(3.3)]{RS9}.

\begin{LE}
\label{lem:crookedmcdonalds}
Let $(G, \Omega)$ be a society and let $\zP$ be a crooked transaction in $(G, \Omega)$.
Let $Q_1$ and $Q_2$ be disjoint paths, each with one endpoint in $V(\Omega)$, the other endpoint in $V(\zP)$, and no other vertex
in $V(\zP)$.
There exists a path $P \in \zP$ and subpaths $R_1, R_2$ of $P$ such that one of the following is a crooked transaction:
\begin{itemize}
\item[{\rm(i)}] $(\zP - P) \cup \{R_1 \cup Q_1\}$,
\item[{\rm(ii)}] $(\zP - P) \cup \{R_2 \cup Q_2\}$, or
\item[{\rm(iii)}] $(\zP - P) \cup \{R_1 \cup Q_1, R_2 \cup Q_2\}$.
\end{itemize}
\end{LE}

\begin{proof}
Let $P$ be the unique member of $\zP$ intersected by $Q_1$.
We may assume that there is no subpath $R_1$ of $P$ such that $(\zP - P) \cup \{R_1 \cup Q_1\}$  is a crooked transaction.
By Lemma~\ref{lem:mcdonaldsprep}  applied to the path $Q_1$  there exist paths $S_1,S_2$
with endpoints $x_1,y_1$ and $x_2,y_2$, respectively, satisfying (i)--(iii) of Lemma~\ref{lem:mcdonaldsprep}.
It follows that $P$  is the only member of $\zP$ that crosses every other member of $\zP$.
Let $P'$ be the unique member of $\zP$ intersected by $Q_2$. Similarly,
by Lemma~\ref{lem:mcdonaldsprep}  applied to the path $Q_2$ we may assume that there exist paths $S_1',S_2'$
with endpoints $x_1',y_1'$ and $x_2',y_2'$, respectively, satisfying (i)--(iii) of Lemma~\ref{lem:mcdonaldsprep}.
Then $P'$ also crosses every other member of $\zP$, and hence $P=P'$ and $\{S_1,S_2\}=\{S_1',S_2'\}$.
For $i=1,2$ let $R_i$ be the subpath of $P$ with one endpoint an endpoint of $P$ and the other endpoint an endpoint of $Q_i$,
which is disjoint from $Q_{3-i}$.
It follows that $(\zP - P) \cup \{R_1 \cup Q_1, R_2 \cup Q_2\}$ is a crooked transaction, as desired.
\end{proof}

We now present the main result of this section.  As we said above, it is a slight weakening
of~\cite[Theorem 6.1]{RS9} which gives a bound of $3p + 7$ on the depth of the rendition.  However, our proof is somewhat shorter.

\begin{LE}\label{lem:GM9}
Let $(G, \Omega)$ be a society and $p \ge 4$ a positive integer.
Then $(G, \Omega)$ has either  a crooked transaction of cardinality $p$, or  a cylindrical rendition
 of depth at most $6p$.
\end{LE}

\begin{proof}
Assume the lemma is false, and let $(G, \Omega)$ be a counterexample with a minimal number of vertices.
Since $(G, \Omega)$ is a counterexample to our lemma, it has depth at least $6p+1$, and therefore it has
a transaction of order $6p+1$.
By Lemma \ref{lem:planarorcrooked}, we may assume that $(G, \Omega)$ has a planar transaction $\zP$ of order $5p+3$.
Let $A$ and $B$ be disjoint segments such that every element of $\zP$ has one endpoint in $A$ and one endpoint in $B$.
Label the elements of $\zP$ as $P_0, P_1, \dots, P_{5p+2}$ and the endpoints of $P_i$ as $a_i\in A$ and $b_i\in B$
in such a way that traversing $A$ from one endpoint to the other in
the order given by $\Omega$, we encounter $a_0, a_1, \dots, a_{5p+2}$ in that order.

Let $H$ denote the subgraph of $G$ obtained from the union of the elements of $\zP$ by adding the elements of  $V(\Omega)$
as isolated vertices.

Let $G_1'$ be the subgraph of $H$ consisting of $P_0\cup P_1\cup\cdots\cup P_{2p+1}$ and all the vertices in the segment
$b_{2p+1}\Omega a_{2p+1}$.
Let us consider all $H$-bridges with at least one attachment in $V(G_1')\setminus V(P_{2p+1})$, and for each such
an $H$-bridge $B$ let $B'$  denote the graph obtained from $B$ by deleting all attachments that do not belong to  $V(G_1')$.
Finally, let $G_1$ denote the union of $G_1'$ and all the graphs $B'$ as above.
Let the cyclic permutation $\Omega_1$ be defined by saying that $V(\Omega_1)=b_{2p+1}\Omega a_{2p+1}$
and that the order on $\Omega_1$ is obtained by following $b_{2p+1}\Omega a_{2p+1}$.
Thus $(G_1,\Omega_1)$ is a society, and we define the society $(G_2,\Omega_2)$ analogously using the paths
$P_{3p+1}, P_{3p+2},\ldots, P_{5p+2}$ instead, with $P_{3p+1}$ playing the role of $P_{2p+1}$.

If the graphs $G_1$ and $G_2$ are not disjoint, then $G_1\cup G_2$ includes a path
from $b_{2p+1}\Omega a_{2p+1}$ to $a_{3p+1}\Omega b_{3p+1}$, which is disjoint from $P_{2p+2}\cup P_{2p+3}\cup\cdots\cup P_{3p}$,
and hence that path together with the paths $P_{2p+2},\allowbreak P_{2p+3},\ldots, P_{3p}$ forms a crooked transaction
of cardinality $p$, a contradiction.
Thus we may assume that the graphs $G_1$ and $G_2$ are disjoint.
If both the societies $(G_1,\Omega_1)$ and $(G_2,\Omega_2)$ have a cross then, since $p\ge4$, the union of the two crosses
and $P_{2p+2}\cup P_{2p+3}\cup\cdots\cup P_{3p-3}$ is a crooked transaction of cardinality $p$, as required.
We may therefore assume from the symmetry that  $(G_1,\Omega_1)$ has no cross.

Fix a vortex-free rendition $\rho = (\Gamma, \zD)$ of $(G_1, \Omega_1)$ in the disk $\Delta$.  For each $i$, $0 \le i \le p+2$, let $T_i$ be the track of $P_i$ in $\rho$.  Define the subgraph $J_i$ of $G_1$ as follows.  Let $\Delta'$ be the closed disk defined by the closure of the connected component of $\Delta \setminus T_i$ which does not contain $\pi^{-1}(b_{2p+3})$.   Let $J_i$ be the subgraph formed by the union of $\bigcup_{\{c \in C(\rho): c \subseteq \Delta'\}} \sigma(c)$ along with the vertices $\pi(N(\rho) \cap T_i)$ treated as isolated vertices.  Note that $J_i$ is a subgraph of $J_{i'}$ for $i' \ge i$.

\begin{claim}\label{cl:Jisolated}
Let $1 \le i \le p+2$.  Let $X \subseteq V(G)$ be the set of vertices $\pi(N(\rho) \cap T_i)$.  There does not exist an edge of $G$ with one endpoint in $V(J_i) \setminus X$ and the other endpoint in $V(G) \setminus V(J_i)$.
\end{claim}

\begin{cproof}
Assume the claim is false, and let $i$ be an index such that there exists an edge $xy$ forming a counterexample to the claim.  Assume $x \in V(J_i)$ and $y \in V(G) \setminus V(J_i)$.  Given the rendition $\rho$, the vertex $y$ is not an element of $V(G_1)$.  Thus, there exists a path $Q_1$ in $J_i$ linking $x$ and a vertex of $V(\Omega)$ in $J_i$ and a path $Q_2$ linking $y$ and a vertex of $V(\Omega)$ in $G - G_1$.  Thus, $Q_1 \cup Q_2 \cup xy$ along with the paths $P_{p+3}, \dots, P_{2p+1}$ forms a crooked transaction of order $p$, a contradiction.
\end{cproof}

We define the society $(G^*, \Omega^*)$ as follows.  Let $Y = \pi(N(\rho) \cap T_1)$, and let $G^* = G - (J_1 - Y)$.  Note that at least the vertex set $\{a_0, b_0\}$ is contained in $V(J_1) \setminus Y$, and thus $|V(G^*)|<|V(G)|$.  We define the society vertices $\Omega^*$ as $V(\Omega^*) = Y \cup a_1\Omega b_1$ where the vertices $a_1\Omega b_1$ are ordered according to $\Omega$ and the vertices of $Y$ are ordered by the order in which they occur when traversing $P_1$ from $b_1$ to $a_1$.  Note that $E(G^*) \cap E(J_1) = \emptyset$ by construction,  and by construction and Claim \ref{cl:Jisolated}, we have that $G = G^* \cup J_1$.

Define $J^{**}$ to be the subgraph $J_{p+2} - (J_1 - Y)$ of $G^*$.  Let $Z = \pi( T_{p+2} \cap N(\rho))$.  We define a society $(J^{**},\Omega^{**})$ with $V(\Omega^{**}) = Y \cup a_1\Omega a_{p+2} \cup  Z \cup b_{p+2}\Omega b_1$ with the cyclic order $\Omega^{**}$ defined by the set $Y$ ordered by traversing $P_1$ from $b_1$ to $a_1$, $a_1 \Omega a_{p+2}$ ordered according to $\Omega$, $Z$ ordered by traversing $P_{p+2}$ from $a_{p+2}$ to $b_{p+2}$ and finally $b_{p+2}\Omega b_1$ ordered by following $\Omega$.  Note that $(J^{**}, \Omega^{**})$ has a vortex free rendition in the disk of with the vertices of $\Omega^{**}$ on the boundary in the appropriate order and that $P_i$ is a subgraph of $J^{**}$ for $2 \le i \le p+1$.

By the minimality of $(G,\Omega)$, the society $(G^*,\Omega^*)$ has either a crooked transaction of order $p$, or  a cylindrical rendition
 of depth at most $6p$.  In the latter case, we can find a cylindrical rendition of depth at most $6p$ of $(G, \Omega)$ by joining the rendition of $(G^*, \Omega^*)$ to the restriction of $\rho$ to $J_1$.
We may therefore assume that $(G^*,\Omega^*)$ has  a crooked transaction of order $p$.
Choose such a transaction $\zQ$ to minimize by containment $E(\zQ) \cup E(P_2)\cup E(P_3)\cup\cdots\cup E(P_{p+2})$.

If no element of $\zQ$ has an endpoint in $Y$, then the lemma is proven.  Thus, we may assume that at least one $Q \in \zQ$ endpoints in a vertex of $Y$.

We first observe that no element of $\zQ$ is contained in $J^{**}$.  To see this, assume $Q \in \zQ$ is a subgraph of $J^{**}$.  The track of $Q$ defines a subdisk $\Delta_Q$ of the disk which is disjoint from $Z$ and it follows that $Q$ is not crossed by any element of $\zQ$.  If we assume that such a path $Q$ is picked to  minimize $\Delta_Q$ (by containment), we conclude that not only $Q$ is not crossed by any element of $\zQ$, but it is also peripheral, contradicting the fact that $\zQ$ is crooked.  This proves that no such $Q$ can exist.

\begin{claim}
\label{cl:meetPs}
If a path $Q\in\zQ$ uses a vertex of $Y$, then it intersects each of  $P_2, P_3,\ldots,P_{p+2}$.  For every $i$, $2 \le i \le p+2$, there exist at least two distinct elements of $\zQ$ which intersect $P_i$.
\end{claim}

\begin{cproof}
The first part of the claim follows from the fact that if such a path $Q$ was disjoint from $P_i$, the rendition of $(J^{**}, \Omega^{**})$ would imply that $Q$ is a subgraph of $J^{**}$, contrary to what we have already seen.

To see the second part of the claim, fix $i$ with $2 \le i \le p+2$.  As at least one element of $Q \in \zQ$ has an endpoint in $Y$, we have that $Q$ intersects $P_i$.  Let the endpoints of $Q$ be $y$ and $y'$ with $y \in Y$.  Assume that no other element of $\zQ$ intersects $P_i$.  Let $R$ be the path in $P_i$ with one endpoint equal to $a_i$, the other endpoint in $Q$, denote it $x$, and no internal vertex in $Q$.  As no element of $\zQ$ is contained in $J^{**}$, again the rendition implies that no element of $\zQ$ has an endpoint in the segment $y\Omega^{**}a_i$.  Thus, the linkage $(\zQ - Q )\cup (R \cup xQy')$ which replaces the path $Q$ with $R \cup xQy'$ has the endpoints in the same circular order on $\Omega$.  Thus, $(\zQ - Q )\cup (R \cup xQy')$ is a crooked transaction and will have strictly fewer edges in $E((\zQ - Q )\cup (R \cup xQy')) \cup E(P_2) \cup \dots \cup E(P_{2p+2})$ as at least one edge of $yQx$ is not contained in any $P_j$.  This contradicts the choice of $\zQ$ and completes the proof of the claim.
\end{cproof}

Note that an immediate consequence of Claim \ref{cl:meetPs} is that none of the paths $P_2, \dots, P_{p+2}$ is also an element of $\zQ$.

For $i=2,3,\ldots,p+2$ let us say that an endpoint $x$ of $P_i$ is {\em available} if no path in $\zQ$ has the vertex $x$ as an endpoint.
Since $\zQ$ has cardinality $p$, there are at least two available vertices.
Let $x$ and $y$ be distinct available vertices, where $x$ is an endpoint of $P_i$ and $y$ is an endpoint of $P_j$ (possibly with $i=j$).  Let $R_x$ be a minimal subpath of $P_i$ with one endpoint equal to $x$ and one endpoint in $\zQ$.  Similarly, let $R_y$ be a minimal subpath of $P_j$ with one endpoint equal to $y$ and one endpoint in $\zQ$.

We claim that $R_x$ is disjoint from $R_y$.  If not, it must be the case that $i = j$, $R_x$ and $R_y$ have a common endpoint, and that this endpoint is the unique vertex of $P_i$ which intersects $\zQ$, contradicting Claim \ref{cl:meetPs}.

Apply Lemma \ref{lem:crookedmcdonalds} to the paths $R_x$, $R_y$, and the transaction $\zQ$.  Thus there exists an element $Q' \in \zQ$ and subpaths $Q_1$ and $Q_2$ of $Q'$ such that one of the three outcomes of Lemma \ref{lem:crookedmcdonalds} holds for $Q'$, $Q_1$, $Q_2$, $R_x$ and $R_y$.  Assume one of the first two outcomes holds, and without loss of generality, that $(\zQ - Q' )\cup (R_x \cup Q_1)$ is a crooked transaction.  By the choice of $\zQ$, the edge set of $(\zQ - Q') \cup (R_x \cup Q_1)$ is the same as the crocked transaction $\zQ$.  Thus, every edge of the path $Q' - Q_1$ is contained in $P_2 \cup \dots \cup P_{p+2}$.  It follows that $V(P_i) = V(R_x) \cup V(Q' - Q_1)$.  Thus, $Q'$ is the unique path of $\zQ$ to intersect $P_i$, contradicting Claim \ref{cl:meetPs}.

We conclude that the third outcome holds and that $\zQ'$ is a crooked transaction of order $p+1$.  Let $Q_3$ be the minimal subpath of $Q'$ with endpoints in $R_x$ and $R_y$ such that $Q' = Q_1 \cup Q_2 \cup Q_3$.  If $Q_3$ has at least one edge of $P_2 \cup \dots \cup P_{p+2}$, then $E(\zQ') \cup E(P_2) \cup \dots \cup E(P_{p+2})$ is a proper subset of $E(\zQ) \cup E(P_2) \cup \dots \cup E(P_{p+2})$, a contradiction since $\zQ'$ contains a crooked transaction of order $p$ by Lemma \ref{lem:crookeddeletions}.

Thus, $E(Q_3)$ is a subset of $E(P_2) \cup \dots \cup E(P_{p+2})$.  But in this case, $i = j$ and both $R_x$ and $R_y$ are subpaths of $P_i$.  It follows that $Q'$ is the unique path of $\zQ$ which intersects $P_i$, a contradiction.  This completes the proof of the lemma.
\end{proof}

\section{Rooting a transaction}

In this section we develop four lemmas with the following common theme.
Suppose we are given a nest in a cylindrical rendition of a society, and a radial linkage orthogonal to the nest.
Given a transaction  in the society, we would like to modify the transaction so as to coincide with the radial linkage,
at least in a substantial portion of the rendition.  Similar results are proved in \cite{RS15}, but our proof is simpler.

We need  different lemmas depending on the properties we want the resulting transaction to satisfy.
In the first lemma  all we require is that the resulting transaction avoid part of the society.

\begin{LE}\label{lem:rootedtransaction}
Let $r, s$ be positive integers with $r \ge 1$ and $s \ge r+1$.  Let $(G, \Omega)$ be a society, $\rho = (\Gamma, \zD, c_0)$
a cylindrical rendition of $(G, \Omega)$ and  let $\zC= (C_1, C_2, \dots, C_s)$ be a nest in $\rho$ around $c_0$.
Let $X_1,X_2, Y_1, Y_2$ be disjoint segments of $\Omega$, such that $X_1, Y_1, X_2, Y_2$ occur in $\Omega$ in that order.  Let $\zP_1,\zP_2, \zR_1, \zR_2$ be radial
 linkages in $\rho$ orthogonal to $\zC$  such that $\zP_1\cup\zP_2 \cup \zR_1 \cup \zR_2$ is a linkage.  Assume that for $i=1,2$ every member of $\zP_i$ has an endpoint in $X_i$ and every member of $\zR_i$ has an endpoint in $Y_i$.
 Assume $|\zP_1| = |\zP_2| = r$ and $|\zR_1|  = |\zR_2| = 2$.
Let $J$  be the subgraph formed by $\zR_1 \cup \zR_2$ and the union of all subpaths of $\zC$ with both endpoints in $V(\zR_1) \cup V(\zR_2)$ and no internal vertex in $\zP_1\cup\zP_2 \cup \zR_1 \cup \zR_2$.
Let $\zP$ be a transaction of order $r$ in $(G, \Omega)$ from the segment $X_1$ to the segment $X_2$ which is disjoint from $J$.
Then there exists a transaction $\zQ$ in $(G, \Omega)$ of order $r$ from $X_1$ to $X_2$ such that
\begin{itemize}
\item $V(\zQ) \subseteq V(\zP)\cup V(\zP_1)\cup V(\zP_2)\cup \left( \bigcup_1^r V(C_i) \setminus V(J)\right)$, and
\item $\zQ$ is coterminal with $\zP_1\cup\zP_2$ up to level $C_{r+1}$.
 \end{itemize}
\end{LE}

\begin{proof}
Let $H$ be subgraph formed by the union of the elements of $\zP_1$, the elements of $\zP_2$ and $\bigcup_1^r C_i - V(J)$.  For $i = 1, 2$, let $H_i$ be the union of the component of $H$ intersecting $X_i$.  Let $P \in \zP$.  As the path $P$ is disjoint from $J$ and given the cylindrical rendition of $(G, \Omega)$, for every $1 \le i \le r$ the path $P$ intersects $C_i \cap H_1$ and $C_i \cap H_2$.

For $P \in \zP$, fix $P'$ to be a minimal subpath of $P$ with one endpoint in $C_r \cap H_1$ and the other endpoint in $C_r \cap H_2$.  It follows that $P'$ is disjoint from $C_{i}$ for all $i \ge r+1$.

Let $G'= H_1 \cup H_2 \cup \bigcup_{P \in \zP} P'$.  By Menger's theorem there exist $r$ disjoint paths in $G'$ from $X_1$ to $X_2$.  To see this, assume otherwise.  Then there exists a set $X$ of size at most $r-1$ intersecting every $X_1 - X_2$ path in $G'$.  There exists $P_1 \in \zP_1$ and $P_2 \in \zP_2$ which are disjoint from $X$ and similarly, there exists $i$, $1 \le i \le r$ and $P' \in \zP$ such that $C_i$ and $P'$ are disjoint from $X$.  It follows that $P_1 \cup P_2 \cup C_i \cup P'$ contains a path from $X_1$ to $X_2$, a contradiction.

Fix $\zQ$ to be a set of $r$ disjoint paths from $X_1$ to $X_2$ in $G'$, and it now satisfies the desired outcome of the lemma.
\end{proof}

The second lemma is a slight  variation of the first.
We omit the almost identical proof.


\begin{LE}\label{lem:rootedlinkage}
Let $r, s$ be positive integers with $r \ge 1$ and $s \ge r+1$.  Let $(G, \Omega)$ be a society, $\rho = (\Gamma, \zD, c_0)$
a cylindrical rendition of $(G, \Omega)$ and  let $\zC= (C_1, C_2, \dots, C_s)$ be a nest  in $\rho$.
Let $\zP_1$ be a radial linkage in $\rho$ orthogonal to $\zC$.
 Assume $|\zP_1| = r$,
and let $\zP$ be a linkage of cardinality $r$ in $(G, \Omega)$ from a set $Z\subseteq V(\sigma(c_0))$ to  $V(\Omega)$.
Let $J:=C_1\cup C_2\cup\cdots\cup C_r$.
Then there exists a linkage $\zQ$ in $(G, \Omega)$ of order $r$ from $Z$ to  $V(\Omega)$ such that
\begin{itemize}
\item $V(\zQ) \subseteq V(\zP)\cup V(\zP_1)\cup V(J)$, and
\item $\zQ$ is coterminal with $\zP_1$ up to level $C_{r+1}$.
 \end{itemize}
\end{LE}


In the third lemma we are given a crooked transaction, and we want the resulting transaction also to be crooked.

\begin{DEF}
Let $(G, \Omega)$ be a society and $H$ a subgraph of $G$.  Let $R$ be any $\Omega$-path, and let $x$ be an endpoint of $R$.  The \emph{$H$-hook of $R$ containing $x$} is the maximal subpath of $R$ containing $x$ contained in $H$.  Note, the hook may be the null graph in the case when the endpoint $x$ is not contained in $X$, and it may as also be the whole path $R$ in the case when $R$ is a subpath of $H$.

Let the endpoints of $R$ be $x$ and $y$, and let $S_x$ and $S_y$ be the $H$-hooks of $R$ containing $x$ and $y$, respectively. The \emph{$H$-unhooked interior of $R$} is the minimal subpath $R^*$ or $R$ such that $R = R^* \cup S_x \cup S_y$.

When there can be no confusion as to the subgraph $H$, we will simply refer to the \emph{hook of $R$} and the \emph{unhooked interior of $R$}.
\end{DEF}

We will also need the following auxiliary result for the third main result of this section.


\begin{LE}
\label{lem:crookinrend}
Let $(G, \Omega)$ be a society and $\rho = (\Gamma, \zD, c_0)$ a cylindrical rendition of
$(G, \Omega)$. If $\zQ$ is a crooked society in $(G, \Omega)$, then no member of  $\zQ$
is a subgraph of $\bigcup(\sigma(c)\,:\,c \in C(\rho)\setminus \{ c_0\})$.
\end{LE}


Recall that a radial orthogonal linkage was defined in Definition \ref{def:radialorthogonal}
\begin{LE}\label{lem:crookedrooted}
Let $r, s$ be positive integers with
$s \ge 2r+7$.
Let $(G, \Omega)$ be a society and $\rho = (\Gamma, \zD, c_0)$ a cylindrical rendition of $(G, \Omega)$.
 Let $\zC= (C_1, C_2, \dots, C_s)$ be a nest of size $s$ in $\rho$.
Let $\zP$ be a radial orthogonal linkage of size $2r+6$.
If there exists a crooked transaction of  cardinality at least $r$ in $(G, \Omega)$, then there exists a crooked transaction
 of cardinality at least $r$
in $(G, \Omega)$ that is coterminal with $\zP$ up to level $C_{2r+7}$.
\end{LE}

\begin{proof}
Assume the lemma is false.
Let $\rho = (\Gamma, \zD, c_0)$ be the cylindrical rendition of
$(G, \Omega)$.
Let $H$ be the union of the subgraph $\bigcup_{P \in \zP} P \cup \bigcup_1^{2r+6} C_i$ along with the vertices of $V(\Omega)$ treated as a subgraph with no edges.  Let $X$ be the set of vertices $V(\zP) \cap V(\Omega)$.

Let $r' = max\{4, r\}$, and let $\zQ$ be a crooked transaction in $(G, \Omega)$ with $r \le |\zQ| \le r'$.  For all $Q \in \zQ$, let $Q^*$ be the $H$-unhooked interior.  We assume that $\zQ$ minimizes $\left| \bigcup E(Q^*) \right|$ among all crooked transactions which satisfy $r \le |\zQ| \le r'$.

Observe that for all $Q \in \zQ$, the unhooked interior $Q^*$ is not contained in $\bigcup(\sigma(c)\,:\,c \in C(\rho)\setminus \{ c_0\})$.  Otherwise, given that each $H$-hook is a subgraph of $H$, we would have $Q$ contained in $\bigcup(\sigma(c)\,:\,c \in C(\rho) \setminus \{ c_0\})$, contrary to Lemma \ref{lem:crookinrend}.  Specifically, the unhooked interior is not empty for all $Q \in \zQ$.

If for all $Q \in \zQ$ we have that $Q^*$ is disjoint from $V(C_{2r + 6})$, then again by Lemma \ref{lem:crookinrend}, the unhooked interior is contained in the inner graph of $C_{2r+6}$ for every element of $\zQ$, a contradiction.  Thus there exists at least one element of $\zQ$ for which the unhooked interior intersects $V(\Omega) \cup V(C_{2r +6})$.  Note that if $Q^*$ intersects $V(\Omega) \cup V(C_{2r+6})$, then $Q^*$ intersects $C_{i}$ for all $1 \le i \le 2r+6$, lest $Q^*$ be contained in $\bigcup(\sigma(c)\,:\,c \in C(\rho) \setminus \{ c_0\})$.
Let $|\zQ| = t$.  Label the elements of $\zQ$ as $Q_1, \dots, Q_{t}$ and the endpoints of $Q_i$ as $x_i, y_i$ such that there exists a permutation $p:[t] \rightarrow [t]$ such that $x_1, x_2, \dots, x_t y_{p(1)}, y_{p(2)}, \dots, y_{p(t)}$ occur in $\Omega$ in that order.

\begin{claim}
There exists a crooked transaction $\zQ'$ of order $t$ with elements $Q_1', \dots, Q_{t}'$ such that the endpoints of $Q_i'$ can be labeled $x_i'$ and $y_i'$ which satisfy the following:
\begin{itemize}
\item[(i)] $x_1', x_2', \dots, x_{t}', y_{p(1)}', y_{p(2)}', \dots, y_{p(t)}'$ occur in that order on $\Omega$;
\item[(ii)] the union of the unhooked interiors of the elements of $\zQ'$ is equal to
$\bigcup_{Q \in \zQ} Q^*$.
\end{itemize}
Moreover, there exist paths $R_1$ and $R_2$ which are pairwise disjoint, internally disjoint from $\zQ'$, and $R_i$ has one endpoint in $X \setminus V(\zQ')$ and one endpoint which is an internal vertex of an unhooked interior of an element of $\zQ'$ for $i = 1, 2$.
\end{claim}

\begin{cproof}
Let $Y$ be the set $\left( \bigcup_{Q \in \zQ} V(Q^*) \right) \cap V(H)$.  Let $J$ be the union over $1 \le i \le t$ of the $x_i$ and $y_i$ hook of $Q_i$.  The subgraph $J$ is a linkage in $H$ from $X$ to $Y$ of order at most $2t \le 2r+4$.  The set $Y$ contains a vertex of $V(C_i)$ for all $1 \le i \le 2r+6$, so it follows that there does not exist a separation in $H$ of order  $\le 2t+2$ separating $X$ and $Y$.

By Menger's theorem there exists an \emph{augmenting path} to $J$ (see \cite{Diestel}[Lemma 3.3.3]).  That is, there exists $k \ge 0$ and internally disjoint paths $R^1, R^2, \dots, R^{2k+1}$ each of length at least one which satisfy the following:
\begin{itemize}
\item $\bigcup_1^{2k+1} R^i$ is a path from $X \setminus V(J)$ to $Y\setminus V(J)$ in $H$;
\item $R^{2i}$ is a subpath of $J$ for $1 \le i \le k$ and $R^{2i+1}$ is internally disjoint from $J$ for $0 \le i \le k$.
\item Let $S$ be a component of $J$ and $1 \le i \le k$ such that $R^{2i}$ is a subpath of $S$ .  Let $x$ and $y$ be the endpoints of $S$ in $X$ and $Y$, respectively. Then the vertices $x, V(R^{2i+1} \cap R^{2i}), V(R^{2i-1} \cap V(R^{2i}), y$ occur on $S$ in that order when traversing $S$ from $x$ to $y$.  If $j$ is another value such that $i<j$ and $R^{2j}$ is a subpath of $S$, then the vertices $x, V(R^{2i}), V(R^{2j}), y$ occur in that order traversing $S$ from $x$ to $y$.
\end{itemize}
Pick $\zQ'$ satisfying (i) and (ii) and $R^1, \dots, R^{2k+1}$ as above to minimize $k$.

We claim that $k = 0$; assume not, and fix $i$ to be the index such that $R^2$ is a subpath of $Q_i'$.  Without loss of generality, assume that $R^2$ is contained in the $x_i'$ hook of $Q_i'$.  Let $a$ and $b$ be the endpoints of $R^2$, and assume that $a$ is also an endpoint of $R^1$.  The path $x_i'Q_i'a \cup R^1$ is an $\Omega$-path and the track of $x_i'Q_i'a \cup R^1$ divides the disk into two connected components, one of which, say $\Delta'$ does not contain $c_0$.  The subgraph of $G = \bigcup_{c \in C(\rho): c \subseteq \Delta'} \sigma(c)$ can intersect the unhooked interior of some element of $\zQ'$ only in the case when $a$ is the endpoint of the $x_i'$ hook of $Q_i'$ in $Y$.  Thus, the vertices of $\Omega$ mapped to the boundary of $\Delta'$ defines a segment of $\Omega$ containing both the vertices $x_i'$ and $V(R_1) \cap V(\Omega)$ which is disjoint from $V(J)$ except for the vertices $x_i'$ and $V(R_1) \cap V(\Omega)$.  Let $Q_i'' = R_1 \cup a Q_i' y_i'$.  We conclude that $(\zQ' - Q_i' )\cup Q_i''$ satisfies (i) and (ii).  Moreover, the path $x_i'Q_i'b \cup R^3 \cup \dots \cup R^{2k+1}$ violates our choice to minimize $k$.

This completes the proof that $k = 0$ and thus $R^1 = R_1$ is a path from $X \setminus V(J)$ to $Y \setminus V(J)$ with is disjoint from $J$ and therefore satisfies the outcome of the claim.  The desired path $R_2$ exists by the same augmenting path argument as above applied to $J \cup R_1$.
This completes the proof of the claim.
\end{cproof}

The lemma now follows easily from Lemma \ref{lem:crookedmcdonalds} applied to the crooked transaction $\zQ'$ and paths $R_1$ and $R_2$ to obtained a crooked transaction $\zQ''$ of order at most $t+1$.  In the case when $|\zQ''| = t+1$ and $t = r'\ge 4$, we can delete an element to find a transaction of order $r'$.  As both $R_1$ and $R_2$ are contained in $H$, the union of the edge sets of the unhooked interiors of the elements of $\zQ''$ will be a proper subset of $\bigcup_{Q \in \zQ} E(Q^*)$, contradicting our choice of $\zQ'$.
\end{proof}


The final result of this section shows how to effectively rotate a transaction coterminal using a given radial linkage to terminate on an alternate set of specified endpoints.

\begin{LE}\label{lem:transrotation2}
Let $p \ge 1$ be a positive integer.
Let $(G, \Omega)$ be a society and $\rho = (\Gamma, \zD, c_0)$ a cylindrical rendition of $(G, \Omega)$.
Let $\zC = (C_1, \dots, C_s)$ be a nest in $(G, \Omega)$ of cardinality $s \ge 4p+2$.
Let $\zP$ be a radial orthogonal linkage of order at least $4p $.
Let $\zQ$ be a linkage of $\Omega$-paths of  cardinality $p$ that is coterminal with $\zP$ up to level $C_1$,
and let $x_1,x_2,\ldots,x_{2p}$ be the endpoints of the paths in~$\zQ$, listed in their order in $\Omega$.
Let $  x'_1,x'_2,\ldots,x'_{2p}\in V(\Omega)$ be distinct vertices, each an endpoint of a member of~$\zP$,
 listed in their order in $\Omega$.
Then there exists a linkage $\zQ'$ of $\Omega$-paths such that
\begin{itemize}
\item $\zQ'$ is coterminal with $\zP$ up to level $C_{4p+2}$, and
\item for every member of $\zQ$ with endpoints  $x_i$ and $x_j$ there exists a member of $\zQ'$ with endpoints $x'_i$ and $x'_j$.
\end{itemize}
\end{LE}


\begin{proof}
For notational convenience let
$X':=\{x'_1,x'_2,\ldots,x'_{2p}\}$.
Let $i\in\{1,2,\ldots,2p\}$. The vertex $x_i$ is the endpoint of a unique path $P\in\zP$.
We define $y_i$ to be the vertex of $P\cap C_1$ that is closest to $x_i$ on $P$,
and we define $P_i=x_iPy_i$.
Let $Y:=\{y_1,y_2,\ldots,y_{2p}\}$.
We shall construct a linkage $\zL$ from $Y$ to $X'$
in a subgraph $H$ of $G$ in such a way that $\zL$ will connect $y_i$ with $x_i'$.
By replacing the paths $P_1,P_2, \ldots,P_{2p}$
by the linkage $\zL$ we will arrive at the desired
linkage $\zQ'$.

For each cycle $C_j$, $1 \le j \le 4p+1$, there are two natural cyclic orderings of the vertices $V(C_j)$.  We define the \emph{clockwise} order of $C_j$ to be the cyclic ordering of $V(C_j)$ such that traversing the vertices of $C_j$ in order we encounter $P_1, P_2, \dots, P_{2p}$ in that order.  The alternate cyclic ordering of $V(C_j)$ is the \emph{counterclockwise} ordering.

To construct $\zL$ we will define an appropriate subgraph $H$ and find the desired paths in $H$. To that end we first define a several linkages.  For $1 \le i \le 2p$, let $P_i'$ be the minimal subpath of $\zP$ with one endpoint equal to $x_i'$ and the other endpoint in $C_1$.  Let $R_i$ be the minimal subpath of $P_i$ with one endpoint equal to $y_i$ and the other endpoint in $C_2p+1$.  Let $S_i$ be the minimal subpath of $P_i'$ with one endpoint equal to $x_i'$ and the other endpoint in $C_{2p+2}$.  Finally, for $P \in \zP$, let $\overline{P}$ be the shortest subpath of $P$ with one endpoint in $C_2$ and the other endpoint in $C_{4p+1}$.  Define the subgraph $H'$ as
$$H' = \bigcup_1^{4p+1} C_i \cup \bigcup_1^{2p} (R_i \cup S_i) \cup \bigcup_{P \in \zP} \overline{P}.$$

We define the path $Q$. If $P_1' = P_1$, let $Q = P_1$.  Otherwise, define $Q$ as follows.
The graph $P_1\cup C_{2p+1}\cup P_1'$ includes two paths from $y_1$ to $x_1'$; let $Q$ denote the one
that intersects no more members of $\zP$ than the other one.  Without loss of generality, we assume that traversing $Q \cap C_{2p+1}$ from a vertex of $P_1'$ to a vertex of $P_1$, we  traverse the vertices in the clockwise order.

Let $H$ be obtained from $H'$ by deleting the following sets of edges:
\begin{itemize}
\item for $i=2,3,\ldots,4p+1$, the edge of $C_i$ that is incident with $C_i\cap Q$ counterclockwise,
\item for every $P\in\zP$ such that $\overline{P}$ intersects $Q$, the edge of $\overline{P}$ incident with $Q$ that lies
on $\overline{P}$ between $C_{2p+1}$ and $C_{2p+2}$.
\end{itemize}

Let $F$ denote the set of edges deleted according to the second rule above. By the choice of $Q$, at least $2p$ distinct paths $P \in \zP$ have $\overline{P}$ disjoint from $F$.
We claim that $H$  has a linkage of cardinality $2p$  from $Y$ to $X'$.
To prove the claim it suffices to show, by Menger's Theorem, that for every  set $Z\subseteq V(H)$ of cardinality
at most $2p-1$ there exists a path in $H-Z$ from $Y$ to $X'$.
To that end let $Z$ be such a set.  Note that by construction, the elements of $\{S_i \cup R_i : 1 \le i \le 2p\}$ are pairwise disjoint.  Thus, there exists an index $i$ such that $S_i \cup R_i$ is disjoint from $Z$.  There exists $P \in \zP$ such that $\overline{P}$ is disjoint from both the edge set $F$ and vertex set $Z$.  Finally, there exist indices $j$ and $j'$ such that $2 \le j \le 2p+1$ and $2p + 2 \le j' \le 4p+1$ such that $C_j$ and $C_{j'}$ are both disjoint from $Z$.  It follows that $R_i \cup S_i \cup C_j \cup C_{j'} \cup \overline{P}$ has the desired path in $H$ avoiding $Z$.  This proves our claim that $H$  has  a linkage $\zL$ of cardinality $2p$ from $Y$ to $X'$.

It follows from the definition of $F$ that the linkage $\zL$ connects $y_i$ with $x_i'$,
and hence it gives rise to the desired linkage $\zQ'$, as indicated earlier.
\end{proof} 
\section{Finding a planar strip}\label{sec:planarstrip}

In this section, we focus on societies which contain a large transaction of pairwise crossing or pairwise non-crossing paths, and consider how the rest of the graph attaches to the transaction.  We show that if we exclude a clique minor and one other specific type of transaction, then we can find a large, nearly planar subgraph containing a large portion of the transaction.   The results in this section are a new ingredient that did not appear in the original proof of Robertson and Seymour and constitutes one of the key ideas in our proof.

To simplify the notation going forward, we give the following definitions.  A transaction $\zP$ in a society $(G, \Omega)$ where every pair of elements cross is called a \emph{crosscap transaction}.  If a transaction $\zP$ is either a crosscap or planar transaction, we say it is \emph{monotone}.

This section will rely heavily on results from \cite{flatwall}.  We reintroduce some of the notation of \cite{flatwall} which we will use going forward.

\begin{DEF} Let $r,s \ge 2$ be positive integers, let $M$ be a graph, and let $P_1, P_2,\dots ,P_r, Q_1, Q_2,\dots , Q_s$ be paths in $M$ such that the following conditions hold for all $ i = 1,2,\dots ,r$ and $j = 1,2,\dots ,s$:
\begin{itemize}
\item[(1)] $P_1, P_2, \dots ,P_r$ are pairwise vertex disjoint, $Q_1, Q_2, \dots ,Q_s$ are pairwise vertex disjoint, and $M = P_1 \cup P_2 \cup \dots \cup P_r \cup Q_1 \cup Q_2 \cup \dots \cup Q_s$;
\item[(2)] $P_i \cap Q_j$ is a path, and if $i \in \{ 1,s\}$ or $j \in \{1,r\}$ or both, then $P_i \cap Q_j$ has exactly one vertex;
\item[(3)] $P_i$ has one endpoint in $Q_1$ and the other endpoint in $Q_s$, and when traversing $P_i$ the paths $Q_1, Q_2 ,\dots , Q_s$ are encountered in the order listed;
\item[(4)] $Q_j$ has one endpoint in $P_1$ and the other endpoint in $P_r$, and when traversing $Q_j$ the paths $P_1 , P_2 ,\dots , P_r$ are encountered in the order listed.
\end{itemize}
In those circumstances we say that $M$ is an \emph{$r \times s$ mesh}. We will refer to $P_1 , P_2 , \dots , P_r$ as \emph{horizontal paths} and to $Q_1 , Q_2 , \dots , Q_s$ as vertical paths. Thus every $r \times s$ grid is an $r \times s$ mesh, and, conversely, every planar graph obtained from an $r \times s$ grid by subdividing edges and splitting vertices is an $r \times s$ mesh. In particular, every $r$-wall is an $r \times r$-mesh.
\end{DEF}

The definition of a clique minor grasped by a wall extends to meshes almost verbatim.  Let $r, s, t$ be positive integers, let $M$ be an $r \times s$ mesh and let $X_1, \dots, X_t$ form the branch sets of a model of a $K_t$ minor.  The model is \emph{grasped} by the mesh $M$ if $t \le s, r$ and for all $1 \le l \le t$, there exist distinct indices $i_1, i_2, \dots, i_t$ and distinct indices $j_1, j_2, \dots, j_t$ such that $X_l$ contains $V(P_{i_k})\cap V(Q_{j_k})$ for $1 \le k \le t$.

\begin{DEF}
Let $M$ be a mesh with horizontal paths $P_1,P_2,\ldots,P_r$ and vertical paths $Q_1,Q_2,\ldots,Q_s$.  Let $\zP$ be the set of horizontal paths and $\zQ$ the set of vertical paths.  Let $x$ be a vertex of $M$.  For a positive index $i$, we say $x$ is \emph{in the vicinity of $P_i$} if either $x \in V(P_i)$ or there exists a path from $x$ to a vertex of $P_i$ which intersects $V(\zP)$ in exactly one vertex.  Similarly, we say that $x$ is in the \emph{vicinity of $Q_j$} if either $x \in V(Q_j)$ or there exists a path in $M$ from $x$ to $V(Q_j)$ intersecting $V(\zQ)$ in exactly one vertex.
\end{DEF}

Note that every vertex in a given mesh is in the vicinity of at least one and at most two horizontal (or vertical) paths.

\begin{DEF}
Let $M$ be a mesh with  horizontal paths $P_1,P_2,\ldots,P_r$ and vertical paths $Q_1,Q_2,\ldots,Q_s$.
Let $d\ge 1$.
We say that a vertex $v\in V(M)$ {\em is at distance at least $d$ from the boundary} if there does not exist an $i$ such that $i \le d$ or $r-d+1 \le i \le r$ for which
$v$ is in the vicinity of a horizontal path $P_i$ and there does not exist $j$ such that $1 \le j \le d$ or $s - d + 1 \le j \le s$ for which $v$ is in the vicinity of the vertical path $Q_j$.
\end{DEF}

The definition of distance at least $d$ from the boundary of a mesh is motivated by the following.  Let $M$ be an $r \times s$ mesh wish horizontal paths $P_1, \dots, P_r$ and vertical paths $Q_1, \dots, Q_s$.  If we let $dist$ be the distance function defined on $V(M)$ in Section 3 of \cite{flatwall}, then a vertex $x \in V(M)$ is at distance $d$ from the boundary implies that for all $z \in V(P_1 \cup P_r \cup Q_1 \cup Q_s)$, $dist(x, z) \ge d$.  This will specifically allow us to apply the results of \cite{flatwall} in the coming arguments without the technical details required to define the distance function used in \cite{flatwall}.

In the next lemma and its proof, by an {\em interval} we mean a set of consecutive
positive integers and the \emph{length} of the interval is simply the order.

\begin{LE}
\label{l:strips}
Let $k,l,p,r,s,t\ge1$ be integers such that
$$s\ge 49152p^{24}(15p^{2} +l)+kl.$$
Then for every graph $G$ and every $r\times s$ mesh $M$ in $G$ with vertical paths $Q_1,Q_2,\ldots,Q_s$,
if G has no $K_p$ minor grasped by $M$, then there exist a set $A\subseteq V(G)$ of size at most
$ 12288(p(p-1))^{12}$ and a set $I\subseteq\{1,2,\ldots,s\}$
such  that
the following conditions hold:
\begin{itemize}
\item the set $I$ is a union of $k$ disjoint intervals of length $l$;
\item if $i\in I$, then $Q_i$ is disjoint from $A$;
\item if $P$ is an $M$-path in $G-A$ with endpoints $x$ and $y$ such that $x$ is at distance at least $10p(p-1)$ from the boundary and $i$ and $j$ are indices such that $x$ is in the vicinity of $Q_i$ and $y$ is in the vicinity of $Q_j$ with $|j-i|> 10p(p-1)+1$,
then $i\not\in  I$;
\end{itemize}
\end{LE}

\begin{proof}
By~\cite[Lemma~4.6]{flatwall}
we may assume that there
exist sets $A\subseteq V(G)$ and $Z\subseteq \{1,2,\dots,s\}$ such that

\begin{itemize}
\item[($*$)]
$|A|\le 12288(p(p-1))^{12}$, $|Z|\le 3\cdot  12288(p(p-1))^{12}$,
and if $x$ and $y$ are the endpoints of an $M$-path in $G-A$ and $i$, $j$ such that $x$ is in the vicinity of $Q_i$ and $y$ is in the vicinity of $Q_j$, then one of the following holds:
\begin{itemize}
\item  the vertex $x$ is not at distance at least $10p(p-1)$ from the boundary, or
\item $|i-j|\le10p(p-1) +1$, or
\item there exists $z \in Z$ such that $|i-z| \le 10p(p-1)$.
\end{itemize}
\end{itemize}

Let us say that an integer $i\in\{1,2,\ldots,s\}$ is {\em good} if $Q_i$ is disjoint from $A$ and
there does not exist an $M$-path $P$ in $G-A$ and an integer $j$ such that one endpoint $x$ of $P$ is in the vicinity of $Q_i$ and the other endpoint $y$ is in the vicinity of $Q_j$, such that
$|j-i|>10p(p-1)+1$ and $x$ is at distance at least $10p(p-1)$ from the boundary.
We say that an interval $I$ is {\em good} if every member of $I$ is good.
Thus it suffices to show that there are $k$ pairwise disjoint good intervals.

Let $S$ be the set of all integers  $i\in\{1,2,\ldots,s\}$ such that $Q_i$ is disjoint from
$A$ and $i$ is not within $10p(p-1) $ from some member of $Z$.
By ($*$), $S$ is good.
Clearly $|S|\ge s-|A|-(20p(p-1)+1)|Z|$ and $S$ breaks into at most $|Z|+|A|+1$ intervals.
Every subinterval of $S$ of length $L$ gives rise to $\lfloor L/l\rfloor$ disjoint good intervals of length $l$.
By ($*$) and the hypothesis on $s$ we have $s\ge 20p(p-1)|Z|+ (|A|+|Z|)l +kl $; thus
$|S|\ge ( k-1)l + (|Z|+|A|+1)(l-1) + 1 $, and hence  there are
$k$  pairwise disjoint good intervals, as desired.
\end{proof}

An obstruction to the presence of the large, nearly planar subgraph for which we are searching is a structure which we will call $k$-nested crosses transaction. In the analysis, we will also need to consider the reversed version which we call a twisted $k$-nested crosses transaction.
\begin{DEF}
Let $(G, \Omega)$ be a society and $k \ge 1$ a positive integer.  A \emph{$k$-nested crosses transaction} is a transaction $\zP$ of order $2k$ such that the elements of $\zP$ can be labeled $P_1, \dots, P_k, Q_1, \dots, Q_k$ such that
\begin{itemize}
\item for all $1 \le i < j \le k$, $P_i$ and $P_j$ do not cross and $Q_i$ and $Q_j$ do not cross;
\item for all $1 \le i \le k$, $1 \le j \le k$, $P_i$ and $Q_j$ cross if and only if $i = j$.
\end{itemize}
\end{DEF}

\begin{DEF}
Let $(G, \Omega)$ be a society and $k \ge 1$ a positive integer.  A \emph{twisted $k$-nested crosses transaction} is a transaction $\zP$ of order $2k$ such that the elements of $\zP$ can be labeled $P_1, \dots, P_k, Q_1, \allowbreak \dots, Q_k$ such that
\begin{itemize}
\item for all $1 \le i < j \le k$, $P_i$ and $P_j$ cross and $Q_i$ and $Q_j$ cross;
\item for all $1 \le i \le k$, $1 \le j \le k$, $P_i$ and $Q_j$ do not cross if and only if $i = j$.
\end{itemize}
\end{DEF}

\begin{DEF}
\label{def:planarstrip}
Let $\zP$ be a planar transaction of cardinality at least two in a society $(G, \Omega)$.
We wish  to define a new society, which we will call the ``$\zP$-strip society of $(G, \Omega)$".
The members of $\zP$ may be denoted by $P_1,P_2,\ldots,P_n$ and the endpoints of $P_i$ by
$a_i$ and $b_i$ in such a way that $a_1,a_2,\ldots,a_n,b_n,b_{n-1},\ldots,b_1$  occur in $\Omega$
in the order listed.  We refer to $P_1$ and $P_n$ as the \emph{external elements of $\zP$}.

Let $H$ denote the subgraph of $G$ obtained from the union of the elements of $\zP$ by adding the elements of  $V(\Omega)$
as isolated vertices.
Let $H'$ be the subgraph of $H$ consisting of $\zP$ and all the vertices  $ a_{1}\Omega a_{n}\cup b_{n}\Omega b_{1}$.
Let us consider all $H$-bridges of $G$ with at least one attachment in $V(H')\setminus V(P_1\cup P_{n})$, and for each such
$H$-bridge $B$ let $B'$  denote the graph obtained from $B$ by deleting all attachments that do not belong to  $V(H')$.
Finally, let $G_1$ denote the union of $H'$ and all the graphs $B'$ as above.
Let the cyclic permutation $\Omega_1$ be defined by saying that
$V(\Omega_1)=a_{1}\Omega a_{n}\cup b_{n}\Omega b_{1}$
and that the order on $\Omega_1$ is that induced by $\Omega$.
Thus $(G_1,\Omega_1)$ is a society, and we call it the {\em $\zP$-strip society of $(G, \Omega)$}.
We say that $P_1$ and $P_n$  are the {\em boundary paths} of the $\zP$-strip society  $(G_1,\Omega_1)$.
We say that the $\zP$-strip society of $(G, \Omega)$ is {\em isolated in $G$} if no edge of $G$ has one
endpoint in $V(G_1)\setminus V(P_1 \cup P_n)$ and the other endpoint in $V(G)\setminus V(G_1)$.
Thus $(G_1,\Omega_1)$ is isolated if and only if every $H$-bridge of $G$ with at least one attachment in $V(H')\setminus V(P_1\cup P_{n})$
has all its attachments in $V(H')$.
\end{DEF}

\begin{DEF}
\label{def:crosscapstrip}
We now define a $\zP$-strip society of $(G, \Omega)$ for a crosscap  transaction  $\zP$.  For a given planar transaction $\zP'$, there are unique minimal (by containment) segments $U_1, U_2$ such that $\zP'$ is a $U_1$ to $U_2$ linkage; this is no longer the case for the crosscap transaction $\zP$.  Thus, when defining the $\zP$-strip society for a crosscap transaction, we will define it with respect to a pair $(X_1, X_2)$ of segments such that $\zP$ is an $X_1$ to $X_2$ linkage.

Assume that $X_1$ and $X_2$ are segments of $\Omega$ which are pairwise disjoint such that $\zP$ is an $X_1$ to $X_2$ linkage.
The members of $\zP$ may be denoted by $P_1,P_2,\ldots,P_n$ and the endpoints of $P_i$ by
$a_i$ and $b_i$ in such a way that $a_1,a_2,\ldots,a_n,b_1,b_{2},\ldots,b_n$  occur in $\Omega$
in the order listed and $a_i \in X_1$ and $b_i \in X_2$ for all $i$.  Again, we refer to $P_1$ and $P_n$ as the \emph{external elements of $zP$}.

Let $H$ denote the subgraph of $G$ obtained from the union of the elements of $\zP$ by adding the elements of  $V(\Omega)$
as isolated vertices.
Let $H'$ be the subgraph of $H$ consisting of $\zP$ and all the vertices  $ a_{1}\Omega a_{n}\cup b_{1}\Omega b_{n}$.
We define $G_1$ as in  Definition~\ref{def:planarstrip} and
 the cyclic permutation $\Omega_1$  by saying that
$V(\Omega_1)=a_{1}\Omega a_{n}\cup b_{1}\Omega b_{n}$
and that the order on $\Omega_1$ is obtained by following $a_{1}\Omega a_{n}$ in the order given by $\Omega$,
and then following $b_{1}\Omega b_{n}$ in the {\em reverse} order from  the one given by $\Omega$.
We call $(G_1, \Omega_1)$ the {\em $\zP$-strip society of $(G, \Omega)$ with respect to $(X_1, X_2)$}.
We define the boundary paths and isolation in the same way as in  Definition~\ref{def:planarstrip}.
When there can be no confusion as to the choice of segments $(X_1, X_2)$, we will omit their reference.
\end{DEF}

Recall that a society is rural if it has a vortex-free rendition.

\begin{LE}\label{lem:inhereitedrurality}
Let $(G, \Omega)$ be a society and
 $\zQ$ a monotone transaction in $(G, \Omega)$.
Let $X^1,X^2$ be disjoint segments of $\Omega$ such that $\zQ$ is a linkage from $X^1$ to $X^2$.  For $i = 1, 2$, let $\zQ_i$ be a subset of $\zQ$ and let $X^j_i$ be the minimal segment contained in $X^j$ containing all the endpoints of $\zQ_i$ for $j = 1, 2$.  If the $\zQ$-strip society with respect to $(X^1, X^2)$ is isolated and rural, then the $\zQ_i$-strip society with respect to $(X^1_i, X^2_i)$ is isolated and rural for $i = 1, 2$.  Moreover, if $X^1_1 \cap X^1_2  = \emptyset$, then the $\zQ_1$-strip society is disjoint from the $\zQ_2$-strip society.
\end{LE}

\begin{proof}
Let the elements of $\zQ$ be enumerated $Q_1, \dots, Q_l$ for some positive integer $l$ such that traversing the elements of $X^1$, we encounter the endpoints of $Q_1, Q_2, \dots, Q_l$ in that order.  Let $i \in \{1, 2\}$.  Let $J_i$ be the subgraph of $G$ formed by the union of $X^1_i$ and $X^2_i$ treated as isolated vertices with $\bigcup_{Q \in \zQ_i} Q$.  Without loss of generality, assume that $i = 1$, and let $Q$ and $Q'$ be the first and last paths, respectively, we encounter in $\zQ_1$ traversing $X^1_1$ according to $\Omega$.

Let $(G', \Omega')$ be the $\zQ$-strip society with respect to $(X^1, X^2)$.  Fix a vortex-free rendition $\rho$ of $(G', \Omega')$ in the disk $\Delta$, and let $T$, $T'$ be the track of $Q$ and $Q'$ in $\Delta$, respectively.  $\Delta \setminus (T \cup T')$ has exactly three connected components.  Let $\Delta'$ be the connected component whose boundary contains both $T$ and $T'$, and let $G'' = (Q \cup Q') \cup \bigcup_{c \in C(\rho): c \subseteq \Delta'} \sigma(c)$ and let $\Omega''$ be the cyclic ordering on $X^1_1 \cup X^2_1$ induced by $\Omega'$.  Given the rendition of $(G', \Omega')$ and the fact that it is isolated, we see that no edge of $G$ has an endpoint in $V(G'') \setminus V(Q \cup Q')$ and an endpoint in $V(G) \setminus V(G'')$.  Thus, every $J_1$-bridge with an attachment in $V(J_1) \setminus V(Q \cup Q')$ must be contained in $G''$.  This implies both that the $\zQ_1$-strip society is a subgraph of $G''$ and also that it is isolated.  To see that the $\zQ_1$-strip society is rural, observe that $(G'', \Omega'')$ is rural and a rendition can be restricted to the $\zQ_1$-strip society to show that the $\zQ_1$-strip society is rural as well.

Assume to reach a contradiction that the $\zQ_1$-strip society intersects the $\zQ_2$-strip society.  It follows that there is a $J_1$-bridge in $G$ with an attachment in $V(J_1) \setminus V(Q \cup Q')$ which also contains a vertex of $X^1_2 \cup X^2_2 \cup V(\zQ_2)$.  But such an bridge would violate the fact that the $\zQ_1$-strip society is isolated. This contradiction completes the proof of the lemma.
\end{proof}

Recall that a radial linkage orthogonal to a nest was defined in Definition~\ref{def:radialorthogonal}.
Let $(G, \Omega)$ be a society and $\rho = (\Gamma, \zD, c_0)$ a cylindrical rendition of $(G, \Omega)$.
Let $\zC = (C_1, \dots, C_{s})$ be a nest  in $\rho$ and let $\zQ$ be a transaction in $(G, \Omega)$.
We say that {\em $\zQ$ is orthogonal to $\zC$} if for all $C\in\zC$ and all $Q \in \zQ$ the graph $C\cap Q$
has exactly two components.
We say that $\zQ$ is \emph{unexposed in $\rho$} if for every $Q \in \zQ$, $Q$ has at least one edge in $\sigma(c_0)$, or simply unexposed when there can be no ambiguity.   Note that any crooked transaction is necessarily unexposed.
If $\zQ$ is unexposed, then every element of $\zQ$ contains exactly two distinct minimal subpaths which each have one endpoint in $V(\Omega)$ and the other endpoint in $V(\sigma(c_0))$ (they will also be necessarily disjoint).  Let $\zP$ be the union of all such minimal subpaths over the elements of $\zQ$.  We call $\zP$ the \emph{radial linkage truncation of $\zQ$} and note that it follows that the radial linkage truncation is a radial linkage orthogonal to the nest $\zC$.
If $\zQ'$ is a transaction in $(G, \Omega)$, we say that it is \emph{coterminal with $\zQ$ up to level $i$} if $\zQ'$ is coterminal with the radial linkage truncation $\zP$ of $\zQ$ up to level $i$.

Finally, we extend the definition of a model of a clique grasped by the union of a nest and a radial linkage.

\begin{DEF}
Let $p, r, s$ be integers with $r, s \ge p$.  Let $(G, \Omega)$ be a society.  Let $\zC = (C_1, \dots, C_s)$ be an $s$-nest in a cylindrical rendition of $(G, \Omega)$, and let $\zP$ be a radial linkage of order $r$.  Let the elements of $\zP$ be labeled $P_1, \dots, P_r$.  We say that $\zC \cup \zP$ \emph{grasps a $K_p$ minor} if there exists the branch sets $X_1, \dots, X_p$ of a model of a $K_p$ minor such that for all $l$, $1 \le l \le p$, there exist distinct indices $i_1, i_2, \dots, i_p$ and distinct indices $j_1, j_2, \dots, j_p$ such that $X_l$ contains $C_{i_k}\cap P_{j_k}$ for all $1 \le k \le p$.
\end{DEF}

\begin{theorem}
\label{thm:RTplanarstrip}
Let $l,k \ge 1$ and $p\ge2$ be integers, $\alpha:=12288(p(p-1))^{12}$ and $$l':=(49152p^{24}+k)(55p^{2} +l).$$
Let $(G, \Omega)$ be a society and $\rho = (\Gamma, \zD, c_0)$ a cylindrical rendition of $(G, \Omega)$.
Let $\zC = (C_1, \dots, C_{s})$ be a nest of cardinality $s \ge 2\alpha + 20p(p-1)+4$ in $\rho$ and
 $\zQ$  an unexposed monotone transaction in $(G, \Omega)$ of order $l'$ which is
orthogonal to  $\zC$.
Let $X_1,X_2$ be disjoint segments of $\Omega$ such that $\zQ$ is a linkage from $X_1$ to $X_2$.
Then one of the following holds:
\begin{itemize}
\item[(i)] $G$ contains a $K_p$ minor grasped by $\zQ \cup \zC$;
\item[(ii)] there exists a transaction $\zQ'$ from $X_1$ to $X_2$ that is coterminal with $\zQ$
up to level $C_{\alpha+8}$, such that $\zQ'$ is either a $k$-nested crosses transaction or a twisted $k$-nested crosses transaction;
\item[(iii)] there exists a set $Z \subseteq V(G)$ with $|Z| \le \alpha$
and a set $\zQ'\subseteq\zQ$ of size at least $l$
such that the $\zQ'$-strip in $(G - Z, \Omega - Z)$ is isolated and rural.  Moreover, $Z$ is disjoint from $\zQ'$ and $Z$ is disjoint from the outer graph of $C_{2\alpha+20p(p-1) + 4}$.
\end{itemize}
\end{theorem}

\begin{proof}
Assume the claim is false.  We begin by labeling the elements of the various linkages in the statement in order to define an appropriate mesh to work with.  Note, we will have to work with the inner and outer graph defined by the nest $\zC$ which were defined in Section 2.4.

Let $X^1 \subseteq X_1$ and $X^2\subseteq X_2$ be the two disjoint minimal segments of $\Omega$ such  that every path
in $\zQ$ has one endpoint in $X^1$ and the other point in $X^2$.
Let $\zQ=\{Q_1,Q_2,\ldots,Q_{l'}\}$, where the paths are numbered in such a way that
their endpoints appear in $X_1$ (and therefore also in $X_2$) in order.  Let $\zP$ be the radial linkage truncation of $\zQ$ and let $\zP_i$ be the subset of $\zP$ defined by elements with an endpoint in $X^i$ for $i = 1,2$.

For $\alpha + 2 \le i \le \alpha + 10p(p-1) +2$, let $R_i$ be the minimal subpath of $C_i$ with one endpoint in $Q_1$, the other endpoint in $Q_{l'}$, and disjoint from $\zP_2$.  For $\alpha + 10p(p-1) + 3 \le i \le \alpha + 20p(p-1)+3$, let $R_i$ be the minimal subpath of $C_i$ with one endpoint in $Q_1$, the other endpoint in $Q_{l'}$, and disjoint from $\zP_1$.  Let $\zR$ be the union of $R_i$ for $\alpha + 2 \le i \le \alpha + 20p(p-1) + 3$.
Let $M$ be the graph
obtained from $\zQ\cup \zR$ by repeatedly deleting vertices of degree at most one.
Then we can regard $M$ as a $20p(p-1) + 2\times l'$-mesh, where each  horizontal path is a subpath
of a member of $\zC$ and each vertical path is a subpath of a distinct member of $\zQ$.
Note that by construction, each horizontal path of $M$ is a subpath of a distinct cycle of $\zC$.  Thus, a model of $K_p$ which grasps $M$ will also grasp $\zQ \cup \zC$.

Let $l_1:=l+40p(p-1)+2$. Then $l'\ge 49152p^{24}(15p^{2} +l_1)+kl_1$.
By Lemma~\ref{l:strips} applied to the integers $k,l_1,p$, the graph $G$ and mesh $M$,
either the graph $G$  has a $K_p$ minor grasped by $M$, in which case $(i)$ holds,
or there exist sets $Z \subseteq V(G)$ with $|Z|\le\alpha$
and  $I\subseteq\{1,2,\ldots,l'\}$
such that the conclusions of Lemma~\ref{l:strips} hold with $l$ replaced by $l_1$.  Assume that we pick $Z$ minimal by inclusion so that the outcomes of Lemma \ref{l:strips} hold.

\begin{claim}\label{cl:Zdisjoint}
The set $Z$ is disjoint from the outer graph of $C_{2\alpha+20p(p-1) + 4}$.
\end{claim}
\begin{cproof}
By construction, the subgraph $M$ is disjoint from $C_i$ for $i \ge \alpha + 20p(p-1) + 4$.  Assume, to reach a contradiction, that $Z$ contains a vertex $z$ contained in the outergraph of $C_{2\alpha + 20p(p-1) + 4}$.  Given the bound on $|Z|$, there exists an index $i$, $\alpha + 20p(p-1) + 4 \le i \le 2\alpha + 20p(p-1) + 3$ such that $C_i$ is disjoint from $Z$.

By the minimality of $Z$, there exists an $M$-path, call it $P$, with violates the conditions in the third outcome of Lemma \ref{l:strips} such that $P$ intersects $Z$ only in the vertex $z$.  As $P$ has both endpoints contained in the inner graph of $C_{\alpha + 20p(p-1) + 4}$, the cylindrical rendition of $\rho$ ensures that each component of $P-z$ intersects every cycle $C_j$ for $\alpha + 20p(p-1) + 4 \le j \le 2\alpha + 20p(p-1) + 3$.  Thus, $(P-z) \cup C_i$ contains an $M$-path with the same endpoints as $P$ which is disjoint from $Z$, a contradiction, proving the claim.
\end{cproof}

Note that specifically, $Z$ is disjoint from $V(\Omega)$.

Let $I_1',I_2',\ldots,I_k'$ be the intervals of length $l_1$ with union $I$ given by the application of Lemma \ref{l:strips}.
For $i\in\{1,2,\ldots,k\}$
let $I_i$ be obtained from $I_i'$ by deleting the first and last $10p(p-1)$ elements.
Thus $|I_i|=l+20p(p-1)+2$.

For $1 \le i \le k$, let $\zQ_i$ be the set $\{Q_j: j \in I_i\}$.  For $1 \le i \le k$, and $1 \le j \le 2$, let $X^j_i$ be the minimum subset of $X^j$ forming a segment of $\Omega$ containing an endpoint of every element of $\zQ_i$.  Let $(H_i, \Omega_i)$ be the $\zQ_i$-strip society of $(G - Z, \Omega )$.  Note that both $\zQ_i$ and $X^j_i$ are disjoint from $Z$.

For $1 \le i \le k$, we define societies $(H_i', \Omega_i')$ which are closely related to the strip society $(H_i, \Omega_i)$.  Fix $i$, $1 \le i \le k$.  Let $J$ be the graph consisting of the union of  $V(\Omega)$ treated as isolated vertices and the $\bigcup_{Q\in \zQ} Q - Z$.  For every $J$-bridge $B$ in $G-Z$, let $B'$ be the subgraph obtained by deleting all attachments of $B$ not in $V(\zQ_i) \cup X^1_i \cup X^2_i$.  Let $c$ be the smallest value in $I_i$.  Let $H_i'$ be the union of $J[V(\zQ_i) \cup X^1_i \cup X^2_i]$ and $B'$ for every $J$-bridge $B$ in $G-Z$ with at least one attachment in $(V(\zQ_i) \cup X^1_i \cup X^2_i) \setminus V(Q_c \cup Q_{c+l+20p(p-1) + 1})$ (that is, an attachment not in the first or last element of $\zQ_i$.)
Let $\Omega_i' = \Omega_i$.  Note that $H_i'$ is a subgraph of $H_i$ the $\zQ_i$-strip society, but it will typically be a proper subgraph as multiple $J$-bridges may merge together when considered as bridges for the subgraph of $J$ formed by $\zQ_i$ and $V(\Omega)$.

\begin{claim}\label{cl:stripdisjoint}
The subgraphs $H_i'$ are pairwise vertex disjoint.
\end{claim}

\begin{cproof}
Assume the claim is false and let $i, i'$ be indices such that $H_i'$ intersects $H_{i'}'$.  Thus, there exists a path $R$ in $G-Z$ with one endpoint in $X^1_i \cup X^2_i \cup V(\zQ_i)$, the other endpoint in $X^1_{i'} \cup X^2_{i'} \cup V(\zQ_{i'})$ and no internal vertex in $V(\Omega) \cup V(\zQ)$.

Fix $j$ to be an index such that $C_j$ is disjoint from $Z$.  Let $J$ be the subgraph formed by the union of $C_j$ and the inner graph of $C_j$.  We fix the path $R$ to minimize the number of edges not contained in $E(J)$.  Let $\bar{R} = R \cap J$.

Consider a maximal subpath $T$ of $R$ with all internal vertices contained in $V(R) \setminus V(\bar{R})$.  There are two possible cases given the cylindrical rendition and the fact that $R$ is internally disjoint from $\zQ$: either $T$ has one endpoint in $\{x_i, x_{i'}\}$ and one endpoint in $C_j$ or, alternatively, $T$ has both endpoints in a component of $C_j - V(\zQ)$.  By replacing any such subpath $T$ with an appropriately chosen subpath of $C_j$, it follows that there exists a path $R'$ contained in $C_j \cup \bar{R}$ with endpoints $x_i'$ and $x_{i'}'$ such that
\begin{itemize}
\item $R'$ is internally disjoint from $\zQ$ and,
\item $x_i'$ is contained in $V(\zQ_i)$ and $x_{i'}'$ is contained in $V(\zQ_{i'}')$.
\end{itemize}
Note that if $x_i$ is in an element of $\zQ$, we can choose $x_i'$ to be in the same element of $\zQ$ as $x_i$.  If $x_i \in V(\Omega)$, we can choose $x_i'$ to be in either of the two elements of $\zQ$ closest to $x_i$ on $\Omega$.

We conclude, by the choice of $R$, that $R$ is a subgraph of $J$.  Thus, both endpoints $x_i'$ and $x_{i'}'$ are contained in $V(\zQ) \cap V(J)$ and thus both contained in $M$.  The path $R$ therefore violates the third condition of Lemma \ref{l:strips}, a contradiction and completing the proof of the claim.
\end{cproof}

\begin{claim}
There exists an index $i$, $1 \le i \le k$, such that $(H_i', \Omega_i')$ is rural.
\end{claim}

\begin{cproof}
Assume not and let $i$ be a positive integer, $i \le k$.  By Theorem \ref{thm:crossreduct} there exists a cross in $(H_i', \Omega_i')$.  We will define the graph $H_i''$ which contains $H_i'$ as a subgraph and using a cross in $(H_i', \Omega_i')$, we show that there exist disjoint paths $R_1'$, $R_2'$ in $H_i''$ which are coterminal with $\zP$ up to level $\alpha + 8$ such that: (1) each path $R_1'$ and $R_2'$ has one endpoint in $X^1$ and one endpoint in $X^2$ and (2)  $R_1', R_2'$ form a cross in the case $\zQ$ is planar, and $R_1', R_2'$ do not cross in the case $\zQ$ is a crosscap transaction.  Moreover, the subgraphs $H_i''$ will be pairwise disjoint for distinct indices $i$.  This suffices to prove the claim, as the union of $R_1', R_2'$ over all $1 \le i \le k$ forms a $k$-nested crosses or twisted $k$-nested crosses transaction.

Fix $i$, $1 \le i \le k$, and let $a$ be the minimal value such that there are at least five distinct indices $j$, $1 \le j \le a$ such that $C_j$ is disjoint from $Z$. First, we show that there is a cross in $(H_i', \Omega_i')$ which is coterminal with $\zP$ up to level $a+1$.  The argument here closely follows the proof of Lemma \ref{lem:crookedrooted}.

Let $R_1, R_2$ form a cross in $(H_i', \Omega_i')$.  Fix $j \in \{1, 2\}$.  We claim that the path $R_j$ contains an edge of $E(\sigma(c_0))$.  In the case when $R_1$ and $R_2$ form a cross in $(G, \Omega)$, this follows from Lemma \ref{l:crookinrend}.  If $R_1, R_2$ is not a cross in $(G, \Omega)$, it must be the case that $\zQ$ is a crosscap transaction and that each of $R_1$ and $R_2$ has one endpoint in $X^1_i$ and the other endpoint in $X^2_i$.  However, every path in $G$ from $X^1_i$ to $X^2_i$ which is disjoint from $E(\sigma(c_0))$ must intersect an element of $\zQ \setminus \zQ_i'$.  Given the definition of $H_i'$, we conclude that each of the paths $R_1$ and $R_2$ must contain an edge of $\sigma(c_0)$.

Define $L$ to be the subgraph of $H_i'$ given by $\left(\bigcup_1^{a} C_i \cup \bigcup_{P \in \zP} P\right) \cap H_i'$. For $i = 1, 2$, let $R_i^*$ be the $L$-unhooked interior of $R_i$.  Note that by the previous paragraph, the $L$-unhooked interior of each $R_i$ is nonempty.  We pick $R_1, R_2$ to be a cross which minimizes $| E(R_1^*) \cup E(R_2^*)|$.

If $R_j^*$ is disjoint from $C_{a+1}$ for $j = 1, 2$, then $R_1$, $R_2$ are coterminal with $\zP$ up to level $a+1$ by the construction of $L$.  Assume, instead, that $R_j^*$ intersects $C_{a+1}$ for some $j$.  It follows that $R^*_j$ intersects $C_l$ for all $l \le a$ as $R_j^*$ contains an edge of $\sigma(c_0)$.  Thus, by the same augmenting path argument used in Lemma \ref{lem:crookedrooted}, we may assume that there exists a path $R$ contained in $L$ with one endpoint in $X^1_i \cup X^2_i$ and the other point which is an internal vertex of either $R_1^*$ or $R_2^*$.  By Lemma \ref{lem:crookedmcdonalds}, there exists a cross with strictly fewer edges in the $L$-unhooked interiors, contrary to our choice of $R_1$ and $R_2$.  We conclude that $R_1$ and $R_2$ are coterminal with $\zP$ up to level $a+1$.

Let $N_i$ be the subgraph formed by the union over all $j$, $a+1\le j \le \alpha+7$ of the union of all subpaths of $C_j$ with endpoints in $V(\zQ_i)$ and no internal vertex in $V(\zQ)$.  Let $H_i'' = H_i' \cup N_i \cup \bigcup_{Q \in \zQ_i} Q$.  Note that $H_i''$ and $H_{i'}''$ are disjoint for distinct indices $i, i'$.  There exist two cycles $C_j$, $C_{j'}$ which are disjoint from $Z$ with $a+1 \le j < j' \le \alpha+7$ and fix $Q, Q'$ in $\zQ_i \setminus \zQ_i'$.  Given that $R_1$, $R_2$ is coterminal with $\zP$ up to level $a+1$ in $H_i'$, the desired paths $R_1', R_2'$ exist in $C_j \cup C_{j'} \cup Q \cup Q' \cup R_1 \cup R_2$, completing the proof.
\end{cproof}

We fix, for the remainder of the proof, an index $i$ such that $(H_i', \Omega_i')$ is rural, and fix a vortex-free rendition $\rho'$ of $(H_i', \Omega_i')$ in the disk $\Delta$.  Let $I^*$ be the interval obtained from $I_i$ by deleting the first and last $10p(p-1) + 1$ elements.  Note that $|I^*| = l$.  Let the first element in $I^*$ be $\beta + 1$.  Let $\zQ^*$ be the set $\{Q_j: j \in I^*\}$, and $X^{1*}$ and $X^{2*}$ be the minimal by containment subsets of $X^1_i$ and $X^2_i$, respectively, containing an endpoint of each element of $\zQ^*$ and forming sections in $(G-Z, \Omega)$.  Let $T_1$ and $T_l$ be the track of $Q_{\beta + 1}$ and $Q_{\beta +l}$ in the rendition $\rho'$.

Of the three connected components of $\Delta \setminus ( T_1 \cup T_l)$, let $\Delta^*$ be the (unique) component whose boundary contains $T_1$ and $T_l$.  Let $H^* = \bigcup_{c \in C(\rho'): c \subseteq \Delta^*} \sigma(c) \cup Q_{\beta + 1} \cup Q_{\beta + l}$.  Let $\Omega^*$ be the cyclically ordered set with $V(\Omega^*) = X^{1*} \cup X^{2*}$ obtained by restricting $\Omega_i'$ to $V(\Omega^*)$.

The rendition $\rho'$ restricted to the disk $\Delta^*$ can be extended to a vortex-free rendition of $(G^*, \Omega^*)$ by mapping the vertices of $Q_{\beta+1}$ and $Q_{\beta+l}$ to the boundary, and thus the following claim immediately follows.

\begin{claim}
$(H^*, \Omega^*)$ is rural.
\end{claim}

We now see that no edge to $V(H^*)$ avoids the paths $Q_{\beta+1}$ and $Q_{\beta+l}$.  Recall that the definition of the subgraph $J$ of $G-Z$ is the subgraph consisting of the union of  $V(\Omega)$ treated as isolated vertices and the $\bigcup_{Q\in \zQ} Q - Z$.

\begin{claim}\label{cl:nojumpedge}
There does not exist an edge $xy$ of $G-Z$ with $x$ in $V(H^*) \setminus (V(Q_{\beta+1}) \cup V(Q_{l}))$ and $y$ in $V(G-Z) \setminus V(H^*)$.
\end{claim}

\begin{cproof}
Assume that such an edge $xy$ exists.  Given the rendition $\rho'$, it follows that the edge $xy \notin E(H_i')$, and thus the vertex $y$ is the attachment of a $J$-bridge $B$ in $G-Z$ such that $y \in (V(\Omega) \setminus V(H_i')) \cup \bigcup_{j \notin I_i} V(Q_j)$ and $B$ has an attachment $x'$ in $V(\Omega^*) \cup V(\zQ^*)$.  Fix a path $P$ in $B$ from $x'$ to $y$ which is internally disjoint from $V(\Omega) \cup V(\zQ)$.  While $P$ may not be an $M$-path (for example, if $x'$ or $y$ is a vertex of $V(\Omega)$), the existence of $P$ implies that $G-Z$ contains an $M$-path $P'$ with one endpoint in the vicinity of an element of $\zQ^*$ and one endpoint in the vicinity of a path of $\zQ \setminus \zQ_i$.  Given that $I^*$ was obtained from $I_i$ by deleting the first and last $10p(p-1) + 1$ elements, the $M$-path violates the third point of Lemma \ref{l:strips}, a contradiction.
\end{cproof}

We now give the final claim in the proof.

\begin{claim}
The $\zQ^*$-strip society of $(G-Z, \Omega)$ is a subgraph of $H^*$.
\end{claim}

\begin{cproof}
Let $J^*$ be the subgraph of $G-Z$ formed by the union of the subgraph formed by $V(\Omega)$ treated as isolated vertices along with the subgraph $\bigcup_{Q \in \zQ^*} Q$.  Note $J^*$ is a subgraph of $J$, and so every $J$-bridge is a subgraph of some $J^*$-bridge.

Consider a $J^*$-bridge $B$ with at least one attachment in $(V(\Omega^*) \cup V(\zQ^*)) \setminus V(Q_{\beta+1} \cup Q_{\beta + l})$.  Assume that $B$ contains a vertex of $V(J) \setminus V(J^*)$.  Then $B$ contains a path from a vertex of $(V(\Omega^*) \cup V(\zQ^*)) \setminus V(Q_{\beta+1} \cup Q_{\beta + l})$ to a vertex of $V(J) \setminus V(J^*)$ which is disjoint from $V(Q_{\beta + 1} \cup Q_{\beta+l})$.  Such a path must contain an edge with one end in $V(H^*) \setminus V(Q_{\beta + 1} \cup Q_{\beta+l})$ and one end in $V(G-Z) \setminus V(H^*)$, contradicting Claim \ref{cl:nojumpedge}.

We conclude that the bridge $B$ is a bridge of $J$ in addition to being a bridge of $J^*$.  The claim now follows from the definition of the strip society and the definition of $H'_i$.
\end{cproof}

The theorem now follows as the vortex-free rendition of $(H^*, \Omega^*)$ restricted to the $\zQ^*$-strip society shows that the strip society is rural, and Claim \ref{cl:nojumpedge} shows that the $\zQ^*$-strip society is isolated.

\end{proof}

\begin{theorem}
\label{thm:planarstripjump}
Let $l,k \ge 1$ and $p\ge2$ be integers, let  $\alpha:=12288(p(p-1))^{12}$
and
$$l':= 24576p^{24}(49152p^{24}+k)l.$$
Let $(G, \Omega)$ be a society and $\rho = (\Gamma, \zD, c_0)$ a cylindrical rendition of $(G, \Omega)$.
Let $\zC = (C_1, \dots, C_{s})$ be a nest of cardinality $s \ge 2 \alpha + 20p(p-1) + 4$ in $\rho$ and
 $\zQ$  an unexposed monotone transaction in $(G, \Omega)$ of order $l'$
orthogonal to $\zC$.  Let $X^1, X^2$ be disjoint segments of $\Omega$ such that $\zQ$ is a transaction from $X^1$ to $X^2$.
Then one of the following holds.
\begin{itemize}
\item[(i)] $G$ contains a $K_p$ minor grasped by $\zQ \cup \zC$;
\item[(ii)]
there exists a transaction $\zQ'$ from $X^1$ to $X^2$ which is coterminal with $\zQ$ up to level $C_{\alpha+8}$ such that $\zQ'$ is either a $k$-nested crosses transaction or a $k$-twisted nested crosses transaction;
\item[(iii)]
there exist a set $Z \subseteq V(G)$ with $|Z| \le \alpha$ and a subset $\zQ'$ of $\zQ$ of size $l$ such that the $\zQ'$-strip society $(G', \Omega')$ is rural and isolated and $Z$ is disjoint from both $\zQ'$ and the outer graph of $C_{2\alpha + 20p(p-1) + 4}$.  Moreover, for all $z \in Z$ there exists an $\Omega$-path $P$ in $G - V(G')$ linking the two segments of $\Omega - V(\Omega')$ such that $P$ intersects $Z$ exactly in the vertex $z$.
\end{itemize}
\end{theorem}

\begin{proof}
Let $l,k,p$ be as stated in the theorem, and let $l_2:=(2\alpha + 1) l$.
By Theorem~\ref{thm:RTplanarstrip} applied with $l$ replaced by $l_2$
we may assume that there exist a set $Z \subseteq V(G)$ with $|Z| \le \alpha$, and a subset $\zQ' \subseteq \zQ$ of order $l_2$ such that $\zQ'$ is disjoint from $Z$ and the $\zQ'$-strip society of $(G-Z, \Omega)$ with respect to $(X^1, X^2)$ is rural and isolated.  Moreover, $Z$ is disjoint from the outer graph of $C_{2\alpha + 20p(p-1) + 4}$.  The remainder of the proof will focus on ensuring that we can pick a rural and isolated planar strip which satisfies the property that the desired $V(\Omega)$ path exists for every vertex in $Z$.

Fix pairwise disjoint subsets $\zQ'_1, \dots, \zQ'_{2\alpha+1}$ of $\zQ'$ each of order $l$.  We may choose the subsets such that if we let $X^1_i$ and $X^2_i$ be the minimal subset of $X^1$ and $X^2$, respectively, such that $\zQ'_i$ is a linkage from $X^1_i$ to $X^2_i$, we have that $X^1_i \cap X^1_j = \emptyset$ for $1 \le i < j \le 2\alpha+1$.  Assume that $\zQ_1', \dots, \zQ_{2\alpha + 1}'$ are labeled such that $X_1^1, \dots, X_{2\alpha+1}^1$ occur on $X^1$ in that order.  Let $(G'_i, \Omega'_i)$ be the $\zQ_i'$-strip society in $(G - Z, \Omega - Z)$ with respect to $(X^1_i, X^2_i)$.  By Lemma \ref{lem:inhereitedrurality}, $(G'_i, \Omega'_i)$ is rural and isolated.

For all $z \in Z$, let $i_z$ be the minimum index $i$ in $[2\alpha + 1]$ such that $G_i'$ contains a neighbor of $z$.  Similarly, let $j_z$ be the maximum index $j$ in $[2\alpha + 1]$ such that $G_j'$ contains a neighbor of $z$.
Fix an index $a \in [2\alpha+1] \setminus \left (\bigcup_{z \in Z} \{i_z, i_z'\}\right)$.  Given the size of $Z$, it follows that such an index $a$ exists and $1 < a < 2\alpha + 1$.

Let $Z' = \{z \in Z: \text{ $V(G_{a}')$ contains a neighbor of $z$}\}$.  The strip society $(G_a', \Omega_a')$ in $(G-Z, \Omega)$ is therefore also the strip society of $\zQ_a'$ in $(G-Z', \Omega)$.  Let $Y_1, Y_2$ be the two segments of $\Omega -  (X^1_a \cup X^2_a)$.  Let $z \in Z'$.  Given that $z \in Z'$, the two indices $i_z$ and $j_z$ are defined and $i_z < a < j_z$.  As every component of $G_{i_z}'$ and $G_{j_z}'$ contains a vertex of either $Y_1$ or $Y_2$, we have the existence of the desired path in $G - V(G_a')$ from $Y_1$ to $Y_2$ intersecting $Z$ in exactly the vertex $z$, completing the proof of the theorem.
\end{proof}

We conclude the section with a result which shows that an unexposed monotone transaction whose strip society is rural and isolated induces a monotone transaction with a rural and isolated strip society in the natural society defined on the inner graph of a cycle in a given nest in a cylindrical rendition.

\begin{DEF}
\label{def:restriction-rendition}
Let $\rho=(\Gamma,\zD)$ be a  rendition of a society $(G,\Omega)$ in a surface $\Sigma$ and let $C$ be a cycle in $G$
that is grounded in $\rho$. Assume that the track of $C$ either bounds a closed disk  $\Delta$ or forms
one boundary component of a cylinder $\Delta $ whose other boundary component is a boundary of $\Sigma$.
Let $\Gamma=(U,V,E)$, let $U'=U\cap\Delta$, $V'=V\cap\Delta$ and $E'=\{e\in E: e\subseteq \Delta\}$.
Then $\Gamma'=(U',V',E')$ is a drawing in $\Delta$.
Let $\zD'=\{D\in\zD:D\subseteq \Delta \}$.
Then  $(\Gamma',\zD')$ is a $\Delta $-decomposition. We will call it the {\em restriction of $\rho$ to $\Delta$.}
\end{DEF}

\begin{DEF}
\label{def:restriction-transaction}
Let $(G, \Omega)$ be a society and $\rho$ a cylindrical rendition. $\zC = (C_1, \dots, C_{s})$ be a nest in $\rho$ of order $s \ge 2$ and let $\zQ$ be an unexposed monotone transaction in $(G, \Omega)$ which is orthogonal to $\zC$.

Fix an index $i$, $2 \le i \le s$, and let $T$ be the track of $C_i$ in $\rho$.  Let $(G', \Omega')$ be the inner \cc{society} of $C_i$, defined in
Definition~\ref{def:innerouter}.
Let $\Delta'$ be the closed subdisk whose boundary is $T$.  Note that the restriction $\rho'$ of $\rho$ to $\Delta'$ gives a cylindrical rendition of $(G', \Omega')$.
Moreover, $(C_1, \dots, C_{i-1})$ is a nest in $\rho'$.

Let $\zQ$ be an unexposed monotone transaction in $(G, \Omega)$ which is orthogonal to $\zC$.  For every $Q \in \zQ$, there exists a unique subpath $Q'$ of $Q$ which is an $\Omega'$-path containing an edge of $\sigma(c_0)$.  Note that uniqueness here requires that $i \ge 2$ and that $\zQ$ be orthogonal to $\zC$.  We define $\zQ' = \{Q' : Q \in \zQ\}$ to be the \emph{restriction of $\zQ$ in the inner society of $C_i$}.
\end{DEF}

\begin{LE}\label{lem:planarstriprestriction}
Let $(G, \Omega)$ be a society and $\rho = (\Gamma, \zD, c_0)$ a cylindrical rendition.  Let $\zC = (C_1, \dots, C_{s})$ be a nest in $\rho$ and let $\zQ$ be an unexposed monotone transaction in $(G, \Omega)$ which is orthogonal to $\zC$ of order at least three.  Let $(H, \Omega_H)$ be the $\zQ$-strip society in $(G, \Omega)$ and let $Y_1$ and $Y_2$ be the two segments of $\Omega \setminus V(H)$.  Assume that there exists a linkage $\zP = \{P_1, P_2\}$ such that $P_i$ links $Y_i$ and $V(\sigma(c_0))$ for $i = 1, 2$, $\zP$ is disjoint from $\zQ$, and $\zP$ is orthogonal to $\zC$.

Let $i \ge 7$.  Let $(G', \Omega')$ be the inner society of $C_i$ and let $\zQ'$ be the restriction of $\zQ$ to $(G', \Omega')$.  Let $\rho' = (\Gamma', \zD', c_0)$ be the restriction of $\rho$ to be a cylindrical rendition of $(G', \Omega')$.  Then $\zQ'$ is unexposed and monotone in $\rho'$; moreover, $\zQ'$ is a crosscap transaction if and only if $\zQ$ is a crosscap transaction.  If the $\zQ$ strip society is rural and isolated in $(G, \Omega)$, then the $\zQ'$ strip society is rural and isolated in $(G', \Omega')$.
\end{LE}

\begin{proof}
That $\zQ'$ is unexposed follows immediately from the definitions.   The fact that $\zQ'$ is monotone, and moreover, that $\zQ'$ is a crosscap transaction if and only if $\zQ$ is a crosscap transaction follows from the rendition $\rho$ and the fact that $\zQ$ is orthogonal to $\zC$.  Let $(H', \Omega_H')$ be the $\zQ'$-strip society in $(G', \Omega')$.  Assume that $(H, \Omega_H)$ is rural and isolated.  It remains to show that $(H', \Omega_H')$ is rural and isolated.  Let $X_1, X_2$ be segments of $\Omega$ and $X_1', X_2'$ segments of $\Omega'$ and such that while traversing every element $Q\in\zQ$ from one end to the other, we encounter $X_1, X_1', X_2', X_2$ in that order.  Moreover, we assume $X_1, X_2, X_1', X_2'$ are all chosen to be minimal by containment.

The graph $H'$ is not necessarily a subgraph of $H$ - specifically, there could be $\zQ$-bridges in $G$ which only attach to an external element of $\zQ$ which given the rendition induced on $(G', \Omega')$ have attachments in $(X_1' \cup X_2') \setminus V(\zQ')$.  We begin by characterizing edges of $H'$ which are not contained in $H$.  Let $G''$ be the inner graph of $C_{i-1}$ and let $Q_1, Q_2$ be the external paths of $\zQ$ and $Q_1', Q_2'$ the external paths of $\zQ'$.

\begin{claim}
Let $e = uv$ be an edge contained in $E(H')$ but not $E(H)$.    Then $e \in E(G') \setminus E(G'')$.
\end{claim}
\begin{cproof}
Let $J$ be the subgraph $\bigcup_{Q \in \zQ} Q$ along with $V(\Omega)$ treated as isolated vertices.  Similarly, let $J'$ be the subgraph $\bigcup_{Q \in \zQ'} Q$  with $V(\Omega')$ treated as isolated vertices.  The edge $e$ is contained in a $J'$-bridge $B'$ in $G'$ with at least one attachment in $\left ( \bigcup_{Q \in \zQ'} V(Q) \cup X_1' \cup X_2' \right) \setminus (Q_1', Q_2')$.  The subgraph $B'$ is contained in a $J$-bridge $B$ in $G$.

Every path of $\zQ'$ is a subpath of an element of $\zQ$, thus if $B'$ has an attachment in an element of $\zQ' \setminus \{Q_1', Q_2'\}$, then $B$ would also be included in $H$, contrary to our assumptions.  Thus we may assume that $B'$ has no such attachment and instead does have an attachment in $(X_1'\cup X_2') \setminus V(\zQ)$.  Given the rendition $\rho'$ of $(G', \Omega')$ and the fact that $B'$ is connected, it follows $B'$ must be disjoint from $C_{i-1}$, lest $B'$ have an attachment on a non-external element of $\zQ'$.  We conclude that every edge of $B'$ is contained in $E(G') \setminus E(G'')$, as desired.
\end{cproof}

\begin{claim}
$(H', \Omega')$ is isolated.
\end{claim}
\begin{cproof}
Assume the claim is false.  Then there exists a path $P$ with one end in $V(\Omega') \setminus (X_1' \cup X_2')$ and the other end in $(V(\zQ') \cup (X_1' \cup X_2')) \setminus (V(Q_1' \cup Q_2'))$ which is internally disjoint from $V(\Omega') \cup V(\zQ')$.  Given the rendition $\rho'$, we may assume that the path $P$ has an endpoint in $V(\zQ') \setminus (V(Q_1' \cup Q_2'))$ by possibly rerouting $P$ through the cycle $C_{i-1}$.

Consider the components of $C_{i-1} - V(\zQ)= C_{i-1} - V(\zQ')$.  There are exactly two such components which do not have an edge in $C_{i-1}$ to a non-external element of $\zQ'$; denote these components $L_1$ and $L_2$.  The rendition $\rho$ of $(G, \Omega)$ ensures that $P_1$ intersects one of $L_1$ and $L_2$, say $L_1$, and $P_2$ intersects $L_2$.  Similarly, the rendition $\rho'$ ensures that $P$ also intersects one of $L_1$ and $L_2$, say $L_1$.  We conclude that $P_1 \cup P \cup L_1$ contains a path from $V(\Omega) \setminus (X_1 \cup X_2)$ to a vertex of $V(\zQ) \setminus V(Q_1 \cup Q_2)$ which is internally disjoint from $V(\Omega) \cup V(\zQ)$, contrary to the fact that $\zQ$ is isolated.
\end{cproof}

We conclude now by showing that $(H', \Omega_H')$ is rural.  If not, there must exist a cross in $(H', \Omega_H')$.  We show that by carefully picking the cross, we can extend it to a cross in $(H, \Omega_H)$, a contradiction to the rurality of $(H, \Omega_H)$

\begin{claim}
$(H', \Omega_H')$ is rural.
\end{claim}
\begin{cproof}
The proof closely follows the proof of Lemma \ref{lem:crookedrooted}.  Assume that the claim is false.  By Theorem \ref{thm:crossreduct} there exists a cross in $(H', \Omega_H')$.

We may assume that for any choice of cross $R_1$ and $R_2$ in $(H', \Omega_H')$, the path $R_j$ contains an edge of $\sigma'(c_0)$ for $j=1,2$.  To see this, assume there exists a cross $R_1$ and $R_2$ such that $R_1$ is edge disjoint from $E(\sigma'(c_0))$.  It follows from Lemma \ref{lem:crookinrend} that $R_1$ and $R_2$ are not a cross in $(G', \Omega')$.   Thus $\zQ'$ is a crosscap transaction and each of $R_1$ and $R_2$ has one endpoint in $X_1'$ and one end in $X_2'$.  However, every path in $G'$ from $X_1'$ to $X_2'$ which is edge disjoint from $\sigma'(c_0)$ must intersect either $P_1$ or $P_2$.  We conclude that either $P_1$ or $P_2$ is contained in $H'$, contrary to the previous claim that $(H', \Omega_H')$ is isolated.

Define $L$ to be the subgraph of $H'$ given by $\left(\bigcup_1^{5} (C_j \cap H') \cup \bigcup_{Q \in \zQ'} Q\right)$. For $j = 1, 2$, let $R_j^*$ be the $L$-unhooked interior of $R_j$.  Note that by the previous paragraph, the $L$-unhooked interior of each $R_j$ is nonempty.  We pick $R_1, R_2$ to be a cross which minimizes $| E(R_1^*) \cup E(R_2^*)|$.

Observe that if $R_j^*$ is disjoint from $C_{6}$ for $j = 1, 2$, then $R_1$, $R_2$ are coterminal with $\zQ'$ up to level $6$ by the construction of $L$.  Assume, instead, that $R_j^*$ intersects $C_6$ for some $j$.  Then $R_j^*$ intersects each of $C_1\cap L, \dots, C_5\cap L$. As in the proof of Lemma \ref{lem:crookedrooted}, we may assume that there exists a path $R$ contained in $L$ with one end in $X_1' \cup X_2'$ and one end which is an internal vertex of either $R_1^*$ or $R_2^*$.  By Lemma \ref{lem:crookedmcdonalds}, there exists a cross with strictly fewer edges in the $L$-unhooked interiors, contrary to our choice of $R_1$ and $R_2$.  We conclude that $R_1$ and $R_2$ are coterminal with $\zQ'$ up to level $6$.

Observe that every edge of $R_1$ and $R_2$ contained in $E(G') \setminus E(G'')$ is an edge contained in $E(\zQ')$ and so $R_1 \cup R_2$ is a subgraph of $H$ as well as $H'$ by Claim 1.   Thus, we can extend $R_1$ and $R_2$ to a cross in $(H, \Omega_H)$ using subpaths in $\zQ$, contradicting our assumptions.
\end{cproof}
This completes the proof of the lemma.
\end{proof}

\section{Leap patterns}\label{sec:augmenting}

The primary technical step in the proof of Theorem \ref{thm:main} is to extend Lemma \ref{lem:GM9} to quantify the obstructions to when a society has a rendition in the disk of bounded width and depth.  One such an obstruction which we have already seen is a large nested crosses transaction.  Another is what we will call a leap pattern.  Leap patterns will play a similar role to what were referred to as ``long jumps'' in the original graph minor papers.  This section introduces leap patterns and presents several technical lemmas in preparation for the coming sections.

\begin{DEF}
\label{def:rho}
Let $(G, \Omega)$ be a society, let $\zP$ be an $\Omega$-linkage, and let $X$ be the set of all vertices $v\in V(\Omega)$
such that $v$ is not an endpoint of any member of $\zP$.
Let $d_{\zP}(x,x)=0$ for every $x\in X$ and for distinct
$x,y\in X$ let $d_{\zP}(x,y)$ be the number of paths $P$ such that the endpoints of $P$ belong to different maximal segments
of $\Omega-\{x,y\}$.
\end{DEF}

\begin{DEF}
Let $(G, \Omega)$ be a society and let $k, l$ be positive integers.  A \emph{$(k, l)$-leap pattern} is a pair $(\zP, \zQ)$ such that:
\begin{itemize}
\item both $\zP$ and $\zQ$ are linkages of $V(\Omega)$-paths and $V(\zP) \cap V(\zQ)= \emptyset$;
\item $\zP$ is a planar transaction;
\item $|\zQ| = k$ and for every $Q \in \zQ$, the endpoints $x$ and $y$ of $Q$ satisfy $d_{\zP} (x, y) \ge l$.
\end{itemize}
Moreover, if the elements of $\zQ$ are $Q_i$, $1 \le i \le k$, we can label the endpoints of $Q_i$ as $s_i$ and $t_i$ such that
\begin{itemize}
\item for all $i, j$, $1 \le i < j \le k$, it holds that $d_{\zP} (t_i, t_j) \ge l$.
\end{itemize}
\end{DEF}

We first give the following observation on the size of a leap pattern.

\begin{LE}\label{lem:leapsandbounds}
Let $(G, \Omega)$ be a society and $k, l \ge 1$ positive integers.  Let $(\zP, \zQ)$ be a $(k,l)$-leap pattern such that for all $P \in \zP$, it holds that $(\zP - P, \zQ)$ is not a $(k,l)$-leap pattern.  Then $|\zP \cup \zQ| \le k(2l+1)$.
\end{LE}

\begin{proof}
Let $|\zP| = m$, and let the elements of $\zP$ be labeled $P_1, \dots, P_m$ with the endpoints of $P_i$ labeled $x_i$ and $y_i$ such that when traversing $\Omega$ in order, we encounter the vertices $x_1, x_2, \dots, x_m, y_m, y_{m-1}, \dots, y_1$ in that order.  Assume, to reach a contradiction, that $m > 2kl$.  Consider the segments $x_{il + 1} \Omega x_{(i+1)l + 1}$ and $y_{il + 1} \Omega y_{(i+1)l + 1}$ for $ 0 \le i \le 2k-1$.  If there exists an index $i$ such that $x_{il + 1} \Omega x_{(i+1)l + 1}$ and $y_{il + 1} \Omega y_{(i+1)l + 1}$ are disjoint from the endpoints of $\zQ$, then if $u$ and $v$ are endpoints of $\zQ$ which are separated by the endpoints of $P_j$ for some $il + 1 \le j \le (i+1)l+1$, then $u$ and $v$ are separated by the endpoints of $P_j$ for all $j$, $il + 1 \le j \le (i+1)l + 1$.  It follows that $(\zP - P_{il + 1}, \zQ)$ is a $(k,l)$-leap pattern, a contradiction.

Thus we may assume that $x_{il + 1} \Omega x_{(i+1)l + 1} \cup y_{il + 1} \Omega y_{(i+1)l + 1}$ contains at least one endpoint of an element of $\zQ$ for all $0 \le i \le 2k-1$.  As the endpoints of $\zQ$ are disjoint from $\{x_1, \dots, x_m, y_1, \dots, y_m\}$ and $|\zQ| = k$, we have that each set $x_{il + 1} \Omega x_{(i+1)l + 1} \cup y_{il + 1} \Omega y_{(i+1)l + 1}$ contains exactly one endpoint of $\zQ$ and moreover, no endpoint of $\zQ$ is not contained in one such set.  Specifically, the segment $y_1 \Omega x_1$ does not contain an endpoint of $\zQ$.  Thus, $d_{\zP - P_1} (u, v) = d_{\zP}(u,v)$ for all endpoints $u$ and $v$ of $\zQ$ implying that $(\zP - P_1, \zQ)$ is a $(k,l)$-leap pattern.  This contradiction completes the proof of the lemma.
\end{proof}

Recall that if $G$ is a graph, $X \subseteq V(G)$, and $\zS$ a partition of $X$, then an $\zS$-path is a path with both endpoints in $X$, no internal vertex in $X$, and endpoints in distinct sets of the partition $\zS$.

\begin{DEF}
\label{def:diverse}
Let $G$ be a graph, $X \subseteq V(G)$, $\zS$ a partition of $X$, and $k \ge 1$ a positive integer.  A set $\zQ = \{Q_1, \dots, Q_k\}$ of $\zS$-paths are \emph{$\zS$-diverse} if they are pairwise disjoint and for all $i$, $1 \le i \le k$, we can label the endpoints of $Q_i$ as $s_i$ and $t_i$ such that for all $i, j$, $1 \le i < j \le k$, it holds that $t_i$ and $t_j$ are not in the same set of $\zS$.
We say that $t_i$ is the {\em major endpoint} and $s_i$ is the {\em minor endpoint} of $Q_i$. We say an $\zS$-diverse set $\zQ'$ of paths is \emph{majority equivalent} to $\zQ$ if the elements of $\zQ'$ can be labeled $Q_1', \dots, Q_k'$ with the endpoints of $Q_i'$ labeled $s_i'$ and $t_i'$ such that $t_i'$ is the major endpoint of $Q_i'$ and $t_i'$ is in the same set of $\zS$ as $t_i$.
\end{DEF}

\begin{LE}\label{lem:indMcDonalds}
Let $G$ be a graph, $\zS$ a partition of a subset $X \subseteq V(G)$, and $\zQ$ an $\zS$-diverse set of paths.
Let $R$ be a path of length at least one with one endpoint in $X$, the other endpoint which is an internal vertex in $Q \in \zQ$,
and which is otherwise disjoint from $V(\zQ)$.
Then there exists an $X$-path $Q' $ such that $Q'\ne Q$, $Q'$ is a  subgraph of $ Q \cup R$ and $(\zQ - Q) \cup Q'$ is an
$\zS$-diverse  set of paths which is majority equivalent to $\zQ$.
\end{LE}

\begin{proof}
Let us assume that the endpoints of every member of $\zQ$ are designated as major and minor as in Definition~\ref{def:diverse}.
Let $t$ be the major endpoint of $Q$, let $s$ be the minor endpoint of $Q$, and let $x,y$  be the endpoints of $R$, where
$x\in X$ and $y\in V(Q)$.
If $x$ and $t$ do not belong to the same member of $\zS$, then let $Q':=yQt\cup R$.  Otherwise, $x$ and $t$ belong to the same member of $\zS$, and we let $Q':=yQs\cup R$.  In the first case, we designate $x$ as the major endpoint of $Q'$ and in the second case, we designate $x$ to be the major endpoint of $Q'$.
In either case, $Q'$ is as desired by the lemma with the major endpoint in the same set of $\zS$ as $t$.
\end{proof}

\begin{LE}\label{lem:Spaths1}
Let $G$ be a graph, $X \subseteq V(G)$, and $\zS$ a partition of $X$.  Let $k \ge 1$ be an integer and assume that there exists an $\zS$-diverse set $\zQ$ of $k-1$ paths.  Let $Y \subseteq V(G) \setminus X$, and assume that there exists a linkage $\zR$ of $2k-1$ disjoint $X-Y$ paths in $G$.  Then there exists an $\zS$-diverse set $\zQ'$ of $k-1$ paths which is majority equivalent to $\zQ$ and a path $R$ from $X$ to $Y$ such that $V(\zQ) \cap V(R) = \emptyset$.
\end{LE}
\begin{proof}
Assume the lemma is false, and pick a counterexample which minimizes $|E(G)|$.  Let the elements of $\zR$ be labeled $R_1, \dots, R_{2k-1}$.  It follows that $E(G) = \bigcup_1^{2k-1} E(R_i) \cup \bigcup_{Q \in \zQ} E(Q)$.  Fix $i$ such that the endpoint of $R_i$ in $X$ is disjoint from $V(\zQ)$ and let $u = V(R_i) \cap X$.  Given that $G$ forms a counterexample, $R_i$ must intersect some $Q \in \zQ$.  Let $v$ be the first vertex of $V(\zQ)$ we encounter traversing $R_i$ from $u$ to its endpoint in $Y$.  Let $Q \in \zQ$ be such that $v \in V(Q)$.  Note that $v$ must be an internal vertex of $Q$ by the fact that $Y$ is disjoint from $X$.  By Lemma \ref{lem:indMcDonalds}, $uR_iv \cup Q$ contains an $\zS$-path $Q'$ which is distinct from $Q$ such that $(\zQ - Q) \cup Q'$ is $\zS$-diverse and majority equivalent to $\zQ$.  The subgraph $\bigcup_1^{2k-1}R_i \cup (\zQ - Q) \cup Q'$ has strictly fewer edges, contradicting our choice of counterexample.
\end{proof}

We now present a result on disentangling a leap transaction from a separate linkage in the graph.

\begin{LE}
\label{lem:augmentleap}
Let $k, l \ge 1$ be integers and $(G, \Omega)$ a society.  Let $(\zP, \zQ)$ be a $(k-1, 2kl)$-leap pattern in $(G, \Omega)$.
Let $Z \subseteq V(G)$,
and let $\zR$ be a set of $2k-1$ pairwise disjoint paths from $V(\Omega)$ to $Z$.
Then there exist $R$, $\zQ'$, $\zP'$ such that
\begin{itemize}
\item $\zP' \subseteq \zP$
and $(\zP', \zQ')$ is a $(k-1, l)$-leap pattern
\item  $R$ is a path from $V(\Omega)$ to $Z$ and is disjoint from $V(\zQ') \cup V(\zP')$.
 \end{itemize}
\end{LE}

\begin{proof}
As a first step, we fix $\zR$ to minimize $E(\zR) \cup E(\zQ) \cup E(\zP)$.  As a result, if $R \in \zR$ intersects an element $P \in \zP$ or $Q \in \zQ$, it follows that at least one endpoint of $P$ or $Q$, respectively, is also an endpoint of an element of $\zR$.

Let the elements of $\zQ$ be labeled $Q_1, \dots, Q_{k-1}$ and label the endpoints of $Q_i$ as $s_i, t_i$ such that $d_\zP(s_i, t_i) \ge 2kl$ and $d_\zP(t_i, t_j) \ge 2kl$ for all $1 \le i \le j \le k-1$, $i \neq j$.  Assume that exactly $t$ elements of $\zQ$ have an endpoint in common with an element of $\zR$.  Fix $\zR'$ to be a subset of $2t+1$ elements of $\zR$ such that $\zR'$ contains every element of $\zR$ which has a common endpoint with an element of $\zQ$.  Assume that the elements of $\zQ$ are labeled $Q_1, \dots, Q_{k-1}$ such that $Q_1, \dots, Q_t$ have an endpoint in common with an element of $\zR$.  As we have already observed, the paths $Q_{t+1}, \dots, Q_{k-1}$ are disjoint from $\zR$.

Let $\zP_0$ be the elements of $\zP$ which are disjoint from $\zR$.  Note that $\zP_0$ is obtained from $\zP$ by deleting at most $(2k-1) - t$ elements.  Let $X = (V(\zR') \cap V(\Omega)) \cup (V(\zQ) \cap V(\Omega))$, i.e. the set of endpoints of $\zR' \cup \zQ$ in $V(\Omega)$.  Note that $|X| \le 2(k-1) + (t+1)$.

We define a sequence $\zP_0, \zP_1, \dots, \zP_a$ for some $a \ge 0$ such that for $i \ge 0$, the linkage $\zP_{i+1}$ is obtained from $\zP_{i}$ by fixing $x, y \in X$ such that $d_{\zP_{i}}(x,y) > 0$ and $d_{\zP_{i}}(x,y) \le l-1$ and deleting the elements of $\zP_{i}$ which separate $x$ and $y$.  Fix such a sequence $\zP_0, \zP_1, \dots, \zP_a$ to maximize $a$.  By the choice to maximize $a$, we have that for all $x, y \in X$, if $d_{\zP_a}(x,y) > 0$, then $d_{\zP_a}(x,y) \ge l$.  Note as well that if $d_{\zP_i}(x,y) = 0$ for some $i \ge 0$, then for all $i' > i$, $d_{\zP_{i'}}(x,y) = 0$.

While $d_{\zP_i}$ does not typically satisfy the triangle inequality on $V(\Omega)$, when restricted to the set $X$ which is disjoint from the endpoints of $\zP_i$, the triangle inequality does hold.  Define $x \sim_i y$ if $d_{\zP_i}(x,y)= 0$.  Given the triangle inequality, $\sim_i$ is an equivalence relation on $X$ for all $0 \le i \le a$.  The definition of $\zP_{i+1}$ from $\zP_i$ implies that $\sim_{i+1}$ has strictly fewer equivalence classes than $\sim_i$.  Thus, $a \le |X| \le 2(k-1) + (t+1) = 2k+t-1$.  We give a stronger bound in the next claim.

\begin{claim}
$a \le k+t$.
\end{claim}

\begin{cproof}
Assume the claim is false and that $a > k+t$.  Consider the linkage $\zP_{k+t}$ which is obtained from $\zP$ by deleting at most
\begin{align*}
(l-1)(t+k) + (2k-1 -t) &= (t+k)l - t - k + 2k -1 -t \\
&= (2k-1)l - (k-t)l + (k-t-1) - t \\
& \le (2k-1) l.
\end{align*}
Thus, for all $i, j$, $d_{\zP_{k+t}}(t_i, t_j) \ge 2kl - (2k-1)l = l$ and the relation $\sim_{k+t}$ has at least $k-1$ distinct equivalence classes.  As $\sim_{k+t}$ also has at most
\begin{align*}
|X| - (k+t) & \le (2k-2) + (t+1) - (k+t) \\
&= k-1
\end{align*}
equivalence classes, we conclude that there are exactly $k-1$ equivalence classes, each of which contains $t_i$ for some index $i$.  As any $x, y \in X$ such that $d_{\zP_{k+t}}(x,y) > 0$ must be contained in distinct equivalence classes, we have that $d(x,y) \ge d(t_i, t_j) \ge l$ for some indices $i$ and $j$.  We conclude that $\zP_{s+t+1}$ is not defined, a contradiction.
\end{cproof}
Define $\zS$ to be the partition of $X$ defined by the equivalence classes of $\sim_a$.  It follows that the set $\zQ$ of paths is $\zS$-diverse and for any pair of points $x, y \in X$ in distinct sets of $\zS$, it holds that $d_{\zP_a}(x,y) \ge l$.

We apply Lemma \ref{lem:Spaths1} to the subgraph formed by $\zR' \cup \bigcup_1^t Q_i$ and the partition $\zS$ of the set $X$.  Thus, there exists a path $R$ from $X$ to $Y$ and a set of $\zS$-diverse $X$-paths $Q_1', \dots, Q_t'$ which is majority equivalent to $Q_1, \dots, Q_t$.  We may assume the majority endpoint of $Q_i'$ is $t_i'$ and that $t_i$ and $t_i'$ are contained in the same set of $\zS$.  Note that for $1 \le i \le t < j \le k-1$, we have that $l \le d_{\zP_a}(t_i, t_j) \le d_{\zP_a}(t_i, t_i') + d_{\zP_a}(t_i', t_j) = d_{\zP_a}(t_i', t_j)$, as $\zP_a$ is disjoint from the endpoints of $\zR$ and $\zQ$.  It follows that $(\zP_a, \{Q_1', \dots, Q_t', Q_{t+1}, \dots, Q_{k-1}\})$ is a $(k-1,l)$-leap pattern disjoint from the path $R$ as desired.
\end{proof}

We finish this section with one final lemma which will be necessary when we attempt to increase the order of a leap pattern.

\begin{LE}\label{lem:sep1-2}
Let $G$ be a graph and $X$ and $Y$ two nonempty subsets of vertices of $G$.
Let $k \ge 0$ be an  integer.
Assume that there
do not exist $k+1$ vertex
disjoint $X-Y$ paths in $G$.  Then there exists a separation $(A, B)$ of $G$ with
$X \subseteq A$ and $Y \subseteq B$ such that for every $v \in A \cap B$, there exists a path from $v$ to $X$ in $G[A]$ and moreover, exactly one of the following holds:
\begin{itemize}
\item[{\rm(i)}] $|A \cap B| = k$ and there exist $k$ disjoint $(A \cap B) - Y$ paths in $G[B]$, or
\item[{\rm(ii)}] $|A \cap B| < k$ and $A \cap B \subseteq Y$.
\end{itemize}
\end{LE}

\begin{proof}
By assumption, there do not exist $k+1$ disjoint $X-Y$ paths, and hence
there exists a separation $(A, B)$ of order at most $k$ such that
\begin{itemize}
\item[(1)]
$X \subseteq A$ and $Y \subseteq B$.
\end{itemize}
If we choose such a separation of minimum order,
then it furthermore satisfies
\begin{itemize}
\item[(2)] for every $v\in A\cap B$  there exists a path from $v$ to $X$
in $G[A]$,  and
\item[(3)]there exist $|A \cap B|$ disjoint $(A \cap B) - Y$ paths in $G[B]$.
\end{itemize}

Let us  choose a separation $(A, B)$ of order at most $k$ that satisfies (1) -- (3),
and, subject to that
\begin{itemize}
\item[(4)]$B\setminus A$ is minimal.
\end{itemize}

If the order of $(A, B)$ is $k$, then $(A, B)$  satisfies (i).
We may therefore assume that $|A\cap B|<k$.
 We may assume that $A \cap B \nsubseteq Y$, lest the
separation $(A, B)$ satisfies (ii).
Let $x$ be a vertex of $B\setminus A$ which is adjacent to a vertex of $A \cap B$.  Such a
vertex exists because there exist $|A \cap B|$ disjoint paths from
$A \cap B$ to $Y$ by (3).
Given that $A \cap B \nsubseteq Y$, we may simply pick the next
vertex on one of these paths.
Let $A':=A \cup \{x\}$.
Then $(A', B)$ is a separation of $G$ of order at most $k$ that clearly  satisfies (1) and (2).
It also satisfies (3), for otherwise there exists a separation $(A^*,B^*)$ of $G[B]$
of order at most $|A \cap B|$ such that $A'\cap B\subseteq A^*$ and $Y\subseteq B^*$,
in which case the separation $(A\cup A^*,B^*)$ contradicts (4).
Thus the separation $(A', B)$ satisfies (1) -- (3), contrary to (4).
\end{proof}

\section{Finding a planar or crosscap transaction}

A classic result of Erd\H{o}s and Szekeres \cite{ES} states that every sequence of length at least $(s-1)(t-1) + 1$ has either a monotonic decreasing subsequence of length $s$ or a monotonic increasing subsequence of length $t$.  An immediate consequence of this is the following lemma.

\begin{LE}\label{lem:transES}
Let $s, t \ge 1$ be positive integers.  Let $\zP$ be a transaction in a society $(G, \Omega)$ of order at least $(s-1)(t-1)+1$.
Then there exists a transacton $\zP' \subseteq \zP$ such that $\zP'$ is either a crosscap transaction of order $s$ or a planar transaction of order $t$.
\end{LE}

In the applications to come, we will often be happy to either find a large $(k, l)$-leap transaction or a large nested crosses transaction.  In this section, we give a version of Lemma \ref{lem:transES} which takes this into account. 

 We first need the following definition.  Let $(G, \Omega)$ be a society, and let $X$ be a segment.  Let the vertices of $X$ be $x_1, \dots, x_n$ and assume we encounter them in that order following $\Omega$.  We define a new cyclic order $\Omega^*$ on $V(\Omega)$ as follows.  Let $y_1$ be the vertex of $\Omega$ immediately preceding $x_1$ and $y_n$ the vertex immediately following $x_n$.  $\Omega^*$ orders $V(\Omega)$ by $y_n \Omega y_1, x_n, x_{n-1}, \dots, x_2, x_1$.  We say that the society $(G, \Omega^*)$ is obtained by \emph{flipping $X$}.

\begin{LE}\label{lem:strongES}
Let $k, l, k^*, l^*, q, q^*, s, s^*$ be positive integers, and let $(G, \Omega)$ be a society.  Let $s'$ be the max $\{s, s^*\}$.  Let $\zP$ be a transaction of order $(k+2q)l(k^*+2q^*)l^*(s')$ in $(G, \Omega)$, and let $X_1$, $X_2$ be disjoint segments of $\Omega$ such that each element of $\zP$ has one endpoint in $X_1$ and the other endpoint in $X_2$.  Let $(G, \Omega^*)$ be the society obtained by flipping $X_1$.  There exists a sublinkage $\zP'$ of $\zP$ satisfying one of the following.
\begin{itemize}
\item[\rm(i)] either $\zP'$ is a $(k, l)$-leap pattern in $(G, \Omega)$ or $\zP'$ is a $(k^*, l^*)$-leap pattern in $(G, \Omega^*)$;
\item[\rm(ii)] either $\zP'$ is a $q$-nested crosses transaction in $(G, \Omega)$ or $\zP'$ is a $q^*$-nested crosses transaction in $(G, \Omega^*)$;
\item[\rm(iii)] either $\zP'$ is a planar transaction of order $s$ in $(G, \Omega)$ or $\zP'$ is a planar transaction of order $s^*$ in $(G, \Omega^*)$.
\end{itemize}
\end{LE}

\begin{proof}
Assume the lemma is false.  We define an equivalence relation on the elements of $\zP$.  We say two elements $P_1, P_2 \in \zP$ are equivalent if for all $Q \in \zP$, $Q \neq P_1, P_2$, we have that $Q$ crosses $P_1$ in $(G, \Omega)$ if and only if $Q$ crosses $P_2$ in $(G, \Omega)$.  By the definition of the relation, the elements of an equivalence class will either all be pairwise crossing or pairwise non-crossing.  Lest we find a planar transaction or crosscap transaction satisfying  (iii), we immediately conclude that the following claim holds.

\begin{claim}\label{cl:es1}
Every equivalence class has size at most $s'$.
\end{claim}

Let $\zP_2$ be a maximal subset of pairwise non-equivalent elements of $\zP$.  By Claim \ref{cl:es1}, $|\zP_2| \ge (2q+k)l\cdot (2q^*+k^*)l^*$.  By Lemma \ref{lem:transES}, $\zP_2$ contains a transaction $\zP_3$ which is either a planar transaction of order $(2q+k)l$ or a crosscap transaction of order $(2q^*+k^*)l^*$.

Assume as a case that $\zP_3$ is a planar transaction.  Let $m = (2q+k)l$, and let the elements of $\zP_3$ be labeled $P_1, \dots, P_m$ with the endpoints of $P_i$ labeled $x_i$ and $y_i$ such that $x_1, x_2, \dots, x_m y_m, \dots, y_2, y_1$ occur on $\Omega$ in that order.  Fix $a \in\{ 0, \dots, (2q+k)-1\}$ and consider the paths $P_{al+1}$ and $P_{al+2}$.  As $P_{al+1}$ and $P_{al+2}$ are not equivalent, there exists a path $Q_a$ which crosses one of $P_{al+1}$ or $P_{al+2}$ but not the other.  Thus, $Q_a$ has exactly one endpoint in $x_{al+1}\Omega x_{al+2} \cup y_{al+2}\Omega y_{al+1}$.  We say that $Q_a$ is \emph{long} if it crosses at least $l$ elements of $\zP_3$, and \emph{short} otherwise. If there exist $k$ distinct indices $a$, $0 \le a \le 2q+k-1$ such that $Q_a$ is long, then $(G, \Omega)$ contains a $(k, l)$-leap pattern transaction, a contradiction.  Note that the endpoint contained in $x_{al+1}\Omega x_{al+2} \cup y_{al+2}\Omega y_{al+1}$ of such a $Q_a$ is the endpoint labeled $t_i$ as in the definition of a $(k,l)$-leap pattern.

Again, let $0 \le a \le 2q+k-1$.  If $Q_a$ is \emph{short}, we say it goes \emph{right} if it crosses $P_{al+2}$ and goes \emph{left} otherwise. Observe that if $Q_{a}$ is short and goes right, then it does not cross $P_{(a+1)l +1}$ and thus if $Q_{a+1}$ is also short and goes right, then $Q_a$ does not cross $Q_{a+1}$.  It follows that if there exist $q$ distinct indices $a$ such that $Q_a$ is short and goes right, then $(G, \Omega)$ contains a $q$-nested crosses transaction, a contradiction.  We conclude that there exist at least $q$ distinct indices $a$ such that $Q_a$ is short and goes left.  By the same argument as above, we see that there exists a $q$-nested crosses transaction, again a contradiction.

We have completed the analysis of the case when $\zP_3$ is planar.  If $\zP_3$ is a crosscap transaction of order $(2q^*+k^*)l^*$, the same arguments show that $(G, \Omega^*)$ contains either a $(k^*, l^*)$-leap pattern transaction or $q^*$-nested crosses transaction, again a contradiction to our choice of counterexample and completing the proof of the lemma.
\end{proof} 
\section{Crosses and leaps yield clique minors}\label{sec:crossesandleaps}

In this section, we show that large cross and leap patterns rooted on an orthogonal linkage to a sufficiently large nest yield a large clique minor grasped by the nest and orthogonal linkage. This is another key step toward our overall goal, and forms another place where we proceed differently than Robertson and Seymour.  In \cite{RS16}, Robertson and Seymour showed that if there are many ``long jumps", then either there exists the desired clique minor or there is one more vertex that yields a ``universal'' apex.  In our proof, the second outcome will not occur, because this case is already taken care of through the application of the results in \cite{flatwall} in Section \ref{sec:planarstrip}.

We first consider the case of large nested cross patterns.  Let $r \ge 1$ be a positive integer and define the graph $H_{2r}^1$ as follows.  Let $V(H_{2r}^1) = \{(i,j): 1 \le i \le 2r, 1 \le j \le 2r\}$.  The edges of $H_{2r}^1$ are added as follows.  Let $(i, j)$ be adjacent $(i', j')$ if $|i-i'| + |j - j'| = 1$.  Moreover, add the edge with endpoints $(i,r)$ and $(i+1, r+1)$ and the edge with endpoints $(i+1, r)$ and $(i, r+1)$ for $1 \le i \le 2r-1$.  Thus, $H_{2r}^1$ is the graph obtained from the $2r \times 2r$-grid by adding a pair of crossing edges to every face in the middle row of faces.

The following result was proven in \cite{flatwall}.

\begin{LE}[\cite{flatwall}, Lemma 3.2]\label{lem:gridwithcrosses}
Let $p \ge 2$ be a positive integer.  The graph $H_{p(p-1)}^1$ contains a $K_p$ which is grasped by the underlying grid.
\end{LE}

\begin{CO}\label{cor:nonestedcrosses}
Let $p, s, t\ge 1$ be integers.  Let $(G, \Omega)$ be a society with a cylindrical rendition $\rho = (\Gamma, \zD, c_0)$.  Let $\zC = (C_1, \dots, C_s)$ be a nest around $c_0$ of order $s\ge 4p^2$ and let $\zP$ be an orthogonal radial linkage in $\rho$ of order $t\ge 8p^2$.  Let $\zQ$ be a $V(\Omega)$-linkage of order $4p^2$ which is either a $2p^2$-nested crosses or $2p^2$-twisted nested crosses transaction which is coterminal with $\zP$ up to level $C_1$.  Then $G$ contains a $K_p$ minor which is grasped by $\zC \cup \zP$.
\end{CO}

\begin{proof}
We construct a $H_{p(p-1)}^1$ minor in $G$ using $\zC$ and $\zQ$, and the claim will follow by Lemma \ref{lem:gridwithcrosses}.  In order to define the model of the $H_{p(p-1)}^1$ minor, we will follow the ideas in the proof of Theorem \ref{thm:RTplanarstrip} to ensure that if the $K_p$-minor grasps the underlying grid of the $H_{p(p-1)}^1$-minor, then it grasps $\zC \cup \zP$ as well.

Let $X_1$ and $X_2$ be minimal (by inclusion) segments of $\Omega$ such that $\zQ$ is an $X_1$-$X_2$-linkage.  For $i = 1, 2$, let $\zP_i$ be the linkage formed by the elements of $\zP$ which both have an endpoint in $X_i$ and intersect $H \cap \zQ$ where $H$ is the outer graph of $C_1$.  Label the elements of $\zP_i$ as $P_1^i, \dots, P_{4p^2}^i$ so that when traversing $X_i$ in the order given by $\Omega$, we encounter the endpoints of $P^i_1, \dots, P^i_{4p^2}$ in that order.

For $1 \le i \le p(p-1)$, let $Z_i^1$ be the minimal subpath in $C_{1+i}$ with one endpoint in $P_1^1$, the other endpoint in $P_{4p^2}^1$, and no internal vertex in $P_1^2$.  Similarly,  let $Z_i^2$ be the minimal subpath of $C_{p^2+i}$ with one endpoint in $P_1^2$, the other endpoint in $P_{4p^2}^2$, and no internal vertex in $P_1^1$.  Thus, for all $i$ and $j$, $Z_i^1$ and $Z_j^2$ are not contained in the same cycle of the nest $\zC$.

We conclude that $\zQ \cup \bigcup_1^{p(p-1)} Z_i^1 \cup Z_i^2$ contains a model of $H_{p(p-1)}^1$ where every horizontal path arises from a subpath of a distinct cycle of the nest $\zC$ and every vertical path arises from an element of $\zQ$.  Lemma \ref{lem:gridwithcrosses} implies that $G$ contains a $K_p$ minor which is grasped by $\zC \cup \zP$ by the construction of the minor model.  Note that the graph $\zQ \cup \bigcup_1^{p(p-1)} Z_i^1 \cup Z_i^2$ is the same when $\zQ$ is a nested crosses or a twisted nested crosses transaction.  This completes the proof of the lemma.
\end{proof}

The next lemma is a restatement of Lemma 4.5 of \cite{flatwall}, using the notation we established in Section \ref{sec:planarstrip}.

\begin{LE}[Lemma 4.5, \cite{flatwall}]\label{lem:meshandpaths}
Let $p \ge 1$ be an integer, $k_0 = 12288(p(p-1)^{12}$, $M$ a mesh in a graph $G$ with vertical paths $Q_1, \dots, Q_s$ for $s \ge 1$, and assume $G$ contains a linkage $\zP$ of $M$-paths of order $k_0$ such that the endpoints of every $P \in \zP$ can be labeled $x(P), y(P)$ so that the following holds:
\begin{itemize}
\item $x(P)$ is at distance $10(p(p-1))$ from the boundary of $M$;
\item for all $P \in \zP$, if $i$ and $j$ be indices such that $x(P)$ is in the vicinity of $Q_i$ and $y(P)$ is in the vicinity of $Q_j$, then $|i-j| \ge 10(p(p-1))$;
\item for all $P, P' \in \zP$ with $P \neq P'$, if $i$ and $j$ are indices such that $x(P)$ is in the vicinity of $Q_i$ and $x(P')$ is in the vicinity of $Q_j$, then $|i - j| \ge 10p(p-1)$.
\end{itemize}
Then $G$ contains a $K_p$ minor grasped by $M$.
\end{LE}

We now use Lemma \ref{lem:meshandpaths} to show that a large $(k,l)$-leap transaction yields a clique minor grasping a nest and orthogonal linkage.  We first give a definition of twisted leap patterns analogously to that of twisted nested crosses.

\begin{DEF}
Let $(G, \Omega)$ be a society, $k,l \ge 0$ integers, and $(\zP, \zQ)$ a pair of linkages of $\Omega$-paths.  The pair $(\zP, \zQ)$ is a \emph{$(k,l)$-twisted leap pattern} if there exists a segment $X$ of $\Omega$ such that every element of $\zP$ has one endpoint in $X$ and one endpoint not in $X$ such that $(\zP, \zQ)$ is a $(k,l)$-leap pattern in the society $(G, \Omega^*)$ obtained by flipping $X$.
\end{DEF}

\begin{LE}\label{lem:nolargeleap}
Let $p, s, t\ge 1$ be integers.  Let $(G, \Omega)$ be a society with a cylindrical rendition $\rho = (\Gamma, \zD, c_0)$.  Let $\zC = (C_1, \dots, C_s)$ be a nest around $c_0$ of order $s\ge 20p^2$ and let $\zP$ be an orthogonal radial linkage in $\rho$ of order $t$.  Let $(\zR, \zQ)$ be a pair of $V(\Omega)$-linkages forming an $(36864p^{24}, 10p^2)$-leap pattern or twisted leap pattern such that $\zR$ and $\zQ$ are coterminal with $\zP$ up to level $C_1$.  Then $G$ contains a $K_p$ minor which is grasped by $\zC \cup \zP$.
\end{LE}

\begin{proof}
In order to apply Lemma \ref{lem:meshandpaths}, we first define the appropriate mesh.  Fix $X_1, X_2$ disjoint segments such that every element of $\zR$ has one endpoint in $X_1$ and one endpoint in $X_2$.  Moreover, in the case that $(\zR, \zQ)$ is a twisted leap pattern, then $(\zR, \zQ)$ is a $(k,l)$-leap pattern in the society obtained by flipping $X_1$.

As in the proof of Corollary \ref{cor:nonestedcrosses}, for $i = 1, 2$, we define $\zP_i$ to be the elements of $\zP$ with an endpoint in $X_i$ and which are contained in $\zR \cap H$ where $H$ is the outergraph of $C_1$.  Let the order of $\zR$ be $T$, and let $P_1^1, \dots, P_T^1$ be the elements of $\zP_1$ ordered by their endpoints in $X_1$ and the ordering given by $\Omega$.  Similarly, let $P_1^2, \dots, P_T^2$ be the elements of $\zP_2$ ordered by their endpoints in $X_2$ via $\Omega$.  Let $S_i$, $i = 1, \dots, 10p(p-1) + 2$, be the minimal subpath of $C_{1+i}$ with one endpoint in $P_1^1$, one endpoint in $P_T^1$, and no internal vertex in $P_1^2$.  Similarly, let $S_i$, $i = 10p(p-1) + 3, \dots, 20p(p-1) + 4$, be the minimal subpath of $C_{1 + i}$ with one endpoint in $P_1^2$, the other endpoint in $P_T^2$, and no internal vertex in $P_1^1$.  Let $M$ be the (uniquely determined) mesh obtained by repeatedly deleting vertices of degree one in $\zR \cup \bigcup_1^{20p(p-1) + 4} S_i$.

Fix $(G', \Omega')$ to be $(G, \Omega)$ in the case that $(\zR, \zQ)$ is a leap pattern and the society obtained by flipping $X_1$ in the case that $(\zR, \zQ)$ is a twisted leap pattern.  Thus, $(\zR, \zQ)$ is a leap pattern in $(G', \Omega')$, and we fix a labeling $s(Q)$, $t(Q)$ of the endpoints of $Q \in \zQ$ in $V(\Omega')$ as in the definition of a leap pattern. For $i = 1, 2$, let $X_i'$ be the minimal subset of $X_i$ forming a segment of $\Omega'$ such that the endpoints of $P_{10p^2+1}^i$ and $P_{T-10p^2}^i$ are contained in $X_i'$.  By the definition of a leap pattern, there are at most two elements $Q, Q'$ of $\zQ$ which have $t(Q)$ and $t(Q')$ in $V(\Omega') \setminus (X_1' \cup X_2')$.  Let $\zQ' = \{Q \in \zQ: t(Q) \in X_1' \cup X_2'\}$.  Thus, $|\zQ'| \ge 12288p^{24} -2$.  There are now two cases to consider: either many of the endpoints $s(Q)$ for $Q \in \zQ'$ are contained in $X_1' \cup X_2'$, or many $s(Q)$ are contained in $V(\Omega') \setminus (X_1' \cup X_2')$.  Fix $\zQ^* = \{Q \in \zQ' : s(Q) \in X_1' \cup X_2'\}$.

We first consider the case that $|\zQ^*| \ge 12288(p(p-1))^{12}$.  For every $Q \in \zQ^*$, $Q$ contains a subpath $Q^*$ which is an $M$-path and has endpoints in $S_1 \cup S_{10p(p-1) + 3}$ (whether the endpoints of $Q^*$ are in $S_1$ or in $S_{10p(p-1) + 3}$ will depend on whether the endpoints of $Q$ are in $X_1'$ or in $X_2'$).  Let $y(Q^*)$ be the endpoint of $Q^*$ closest to $s(Q)$ on the path $Q$ and define $x(Q^*)$ analogously to be the endpoint closest to $t(Q)$.  As $t(Q) \in X_1' \cup X_2'$ and $x(Q^*) \in V(S_1)\cup V(S_{10p(p-1)+3})$, it follows that the vertex $x(Q^*)$ is at distance at least $10p(p-1)$ from the boundary of $M$.  From the properties of a $(36864p^{24}, 10p^2)$-leap pattern and the construction of the paths $Q^*$, it follows that both the second and third conditions of Lemma \ref{lem:meshandpaths} hold and thus, $G$ contains a $K_p$ minor grasped by $M$ and, consequently, grasped by $\zC \cup \zP$.

We conclude that we may assume that $|\zQ' \setminus \zQ^*| \ge 24576p^{24}$.  Define $J$ to be the subgraph of $C_1$ defined by the union of the subpaths of $C_1$ with one endpoint in $\zP_1$, one endpoint in $\zP_2$, and no internal vertex in $\zP_1 \cup \zP_2$.  Note that by the rendition of $(G, \Omega)$, $J$ has exactly two connected components.  For every $Q \in \zQ' \setminus \zQ^*$, again let $Q^*$ be the (unique) subpath with one endpoint in $S_1 \cup S_{10p(p-1) +3}$, one endpoint in $J$, and no internal vertex in $J \cup M$.  Let $x(Q^*)$ be the endpoint of $Q^*$ in $S_1 \cup S_{10p(p-1) + 3}$.  As $J$ has exactly two components, there exists a linkage $\zL$ of order $(|\zQ' \setminus \zQ^*|-2) /2$ contained in $\{Q^*: Q \in \zQ' \setminus \zQ^*\} \cup J$ formed by linking up pairs of elements of $\zQ' \setminus \zQ^*$ with a subpath in $J$.  In this way, we can ensure that for every $L \in \zL$ there exist $Q_1, Q_2 \in \zQ' \setminus \zQ^*$ such that the endpoints of $L$ are $x(Q_1)$ and $x(Q_2)$.  It follows that $\zL$ is a set of disjoint $M$-paths such that for every $L \in \zL$, both endpoints of $L$ are at distance at least $10p(p-1)$ from the boundary of $M$ by the choice of $\zQ'$ and the position of $S_1$ and $S_{10p(p-1)+3}$ in the construction of $M$.  Moreover, the rendition and the fact that the endpoints in $X_1' \cup X_2'$ of the paths $Q \in \zQ' \setminus \zQ^*$ are separated by at least $10p^2$ elements of $\zR$, we conclude that $\zL$ satisfies the second and third conditions in the statement of Lemma \ref{lem:meshandpaths} and as a result, $G$ contains a $K_p$ minor grasped by $M$ and thus, grasped by $\zC \cup \zP$, as desired.
\end{proof}

\section{Forcing a rendition of bounded width and depth}

In Section \ref{sec:GM9},  we showed that a society either has a rendition of width one and bounded depth or it must have a large transaction.  In this section, we prove an extension of this result and quantitatively describe the canonical obstructions to the existence of a rendition of bounded width and depth.  This is the main technical step in the proof of Theorem \ref{thm:main}.

We first define the following parameters.  Given fixed non-negative integers $p, r \ge 1$ and values $l,k,p \ge 0$, let
$$f_{r,p}(l,k,q) = (5 \cdot 10^{21} p^{100})^{k+q}\left [ (2k)^{k+q}(l + k^2 l) + r + q + k + 1\right ].$$

\begin{theorem}\label{thm:maintransstrong}
Let $k, p, q, r \ge 0$ and $l \ge 1$ be integers
and let $s,t$ be integers satisfying $s,t\ge f_{r,p}(l,k,q)$.
Let $(G, \Omega)$ be a society and $\rho$ be a cylindrical rendition of $(G, \Omega)$.
Let $\zC = (C_1,C_2, \dots, C_s)$ be a nest in $\rho$ of cardinality $s$ and $\zP$ an orthogonal radial linkage in $\rho$ of cardinality $t$.
Then one of the following holds.
\begin{itemize}
\item[\rm{(i)}] There exists a handle or crosscap transaction $\zQ$ in $(G, \Omega)$ of thickness $r$;
\item[\rm{(ii)}]  there exists a  $q$-nested crosses transaction $\zQ$  in $(G, \Omega)$;
\item[\rm{(iii)}] there exists  a $(k, l)$-leap pattern  $(\zP_0,\zQ)$  in $(G, \Omega)$;
\item[\rm{(iv)}] $G$ contains a $K_p$ minor grasped by $\zC \cup \zP$;
\item[\rm{(v)}] there exists a set $Z \subseteq V(G)\setminus V(\Omega)$ with $|Z| \le qk\left [ 12288p^{24}+ 368650^{2}p^{52}r+4k\right]$
such that $(G - Z, \Omega )$ has a rendition in the disk of breadth at most $q-1$ and depth at most
$f_{r,p}(l,k,q)$.
\end{itemize}
Moreover, in outcomes \rm{(i) -- (iii)}, the transaction $\zQ$ or linkage $\zP_0\cup\zQ$ is coterminal with $\zP$ up to level $f_{r,p}(l,k,q)$.
\end{theorem}

\begin{proof}
Assume the theorem is false, and pick a counterexample to minimize $k+q$.  If $k  = 0$, $q = 0$, or $r = 0$, then the empty graph is a transaction satisfying one of (i) -- (iii).  Thus, we have that $k, q, r \ge 1$.  We also have that $p\ge 3$, for otherwise (iv) holds.  Fix the cylindrical rendition $\rho$ in the unit disk $\Delta$.

We fix the following values.  Let
\begin{align*}
\mu &= max\{f_{r,p}(2kl, k-1, q), f_{r,p}(l,k, q-1)\}, \text{ and} \\
\beta &= 368650^2p^{52}.
\end{align*}
Note that $\beta \ge \left[(36864p^{24} + 4p^2) 10p^2\right]^2$ and
$$\mu \le (5 \cdot 10^{21} p^{100})^{k+q-1}\left [ (2k)^{k+q-1}(2kl + k^2 2kl) + r + q + k + 1\right ].$$  Define
\begin{align*}
m_1&=\beta\cdot (2m_2+2);\\
m_2&= 24576p^{24}(49152p^{24} + 2p^{2})m_3;\\
m_3&=2 \beta r  + 4k + 11 + 2 \mu + 32k^2l + 16 q + l.
\end{align*}

Our first claim is simply a bound on $f_{r,p}(l,k,q)$.  We include the proof for completeness.
\begin{claim}
\label{cl:fm1}
We have $f_{r,p}(l,k,q) \ge 6m_1$. 
\end{claim}
\begin{cproof}
We first claim that $6m_1 \le 2 \cdot 10^{21}p^{100}m_3$.  To see this:
\begin{align*}
m_2 &= 24576p^{24}(49152p^{24} + 2p^2) m_3 \\
  &\le 24576p^{24}(49153p^{24})m_3 \\
  &\le 1.21\cdot10^9p^{48}m_3.
\end{align*}
Thus,
\begin{align*}
6m_1 &= 12 \cdot 368650^2p^{52} m_2 + 12 \beta\\
& \le 12 \cdot 368650^2\cdot1.21\cdot10^9p^{100}m_3 + 6m_3 \\
    &\le 2 \cdot 10^{21} p^{100}m_3,
\end{align*}
as claimed, where we used the inequality $2\beta \le m_3$.
Thus, to prove the claim, it suffices to show that $$f_{r,p}(l,k,q) \ge 2 \cdot 10^{21}p^{100}m_3.$$
Expanding out the definition of $m_3$ gives
\begin{align*}
2 \cdot 10^{21}p^{100}m_3 &\le \left[ (4 \cdot 10^{21}p^{100})(5 \cdot 10^{21}p^{100})^{k+q-1} + 4 \cdot 10^{21}p^{100}368650^2p^{52} \right ] r\\
& + \left[ (4 \cdot 10^{21}p^{100})(5 \cdot 10^{21}p^{100})^{k+q-1}(2k)^{k+q}  + 2 \cdot 10^{21}p^{100}
\right ] l \\
& + \left [ (4 \cdot 10^{21}p^{100})(5 \cdot 10^{21}p^{100})^{k+q-1} + 8 \cdot 10^{21}p^{100} \right ] k \\
& + \left [ (4 \cdot 10^{21}p^{100})(5 \cdot 10^{21}p^{100})^{k+q-1} (2k)^{k+q} + 64 \cdot 10^{21}p^{100} \right] k^2 l \\
& + \left [ (4 \cdot 10^{21}p^{100})(5 \cdot 10^{21}p^{100})^{k+q-1} + 32 \cdot 10^{21}p^{100} \right ] q \\
& + \left [ (4 \cdot 10^{21}p^{100})(5 \cdot 10^{21}p^{100})^{k+q-1} + 22 \cdot 10^{21}p^{100} \right ] \\
& \le f_{r,p}(l,k,q),
\end{align*}
where the final inequality follows by considering the definition of $f$ term by term.
\end{cproof}

We break the proof of the theorem into a series of steps.  In Steps 1--5, we find a large planar transaction which is orthogonal to a significant portion of the nest.  Moreover, later when we apply the minimality of our counterexample, it will be important that the planar transaction is contained in a larger crooked transaction.


\noindent{\bf Step 1.}
If $(G, \Omega)$ has no cross, then by Theorem~\ref{thm:crossreduct} it satisfies (v) with $Z=\emptyset$.  Given a cross in $(G, \Omega)$,
since $f_{r,p}(l,k,q)\ge14$,
by Lemma~\ref{lem:crookedrooted} there is a cross in  $(G, \Omega)$ that is coterminal with $\zP$ up to level $C_{11}$.  Thus, (ii) holds in the case $q \le 1$ and (i) holds in the case $r\le1$.  We conclude that $q, r \ge 2$.


\noindent{\bf Step 2.}  There exists a crooked transaction $\zQ_1$ of cardinality $m_1$.  To see this, by
Lemma~\ref{lem:GM9} the society $(G, \Omega)$ has either a cylindrical rendition of depth $6m_1$,
or a crooked transaction of cardinality $m_1$.  In the former case (v) holds by Claim~\ref{cl:fm1}, because $q\ge2$. Note we are using the fact that $m_1 \ge 4$.


\noindent{\bf Step 3.}
We have $s\ge f_{r,p}(l,k,q)\ge 2m_1+7$ and $t\ge f_{r,p}(l,k,q)\ge  4m_1+6$.
By Lemma~\ref{lem:crookedrooted} applied to the society $(G, \Omega)$, nest $\zC$, rendition $\rho$, linkage $\zP$
and transaction $\zQ_1$, with $r$ replaced by $m_1$, we may assume that $\zQ_1$ is coterminal with $\zP$
up to level $C_{2m_1+7}$.  Moreover, as $\zQ_1$ is a crooked transaction, it must be unexposed in the rendition $\rho$.


\noindent{\bf Step 4.}
Since $m_2\ge r$ and $m_1 \ge (36864p^{24} + 4p^2) 10p^2 (36864p^{24} + 4p^2)10p^2(2m_2+2)$, by Lemma~\ref{lem:strongES}
the transaction $\zQ_1$ contains a transaction $\zQ_2$ which is either
\begin{itemize}
\item a planar transaction of cardinality $2m_2+2$, or
\item a crosscap transaction of cardinality $r$, or
\item  a  $2p^2$-nested crosses transaction, or
\item a  $2p^2$-twisted crosses transaction, or
\item a $(36864p^{24}, 10p^2)$-leap pattern, or
\item a twisted $(36864p^{24}, 10p^2)$-leap pattern.
\end{itemize}
As $\zQ_1$ is coterminal with $\zP$ up to level $2m_1 + 7$ and since $f_{r,p}(l,k,q)\ge 2m_1+7$, the second case satisfies outcome (i) of the theorem.

Since $f_{r,p}(l,k,q)\ge2m_1+6+4p^2$ and $f_{r,p}(l,k,q)\ge8p^2$,
by Corollary \ref{cor:nonestedcrosses}
applied to the nest $(C_{2m_1+7}, C_{2m_1+8},\ldots ,C_s)$,
 if either the third or fourth case holds, we would satisfy outcome (iv) of the theorem.
Since $f_{r,p}(l,k,q)\ge2m_1+6+20p^2$,
by Lemma \ref{lem:nolargeleap}
applied to the nest $(C_{2m_1+7}, C_{2m_1+8},\ldots ,C_s)$,
 if the fifth or sixth case holds, we would again satisfy outcome (iv) of the theorem.  We conclude that
$\zQ_2$  is an unexposed planar transaction of cardinality~$2m_2+2$.


\noindent{\bf Step 5.}
Let $X_1$ and $X_2$ be disjoint segments of $\Omega$ such that $\zQ_2$ is a linkage from $X_1$ to $X_2$.  Consider the elements of $\zQ_2$ ordered by the ordering of their endpoints in $X_1$.  Let $\zQ_2'$ be the linkage obtained from $\zQ_2$ by deleting the first and last $\lceil m_2/2 \rceil$ elements of $\zQ_2$.  Thus, $\zQ_2'$ is an unexposed planar transaction of order $m_2$.  Moreover, let $X_i'$ be the minimal segment of $\Omega$ contained in $X_i$ containing the endpoints of elements of $\zQ_2'$.  Let $Y'$ and $Y''$ be the two segments comprising $\Omega - (X_1' \cup X_2')$.  Observe that there exist $m_2$ elements of $\zP$ with an endpoint in $Y'$ and $m_2$ elements of $\zP$ with an endpoint in $Y''$.


For the remainder of the proof, fix
\begin{align*}
\alpha &= 12288(p(p-1))^{12}, \hbox{ and}\\
\gamma &= 2m_1 + 2\alpha +20p(p-1) +  19.
\end{align*}
In Steps 6 - 12, we find a bounded set of vertices $Z$ such that
$(G-Z, \Omega - Z)$ has a rendition $\rho_1$ in the disk $\Delta$ such that:
\begin{itemize}
\item the decomposition has exactly two vortices,
\item the cycles $C_{\gamma+1}, \dots, C_s$ form a nest around both vortices, and
\item each vortex is itself surrounded by a large nest.
\end{itemize}


\noindent{\bf Step 6.} Apply Theorem \ref{thm:planarstripjump} to the transaction $\zQ_2'$ in $(G, \Omega)$ with the rendition $\rho$,
$l'=m_2$, $k=2p^2$, $l=m_3$ and the nest $(C_{2m_1+7}, C_{2m_1+8}, \dots, C_s)$.
We can do so because $s-2m_1-6\ge f_{r,p}(l,k,q)-2m_1-6\ge 3m_1\ge2\alpha+20p(p-1)+4$ by Claim~\ref{cl:fm1}.
If we find the clique minor in the first outcome of Theorem \ref{thm:planarstripjump}, then we satisfy (iv).
If we find a $2p^2$ nested crosses transaction or twisted $2p^2$ nested crosses transaction, we satisfy outcome (iv) by Corollary \ref{cor:nonestedcrosses},
because $s\ge f_{r,p}(l,k,q)\ge 2m_1+6+\alpha+8$.  Thus, we may assume that we find $Z_1$, $|Z_1| \le \alpha$ and disjoint from the outer graph of $C_{\gamma-8}$ and linkage $\zQ_3$ of order $m_3$ contained in $\zQ_2'$ such that $\zQ_3$ is disjoint from $Z_1$ and the $\zQ_3$-planar strip society in $(G-Z_1, \Omega)$ is rural and isolated.


\noindent{\bf Step 7.}
We have $s\ge f_{r,p}(l,k,q)\ge\gamma$.
Let $(H_1, \Omega_1^H)$ be the inner society of $C_{\gamma}$ in $(G-Z_1, \Omega)$ with respect to the rendition $\rho$ restricted to $G-Z_1$.  Let $\Delta_\gamma$ be the closed subdisk bounded by the track of $C_\gamma$.  Let $\zQ_3^H$ be the restriction of $\zQ_3$ to $(H_1, \Omega_1^H)$ as defined in
Definition~\ref{def:restriction-transaction}.
Let $(J_1, \Omega_1^J)$ be the $\zQ_3^H$-strip society in $(H_1, \Omega_1^H)$.

By Lemma \ref{lem:planarstriprestriction} applied to the society $(G-Z_1, \Omega)$, nest $(C_{\gamma-7}, \dots, C_{\gamma})$, linkage $\zQ_3$ and cycle $C_\gamma$,
the society $(J_1, \Omega_1^J)$ is rural and isolated in $(H_1, \Omega_1^H)$.  There exists
 a vortex-free rendition $\rho_{J_1}$ of $(J_1, \Omega_1^J)$ in $\Delta_\gamma$
 such that
$\pi_\rho^{-1}(v)=\pi^{-1}_{\rho_{J_1}}(v)$ for all $v\in V(\Omega _1^J)$.
  Note that we are using the fact that $\zQ_3$ is planar to ensure that the cyclic ordering $\Omega_1^J$ is the same as the cyclic ordering of $V(\Omega_1^J)$ in $\Omega_1^H$.

Let the elements of $\zQ_3$ be enumerated $Q_1, Q_2, \dots, Q_{m_3}$ by the order in which their endpoints occur on $X_1'$.  Let $T'$ ($T''$) be the track of $Q_2 \cap \zQ_3^H$ ($Q_{m_3 -1} \cap \zQ_3^H$) in $\rho_{J_1}$.   There is a unique connected component of $\Delta_\gamma \setminus(T' \cup T'')$ whose boundary includes both $T'$ and $T''$; let $\Delta^{**}$ be the closure of this connected component.   Let $J_2$ be the subgraph formed by
$\bigcup \{\sigma_{\rho_{J_1}}(c)\,:\,c \in C(\rho_{J_1}), c \subseteq \Delta^{**}\}$ along with any vertices $v$ of $J_1$ such that $\pi^{-1}_{\rho_{J_1}}(v)$ exists and belongs to $\Delta^{**}$.  Let $\Omega_2^J$ be the cyclically ordered set of vertices with $V(\Omega_2^J) = V(\Omega_1^J) \cap V(J_2)$ with the cyclic order induced by $\Omega_{1}^J$.  Let $\rho_{J_2}$ be the restriction of $\rho_{J_1}$ to the disk $\Delta^{**}$.


\noindent{\bf Step 8.} We define the society $(G_1, \Omega_1)$ as follows.  Let $G_1$ be the union of $J_2$ and the outer graph of $C_\gamma$ with respect to $\rho$.  Let $\Omega_1 = \Omega$.
Thus, the union of $\rho_{J_2}$ along with the restriction of $\rho$ to the complement of the
interior of $\Delta_\gamma$
gives us a vortex-free rendition $\rho^*$ of $(G_1, \Omega_1)$ in the  disk $\Delta$.

Consider $\Delta_\gamma \setminus \Delta^{**}$.  There is one connected component which contains $T'$ in its boundary and one which contains $T''$ in its boundary.  Let $U'$ be the set of vertices $u \in V(G_1)$ such that
$\pi_{\rho*}^{-1}(u)$
exists and is contained in the boundary of the connected component of $\Delta_\gamma \setminus \Delta^{**}$ with $T'$ in the boundary.  Similarly, define $U''$ to be the set of vertices $u \in V(G_1)$ such that  $\pi_{\rho^*}^{-1}(u)$
exists and is contained in the boundary of the connected component of $\Delta_\gamma \setminus \Delta^{**}$ with $T''$ in the boundary.  Thus by construction, ($U', U'')$ is a partition of the set of vertices $\left( V(\Omega_1^H) \setminus V(\Omega_2^J)\right) \cup \pi_{\rho_{J_2}}((T' \cup T'') \cap N(\rho_{J_2}))$.

Let $L$ be the minimal subgraph of $G-Z_1$ such that $G-Z_1 = G_1 \cup L$.  Note that $L$ is a subgraph of $H_1$.  We claim as well that $V(L) \cap V(G_1) \subseteq U' \cup U''$.  To see this, consider a vertex $x \in V(L) \cap V(G_1)$.  By the minimality of $L$, the vertex $x$ cannot be deleted from $L$ and still have the property that $L \cup G_1 = G-Z_1$.  Thus, there exists an edge $xy$ incident to $x$ which is not contained in $G_1$ and consequently, $xy$ is in $H_1$, the inner graph of $C_\gamma$.  The vertex $x$ is either in the outer graph of $C_\gamma$ or in $J_2$.  If $x$ is in the outer graph of $C_\gamma$, but not in $J_2$, then $x \in V(\Omega_1^H)\setminus V(\Omega_2^J)$
 and thus $x \in U' \cup U''$.  Thus, we may assume $x \in V(J_2)$.  However, in this case as $(J_2, \Omega_2^J)$ is isolated in $(H_1, \Omega_1^H)$, it follows that $x \in \pi_{\rho_{J_2}} \left ( ( T' \cup T'') \cap N(\rho_{J_2}) \right)$.  We conclude that $x \in U' \cup U''$ as claimed.

Finally, fix closed disks $\Delta'$ and $\Delta''$ in $\Delta_\gamma \setminus \Delta^{**}$ such that $\Delta' \cap( bd(\Delta_\gamma) \cup T' ) = \pi_{\rho^*}^{-1}(U')$ and $\Delta'' \cap (bd(\Delta_\gamma) \cup T'') = \pi_{\rho^*}^{-1}(U'')$.


\noindent{\bf Step 9.}  Observe that for all $i$, $3 \le i \le m_3 -2$, the path $Q_i$ is contained in $G_1$
and that $s\ge f_{r,p}(l,k,q)\ge \gamma +m_3-2$.  The subgraph $Q_i \cup C_{\gamma + i}$ contains exactly two cycles which are distinct from $C_{\gamma + i}$.  If we consider the track in $\rho^*$ of each these two cycles, one bounds a disk in $\Delta$ which contains $\Delta'$ and the other bounds a disk which contains $\Delta''$.  We define $C_i'$ to be the cycle contained in $Q_i \cup C_{\gamma + i}$
and distinct from $C_{\gamma+i}$ such that the disk bounded by the track of $C_i'$ in $\rho^*$ contains $\Delta'$.  Similarly, we define $C_i''$ to be the cycle contained in $Q_{m_3 -i} \cup C_{\gamma + i}$
and distinct from $C_{\gamma+i}$ such that the disk bounded by the track of $C_i''$ in $\rho^*$ contains $\Delta''$.  Observe that $(C_3', C_4', C_5', \dots, C_{m_3 - 2}')$ forms a nest around $\Delta'$ and similarly, $(C_3'', C_4'', \dots, C_{m_3 -2}'')$ forms a nest around $\Delta''$.

Consider the segments $Y'$ and $Y''$ defined in Step 5.
As $\zQ_3$ is coterminal with $\zP$ up to level $2m_1+7$, every element of $\zP$ with an endpoint in $Y' \cup Y''$ is disjoint from the intersection of $\zQ_3$ with the outer graph of $C_{\gamma-1}$.  Thus, without loss of generality, we may assume that every element of $\zP$ with an endpoint in $Y'$ intersects $U'$ and similarly, every element of $\zP$ with an endpoint in $Y''$ intersects $U''$.  Let $\zP_1'$ be the linkage formed by the $Y' - U'$ paths contained in $\zP$.  Similarly, let $\zP_1''$ be the linkage formed by the $Y'' - U''$ paths contained in $\zP$.
Both $\zP_1'$ and $\zP_1''$ have order at least $m_2$ by the choice of $\zQ_2'$ in Step 5.  Moreover, $\zP_1'$ is orthogonal to $(C_3', C_4', \dots, C_{m_3 - 2}')$ and $\zP_1''$ is orthogonal to $(C_3'', C_4'', \dots, C_{m_3 -2}'')$.


\noindent{\bf Step 10.} Let $H_2'$ ($H_2''$) be the inner graph of $C_{\beta r + 3}'$ ($C_{\beta r + 3}''$) in $(G_1, \Omega_1)$ with respect to the rendition $\rho^*$.   Consider the subgraph $L_2 = \zP_1' \cup \zP_1'' \cup L \cup H_2' \cup H_2''$.  We claim there exists a set $Z_2$ of size at most $\beta r-1$ intersecting all $Y' - Y''$ paths in this subgraph.  Assume no such set $Z_2$ exists.  Then there exists a linkage $\zP^*$ of $\beta r$ disjoint $Y_1-Y_2$ paths in $L_2$.  Note that $\zP^*$, by construction, is coterminal with $\zP$ up to level $\gamma + \beta r + 4$ for the nest $(C_1, C_2, \dots, C_s)$ in $(G, \Omega)$ and that $s\ge f_{r,p}(l,k,q)\ge\gamma+\beta r+3+20p^2$.
Recall  that $\beta \ge \left[(36864p^{24} + 4p^2) 10p^2\right]^2$.
 Apply Lemma~\ref{lem:strongES} to the linkage $\zP^*$.  Lest we have a $2p^2$-nested crosses transaction,
a $2p^2$-twisted nested crosses transaction, a $(36864p^{24},10p^2)$-leap pattern, or a
 $(36864p^{24},10p^2)$-twisted leap pattern,
  which via Corollary \ref{cor:nonestedcrosses} or Lemma \ref{lem:nolargeleap}
applied to the nest $(C_{\gamma+\beta r+4},C_{\gamma+\beta r+5},\ldots ,C_{s})$
would imply the graph satisfies (iv), we can assume that $\zP^*$ contains either a crosscap transaction of order $r$ or a planar transaction of order $r$.  If it is a crosscap transaction, we satisfy (i).
Since $\beta r +r+3<m_3-\beta r-3$,
if it is a planar transaction, then along with $Q_{\beta r+4}, Q_{\beta r + 5}, \dots, Q_{\beta r + r+3}$, we have a handle transaction again satisfying (i).  We conclude that no such linkage exists and therefore $L_2$ contains the desired set $Z_2$ hitting all $Y' - Y''$ paths.  Assume that we pick the set $Z_2$ to be minimal (by containment) with the property that it intersects all $Y' - Y''$ paths in the graph $L_2$.


\noindent{\bf Step 11.}  Fix $T_2'$ and $T_2''$ to be the track of $C_{\beta r + 3}'$ and $C_{\beta r + 3}''$
in $\rho^*$, respectively.  Define $U_2' =\pi_{\rho^*}(T_2'\cap N(\rho^*))$
and $U_2''= \pi_{\rho^*}(T_2''\cap N(\rho^*))$.
Then $U_2' \subseteq V(H_2')$ and $U_2'' \subseteq V(H_2'')$.    We claim that there does not exist a path from $U_2'$ to $U_2''$ in $(H_2' \cup H_2'' \cup L)$ with no internal vertex in $Z_2$.  Assume otherwise, and let $R$ be such a path.
As $2(\beta r+3)\le m_3$, the graph
$H_2'$ is vertex-disjoint from $H_2''$, and  it follows that the path $R$ uses at least one vertex of $L$.  Thus, $R$ intersects each of the cycles $C_{3}', \dots, C_{\beta r + 2}'$ in an internal vertex and similarly, each of the cycles $C_{3}'', \dots, C_{\beta r + 2}''$.  Given the size of $Z_2$, there exists an index $i'$, and $i''$ such that $C_{i'}'$ and $C_{i''}''$ are disjoint from $Z_2$ and $3 \le i', i'' \le \beta r + 2$.  Also,
since $m_2\ge \beta r$
there exist $P_1 \in \zP_1'$ and $P_2 \in \zP_1''$ which are disjoint from $Z_2$.  Thus, $P_1 \cup P_2 \cup R \cup C_{i'}' \cup C_{i''}''$ contains a path from $Y'$ to $Y''$ which is disjoint from $Z_2$, a contradiction.

It follows as well that $Z_2$ is contained in $(H_2' \cup H_2'' \cup L)- ( U_2' \cup U_2'' )$.  To see this, assume there exists $z \in Z_2 \setminus( V(H_2' \cup H_2'' \cup L) \setminus (U_2'\cup U''_2))$.  By minimality, there exists a path $R'$ from $Y'$ to $Y''$ in $H_2' \cup H_2'' \cup L \cup \zP_1' \cup \zP_1''$ which intersects $Z_2$ exactly in the vertex $z$.  It follows that $R'$ has a subpath in $H_2' \cup H_2'' \cup L$ from $U_2'$ to $U_2''$ with no internal vertex in $Z_2$, contrary to the argument above.


\noindent{\bf Step 12.}  We are now ready to define the desired rendition $\rho_1 = (\Gamma_1, \zD_1)$ of $(G - (Z_1 \cup Z_2), \Omega)$.  Let $\Delta_2'$ ($\Delta_2''$) be the closed disk bounded by $T_2'$ ($T_2''$).  To define $\rho_1$, we essentially let $\rho_1$ be the restriction of $\rho^*$ outside $\Delta_2'$ and $\Delta_2''$, and then partition the remaining components of $(H_2' \cup H_2'' \cup L )- Z_2$ between $\Delta_2'$ and $\Delta_2''$, depending on whether they attach to $U_2'$ or $U_2''$.  Explicitly, let $K'$ be the union of all components of $(H_2' \cup H_2'' \cup L) - Z_2$ which contain a vertex of $U_2'$ and let $K''$ be the union of all components of $(H_2' \cup H_2'' \cup L) - Z_2$ which do not contain a vertex of $U_2'$.
Let $\rho^*=(\Gamma^*,\zD^*)$.
Define $\zD_1 = \{D \in \zD^*: \hbox{int}(D) \subseteq \Delta \setminus (\Delta_2' \cup \Delta_2'')\} \cup \{\Delta_2', \Delta_2''\}$.  The drawing $\Gamma_1$ is obtained from the restriction of $\Gamma^*-Z_2$ to $\Delta \setminus (\hbox{int}(\Delta_2') \cup \hbox{int}(\Delta_2''))$ along with an arbitrary drawing of $K'$ in $\Delta_2'$ such that the only points on the boundary of $\Delta_2'$ are exactly the vertices of $U_2'\cap V(K')$ and similarly, an arbitrary drawing of $K''$ in $\Delta_2''$ such that the only points on the boundary of $\Delta_2''$ are the vertices of $U_2''\cap V(K'')$.

Thus $\rho_1$ is the desired rendition of $G- (Z_1 \cup Z_2)$ of breadth two with the two vortices formed by
$\Delta_2' \setminus N(\rho_1)$ and $\Delta_2'' \setminus N(\rho_1)$.  The ordered set of cycles $(C_{\gamma + \beta r + 4}, C_{\gamma + \beta r + 5}, \dots, C_s)$ forms a nest around both $\Delta_2'$ and $\Delta_2''$.  Moreover, $(C_{\beta r + 4}', C_{\beta r + 5}', \dots, C_{m_3- \beta r + 4}' )$ forms a nest around $\Delta_2'$ and
$(C_{\beta r + 4}'',\allowbreak C_{\beta r + 5}'', \dots, C_{m_3 - \beta r + 4}'')$ forms a nest around $\Delta_2''$.


We pause for a moment to discuss how the rest of the proof will proceed.  Consider the inner society $(G', \Omega')$ of $C_{\beta r + 4 + \mu}'$ under the rendition $\rho_1$ of $(G - (Z_1 \cup Z_2), \Omega)$.  Similarly, let $(G'', \Omega'')$ be the inner society of $C_{\beta r + 4 + \mu}''$.  First, observe that if both $(G', \Omega')$ and $(G'', \Omega'')$ are rural, then we have derived a contradiction as $(G-(Z_1 \cup Z_2), \Omega)$ is rural as well.  Thus, we may assume that one or the other (or both) has a cross.

If we restrict $\rho_1$ to $(G', \Omega')$ and $(G'', \Omega'')$, we get a cylindrical rendition of each society with a nest of order $\mu$.  Moreover, if we restrict $\zP_1'$ and $\zP_1''$ to $(G', \Omega')$ and $(G'',\Omega'')$, respectively, we get a large linkage orthogonal to the nest.  Consider the case when both $(G', \Omega')$ and $(G'', \Omega'')$ have a cross.  By the minimality of our counterexample, one of the outcomes (i) -- (v) holds for $(G', \Omega')$ with parameters $r, p,(q-1), k, l$.  If (i), (iii), or (iv) holds, then it holds for $(G, \Omega)$ as well, a contradiction.  If (ii) holds for $(G', \Omega')$, then along with a cross in $(G'', \Omega'')$, we can find a $q$-nested crosses transaction in $(G, \Omega)$ satisfying (ii), a contradiction.  We conclude that both $(G', \Omega')$ and $(G'', \Omega'')$ must contain the rendition in (v).  Although some care must be taken to ensure that the desired bounds on the parameters are satisfied, we are able to put the renditions together to get a rendition of the original society $(G, \Omega)$ after deleting a bounded set of vertices.

Thus, we are left with the case where one of $(G', \Omega')$ and $(G'', \Omega'')$ is rural and one is not.  Say $(G'', \Omega'')$ is rural.  It now follows that $Z_1 \cup Z_2$ must be non-empty; we specifically chose $\zQ_3'$ to be a sublinkage of crooked transaction and so if $(G'',\Omega'')$ is rural, it must be the case that in the original graph $(G, \Omega)$ there is some path which leaps ``over" the linkage $\zQ_3'$ and which intersects the set $Z_1 \cup Z_2$.  As above, by minimality, we find one of the outcomes for $(G', \Omega')$ which would also be a valid outcome for $(G, \Omega)$, or we find a $(k-1, l)$-leap pattern in $(G', \Omega')$.  To complete the proof, we use the results of Section \ref{sec:augmenting} to find a $(k, l)$-leap pattern in $(G, \Omega)$ by adding a subset of the paths in $\zQ_3'$ along with a path jumping from $V(G'')$ to $V(G')$ through the set $Z_1 \cup Z_2$.

However, in order for this outline to work, we will need to further refine our rendition $\rho_1$ to find a larger subset $Z \supseteq Z_1 \cup Z_2$ and rendition $\rho_2$ of $(G-Z, \Omega)$ which will allow us to extend a $(k-1, l)$-leap pattern on one side.  We define the set $Z$ of vertices to be deleted and the rendition $\rho_2$ in Steps 13 - 16, and then fix the definition of the societies $(G', \Omega')$ and $(G'', \Omega'')$ before we complete the proof following the steps in the outline above.


\noindent{\bf Step 13.}  Fix the vertex $z$ as follows for the remainder of the proof.  If $Z_2 \neq \emptyset$, let $z$ be an arbitrary vertex of $Z_2$.  If $Z_2 = \emptyset$ and $Z_1 \neq \emptyset$, let $z$ be an arbitrarily chosen vertex of $Z_1$.  If $Z_1 \cup Z_2 = \emptyset$, the vertex $z$ is not defined.

Let $T_3'$ ($T_3''$) be the track of $C_{\beta r + 2k + 4}'$ ($C_{\beta r + 2k + 4}''$) in the rendition $\rho_1$.
Let $J_3'$ be the inner graph of $C_{\beta r + 2k + 4}'$ and $J_3''$ the inner graph of $C_{\beta r + 2k + 4}''$ for the rendition $\rho_1$ of $(G-(Z_1 \cup Z_2), \Omega)$.
Since $m_3\ge2(\beta r+2k+4)$, the graphs $J_3'$ and $J_3''$ are vertex-disjoint.
Let $\zP_2'$ ($\zP_2''$) be the set of all $Y' - U_2'$ ($Y'' - U_2''$) paths contained in $\zP_1'$ ($\zP_1''$).  If the vertex $z$ is not defined, let $H_3'$ be the graph $J_3' \cup \zP_2'$.  If $z$ is defined, let $H_3'$ be the graph formed by $J_3' \cup \zP_2'$ along with all the edges of $G$ with one endpoint equal to $z$ and the other endpoint in $V(\sigma_{\rho_1}(\Delta_2'\setminus N(\rho_1)))$.  Similarly, let $H_3''$ be  the graph $J_3'' \cup \zP_2''$
 along with all the edges from $z$ to $V(\sigma_{\rho_1} (\Delta_2'' \setminus N(\rho_1)))$ in the case when $z$ is defined.


\noindent{\bf Step 14.}  Assume the vertex $z$ is defined.  We claim that there exists a path in $H_3'$ from $z$ to $Y'$.  Similarly, in $H_3''$, there exists a path from $z$ to $Y''$. It suffices to prove the claim that there exists a path from $z$ to $Y'$ in $H_3'$.  There are two possible cases to consider: $z \in Z_2$ or $z \in Z_1$.

Assume $z \in Z_2$.  Given the construction of the graph $L_2$ in Step 10, we observe that $L_2 - Z_2$ is a subgraph of $H_3' \cup H_3''$.  By the minimality of $Z_2$, there exists a path in $L_2$ from $Y'$ to $Y''$ in $L_2$ intersecting $Z_2$ exactly in the vertex $z$; let $P$ be the subpath from $Y'$ to $z$ contained in this path.  By construction, $P-z$ is contained in $H_3'$.  To see this, let $y$ be the vertex on $P$ adjacent to $z$.
Since $z\in V(H_2' \cup H_2'' \cup L)\setminus (U_2'\cup U_2'')$ by Step 11, the vertex $y$ is contained in a component of $(H_2' \cup H_2'' \cup L) -Z_2$ which intersects $U_2'$.  Consequently, $y \in V(\sigma_{\rho_1} (\Delta'_2\setminus N(\rho_1)))$ and $P$ is a subgraph of $H_3'$ as desired.

The more difficult case is when we assume $z \in Z_1$ and consequently, $Z_2 = \emptyset$.  Recall the definition of $Z_1$ from the application of Theorem \ref{thm:planarstripjump} in Step 6.  Let $(J_3, \Omega_3^J)$ be the $\zQ_3$-strip society in $(G, \Omega)$.  Outcome (iii) of Theorem \ref{thm:planarstripjump} ensures that there exists a path $P$ in $G-V(J_3)$ connecting the two segments of $\Omega \setminus V(\Omega ^J_3)$ which intersects the set $Z_1$ exactly in the vertex $z$.

Since $Z_1$ belongs to the inner graph of $C_{\gamma -8}$ in the rendition $\rho$ it follows by considering
the rendition $\rho $ that both components $P',P''$ of $P\setminus z$ intersect $C_{\gamma+\beta r+3}$
and that the neighbors of $z$ in $P$ belong to the inner graph of $C_{\gamma -7}$ in the rendition $\rho $.
Let $y'\in V(P')$ and $y''\in V(P'')$ be the neighbors of $z$ in $P$.
We may assume that $P'$ ends in the segment of $V(\Omega )\setminus V(\Omega _3^J)$ containing $Y'$ and
that $P''$ ends in the segment of $V(\Omega )\setminus V(\Omega_3^J)$ containing $Y''$.
Since $y',y''$  belong to the inner graph of $C_{\gamma -7}$ in the rendition $\rho $ and $P$ is
disjoint from the $\zQ_3$-strip society, and hence from $J_2$, it follows that $y',y''\in V(L)$.
Since $P'$  intersects $C_{\gamma+\beta r+3}$ and is disjoint from the $\zQ_3$-strip society in $(G,\Omega)$,
it follows that $P'$ intersects $U_2'$, and, since $Z_2=\emptyset$, that the subpath of $P'$ from $y'$ until the first intersection with
$U_2'$ is a subgraph of $H_2'\cup L$.
It follows that
$y'\in V(\sigma_{\rho_1} (\Delta '_2\setminus N(\rho_1)))$.
Thus $P\cup C_{\gamma+\beta r+3}\cup \zP_2'$ includes a path from $z$ to $Y'$ in $H_3'$, as desired.


\noindent{\bf Step 15.}
Define $U_3'$ to be $\pi_{\rho_1}(T_3' \cap N(\rho_1))$ and similarly, $U_3'' = \pi_{\rho_1}(T_3'' \cap N(\rho_1))$.
 Assume that the vertex $z$ is defined.
If there exist $2k$ paths in $H_3'$ from $z$ to $Y'$, disjoint except for $z$, then let $B'=V(H_3')$, $A'=\{z\}$,
and let $W'$ be a set of neighbors of $z$ in $H_3'$ of size $2k-1$ such that there exists a set $\zR'$ of  $2k-1$ disjoint $W'-Y'$ paths in
$H_3'$. Otherwise we proceed as follows.
We apply Lemma \ref{lem:sep1-2} to the graph $H_3'$ with $X$
the set of neighbors of $z$ and $Y = Y'$ to find a separation $(A', B')$ of order at most $2k-1$ satisfying either outcome (i) or (ii). We may assume that $z\in A'$.  Recall $J_3'$ is the inner graph of $C_{\beta r + 2k + 4}'$ and observe that by construction, $C_{\beta r + 4}', C_{\beta r + 5}', \dots, C_{\beta r + 2k + 3}'$ are contained in $J_3'$.  We claim that $A' \cap B' \subseteq V(J_3') \setminus U_3'$.  As $(A', B')$ has order at most $2k-1$, there exists an index $i$, $\beta r + 4 \le i \le \beta r + 2k + 3$ such that $C_i'$ is disjoint from $A' \cap B'$.  Moreover, as $\zP_2'$ has at least $m_2 > 2k$ disjoint paths from $Y'$ to $C_i'$, we conclude that $V(C_i') \subseteq B'\setminus A'$.  Since $V(C_i')$ forms a cutset separating $z$ from every vertex not contained in the inner graph of $C_i'$ with respect to $\rho_1$ and the fact that every vertex of $A' \cap B'$ is connected by a path in $H_3'[A']$ to $z$, we conclude that $A' \cap B'$ is contained in the vertex set of the inner graph of $C_i'$.  It follows that $(A', B')$ satisfies outcome (i) in Lemma \ref{lem:sep1-2} and $|A' \cap B'| = 2k-1$.  We fix $W'$ to be a set of $2k-1$ vertices in $B' \setminus A'$ such that every vertex of $W'$ has a neighbor in $A' \cap B'$ and a linkage $\zR'$ of order $2k-1$ from $W'$ to $Y'$ in $H_3'$.  The set $W'$ and linkage $\zR'$ can be found by taking a linkage of size $2k-1$ from $A \cap B'$ to $Y'$ in $H_3'[B']$ and deleting the vertices $A' \cap B'$.  By construction, $W' \subseteq V(J_3')$.

Similarly, we find $(A'', B'')$ a separation of $H_3''$ separating $z$ and $Y''$ of order either exactly $2k-1$ or exactly one such that every vertex of $A'' \cap B''$ is linked by a path to $z$ in $H_3''[A'']$ and $A''\cap B'' \subseteq V(J_3'') \cup \{z\})$, and we find a set $W'' \subseteq V(J_3'') \setminus A''$ of $2k-1$ vertices each adjacent to an element of $A'' \cap B''$ and a linkage $\zR''$ of order $2k-1$ from $W''$ to $Y''$ in $H_3''$.  Note that both $\zR'$ and $\zR''$ are coterminal with $\zP$ up to level $C_{\gamma + \beta r + 2k + 5}$ by the construction of $W_3'$ and $W_3''$.

When $z$ is not defined, fix $B' = V(J_3')$, $A' = \emptyset$ and $B'' = V(J_3'')$ and $A'' = \emptyset$.
The sets $W'$and $W''$ are not defined in this case.


\noindent{\bf Step 16.}   Define our new rendition $\rho_2=(\Gamma _2,\zD_2)$ in $\Delta$ of the society $(G-(Z_1 \cup Z_2 \cup A' \cup A''), \Omega)$ as follows.  Let $\Delta_3'$ and $\Delta_3''$ be the closed disks bounded by $T_3'$ and $T_3''$, respectively.  Let $\Gamma_2$ be formed by the restriction of $\Gamma_1$ to
$\Delta \setminus (\hbox{int}(\Delta_3') \cup \hbox{int}(\Delta_3''))$ along with an arbitrary drawing of $J_3'[B']$ in $\Delta_3'$ which intersects the boundary of $\Delta_3'$ exactly in the vertices $U_3'$ and, similarly, an arbitrary drawing of $J_3''[B'']$ in $\Delta_3''$ which intersects the boundary of $\Delta_3''$ exactly in the vertices $U''_3$.  Let $\zD_2 = \{D \in \zD_1 : D \subseteq \Delta \setminus  (\hbox{int}(\Delta_3') \cup \hbox{int}(\Delta_3''))\} \cup \{\Delta_3', \Delta_3''\}$.  Note that if $z$ is defined, then $W' \subseteq V(\sigma_{\rho_2}(\Delta_3' \setminus N(\rho_2)))$ and similarly, $W'' \subseteq V(\sigma_{\rho_2}(\Delta_3'' \setminus N(\rho_2))$.


\noindent{\bf Step 17.}  Define
$$\gamma_2 = \beta r + 2k + 4 + \mu + 4(4k^2l+2q) + 1.$$
Since $s\ge f_{r,p}(l,k,q)\ge6m_1\ge\gamma +\gamma_2$ by Claim~\ref{cl:fm1}, we may
fix $T_4'$ and $T_4''$ to be the track of $C_{\gamma_2}'$ and $C_{\gamma_2}''$, respectively in the rendition $\rho_2$.  Define $(H_4', \Omega_4')$ and $(H_4'', \Omega_4'')$ to be the inner society of $C_{\gamma_2}'$ and $C_{\gamma_2}''$, respectively, in the rendition $\rho_2$ of $(G-(Z_1 \cup Z_2 \cup A' \cup A''), \Omega)$.  Thus, $V(\Omega_4') = \pi_{\rho_2}(T_4' \cap N(\rho_2))$ and $V(\Omega_4'') = \pi_{\rho_2}(T_4'' \cap N(\rho_2))$.
Since $m_3=2\gamma_2 + l + 1$, the graphs $H_4'$ and $H_4''$ are vertex-disjoint.
Note that each of $(H_4', \Omega_4')$ and $(H_4'', \Omega_4'')$ has a cylindrical rendition obtained from a restriction of $\rho_1$.  Each society also has a nest of depth $\mu + 4(4k^2l+2q)$, namely $\zC' = (C_{ \beta r + 2k + 5}', C_{\beta r + 2k + 6}', \dots, C_{\gamma_2 - 1}')$ in $(H_4', \Omega_4')$ and $\zC'' = (C_{ \beta r + 2k + 5}'', C_{\beta r + 2k + 6}'', \dots, C_{\gamma_2 - 1}'')$ in $(H_4'', \Omega_4'')$.  Let $\zP_3'$ be the linkage of $V(\Omega_4')-U_3'$ paths contained in $\zP_2'$.  Similarly, let $\zP_3''$ be the linkage of $V(\Omega_4'') - U_3''$ paths contained in $\zP_2''$.  Note that $\zP_3'$ and $\zP_3''$ are orthogonal to the nests $\zC'$ and $\zC''$ and each have order at least $m_2$.  Finally, let $\Delta_4'$ be the closed disk bounded by $T_4'$ and $\Delta_4''$ the closed disk bounded by $T_4''$.


\noindent{\bf Step 18.} Assume as a case that $(H_4', \Omega_4')$ and $(H_4'', \Omega_4'')$ are both rural.  Fix vortex-free renditions $\rho'$ and $\rho''$ of $(H_4', \Omega_4')$ and $(H_4'', \Omega_4'')$ in $\Delta_4'$ and $\Delta_4''$, respectively.  By joining the renditions $\rho'$ and $\rho''$ to the restriction of $\rho_2$ to $\Delta \setminus (\hbox{int}(\Delta_4') \cup \hbox{int(}\Delta_4''))$, we get a vortex-free rendition of $(G-(A' \cup A'' \cup Z_1 \cup Z_2), \Omega)$.  Note that every vertex of $A' \setminus ((A' \cap B') \cup\{ z\})$ has every neighbor contained in the set $A' \cup Z_1 \cup Z_2$,
 and the analogous statement holds for $A'' \setminus ((A'' \cap B'') \cup \{z\})$.  Thus, by adding the isolated subgraphs $G[A' \setminus ((A' \cap B') \cup\{ z\})]$ and $G[A'' \setminus ((A'' \cap B'') \cup \{z\})]$ to an arbitrary cell of the rendition, we get a vortex-free rendition of $(G-(Z_1 \cup Z_2 \cup (A' \cap B') \cup (A'' \cap B'')), \Omega)$, satisfying (v).


\noindent{\bf Step 19.} Assume as a case that $(H_4', \Omega_4')$ and $(H_4'', \Omega_4'')$ are both not rural.  Fix $q'$ ($q''$) to be the largest value such that there exists a $q'$- ($q''$-)nested crosses transaction in $(H_4', \Omega_4')$ ($(H_4'', \Omega_4'')$) coterminal with $\zP_3'$ ($\zP_3''$) up to level $C_{\beta r + 2k + 4 + \mu}'$ ($C_{\beta r + 2k + 4 + \mu}''$).  Let $\zR'$ and $\zR''$, respectively, be the $q'$- and $q''$-nested crosses transactions.
Given that a cross is a crooked transaction of order two
and that $s\ge  f_{r,p}(l,k,q)\ge6m_1\ge\gamma + \beta r + 2k + 4+11$ and $m_3\ge14$,
Theorem \ref{thm:crossreduct} and Lemma \ref{lem:crookedrooted} imply that both $q', q'' \ge 1$.

We want to extend $\zR'$ and $\zR''$ to transactions in the original society $(G, \Omega)$.  Let $\overline{\zP_3'}$, ($\overline{\zP_3''}$) be the $Y' - V(C_{\gamma_2 -2}')$ ($Y' - V(C_{\gamma_2 - 2}'')$) linkage contained in $\zP_1'$, ($\zP_1''$).  There exists a (uniquely defined) $q'$- ($q''$)-nested crosses transaction $\overline{\zR'}$ ($\overline{\zR''}$)
in $(G,\Omega )$
 contained in $\zR' \cup \overline{\zP_3'}$ ($\zR'' \cup \overline{\zP_3''}$).  Note that we are using the fact that $\zR'$ ($\zR''$) is co-terminal with $\zP_3'$ ($\zP_3''$) up to level  $C_{\gamma_2 - 4(4k^2l+2q)}'$ and $C_{\gamma_2- 4(4k^2l+2q)}''$, respectively.  Thus, we see that both $q' < q$ and $q'' < q$,
for otherwise statement (ii) of the theorem holds, because $f_{r,p}(l,k,q)\ge \gamma +\beta r + 2k + 4 + \mu$.

It need not be the case that $\overline{\zR'} \cup \overline{\zR''}$ is a $q' + q''$-nested crosses transaction.
In fact, it need not be a transaction.
However, given that both $\zR'$ and $\zR''$ have at most $2(q -1)$ elements,
$4(4k^2l+2q)\ge8(q-1)+2$ and $m_2\ge8(q-1)$,
by applying Lemma \ref{lem:transrotation2} to
the nest $(C'_{\beta r + 2k + 4 + \mu},C'_{\beta r + 2k + 5 + \mu},\allowbreak\ldots ,C'_{\gamma_2-1})$, linkage $\zP_3'$
and transaction
$\zR'$
in $(H_4', \Omega_4')$ and the analogous quantities in $(H_4'', \Omega_4'')$, we see that there exists a $q'$-nested cross transaction $\zQ_5'$ in $(H_4', \Omega_4')$ coterminal with $\zP_3'$ up to level $C_{\gamma_2 - 4(4k^2l) - 7}$ and a $q''$-nested crosses transaction $\zQ_5''$  in $(H_4'', \Omega_4'')$ which is coterminal with $\zP_3'$ up to level $C_{\gamma_2 -  4(4k^2l) - 7}$ such that the transactions $\overline{\zQ_5'}$ and $\overline{\zQ_5''}$ in $(G,\Omega )$
defined analogously to the previous paragraph have the property that $\overline{\zQ_5'} \cup \overline{\zQ_5''}$ is a $q' + q''$-nested crosses transaction.
Furthermore, $\overline{\zQ_5'} \cup \overline{\zQ_5''}$ is coterminal with $\zP$ up to level
$\gamma +\beta r + 2k + 4 + \mu + 8q-6\le 6m_1\le f_{r,p}(l,k,q)$. Since we are assuming that (ii) of the present theorem does not hold,
we conclude that that $2 \le q' + q'' \le q-1$ and $q' + 1, q'' + 1 < q$.

By the choice of our counterexample to minimize $q+k$, it follows that one of (i) -- (v) must hold for $(H_4', \Omega_4')$ with the parameters $p, r, q'+1, k, l$ and for $(H_4'', \Omega_4'')$ with the parameters $ p, r, q''+1, k, l$.  Note that we are using the choice of $\mu \ge f_{r,p}(l,k,q-1)$.  Observe that any transaction satisfying (i) -- (iii) can be extended to a transaction in $(G, \Omega)$ coterminal with $\zP$
up to level $\gamma +f_{r,p}(l,k,q-1)\le \gamma +\mu\le 6m_1\le f_{r,p}(l,k,q)$
by extending along subpaths of $\zP$.    Moreover, a clique minor satisfying (iv) will satisfy (iv) for $(G, \Omega)$ with the nest $\zC$ and linkage $\zP$.  Thus, by the definition of $q'$ and $q''$, we conclude that outcome (v) must hold for each of $(H_4', \Omega_4')$ and $(H_4'', \Omega_4'')$.  Thus, there exist $Z' \subseteq V(H_4') \setminus V(\Omega_4')$ and $Z'' \subseteq V(H_4'') \setminus V(\Omega_4'')$ such that
$|Z'|\le (12288p^{24} + \beta r + 4k)q'k$.  $|Z''|\le (12288p^{24} + \beta r + 4k)q''k$ and
$(H_4' - Z', \Omega_4')$ ($(H_4'' - Z'', \Omega_4'')$) has a rendition $\rho'$ in $\Delta_4'$ ($\rho''$ in $\Delta_4''$) of breadth at most $q'$ ($q''$) and depth at most $f_{r,p}(l,k,q-1))$.  Note that $|Z' \cup Z''| \le (12288p^{24} + \beta r + 4k)(qk-1)$.  Fix $Z = Z' \cup Z'' \cup Z_1 \cup Z_2 \cup (A' \cap B') \cup (A'' \cap B'')$.  We conclude that $(G - Z, \Omega)$ has a rendition in $\Delta$ satisfying (v) by restricting $\rho_2$ to $\Delta \setminus (\hbox{int}(\Delta_4') \cup \hbox{int}(\Delta_4''))$ and using $\rho'$ and $\rho''$ in the disks $\Delta_4'$ and $\Delta_4''$, respectively.  Note, that as in Step 18, we must add the isolated subgraphs $G[A' \setminus ((A' \cap B') \cup \{z\})]$ and $G[A'' \setminus ((A'' \cap B'') \cup \{z\})]$ to an arbitrarily chosen cell of the rendition.  This contradiction completes the analysis of the case.


\noindent{\bf Step 20.}
In preparation for the last step of the proof we claim that if $(H_4', \Omega_4')$ or $(H_4'', \Omega_4'')$ is rural,
then $Z_1\cup Z_2\ne\emptyset $. Indeed, to prove the claim we may assume from the symmetry that
$(H_4'', \Omega_4'')$ is rural and we may assume for a contradiction that $Z_1=Z_2=\emptyset$.
 Let $X_3'$ and $X_3''$ be the minimal segments
of $\Omega$ such that $Q_2,Q_3,\ldots,Q_{m_3-1}$ is a linkage from $X_3'$ to $X_3''$.
Let $Y_3',Y_3''$ be the two maximal segments of $\Omega - (X_3' \cup X_3'')$ such that $Y'\subseteq Y_3'$
and $Y''\subseteq Y_3''$.
Since $\zQ_3$ is a subset of the crooked transaction $\zQ_1$, there exists a path $P_1\in \zQ_1$ with an endpoint in
$Y''_3$ that is crossed by another path  $P_2\in\zQ_1$.
If $P_1$ crosses $Q_i$ for some $i\in\{1,2,\ldots ,m_3\}$, then, since the $\zQ_3$-strip society in $(G,\Omega )$
is isolated, it follows that the other endpoint of $P_1$ is in $Y_3'$.
Consequently, by considering the rendition $\rho_1$, we deduce that $P_1$ intersects both $U_2'$ and $U_2''$.
Thus some subpath of $P_1$ joins $U_2'$ to $U_2''$ in $H_2'\cup H_2''\cup L$, contrary to Step~11.
This proves that $P_1$ crosses no $Q_i$ and the same holds for $P_2$. It follows that $P_1$ and $P_2$
have both endpoints in $Y_3''$. Since $Z_1=Z_2=\emptyset$ it follows that $\rho_2$ is a rendition of $(G,\Omega )$.
Thus some subpaths of $P_1$ and $P_2$ form a cross in $(H_4'', \Omega_4'')$, a contradiction.
This proves the claim  that if $(H_4', \Omega_4')$ or $(H_4'', \Omega_4'')$ is rural,
then $Z_1\cup Z_2\ne\emptyset $.


\noindent{\bf Step 21.}  The final case to consider, up to symmetry, is when $(H_4', \Omega_4')$ is not rural and $(H_4'', \Omega_4'')$ is rural.  By the minimality of $k+q$, one of the outcomes (i) -- (v) must hold for $(H_4', \Omega_4')$, $\zP_3'$, the nest $(C_{\beta r + 2k +5,}' \dots, C_{\gamma_2 - 1}')$, and parameters $p$, $r$, $k-1$, $2kl$, and $q$.
Let us note that  $\gamma +f_{r,p}(2kl,k-1,q)\le \gamma +\mu\le 6m_1\le f_{r,p}(l,k,q)$.
 By extending $\zP_3'$ along $\zP_1'$, if we find a transaction satisfying (i) or (ii), then there exists a transaction in the original society $(G, \Omega)$ satisfying  the same outcome.  If there exists a model of a clique minor satisfying (iv), then the same outcome will be satisfied for $(G, \Omega)$ and the original nest and $\zP$ due to the construction of the cycles $C_i'$ and the linkage $\zP_1'$.
If there exists a rendition satisfying outcome (v), then outcome (v) is satisfied for the original society as well.  The construction is very similar to the previous case.  There exists $Z' \subseteq V(H_4') \setminus V(\Omega_4')$ such that $(H_4' - Z', \Omega_4')$ has a rendition $\rho'$ in $\Delta_4'$ of breadth at most $q$ and depth at most $f_{r,p}(2kl, k-1, q) \le f_{r,p}(l,k,q)$.  We have that $|Z'| \le (12288p^{24} + \beta r + 4k)(qk-1)$.  Fix $Z = Z' \cup Z_1 \cup Z_2 \cup (A' \cap B') \cup (A'' \cap B'')$.  We conclude that $(G - Z, \Omega)$ has a rendition in $\Delta$ satisfying (v) by restricting $\rho_2$ to $\Delta \setminus (\hbox{int}(\Delta_4') \cup \hbox{int}(\Delta_4''))$, using $\rho'$  in the disk $\Delta_4'$ and fixing a vortex-free rendition of $(H_4'', \Omega_4'')$ in $\Delta_4''$.  And again, as in the previous steps of the proof, we add the isolated subgraphs $G[A' \setminus ((A' \cap B') \cup \{z\})]$ and $G[A'' \setminus ((A'' \cap B'') \cup \{z\})]$ to an arbitrarily chosen cell of the rendition.

The final case to consider is when we get outcome (iii) and find a $(k-1, 2kl)$-leap pattern $(\zR_1, \zS_1)$ in $(H_4', \Omega_4')$ which is coterminal with $\zP_3'$ up to level $C_{\gamma_2 - 16k^2l - 8q}'$.  We may assume that for all $R \in \zR_1$, $(\zR_1 - R, \zS_1)$ is not a $(k-1, 2kl)$-leap pattern and by Lemma \ref{lem:leapsandbounds}, we conclude that $|\zR_1 \cup \zS_1| \le 4k^2l$.

As in Step 19, we can extend $(\zR_1, \zS_1)$ to a transaction in $(G, \Omega)$ via the linkage $\overline{\zP_3'}$ to get a linkage $(\overline{\zR_1}, \overline{\zS_1})$.  However, we want to ensure that $\overline{\zR_1}$ forms a planar transaction along with $Q_{\gamma_2 + 1}$, $Q_{\gamma_2+2}, \dots, Q_{\gamma_2 + l}$ from the original linkage $Q_1, \dots, Q_{m_3}$ defined in Step 7 via Lemma \ref{lem:transrotation2}.  Note that $Q_{\gamma_2 + 1}, \dots, Q_{\gamma_2 + l}$ are disjoint from $H_4'$ and $H_4''$ by the fact that they are the inner graph of $C_{\gamma_2}'$ and $C_{\gamma_2}''$ and the choice of the parameter $m_3$.  By applying Lemma \ref{lem:transrotation2} to the society $(H_4', \Omega_4')$, the linkage $\zR_1 \cup \zS_1$ and the nest $(C_{\gamma_2 - 16k^2l - 8q}', \dots, C_{\gamma_2 -1}')$, we see that there exists a $(k-1, 2kl)$-leap pattern $(\zR_2, \zS_2)$ in $(H_4', \Omega_4')$ which is coterminal with $\zP_3'$ up to $C_{\gamma_2 - 8q + 2}'$ such that if we let $\overline{\zR_2}$ and $\overline{\zS_2}$ be the continuations of $\zR_2$ and $\zS_2$, respectively, along $\overline{\zP_3'}$, we have that $\overline{\zR_2} \cup Q_{\gamma_2+1} \cup \dots \cup Q_{\gamma_2 + l}$ is a planar linkage.  Note that by construction, $\overline{\zR_2} \cup \overline{\zS_2} \cup Q_{\gamma_2+1} \cup \dots \cup Q_{\gamma_2 + l}$ is coterminal with $\zP$ in the original society $(G, \Omega)$ up to level $C_{\gamma + \gamma_2 + 1}$.

Recall the set $W'$ and subgraph $H_3' = J_3' \cup \zP_1'$ along with the separation $(A', B')$ were defined in Step 15.
By Step 20, the vertex $z$ is defined along with the set $W'$ and linkage $\zS^*$  of order $2k-1$ from $W'$ to $Y'$ in $H_3'[B']$.  Note that $\zR'$ is orthogonal to $C_{\beta r  + 2k + 5}', \dots, C_{\gamma_2}'$ in $(H_4', \Omega_4')$, and in the outergraph of $C_{\gamma+\gamma_2}$ in the original rendition of $(G, \Omega)$, the linkage $\zS^*$ is a subgraph of $\zP_1'$.
Apply Lemma~\ref{lem:augmentleap} to the linkage $\zS^*$ and the $(k-1, 2kl)$-leap pattern $(\overline{\zR_2}, \overline{\zS_2})$ to get a $(k-1, l)$-leap pattern $(\zR_3, \zS_3)$ in $(G, \Omega)$ and a $W' - Y'$ path $R$ such that $R$ is disjoint from $\zR_3 \cup \zS_3$ and $\zR \cup \zS \cup \{ R\}$ are contained in $\overline{\zR_2} \cup \overline{\zS_2} \cup \zR^*$.
Moreover, the lemma ensures that $\zR_3$ is contained in $\overline{\zR_2}$.  By construction, $ \zR \cup \zS \cup \{ R \}$ are coterminal with $\zP_1'$ up to level $C_{\gamma+\gamma_2 + 1}$ with respect to the original nest $(C_1, \dots, C_s)$ in the rendition of $(G, \Omega)$.  Let $w\in W'$ be the endpoint of $R$.  There exists a path $R'$ through the vertex $z$ linking $w$ and a vertex of $W''$ with all internal vertices in $A' \cup A''$.  There exists a path in $H_3''$ linking the endpoint of $R'$ in $W''$ and a vertex of $Y''$.  We conclude that
$(\zR_3 \cup \{Q_{\gamma_2 + 1} , \ldots ,Q_{\gamma_2+ l}\}, \zS_3 \cup \{R \cup R' \cup R''\})$ is a $(k,l)$-leap pattern which is coterminal with $\zP$ up to level $C_{\gamma + \gamma_2 +1 }$, completing the analysis of the case.  This completes the proof of the theorem.
\end{proof}

\section{A Refinement of Theorem \ref{thm:maintransstrong}}

We saw that in Section \ref{sec:crossesandleaps} that a sufficiently large nested crosses transaction or leap pattern orthogonal to a nest yields a large clique minor.  Using these results and by appropriately choosing values for the parameters, we can eliminate the possible outcomes of a large nested crosses transaction or large leap transaction in Theorem \ref{thm:maintransstrong}.  By applying the work of Section \ref{sec:planarstrip}, we can furthermore ensure that the handle or crosscap transactions which may arise as an outcome are also isolated and rural.  We conclude the section by showing how such an isolated and rural crosscap transaction leads naturally to a rendition in the disk plus a crosscap and similarly, how a handle transaction comprised of isolated and rural constituent transactions yields a rendition in the disk plus a handle.

\begin{theorem}
\label{thm:maintrans2}
For all $p, \theta \ge 1$, let $K:= p^{10^{7}p^{24}}\theta$, $\alpha:=12288(p(p-1))^{12} $, and $s,t \ge K$.
Then the following holds.
Let $(G, \Omega)$ be a society and $\rho$ a cylindrical rendition of $(G, \Omega)$.
Let $\zC = (C_1, \dots, C_{s})$ be a nest in $\rho$ of cardinality $s$ and $\zP$ an orthogonal linkage in $\rho$ of cardinality $t$.  Then one of the following holds.
\begin{itemize}
\item[\rm{(i)}] There exist a set $Z\subseteq V(G)$, $|Z|\le \alpha$, which is disjoint from the outer graph of $C_K$
and a transaction $\zQ$ in $(G - Z, \Omega )$ such that
$\zQ$ is  a crosscap transaction of thickness $\theta$ coterminal with $\zP$ to level $C_K$ and the $\zQ$-strip in $(G - Z, \Omega)$ is isolated and rural.
\item[\rm{(ii)}] There exist a set $Z\subseteq V(G)$, $|Z|\le 2\alpha$, which is disjoint from the outer graph of $C_K$
and a transaction $\zQ$ in $(G - Z, \Omega)$ such that
$\zQ$ is  a handle transaction of thickness $\theta$ which is cotermianl with $\zP$ up to level $C_K$ and if $\zQ_1,\zQ_2$ are the constituent planar transactions of $\zQ$, then for $i=1,2$ the $\zQ_i$-strip in $(G - Z, \Omega)$ is isolated and rural.
\item[\rm{(iii)}] There exists a $K_p$ minor in $G$ which is grasped by $\zC \cup \zP$.
\item[\rm{(iv)}] There exists a set $Z \subseteq V(G) \setminus V(\Omega)$ with
$|Z| \le 5\cdot 10^{20}p^{102}(\theta+\alpha)$ such that $(G - Z, \Omega)$ has a rendition in the disk of depth  at most $K$ and breadth at most $2p^2$.
\end{itemize}
\end{theorem}

\begin{proof}
We first observe that if $p =1$, then $K \ge 1$ and the intersection of an arbitrary $P \in \zP$ and $C \in \zC$ satisfies outcome (iii).  Thus, we may assume going forward that $p \ge 2$.

We apply Theorem~\ref{thm:maintransstrong} to $(G, \Omega)$ with nest $\zC$ and linkage $\zP$ with parameters $p$, $r= (49152p^{24}+2p^2)\left ((\theta + \alpha) +55p^2 \right) + \alpha$ (as in Theorem~\ref{thm:RTplanarstrip}),
$k=36864p^{24}$, $l=10p^2$, $q=2p^2$,
$s$ and $t$.  We begin with the following claim.

We include the following calculation for the sake of completeness.
\begin{claim}
$ f_{r,p}(l,k,q) \le K - 4\alpha - 40p(p-1) - 10$. 
\end{claim}
\begin{cproof}
Let $k, l, q, p, r$ be as stated.  Define
\begin{align*}
T_1 &:= \left( 5 \cdot 10^{21}p^{100}\right)^{k+q}\\
T_2 &:=  \left ( 2k \right) ^{k+q} \left ( l + k^2 l \right) \\
T_3 &:= (r + q + k + 1).
\end{align*}
Thus, $f_{r,p} = T_1 ( T_2+ T_3)$.  We bound each of the terms $T_i$ individually.  Note that $k+q = 36963p^{24} + 2p^2 \le 36865p^{24}$.  Thus, $T_1 \le \left(5 \cdot 10^{21} \right)^{36865p^{24}} p^{3686500p^{24}}$.
Considering the next term,
$$T_2 \le \left( 2 \cdot 36864p^{24}\right)^{36865p^{24}}\left(10p^2 +  \left( 36864p^{24}\right)^2 \cdot 10p^2\right)$$.
We have that $10p^2 + \left( 36864p^{24}\right) ^2 \cdot 10p^2 \le 1.36\cdot 10^{10}p^{50}$ as $36864^2 \le 1.359 \cdot 10^9$.  Thus,
\begin{align*}
T_2 & \le \left [ \left ( 2 \cdot 36864\right)^{36865p^{24}}\cdot 1.36 \cdot 10^{10}\right ]p^{884760p^{24}+50}.
\end{align*}

We now consider $T_3$.  Substituting in, we have that
\begin{align*}
T_3 \le \left( 49152p^{24} + 2p^2\right)\left( 55p^2 + \theta + 12288p^{24} \right)+ 12288p^{24} + 36864p^{24} + 2p^2 + 1.
\end{align*}
Using the fact that $49152p^{24} + 2p^2 \le 49153p^{24}$ and multiplying everything out, we get that:
\begin{align*}
T_3 &\le \theta\left( 49153p^{24}\right) + 110p^4 + 2703360p^{26} + 24576p^{26} + 603979776p^{48} + 49152p^{24} + 2p^2 + 1 \\
&= \theta \left (49153p^{24}\right) + 603979776p^{48} + 2727936p^{26} + 49152p^{24} + 110p^4 + 2p^2 + 1.
\end{align*}
Given that $49152p^{24} + 110p^4 + 2p^2 + 1 \le 49153p^{24} \le 12289p^{26}$ as $p \ge 2$, we have:
\begin{align*}
T_3 &\le (49153p^{24})\theta + 603979776p^{48} + 2740225p^{26} \\
&\le \left( 49153p^{24}\right) \theta + 603979777p^{48}
\end{align*}
given that $p^{22} \ge 4194304$.

Thus,
\begin{align*}
T_2 + T_3 & \le \left [ \left (2 \cdot 36864 \right )^{36865p^{24}} \cdot 1.36 \cdot 10^{10} \right] p^{884760p^{24} + 50} + 603979777p^{48} + 49153p^{24} \theta
\end{align*}
and putting it together, we get
\begin{align*}
T_1(T_2 + T_3)
&\le p^{4571261p^{24}}\theta \left[ \left( 5 \cdot 10^{21}\cdot 2 \cdot 36864 \right)^{36865p^{24}}\cdot 1.36 \cdot 10^{10} + \right . \\
& \left . \qquad + 604028930\left( 5 \cdot 10^{21}\right)^{36865p^{24}}\right].\\
\end{align*}
We convert all constants to powers of two to help combine terms.
\begin{align*}
\left( 5 \cdot 10^{21}\cdot 2 \cdot 36864 \right)^{36865p^{24}}\cdot 1.36 \cdot 10^{10} +&  \\ + 604028930\left( 5 \cdot 10^{21}\right)^{36865p^{24}} &
 \le p^{3280987p^{24}}.
\end{align*}
Substituting this into the previous equation, we get:
\begin{align*}
f_{r,p}(l,k,q) &\le p^{4571261p^{24}}\theta\cdot p^{3280987p^{24}}\\
&\le p^{10^{7}p^{24}}\theta - 4\alpha - 40p(p-1) - 10
\end{align*}
as desired.
\end{cproof}

We consider separately the possible outcomes of applying Theorem \ref{thm:maintransstrong}.

The first possible outcome has two distinct cases: either a crosscap transaction or a handle transaction.  Assume we are in the first case and there exists a crosscap transaction $\zQ$ in $G$ of thickness $r$ which is coterminal with $\zP$ up to level $K - 4\alpha \ 40p(p-1) - 10$.  Apply Theorem~\ref{thm:RTplanarstrip} to $k=q$, $l=\theta$, $p$, the society $(G, \Omega)$, rendition $\rho$, nest $(C_{K-4\alpha-40p(p-1)-9}, \dots, C_{s})$ and transaction $\zQ$ to obtain one of the outcomes therein.  If we get outcome (i) of Theorem \ref{thm:RTplanarstrip}, then we satisfy outcome (iii) of the present theorem.  As the next case, assume we get outcome (ii) of Theorem \ref{thm:RTplanarstrip} and let $\zQ'$ be a $2p^2$ nested crosses or twisted nested crosses transaction which is coterminal with $\zQ$ up to level $C_{K - 3\alpha - 40p(p-1)}$.  By Corollary \ref{cor:nonestedcrosses} applied to the transaction $\zQ'$ and nest $(C_{K - 3\alpha - 40p(p-1)}, \dots, C_s)$, we see that $G$ contains a $K_p$ minor grasped by $\zC \cup \zQ$, satisfying outcome (iii) of the present theorem.  The final possible outcome of Theorem \ref{thm:RTplanarstrip} is that there exists a set $Z \subseteq V(G)$, $|Z| \le \alpha$, and a set $\zQ' \subseteq \zQ$ of size at least $\theta$ (actually $\theta + \alpha$, however we will only need the additional $\alpha$ paths in the next case) such that $Z$ is disjoint from the outergraph of $C_{K - 2 \alpha -  20p(p-1) - 5}$ and the $\zQ'$-strip is rural and isolated in $(G-Z, \Omega)$, satisfying outcome (i).

The other possibility in outcome (i) of Theorem \ref{thm:maintransstrong} is that there exist $\zQ_1, \zQ_2$, each of order $r$ forming an $r$-handle transaction which is coterminal with $\zP$ up to level $C_{K- 4\alpha - 40p(p-1) - 10}$.  The analysis is similar to the previous case. Apply Theorem \ref{thm:RTplanarstrip} to the transaction $\zQ_1$ and nest $(C_{K- 4\alpha - 40p(p-1)-9}, \dots, C_{s})$.  As above, if the result is either outcome (i) or (ii) of Theorem \ref{thm:RTplanarstrip}, then $G$ contains a $K_p$ minor grasped by $\zC \cup \zP$, satisfying outcome (iii) of the present theorem.  Thus, we may assume there exists $Z_1 \subseteq V(G)$, $|Z_1| \le \alpha$ and $Z_1$ disjoint from the outer graph of $C_{K + 2\alpha + 20p(p-1) + 4}$, and $\zQ_1' \subseteq \zQ$, $|\zQ_1'| = \theta + \alpha$ such that $\zQ_1'$ is rural and isolated in $(G-Z_1, \Omega)$.  We now turn our attention to $\zQ_2$.  Let $\bar{\zQ_2}$ be the elements of $\zQ_2$ which are disjoint from $Z_1$.  Note that $|\bar{\zQ_2}| \ge |\zQ_2| - \alpha \ge (49152p^{24}+2p^2)\left ((\theta + \alpha) +55p^2 \right)$.  Again, we apply Theorem \ref{thm:RTplanarstrip} to $\bar{\zQ_2}$ in the society $(G-Z_1, \Omega)$ with nest $(C_{K - 2\alpha - 20p(p-1) - 4}, \dots, C_s)$.  Again, if we get either outcome (i) or (ii) of Theorem \ref{thm:RTplanarstrip}, there exists a clique minor satisfying outcome (iii) of Theorem \ref{thm:maintrans2} as in the previous cases.  Thus, we may assume there exists $Z_2 \subseteq V(G) \setminus Z_1$ which is disjoint from the outer graph of $C_{K}$ and $\bar{\zQ_2'} \subseteq \bar{\zQ_2}$ of order $\theta$ such that $\bar{\zQ_2'}$ is isolated and rural in $(G-(Z_1 \cup Z_2), \Omega)$.  Let $\bar{\zQ_1'} \subseteq \zQ_1'$ be the elements of $\zQ_1'$ which are disjoint from $Z_2$.  Then $|\bar{\zQ_1'}| \ge \theta$.  By Lemma \ref{lem:inhereitedrurality}, $\bar{\zQ_1'}$ is rural and isolated in $(G-Z_1, \Omega)$ and is thus also rural and isolated in $(G-(Z_1 \cup Z_2), \Omega)$.  We conclude that $(\bar{\zQ_1'}, \bar{\zQ_2'})$ is a handle transaction satisfying (ii) of Theorem \ref{thm:maintrans2}.

In outcome (ii) Theorem \ref{thm:maintransstrong}, we have a either a $2p^2$-nested crosses transaction or twisted nested crosses transaction which is coterminal with $\zP$ up to level $K - 4\alpha - 40p(p-1) -10$.  By Corollary \ref{cor:nonestedcrosses}, $G$ has a $K_p$ minor grasped by $\zC \cup \zP$, as desired in outcome (iii) of Theorem \ref{thm:maintrans2}.  Similarly, in outcome (iii) of Theorem \ref{thm:maintransstrong}, we have a $(36864p^{24}, 10p^2)$-leap pattern.  By Lemma \ref{lem:nolargeleap}, $G$ has a $K_p$ minor grasped by $\zC \cup \zP$.

Outcome (iv) of Theorem \ref{thm:maintransstrong} is the same as outcome (iii) of Theorem \ref{thm:maintrans2}.

As the final case, assume we get outcome (v) in the application of Theorem \ref{thm:maintransstrong} and there exists $Z \subseteq V(G) \setminus V(\Omega)$ such that $(G-Z, \Omega)$ has a rendition in the disk of depth at most $f_{r,p}(l,k,q) \le K$ and breadth at most $q = 2p^2$.  Moreover,
\begin{align*}
|Z| & \le (qk)\left [12288p^{24} + 368650^2p^{52}r + 4k \right ] \\
& \le 36865p^{24} \left [ (368650^2 + 1) p^{52}r \right ] \\
& \le 5 \cdot 10^{20}p^{102}( \theta + \alpha)\\
\end{align*}
and we see that outcome (iv) of Theorem \ref{thm:maintrans2} is satisfied.  This completes the analysis of the various cases and the proof of the theorem.
\end{proof}

We conclude this section by examining how the large rural and isolated crosscap and handle transactions in the previous theorem naturally allow one to embed the society in either the projective plane minus a disk or the torus minus a disk, respectively.  These are somewhat technical lemmas which we will need in the next section as we inductively find a $\Sigma$-decomposition in progressively higher and higher genus surfaces.  As such, the statements of the lemmas include the technical nests and systems of paths we will need to maintain in the coming proofs.

\begin{LE}
\label{lem:xcaptrans}
Let $s$ and $\theta$ be non-negative integers.
Let $(G, \Omega)$ be a society and $\rho = (\Gamma, \zD, c_0)$ a cylindrical rendition of $(G, \Omega)$ in the disk $\Delta$.
Let $X_1,X_2$ be disjoint segments of $V(\Omega)$.
Let $\zC = (C_1, \dots, C_{s+8})$ be a nest in $\rho$ of cardinality $s+8$ and $\zP$ a linkage from $X_1$
to $\widetilde{c_0}$ orthogonal to $\zC$.  Let $\Sigma^*$ be a surface homeomorphic to the projective plane minus an open disk obtained from $\Delta$ by adding a crosscap to the interior of $c_0$.
If there exists a crosscap transaction $\zQ$ in $(G , \Omega )$ orthogonal to $\zC$ and of thickness $\theta+2s+7$ disjoint from
$\zP$ such that every member of $\zQ$ has both endpoints in $X_2$ and
  the $\zQ$-strip in $(G , \Omega )$ is isolated and rural, then there exist a subset $\zQ'$ of $\zQ$
of cardinality $\theta$ and a rendition $\rho'$ of $(G, \Omega)$
in $\Sigma^*$ such that there exists a unique vortex $c_0'\in C(\rho')$ and the following hold:
\begin{itemize}
\item[\rm{(i)}] $\zQ'$ is disjoint from $\sigma (c_0')$,
\item [\rm{(ii)}] the vortex society of $c_0'$ in $\rho'$ has a cylindrical rendition $\rho''=(\Gamma'',\zD'',c_0'')$,
\item [\rm{(iii)}] every element of $\zP$ has an endpoint in $V(\sigma_{\rho''}(c_0''))$,
\item [\rm{(iv)}] $\rho''$ has a nest $\zC'' = (C_1'', \dots, C_s'')$ of cardinality $s$ such that $\zP$ is orthogonal to $\zC''$ and for every $i$, $1 \le i \le s$, and for all $P \in \zP$, $C_i'' \cap P = C_{i+7} \cap P$.  Moreover,
\item [\rm{(v)}] let $\zP=\{P_1,P_2,\ldots ,P_m\}$. For $i=1,2,\ldots ,m$ let $x_i$ be the endpoint 
of $P_i$ in $X_1$, and let $y_i$ be the last entry of $P_i$ into $c_0'$;
then if $x_1,x_2,\ldots ,x_m$ appear in $\Omega$ in the order listed,
then $y_1,y_2,\ldots ,y_m$ appear on $\widetilde{c_0'}$ in the order listed.
\item [\rm{(vi)}] Let $\Delta'$ be the open disk bounded by the track of $C_{s+8}$.  Then $\rho$ restricted to $\Delta \setminus \Delta'$ is equal to $\rho'$ restricted to $\Delta \setminus \Delta'$.
\end{itemize}
\end{LE}


\begin{proof}
Let $(G_1, \Omega_1)$ be the vortex society of the cycle $C_7$ and let $\Delta_7$ be the closed disk bounded by the track of $C_7$ in $\Delta$.  Let $t = \theta + 2s + 6$ and let $\zQ = \{Q_0, \dots, Q_t\}$ with the elements of $\zQ$ ordered according to the first time we encounter an endpoint traversing $X_2$ in the order given by $\Omega$.  For every $i$, let $\bar{Q}_i$ be the $V(\Omega_1)$-path in $Q_i$ which contains at least one edge of $\sigma_\rho(c_0)$.  Note that $\bar{Q}_i$ is uniquely defined as $\zQ$ is unexposed and orthogonal to $\zC$.

Apply Lemma \ref{lem:planarstriprestriction} to the society $(G, \Omega)$, $\rho$, nest $\zC$, and transaction $\{Q_1, Q_2, \dots, Q_t\} $.  Note that the two distinct $V(\Omega) - \widetilde{c_0}$ paths contained in $Q_0$ serve as the linkage $\zP$ in the statement of Lemma \ref{lem:planarstriprestriction}.  Let $\zQ_1 = \{\bar{Q}_1, \dots, \bar{Q}_t\}$.  The lemma implies that the $\zQ_1$-strip society $(H, \Omega_H)$ in $(G_1, \Omega_1)$ is rural and isolated.  Moreover, $\zQ_1$ is a crosscap transaction.  Let $(J_1, \Omega_{J_1})$ be the society with $J_1$ equal to the union of $H$ and the graph with vertices $V(\Omega_1)$ and no edges; let $\Omega_{J_1} = \Omega_1$.  Let $\Delta^*$ be a surface obtained from $\Delta_7$ by adding a crosscap to the interior of $c_0$.  As $(H, \Omega_H)$ is rural and isolated and $\zQ_1$ is a crosscap transaction, there exists a rendition $\delta_1 = (\Gamma_1, \zD_1)$ of $(J_1, \Omega_{J_1})$ in $\Delta^*$.  Moreover, we may assume that $\pi_{\delta_1}^{-1}(v) = \pi_{\rho}^{-1}(v)$ for all $v \in V(\Omega_{J_1}) = V(\Omega_1)$.

Let $L_1$ be the minimum subgraph of $G_1$ such that $G_1 = J_1 \cup L_1$.  By the definition of the $\zQ_1$-strip society, $V(L_1) \cap V(J_1) \subseteq (V(\Omega_1) \setminus V(\Omega_{H})) \cup V(\bar{Q}_1) \cup V(\bar{Q}_t)$.  Note, however, that there may be vertices $V(L_1) \cap V(J_1)$ contained in either $\bar{Q}_1$ or $\bar{Q}_t$ for which the corresponding point in the drawing of $\delta_1$ is not an element of $N$ (and instead drawn in the interior of some cell).  Thus, we cannot easily add a single cell containing $L_1$ to get a rendition of $(G_1, \Omega_1)$ in $\Delta^*$.  To resolve this issue, we will slightly reduce the subgraph $J_1$.

Fix $T_1$ and $T_2$ to be the track of $\bar{Q}_2$ and $\bar{Q}_{t-1}$, respectively, in $\rho_1$.  The surface $\Delta^* \setminus (T_1 \cup T_2)$ has exactly two connected components; let $\Delta_1'$ and $\Delta_2'$ be the closures of these two components.  Each of $\Delta_i'$ is homeomorphic to the closed unit disk for $i = 1, 2$.  One of the components, say $\Delta_1'$, contains the images of the endpoints of $\bar{Q}_3, \dots, \bar{Q}_{t-2}$ under the map $\pi^{-1}_{\rho_1}$.  It follows that the other, $\Delta_2'$, contains the drawing of $\bar{Q}_1$, $\bar{Q} _t$, and the vertices of $V(\Omega_1) \setminus V(\Omega_{H})$.

Let $J_2$ be the subgraph formed by the union of $\bigcup_{\{c \in C(\rho_1): c \subseteq \Delta_1'\} }\sigma(c)$ and the set of vertices $\pi_{\rho_1}(N(\rho_1) \cap \partial(\Delta_1'))$ included as isolated vertices.  Let $L_2$ be the minimal subgraph of $G_1$ such that $G_1 = L_2 \cup J_2$.
Note that the drawings of the paths $\bar{Q}_3, \dots, \bar{Q}_{t-2}$ are disjoint from the disk $\Delta_2'$.

Define $\Sigma^*$ to be the surface $(\Delta \setminus int(\Delta_7) ) \cup \Delta^*$.  We now define a rendition $\rho_2 = (\Gamma_2, \zD_2, c_2)$ of $(G, \Omega)$ in $\Sigma^*$.  Let $\Gamma_2$ be equal to the restriction of $\Gamma$ in $\Delta \setminus int(\Delta_7)$, equal to the restriction of $\Gamma_1$ in $\Delta_1'$, and equal to an arbitrary drawing of $J_2$ in the disk $\Delta_2'$.  Note that we may assume that the drawing of $J_2$ is consistent with the drawing of $L_2$ and the outer graph of $C_7$ on the boundary of $\Delta_2'$ and that the drawing intersects the boundary of $\Delta_2'$ only in points corresponding to the vertices in the intersection of $J_2$ and the union of $L_2$ and the outer graph of $C_7$.  Let $\zD_2 = \{D \in \zD: D \nsubseteq \Delta_7\} \cup \{D \in \zD_1 : D \subseteq \Delta_1'\} \cup \{\Delta_2'\}$.  Finally, let $c_2$ be the cell of the $\Sigma^*$-decomposition corresponding to  $\Delta_2'$.

As every vertex of $\widetilde{c_2}$ is contained in either $Q_2 \cup Q_{t-1}$ or in $C_7$, we have that the paths $Q_3, Q_4, \dots, Q_{t-2}$ are all disjoint from $\sigma_{\rho_2}(c_2)$.  For every $i$, $1 \le i \le s+1$, define $C_i''$ to be the unique cycle contained in $Q_{2+i} \cup Q_{t-1-i} \cup C_{7+i}$ which is disjoint from $Q_{3+i}, Q_{4+i}, \dots, Q_{t-2-i}$.  Note that each $C_i''$ bounds a disk in $\Delta^*$ and the inner graph of $C_i''$ contains $\sigma_{\rho_2}(c_2)$ and is disjoint from $Q_{3+i}, Q_{4+i}, \dots, Q_{t-2-i}$.

We are now in a position to define our rendition $\rho' = (\Gamma', \zD', c_0')$.  Let $\Gamma' = \Gamma_2$.  Let $D'$ be the closed disk bounded by the track of $C_{s+1}''$.  Let $\zD' = \{D \in \zD_2: D \nsubseteq D'\} \cup \{D'\}$ and $c_0' = D' \setminus ( N(\rho_2) \cap \partial(D'))$.  Let $\zQ' = \{Q_{4+s}, \dots, Q_{3+s+\theta}\}$.

By construction, the vortex society of $c_0'$ in $\rho'$ is simply the inner society of $C_{s+1}''$ with respect to the rendition $\rho_2$.  Given the order of the original linkage $\zQ$ and the construction of the $C_i''$, we have that $\zQ'$ is disjoint from $\sigma_{\rho'}(c_0')$, satisying \rm{(i)}.  The restriction of $\rho_2$ to $D'$ implies that the vortex society of $c_0'$ has the desired cylindrical rendition $\rho'' = (\Gamma'', \zD'', c_0'')$ in (ii).  Moreover, the cycles $(C_1'', \dots, C_s'')$ form a nest of order $s$ in $\rho''$.

We now turn our attention to the linkage $\zP$.  As the $\zQ$-strip is rural and isolated, and since $\zP$ is disjoint from $Q_2$ and $Q_{t-1}$, we have that each element of $\zP$ has an endpoint in $V(L_2) = V(\sigma_{\rho_2}(c_2))$ and thus (iii) is satisfied.  Moreover, by construction of the $C_i''$, for every $1 \le i \le k$ and $P \in \zP$, $C_i'' \cap P = C_{7+i} \cap P$.   We conclude that $\zP$ is orthogonal to the nest $(C_1'', \dots, C_s'')$ and that (iv) holds.

Conclusion (v) holds by the fact that $\zP$ and $\zQ$ have their endpoints in disjoint segments of the original society vertices $\Omega$ and the construction of the cycles $C_i''$.  Conclusion (vi) holds by the construction of $\rho'$, completing the proof of the lemma.
\end{proof}

The next lemma is the analog of the previous in the case of a large handle transaction.  The proof is much the same, however there is one additional technicality to consider.  In the proof of Lemma \ref{lem:xcaptrans}, after finding the planar strip in the crosscap transaction, the definition of the cycles $C_i''$ followed naturally.  In the variant of Lemma \ref{lem:xcaptrans} for handle transactions, when we find the planar strips in the constituent transaction of handle transaction, we will need an additional property of the handle transaction in order to define the cycles $C_i''$ which give the desired nest.  This leads to the following definition.

\begin{DEF}
Let $k \ge 1$ be a positive integer, $(G, \Omega)$ be a society, $X$ a segment of the society, and $\zP$ be a $k$-handle transaction such that every element of $\zP$ has both endpoints contained in $X$.  We say that $\zP$ is \emph{consistent} with $X$ if the elements of $\zP$ can be labeled $P_1, \dots, P_{2k}$ and the endpoints of $P_i$ as $s_i, t_i$ for $1 \le i \le 2k$ such that $s_1, s_2, \dots, s_{2k}, t_k, t_{k-1}, \dots, t_1, t_{2k}, t_{2k-1}, \dots, t_{k+1}$ occur on $X$ in the linear order induced on $X$ by $\Omega$.
\end{DEF}

\begin{LE}\label{lem:consistent}
Let $k \ge 1$ be a positive integer.  Let $(G, \Omega)$ be a society, $X$ a segment of the society, and $\zP$ a $2k$x-handle transaction such that every element has both endpoints contained in $X$.  Then there exists a $k$-handle transaction $\zP'$ comprised of elements of $\zP$ such that $\zP'$ is consistent with $X$.
\end{LE}

\begin{cproof}
Let $\zP_1, \zP_2$ be the constituent transactions of $\zP$, and assume that traversing $X$ according to the linear order induced by $\Omega$, we first encounter an element of $\zP_1$.  Fix $Y_2$ to be the minimal segment of $\Omega$ contained in $X$ such that every element of $\zP_2$ has both endpoints contained in $Y_2$.  It follows that every element of $\zP_1$ has one endpoint contained in $Y_2$ and one endpoint in $X \setminus Y_2$.  The set $X \setminus Y_2$ is comprised of two segments; we fix $Y_1$ to be one of the two segments of $X \setminus Y_2$ which contains an endpoint of at least half of the elements of $\zP_1$.  Let $\zP_1'$ be elements of $\zP_1$ with an endpoint in $Y_1$ and let $\zP_2'$ be a subset of $\zP_2$ of order $|\zP_1'|$.  Then $\zP' = \zP_1 \cup \zP_2$ is a handle transaction of order at least $k$ which is consistent with $X$, as claimed.
\end{cproof}

We now present the analog of Lemma \ref{lem:xcaptrans} for handle transactions, with the additional factor two in the parameters arising from an application of Lemma \ref{lem:consistent}.

\begin{LE}
\label{lem:handletrans}
Let $s$ and $\theta $ be non-negative integers.
Let $(G, \Omega)$ be a society and $\rho = (\Gamma, \zD, c_0)$ a cylindrical rendition of $(G, \Omega)$ in a disk $\Delta$.
Let $X_1,X_2$ be disjoint segments in $V(\Omega)$.
Let $\zC = (C_1, \dots, C_{s+8})$ be a nest in $\rho$ of cardinality $s+8$ and $\zP$ a linkage from $X_1$
to $\widetilde{c_0}$ orthogonal to $\zC$.  Let $\Sigma^+$ be a surface homeomorphic to the torus minus an open disk obtained by adding a handle to $\Delta$ in the interior of $c_0$.
Let $\zQ$ be a handle transaction in $(G , \Omega )$ orthogonal to $\zC$ and of thickness $2\theta+4s+12$ which is disjoint from
$\zP$ such that every member of $\zQ$ has both endpoints in $X_2$.  Let $\zQ_1$ and $\zQ_2$ be the constituent planar transactions of $\zQ$ and let $\zQ_i$ have the property that the $\zQ_i$-strip in $(G , \Omega )$ is isolated and rural for $i = 1,2$.
Then for $i = 1,2$, there exists a subset $\zQ_i' \subseteq \zQ_i$
of cardinality $\theta$ and a rendition $\rho'$ of $(G, \Omega)$
in $\Sigma^+$ with a unique vortex $c_0'\in C(\rho')$ which satisfy the following:
\begin{itemize}
\item[\rm{(i)}] $\zQ' = \zQ_1' \cup \zQ_2'$ is disjoint from $\sigma (c_0')$,
\item [\rm{(ii)}] the vortex society of $c_0'$ in $\rho'$ has a cylindrical rendition $\rho''=(\Gamma'',\zD'',c_0'')$,
\item [\rm{(iii)}] every element of $\zP$ has an endpoint in $\sigma_{\rho''}(c_0'')$,
\item [\rm{(iv)}] $\rho''$ has a nest $\zC'' = (C_1'', \dots, C_s'')$ of cardinality $s$ such that $\zP$ is orthogonal to $\zC''$ and for every $i$, $1 \le i \le s$, and for all $P \in \zP$, $C_i'' \cap P = C_{i+7} \cap P$.  Moreover,
\item[\rm{(v)}] let $\zP=\{P_1,P_2,\ldots ,P_m\}$. For $i=1,2,\ldots ,m$ let $x_i$ be the endpoint 
of $P_i$ in $X_1$, and let $y_i$ be the last entry of $P_i$ into $c_0'$;
if $x_1,x_2,\ldots ,x_m$ appear in $\Omega$ in the order listed,
then $y_1,y_2,\ldots ,y_m$ appear on $\widetilde{c_0'}$ in the order listed.
\item [\rm{(vi)}] Let $\Delta'$ be the open disk bounded by the track of $C_{s+8}$.  Then $\rho$ restricted to $\Delta \setminus \Delta'$ is equal to $\rho'$ restricted to $\Delta \setminus \Delta'$.
\end{itemize}
\end{LE}

\begin{proof}
First, by Lemma \ref{lem:consistent}, we may assume that $\zQ$ has thickness $t = \theta + 2s + 6$ and is consistent with $X_2$.  Let the elements of $\zQ_i$ be labeled $Q_1^i, \dots, Q_t^i$ for $i = 1, 2$ and assume that traversing $X_2$ in the order induced by $\Omega$ we encounter them in that order.  Note that we are relying on Lemma \ref{lem:inhereitedrurality} to maintain that the constituent transactions are rural and isolated even when we restrict to a subset of the original transaction.

Let $(G_1, \Omega_1)$ be the vortex society of the cycle $C_7$ and let $\Delta_7$ be the closed disk bounded by the track of $C_7$ in $\Delta$.  Let $t = \theta + 2s + 6$ and for $i = 1, 2$ let $\zQ_i = \{Q^i_1, \dots, Q^i_t\}$.  For every $i=1,2$ and $1 \le j \le t$, let $\bar{Q}^i_j$ be the $V(\Omega_1)$-path in $Q^i_j$ which contains at least one edge of $\sigma_\rho(c_0)$.  Note that $\bar{Q}^i_j$ is uniquely defined as $\zQ$ is unexposed and orthogonal to $\zC$.

Fix $i \in \{1,2\}$.  As in the previous lemma, we apply Lemma \ref{lem:planarstriprestriction} to the society $(G, \Omega)$, $\rho$, nest $\zC$, and transaction $\{Q^i_1, Q^i_2, \dots, Q^i_t\}$.  Note that the two distinct $V(\Omega) - \widetilde{c_0}$ paths contained in $Q^{3-i}_1$ serve as the linkage $\zP$ in the statement of Lemma \ref{lem:planarstriprestriction}.  Let $\bar{\zQ}_i= \{\bar{Q}^i_1, \dots, \bar{Q}^i_t\}$.  The lemma implies that the $\bar{\zQ}_i$-strip society $(H_i, \Omega_{H_i})$ in $(G_1, \Omega_1)$ is rural and isolated.  Moreover, $\bar{\zQ}_i$ is a planar transaction.

Let $(J_1, \Omega_{J_1})$ be the society with $J_1$ equal to the union of $H_1$, $H_2$, and the graph with vertices $V(\Omega_1)$ and no edges; let $\Omega_{J_1} = \Omega_1$.  Let $\Delta^+$ be a surface obtained from $\Delta_7$ by adding a handle to the interior of $c_0$.  As $(H_i, \Omega_{H_i})$ is rural and isolated and the order of the vertices in $\Omega_{H_i}$ coincides with the ordering in $\Omega_{J_1}$, we conclude that there exists a rendition $\delta_1 = (\Gamma_1, \zD_1)$ of $(J_1, \Omega_{J_1})$ in $\Delta^+$.  Moreover, we may assume that $\pi_{\delta_1}^{-1}(v) = \pi_{\rho}^{-1}(v)$ for all $v \in V(\Omega_{J_1}) = V(\Omega_1)$.

Let $L_1$ be the minimum subgraph of $G_1$ such that $G_1 = J_1 \cup L_1$.  By the definition of the $\bar{\zQ}_i$-strip society, $V(L_1) \cap V(J_1) \subseteq (V(\Omega_1) \setminus (V(\Omega_{H_1})\cup V(\Omega_{H_2}))) \cup V(\bar{Q}^1_1) \cup V(\bar{Q}^1_t)\cup V(\bar{Q}^2_1) \cup V(\bar{Q}_t^2)$.  Note that, as in the previous lemma, there may be vertices $V(L_1) \cap V(J_1)$ which are drawn in the interior of some cell and do not correspond to points in $N(\rho_1)$.

For $i = 1, 2$, fix $T^i_1$ and $T^i_2$ to be the track of $\bar{Q}^i_2$ and $\bar{Q}^i_{t-1}$, respectively, in $\rho_1$.  The surface $\Delta^+ \setminus (T^1_1 \cup T^1_2 \cup T^2_1 \cup T^2_2)$ has exactly three connected components.  Denote the  closures of the three components as $\Delta_1'$, $\Delta_2'$, and $\Delta_3'$; each $\Delta_i'$ is homeomorphic to the closed unit disk for $i = 1, 2, 3$.  For $i = 1, 2$, we may assume that $\Delta_i'$ contains the images of the endpoints of $\bar{Q}^i_3, \dots, \bar{Q}^i_{t-2}$ under the map $\pi^{-1}_{\rho_1}$.  Thus, the disk $\Delta_3'$ contains the drawing of $\bar{Q}^i_1$, $\bar{Q}^i_t$ for $i = 1,2$, as well as the points corresponding to the vertices of $V(\Omega_1) \setminus (V(\Omega_{H_1}) \cup V(\Omega_{H_2}))$.

Let $J_2$ be the subgraph formed by the union of $\bigcup_{\{c \in C(\rho_1): c \subseteq \Delta_1' \cup \Delta_2'\} }\sigma(c)$ and the set of vertices $\pi_{\rho_1}(N(\rho_1) \cap (\partial(\Delta_1') \cup \partial(\Delta_2')))$ included as isolated vertices.  Let $L_2$ be the minimal subgraph of $G_1$ such that $G_1 = L_2 \cup J_2$.
Note that the drawings of the paths $\bar{Q}^i_3, \dots, \bar{Q}^i_{t-2}$ are disjoint from the disk $\Delta_3'$.

Define $\Sigma^+$ to be the surface $(\Delta \setminus int(\Delta_7) ) \cup \Delta^+$.  We now define a rendition $\rho_2 = (\Gamma_2, \zD_2, c_2)$ of $(G, \Omega)$ in $\Sigma^+$.  Let $\Gamma_2$ be equal to the restriction of $\Gamma$ in $\Delta \setminus int(\Delta_7)$, equal to the restriction of $\Gamma_1$ in the union of $\Delta_1'$ and  $\Delta_2'$, and equal to an arbitrary drawing of $J_2$ in the disk $\Delta_3'$.  Note that we may assume that the drawing of $J_2$ is consistent with the drawing of $L_2$ and the outer graph of $C_7$ on the boundary of $\Delta_3'$ and that the drawing intersects the boundary of $\Delta_3'$ only in points corresponding to the vertices in the intersection of $J_2$ and the union of $L_2$ and the outer graph of $C_7$.  Let $\zD_2 = \{D \in \zD: D \nsubseteq \Delta_7\} \cup \{D \in \zD_1 : D \subseteq \Delta_1' \cup \Delta_2'\} \cup \{\Delta_3'\}$.  Finally, let $c_2$ be the cell of the $\Sigma^+$-decomposition corresponding to  $\Delta_3'$.

As every vertex of $\widetilde{c_2}$ is contained in either $Q^1_2 \cup Q^1_{t-1} \cup Q^2_2 \cup Q^2_{t-1}$ or in $C_7$, we have that the paths $Q^i_3, Q^i_4, \dots, Q^i_{t-2}$ are all disjoint from $\sigma_{\rho_2}(c_2)$ for $i = 1, 2$.  For every $i$, $1 \le i \le s+1$, define $C_i''$ to be the unique cycle contained in $Q^1_{2+i} \cup Q^1_{t-1-i} \cup Q^2_{2+i} \cup Q^2_{t-1-i} \cup C_{7+i}$ which is disjoint from $Q^1_{3+i}, Q^1_{4+i}, \dots, Q^1_{t-2-i}, Q^2_{3+1}, Q^2_{4+i}, \dots, Q^2_{t-2-i}$.  Note that each $C_i''$ bounds a disk in $\Delta^+$ and the inner graph of $C_i''$ contains $\sigma_{\rho_2}(c_2)$ and is disjoint from $Q^i_{3+i}, Q^i_{4+i}, \dots, Q^i_{t-2-i}$ for $i = 1, 2$.  Moreover, as $\zQ$ is consistent with $X_2$, each $C_i''$ will intersect every element of $\zP$.

We are now in a position to define our rendition $\rho' = (\Gamma', \zD', c_0')$.  Let $\Gamma' = \Gamma_2$.  Let $D'$ be the closed disk bounded by the track of $C_{s+1}''$.  Let $\zD' = \{D \in \zD_2: D \nsubseteq D'\} \cup \{D'\}$ and $c_0' = D' \setminus ( N(\rho_2) \cap \partial(D'))$.  For $i = 1, 2$, let $\zQ_i' = \{Q^i_{4+s}, \dots, Q^i_{3+s+\theta}\}$ and let $\zQ' = \zQ_1' \cup \zQ_2'$.

By construction, the vortex society of $c_0'$ in $\rho'$ is simply the inner society of $C_{s+1}''$ with respect to the rendition $\rho_2$.  Given the order of the original linkage $\zQ$ and the construction of the $C_i''$, we have that $\zQ'$ is disjoint from $\sigma_{\rho'}(c_0')$, satisying \rm{(i)}.  The restriction of $\rho_2$ to $D'$ implies that the vortex society of $c_0'$ has the desired cylindrical rendition $\rho'' = (\Gamma'', \zD'', c_0'')$ in (ii).  Moreover, the cycles $(C_1'', \dots, C_s'')$ form a nest of order $s$ in $\rho''$.

We now turn our attention to the linkage $\zP$.  As the constituent linkages $\zQ_1$ and $\zQ_2$ satisfy the property that the $\zQ_i$-strip is rural and isolated, and since $\zP$ is disjoint from $\zQ$, we have that each element of $\zP$ has an endpoint in $V(\sigma_{\rho_2}(c_2)) = V(L_2)$ and thus (iii) is satisfied.  Moreover, by construction of the $C_i''$, for every $1 \le i \le k$ and $P \in \zP$, $C_i'' \cap P = C_{7+i} \cap P$.   We conclude that $\zP$ is orthogonal to the nest $(C_1'', \dots, C_s'')$ and that (iv) holds.

Conclusion (v) holds by the fact that $\zP$ and $\zQ$ have their endpoints in disjoint segments of the original society vertices $\Omega$ and the construction of the cycles $C_i''$.  Conclusion (vi) holds by the construction of $\rho'$, completing the proof of the lemma.
\end{proof}

\section{Finding a $\Sigma$-decomposition}

We are now ready to finish our proof which will proceed, in general terms, by induction on the genus of the surface.
We begin by introducing the technical structure we will inductively find while building a $\Sigma$-decomposition in progressively higher and higher genus surfaces.

\begin{DEF}
Let $G$ be a graph, $W$ a wall in $G$, $\zC$ a nest in $G$, and $\zP$ a linkage which is orthogonal to $\zC$.  We say that $(\zP, \zC)$ \emph{overlays $W$} if we can fix the horizontal paths of $W$ to be $H_1, \dots, H_k$ and the vertical paths $V_1, \dots, V_l$ such that there exists injective maps $\phi_{hor}: \zC \rightarrow \{H_1, \dots, H_k\}$ and $\phi_{vert}: \zP \rightarrow \{V_1, \dots, V_l\}$ such that for all $P \in \zP$ and $C \in \zC$, $P \cap C \supseteq \phi_{hor}(C) \cap \phi_{vert}(P)$.
\end{DEF}

\begin{DEF}
Let $G$  be a graph, let $\Sigma$ be a surface, and let $\delta=(\Gamma,\zD)$
be a $\Sigma$-decomposition of $G$.  Let $W$ be a wall in $G$, $W_1$ a subwall of $W$, and $W_0$ an $r$-subwall of $W_1$.  Assume that there exist  cells $c_0,c_1\in C(\delta)$ and a closed
disc $\Delta_0\subseteq\Sigma$ such that:
\begin{itemize}
\item no cell $c\in C(\delta)\setminus\{c_0,c_1\}$ is a vortex.
\item The vortex society $(G_0,\Omega_0)$
of $c_0$ in $\delta$ is rural, $W_0 \subseteq G_0$, and there exists a vortex free rendition of $(G_0, \Omega_0)$ for which $W_0$ is flat.
\item The cell $c_0\subseteq \Delta_0\subseteq \Sigma \setminus c_1$ and the boundary of $\Delta_0$ intersects no cell of $\delta$.
\item The complementary society $(G_2,\Omega_2)$ to the vortex society of $c_0$ has a cylindrical rendition
 $\rho_2=(\Gamma_2,\zD_2,c_2)$ with
a nest $\zC_2$,
where $\Gamma_2 = \Gamma$, $C(\rho_2)\setminus \{c_2\}=\{c\in C(\delta):c\subseteq \Delta_0\}\setminus \{c_0\}$,
$N(\rho_2)=N(\delta)\cap\Delta_0$, $\sigma_{\rho_2}(c_2)=\bigcup\{\sigma_\delta(c):c\not \subseteq \Delta _0\}$, and
$\sigma_\delta (c)=\sigma_{\rho_2}(c)$ for every $c\in C(\rho_2)\setminus\{c_2\}$.
\item The vortex society
of $c_1$ in $\delta$  has a cylindrical rendition $\rho_1=(\Gamma_1,\zD_1,c_1')$ with
a nest $\zC_1$.
\end{itemize}
Let $\tilde\Delta_0$ be the vertices of $N(\rho)$ on the boundary of $\Delta$.  Let $X_1,X_2$ be two complementary segments of $\Omega_2$.
Assume  that there exists a set $\gamma=\{ (\zQ_i, X_i^2): 1 \le i \le c+h\}$ and disjoint linkages $\zP_1, \zP_2$ in $(G_2,\Omega_2)$ such that:
\begin{itemize}
\item $\zP_1$ is a linkage from $X_1$ to $\sigma_{\rho_1}(c_1')$ which is orthogonal to both $\zC_1$ and $\zC_2$ and $(\zP_1, \zC_1)$ overlays $W_1$.
\item The linkage $\zP_2$ is a linkage from $X_2$ to $\tilde\Delta_0$ which is orthogonal to $\zC_2$ and $(\zP_1 \cup \zP_2, \zC_2)$ overlays $W_1$.
\item For $1 \le i \le c+h$, $\zQ_i$ is a transaction in $(G_2, \Omega_2)$ and $X_i^2 \subseteq X_2$ is a segment of $\Omega_2$.  The transactions $\zQ_1,\zQ_2,\ldots,\zQ_{c+h}$ are pairwise disjoint with endpoints in $X_i^2$ and $X_1^2, \dots, X_{c+h}^2$ are also pairwise disjoint.  For all $i$, $\zQ_i$ is coterminal with $\zP_2$ up to level 1 of $\zC_2$ in $\rho_2$ and $\zQ_i$ is disjoint from $\sigma_{\delta}(c_1)$ and disjoint from $\zP_1$.
\item The transactions $\zQ_1,\zQ_2,\ldots,\zQ_{h}$ are handle transactions, each of thickness $\theta$, and
\item $\zQ_{h+1},\zQ_{h+2},\ldots,\zQ_{h+c}$ are crosscap transactions, each of thickness $\theta$.
\item Let $\zP_1=\{P_1,P_2,\ldots ,P_m\}$, for $i=1,2,\ldots ,m$ let $x_i$ be the endpoint 
of $P_i$ in $X_1$ and let $y_i$ be the last vertex  of $P_i$ in $\widetilde{c_1}$.
If $x_1,x_2,\ldots ,x_m$ appear in $\Omega _2$ in the order listed,
then $y_1,y_2,\ldots ,y_m$ appear on $\widetilde{c_1}$ in the order listed.
\end{itemize}

In those circumstances we say that
$(W_1, W_0, \delta, c_0, c_1, \Delta_0, \zC_1,\zC_2,\zP_1, \zP_2,\gamma)$ is
a \emph{$\Sigma$-configuration} with parameters $(r, |\zC_1|,|\zC_2|,|\zP_1|,\theta,c,h)$.
We  will say  that $c_1$ is the {\em exceptional cell} of $\delta$.
\end{DEF}

We now present a lemma which will allow us to find a clique minor in $\zC_2 \cup \zQ_1 \cup \dots \cup \zQ_{c + h}$ of a $\Sigma$ configuration with appropriately chosen parameters in our given society $(G, \Omega)$.


\begin{LE}\label{lem:findingaclique}
Let $p \ge 2$, $t \ge 3$.  Let $(G, \Omega)$ be a society with a cylindrical rendition $\rho = (\Gamma, \zD, c_0)$ with a nest $\zC$ of order $2p$.  Let $X_1, \dots, X_t$ be pairwise disjoint segments of $\Omega$.  Let $\zQ_1, \zQ_2, \dots, \zQ_t$ be pairwise disjoint transactions such that for all $i$, $1 \le i \le t$, we have that
\begin{itemize}
\item $\zQ_i$ is orthogonal to $\zC$ and every element of $\zQ_i$ has both endpoints contained in $X_i$, and
\item $\zQ_i$ is either a handle or crosscap transaction of thickness $p$.
\end{itemize}
If $t \ge \binom{p}{2} + p + 1$, then $G$ has a $K_p$ minor grasped by $\zC \cup \bigcup_1^t \zQ_i$.
\end{LE}
\begin{proof}
Let $\zC = (C_1, \dots, C_t)$.  For $i$, $1 \le i \le p$, let $x^i_1, \dots, x^i_t$ be vertices of $C_i$ such that traversing the vertices of $C_i$ we encounter $x_1^i, \dots, x_t^i$ in that order.  Let $R_j^i$ be the subpath of $C_i$ from $x_j^i$ to $x_{j+ 1}^i$ which is disjoint from $x_{j'}^i$ for all $j' \neq j, j+1$.  After possibly reordering the $\zQ_i$ and by appropriately choosing $x^i_j$, we may assume that the path $R^i_j$ contains $C_i \cap \zQ_j$ and that $R^i_j$ is internally disjoint from $\zQ_{j'}$ for $j' \neq j$.  Let $\zR_j = \{ R^i_j : 1 \le i \ \le p\}$ and let $H_j$ be the subgraph formed by the union of the elements of $\zQ_j$ and $\zR_j$.  Define $S_j = \{x^i_j: 1 \le i \le p\}$.  Note that since each $R^i_j$ is a non-trivial path, the sets $S_j$ are pairwise disjoint and both $S_j$ and $S_{j+1}$ are contained in $V(H_j)$.

Pick a triple $(\zR^*, \zL, k)$ with the following properties.  The value $k\ge 1$ is an integer.  The first term $\zR^*$ is a linkage of order $p$ with components $R^*_1, \dots, R^*_p$ such that:
\begin{itemize}
\item for all $1 \le i \le p$, $R^*_i$ is a path contained in $\bigcup_{l = p+1}^{p+ k}H_l$ with one endpoint equal to $x^i_{p+1}$, and
\item for all $1 \le i \le p$ and $p+1 \le l \le p+k$, if $H_l \cap R^*_i$ is non-empty, then it is either a path from $S_l$ to $S_{l+1}$ or a single vertex of $S_l$.
\end{itemize}
The second term $\zL$ is a linkage of order $k$ with with elements $L_1, \dots, L_k$ such that $L_i$ is a path in $H_{p+ i} - (S_{p+i}\cup S_{p+i+1})$ with endpoints in distinct elements of $\zR^*$ and with no internal vertex in $\zR^*$.  Moreover,
\begin{itemize}
\item for $i, i', j$, $1 \le i \le i' \le p$, $1 \le j < k$ such that $L_j$ has endpoints in $R_i^* \cup R_{i'}^*$, then no index $j' \neq j$ has the property that both endpoints of $L_{j'}$ are contained in $R^*_{i} \cup R^*_{i'}$.
\end{itemize}
We say that an element $R^*_i$ is \emph{uncovered} if there exists an index $i'\neq i$ such that for all $j \le k$, the path $L_j$ does not have both endpoints contained in $R_i^* \cup R_{i'}^*$.  Otherwise, we say $R^*_i$ is \emph{covered}.  Finally, we assume that
\begin{itemize}
\item for all $i$, $1 \le i \le p$, if there exists an index $j$ such that $R^*_i \cap H_j$ is either trivial or empty, then $R^*_i$ is covered.
\end{itemize}

Assume that we pick the triple $(\zR^*, \zL, k)$ to maximize $k$. Note that such a choice exists because we could fix $k=1$, $\zR^* = \zR_{p+1}$, and let $L_1$ be a path in $H_{p+1}$ linking $R^1_{p+1}$ and $R^2_{p+1}$ which is disjoint from $S_{p+1}\cup S_{p+2}$.

We claim that $k = \binom{p}{2}$ in our choice of $(\zR^*, \zL, k)$ to maximize $k$.  Obviously, by the constraint that each element of $\zL$ must link a distinct pair of paths in $\zR^*$, we have that $k \le \binom{p}{2}$.  Assume, to reach a contradiction, that $k < \binom {p}{2}$.  We will extend $\zR^*$ into $H_{k+1}$ and find a path $L_{k+1}$ violating our choice to maximize $k$.  The analysis of the subgraph $H_{k+1}$ will differ slightly depending on whether $\zQ_{k+1}$ is a crosscap or leap transaction.

Assume, as a case, that $\zQ_{k+1}$ is a crosscap transaction.  Fix $i$ to be the minimum index such that $x_{k+1}^i$ is contained in a path $R^*_j$
which is uncovered.  Fix the index $j$ and let $j'$ be such that for all $l\le k$, $L_l$ does not have endpoints in $R^*_j \cup R^*_{j'}$.
Let $i'$ be the index such that $R^*_{j'}$ has $x_{k+1}^{i'}$ as an endpoint.  Note that both $i$ and $i'$ are defined because, by the definition of $\zR^*$, both $R^*_j$ and $R^*_{j'}$ are non-trivial paths in $H_k$.  To extend the paths of $\zR^*$ to $S_{k+2}$, we route the elements of $\zR^*$ with endpoints in $x_{k+1}^i, \dots, x_{k_+1}^{i'-1}$ through $\zQ_{k+1}$ to terminate on $x_{k+2}^{i'-1}, \dots, x_{k+2}^{i}$, respectively.  For every element of $\zR^*$ terminating on $x_{k+1}^{i'},\dots, x_{k+1}^p$, we extend the path along $\zR_{k+1}$ to terminate on $x_{k+2}^{i'}, \dots, x_{k+2}^p$.  Note that any path of $\zR^*$ with endpoint in $x_{k+1}^1, \dots, x_{k+1}^{i-1}$ will not be extended and thus be a trivial path when restricted to $H_{k+1}$; however, by the choice of $i$, such paths are covered and we maintain the desired properties for $(\zR, \zL, k)$ in the extension.  Note as well that after extending the paths, $R^*_j$ terminates on $x_{k+2}^{i'-1}$ and $R^*_{j'}$ terminates on $x_{k+2}^{i'}$.  Thus, we can choose a subpath of $\zQ_{k+1}$ with endpoints in $R_{i'-1}$ and $R_{i'}$ to serve as $L_{k+1}$ and derive a contradiction to our choice of $(\zR, \zL, k)$ to maximize $k$.

The other case to consider is that $\zQ_{k+1}$ is a handle transaction; let $\zQ_{k+1}'$ and $\zQ_{k+1}''$ be the two constituent planar transactions of the handle $\zQ_{k+1}$, and assume that traversing $\zR_{k+1}$ from $S_{k+1}$ to $S_{k+2}$, we first encounter $\zQ_{k+1}'$.  Fix $i$ again to be the minimum index such that $x_{k+1}^i$ is contained in a path $R^*_j$ which is uncovered.  We will define $j, j', i'$ as in the previous paragraph.  Fix the index $j$, and let $j'\neq j$ be such that for all $l \le k$, $L_l$ does not have endpoints in $R^*_j \cup R^*_{j'}$.  Let $i'$ be the index such that $R^*_{j'}$ has $x^{i'}_{k+1}$ as an endpoint.  We extend the path $R^*_j$ along $R^i_{k+1}$ to $\zQ_{k+1}'$, through $\zQ_{k+1}'$ down to $R^{i'-1}_{k+1}$ and along $R^{i'-1}_{k+1}$ to $S_{k+2}$.  We extend the paths of $\zR^*$ with endpoints in $x^{i+1}_{k+1}, \dots, x^{i'-1}_{k+1}$ through $R^{i+1}_{k+1}, \dots, R^{i'-1}_{k+1}$ to $\zQ_{k+1}''$, through $\zQ''_{k+1}$ up to $R^{i}_{k+1}, \dots, R^{i'-2}_{k+1}$, and along $R^{i}_{k+1}, \dots, R^{i'-2}_{k+1}$ to $S_{k+2}$.  Finally, the paths of $\zR^*$ terminating on $x^{i'}_{k+1}, \dots, x^p_{k+1}$ extend along $\zR_{k+1}$ to their endpoints in $S_{k+2}$.  After doing so, $R^*_j$ terminates on $x^{i'-1}_{k+2}$ and $R^*_{j'}$ endpoints on $x^{i'}_{k+2}$.  Thus, we can fix a path of $\zQ_{k+1}''$ with one endpoint in $R^*_j$ and the other endpoint in $R^*_{j'}$ as $L_{k+1}$ and again derive a contradiction of our choice of triple to maximize $k$.  Note that we are using only one path of the linkage $\zQ_{k+1}'$ in routing the paths and so the argument is independent of whether or not the transaction $\zQ_{k+1}$ is consistent with $X_{k+1}$.

We conclude that $k = \binom{p}{2}$ and thus the elements of $\zR^*\cup \zL$ contains the branch sets of a $K_p$ minor.  We will use the subgraphs $H_1, \dots, H_p$ to extend the elements of $\zR^*$ to guarantee that the $K_p$ minor grasps $\zC \cup \zP$.  For $1 \le i \le p$, we add to $R^*_i$ a path contained in $R_i \cup H_i$ linking $R_i$ and $C_{2p}$ containing $Q \cap C_i$ for $p+1 \le i \le 2p$ and $p$ distinct elements $Q$ of $\zQ_i$.
These extensions can be chosen pairwise disjoint, completing the proof that $G$ contains a $K_p$ minor grasping $\zC \cup \bigcup_1^t\zQ_i$.

\end{proof}


We will deduce Theorem~\ref{thm:main} from the following result.

\begin{theorem}
\label{thm:main3}
Let $r, s_2,M ,p \ge 1$ be integers with $M \ge p, s_2 \ge 2p$.  Let $\mu = \left \lfloor \frac{1}{10^5p^{26}}p^{10^7p^{26}} \cdot M \right \rfloor$.  Let
$$R=49152p^{24}(r+ 2s_2) + \mu.$$
Let $G$ be a graph, and
let $W$ be an $R$-wall in $G$.
Then either $G$ has a model of a $K_p$ minor grasped by $W$,
or there exist integers $c,h \ge0$ with $c + h \le \binom{p+1}{2}$, a set $A\subseteq V(G)$ of size at most $\mu$, a surface $\Sigma$ obtained by adding $c$ crosscaps and $h$ handles to the sphere, and
a $\Sigma$-configuration $(W_1, W_0, \delta,c_0,c_1,\Delta_0,\zC_1,\zC_2,\zP_1, \zP_2,\gamma)$
of $G-A$ with parameters $(r, M,s_2,M,M,c,h)$ such that the vortex society of $c_1$ in $\delta$ has
depth at most $\mu$ and  breadth at most $2p^2$.
\end{theorem}

Note that in the proof of Theorem \ref{thm:main3}, the parameters $s_1$, $\theta$, and $|\zP_1|$ of the $\Sigma$-configuration all have the same asymptotic behavior, so to simplify the presentation, we use a single parameter $M$ for all three.  Similarly, the multiplicative term $\frac{1}{10^5p^{26}}$ in the definition of $\mu$ is not optimal, but rather chosen to simplify the calculations in the application of Theorem \ref{thm:main3} to the proof of Theorem \ref{thm:main} in the next section.

\begin{proof}
We first observe that the theorem holds trivially for $p \le 4$ because the $R$-wall contains $K_4$ as a minor for $R \ge 3$ and we can easily find a model of $K_4$ grasping the wall for slightly larger values of $R$.  Thus, we may assume that $p \ge 5$, and furthermore, we assume going forward that $G$ does not contain a $K_p$ minor grasped by $W$.

Let $c, h \ge 0$ be integers such that $c + h \le p(p+1)/2+1$ and let $g = c+h$.  We define the following:
\begin{equation*}
s(g) = t(g) = \left(6p^{10^7p^{24}}\right)^{\binom{p+1}{2} + 2-g}M.
\end{equation*}
Let $\alpha = 12288(p(p-1))^{12}$.  Note that, as $p \ge 5$, we have
\begin{align*}
49152p^{24}\left(40p^2 + (r + t(0) + 10 + 2s_2 + 2s(0)) \right)& \le 49152p^{24}\left[(r+ 2s_2) + 41p^2 + 3s(0)\right]\\
& \le 49152p^{24}(r+2s_2) + \\
& \qquad + 49152p^{24}\left[41p^2 + 3\left(6p^{10^7p^{24}}\right)^{\binom{p+1}{2} + 2} M \right] \\
& \le 49152p^{24}(r + 2s_2) + \\
& \qquad + 147457p^{24} \left(6p^{10^7p^{24}}\right)^{\binom{p+1}{2} + 2} M \\
& \le 49152p^{24}(r + 2s_2) + p^{32}\left( p^{10^7p^{24} + 2}\right)^{\binom{p+1}{2} + 2}M \\
& \le 49152p^{24}(r + 2s_2) + \frac{1}{10^5 p^{26}}\left( p^{10^7p^{24}}\right)^{\binom{p+1}{2} + 4}M  \\
& \le R.
\end{align*}.


\begin{claim}
\label{cl:basedec}
There exist a set $A\subseteq V(G)$ of size at most $\alpha$,
and an ${\mathbb S}^2$-configuration of $G-A$ $(W_0, W_1, \allowbreak  \delta', c_0, c_1, \Delta_0, \zC_1, \zC_2, \zP_1, \emptyset, \emptyset )$ with parameters
$(r,s(0),s_2,t(0),M,0,0)$.
\end{claim}


\begin{cproof}
Let $r'= \lceil r/2 \rceil + \lceil t(0)/2\rceil + 4 + s_2 + s(0)$.  By Theorem \ref{thm:flatwall} and the above calculation, we may assume that there exists a subset $A \subseteq V(G)$ and an $2r'$-subwall $W_1$ of $W$ such that $W_1$ is flat in $G-A$.  Let $(X, Y)$ be a separation of $G-A$ such that there exists a set $B$ of pegs of $W_1$ such that $B \subseteq X$, $X \cap Y$ is contained in the outer cycle of $W_1$, $V(W_1) \subseteq X$, and $(G[X], \Omega')$ is rural where $\Omega'$ is the cyclic ordering of $V(X \cap Y)$ induced by the outer cycle of $W_1$.  Let $H_1, \dots, H_{2r'}$ and $U_1, \dots, U_{2r'}$ be the  horizontal and vertical paths of $W_1$, respectively.
Define an ${\mathbb S}^2$ decomposition $\rho'$ of $G$ by combining a vortex-free rendition of $(G[X], \Omega)$ in the disc with a single vortex, call it $c'$, such that $\sigma_{\rho'}(c') = G[Y] - E(G[X])$.  For $1 \le i \le r'$, let $C_i$ be the unique cycle contained in $U_i \cup U_{2r' - i + 1} \cup H_i \cup H_{2r' - i + 1}$.  Note that $(C_2, \dots, C_{r'})$ forms a nest around the vortex $c'$.

We modify the decomposition $\rho'$ as follows: let $T_0$ be the disc bounded by the track of $C_{r' - \lceil r/2 \rceil - \lceil t(0)/2\rceil }$ which does not contain $c'$ and let $T_1$ the disc which contains $c'$ bounded by the track of $C_{s(0) + 2}$ which contains $c'$.  Define $\delta'$ to be the decomposition merging all cells in $T_1$ into a single vortex $c_1$ and merging all the cells in $T_0$ into a single single vortex $c_0$.  Thus $\delta'$ has exactly two vortices, namely $c_0$ and $c_1$.

The restriction of $\rho'$ to $T_0$ implies that the vortex society of $c_0$ is rural, and the choice of parameters implies that it contains a subwall $W_0$ of $W_1$ which is an $r + t(0)$ wall with outer cycle equal to $C_{r' - \lceil r/2 \rceil - \lceil t(0)/2\rceil -1}$.  The restriction of $\rho'$ to $T_0$ provides the desired vortex-free rendition of the vortex society of $c_0$ with $W_0$ flat.

Define the disc $\Delta_0$ to be the closed disc bounded by the track of $C_{s(0) +3}$, and let $\zC_2 = (C_{s(0) + 4}, \dots, \allowbreak C_{s(0) + s_2 + 3})$.  Then the complementary society of $c_0$ has the desired cylindrical rendition with the nest $\zC_2$.  Note $s(0) + 3 + M = r' - \lceil r/2 \rceil - \lceil t(0)/2 \rceil - 1$, motivating the choice of $r'$.

Let $\zC_1 = (C_2, \dots, C_{s(0) + 1})$.  The restriction of $\rho'$ to $T_1$ instead yields the desired cylindrical rendition of the vortex society of $c_1$.

For each path $U_{r'- \lceil t(0)/2 \rceil + 1}, \dots, U_{r' + \lceil t(0)/2 \rceil}$, define $U_{r'- \lceil t(0)/2 \rceil + 1}', \dots, U_{r' + \lceil t(0)/2 \rceil}'$ to be a path from $H_{2r'}$ to $Y$ obtained by extending $U_i$ along $H_{1}$ to the first vertex in $Y$.  Define $\zP_1$ to be the restriction of the vertical paths $U_{r'- \lceil t(0)/2 \rceil + 1}', \dots, U_{r' + \lceil t(0)/2 \rceil}'$ from their last vertex in $\widetilde{c_0}$ to the first vertex in $Y$ when traversing each $U_i'$ from $H_{2r'}$ to $Y$. Thus, by the choice of $\zP_1$ and $\zC_1$, $\zP_1$ is orthogonal to $\zC_1$ and $\zC_2$ and $(\zP_1, \zC_1)$ and $(\zP_1, \zC_2)$ both overlay $W_1$.

As $c+h = 0$, setting $\gamma = \emptyset$ and $\zP_2 = \emptyset$ completes the definition of the desired $\mathbb{S}^2$-configuration.
\end{cproof}


By Claim~\ref{cl:basedec} we may select subwalls $W_0$ and $W_1$, a surface $\Sigma$ obtained from the sphere by the addition of $h$ handles and $c$ crosscaps,
a set $A\subseteq V(G)$ of size at most $2\alpha(g+1)$,
 and a $\Sigma$-configuration $\kappa:=(W_1, W_0, \delta,c_0,c_1,\Delta_0,\zC_1,\zC_2,\zP_1, \zP_2,\gamma)$
of $G-A$ which overlays $W_1$ with parameters $(r, s(g),s_2,t(g),\allowbreak M,c,h)$.  Let $\delta = (\Gamma, \zD)$. Let $X_1$ and $X_2$ be complementary segments of the complementary society to the vortex society of $c_0$ as in the definition of a configuration.  We assume that the $\Sigma$ configuration is chosen so that $g \le \binom{p+1}{2} + 1$ is maximum.

We will show that there exist $A'$ with $A \subseteq A'$ and a $\Sigma$-configuration $\kappa'$ of $(G-A', \Omega)$ which satisfy the conclusion of the theorem.

We first observe $g \le \binom{p+ 1}{2}$.  To see this, if $g = \binom{p+1}{2} + 1$ by the rendition in the disc of the outer society of the vortex $c_0$, the nest $\zC_2$ of order $s_2 \ge 2p$, and the handle and crosscap transactions of $\gamma$, each of order $M \ge p$, by Lemma \ref{lem:findingaclique} there exists a $K_p$ minor grasped by $\zC_2 \cup \zP_2$.  As $(\zP_1 \cup \zP_2, \zC_2)$ overlays $W_1$, the clique minor grasps $W$ as well, a contradiction.  Note, here we are also using the assumption that $s_2 \ge 2p$.

Let $\rho_1 = (\Gamma_1, \zD_1, c_1')$ be a cylindrical rendition of the vortex society $(G_1, \Omega_1)$ of $c_1$ in the closure of $c_1$.  Let $\zC_1 = (C_1, \dots, C_{s(g)})$ be the nest given by $\kappa$.  We may assume that $\pi_\kappa^{-1}(v) = \pi_{\rho_1}^{-1}(v)$ for all $v \in V(\Omega_1)$.

Apply Theorem~\ref{thm:maintrans2}
to  the society $(G_1,\Omega_1)$, the cylindrical rendition
$\rho_1$, $p$,  $\theta:=\theta(g)=2M +4s(g+1) + t(g+1)+12$, nest
 $(C_1, \dots, C_{s(g)})$ and the set $\zR_1$ of paths obtained by restricting each element of $\zP_1$ from the last vertex in $V(\Omega_1)$ to the endpoint in $\sigma_{\rho_1}(c_1')$.
Let $K_\theta = p^{10^7p^{24}} \theta$.
We observe that for $g \le \binom{p+1}{2}$, $p \ge 5$ and $M \ge p$,
\begin{align*}
K_\theta + 2s(g+1) +8 &=\left[ 2M + 4s(g+1) + t(g+1) + 12 \right] p^{10^7p^{24}} + 2 s(g+1) + 8 \\
&\le 5\cdot p^{10^7p^{24}}s(g+1) + (2M+14)p^{10^7p^{24}} + 2s(g+1) \\
&\le 5\cdot p^{10^7p^{24}}s(g+1) + 3s(g+1) \\
&\le  s(g) - M - 1.
\end{align*}

We consider the four possible outcomes of Theorem \ref{thm:maintrans2}.  The analysis of outcomes (i) and (ii) of Theorem \ref{thm:maintrans2} follows the same line of reasoning in each case.  In outcome (i), we use Lemma \ref{lem:xcaptrans} to find a subset $Z$ of at most $\alpha$ vertices and a rendition of $(G_1 - Z, \Omega_1)$ in the disc plus a crosscap with a unique vortex which is disjoint from a crosscap transaction of order $M$.
We glue the rendition of $(G_1-Z, \Omega_1)$ to the restriction of $\kappa$ to $\Sigma \setminus c_1$ to find a $\Sigma^*$-configuration in the surface $\Sigma^*$ obtained from $\Sigma$ by adding a single crosscap, contradicting our choice to maximize $g$.  In outcome (ii), the set $|Z|$ has at most $2\alpha$ vertices and instead we use Lemma \ref{lem:handletrans} to find a configuration in the surface obtained from $\Sigma$ by adding a single handle.  The value of $\theta$ is chosen with the application of Lemma \ref{lem:handletrans} in mind.
We present the analysis of outcome (ii) in detail and in the interest of brevity, we omit the analogous argument for outcome (i).

Assume, as a case, outcome (ii) of Theorem \ref{thm:maintrans2} holds.
Thus there exist a set $Z\subseteq V(G_1)\setminus V(\Omega_1)$
and a handle transaction $\zQ$ of thickness $\theta$ in $(G_1- Z, \Omega_1)$ with constituent transactions $\zQ_1$, $\zQ_2$ such that $|Z|\le 2\alpha$, the $\zQ_i$-strip in $(G_1- Z, \Omega_1)$ is isolated and rural for $i = 1,2$,
$\zQ_i$ is coterminal with $\zR_1$ up to level $K_\theta$ for $i = 1, 2$, and $Z$ is disjoint from the outer graph of $C_{K_\theta}$
in the cylindrical rendition $\rho_1$.

We wish to apply Lemma \ref{lem:handletrans}.  To do so, we first have to modify the rendition $\rho_1$ to get a rendition $\rho_2$ suitable for the application of Lemma \ref{lem:handletrans} as well as define our two linkages and their corresponding segments of the society vertices.  Lemma \ref{lem:handletrans} will yield a third rendition $\rho_3$ of $(G_1-Z, \Omega_1)$ which we can use to derive a contradiction to our choice to maximize $g$.  Define the cylindrical rendition $\rho_2$ of $(G_1 - Z, \Omega_1)$ as follows.  Let $\Delta_{K_\theta} \subseteq c_1$ be the closed disc bounded by the track of $C_{K_\theta}$ in the rendition $\rho_1$.  Let $\rho_2 = (\Gamma_1, {\zD_2}, {c_2})$ such that ${c_2} = \Delta_{K_\theta} \setminus \pi_{\rho_1}^{-1}(V(C_{K_\theta}))$ and ${\zD_2} = \{ c \in \zD_1: c \nsubseteq \Delta_{K_{\theta}}\} \cup \{\Delta_{K_\theta}\}$.

The transaction $\zQ$ is coterminal with $\zR_1$, and so we can extend the paths along $\zP_1$ to get a trasaction with endpoints in $X_1$.  We partition $X_1$ into two segments, $Y_1^*$ and $Y_2^*$ so that exactly $t(g+1)$ elements of $\zP_1$ which intersect $\zQ$ have an endpoint in $Y^*_1$.  By the properties of the $\Sigma$-configuration and that the cyclic order of the elements of $\zP_1$ on $\Omega_2$ and $\Omega_1$ are the same, we can find disjoint segments $Y_1$ and $Y_2$ of $\Omega_1$ such that for every $P \in \zP_1$ such that $P$ intersects an element of $\zQ$, $V(P) \cap V(\Omega_1) \subseteq Y_1$ if $P$ has an end in $Y_1^*$ and $V(P) \cap V(\Omega_1) \subseteq Y_2$ if $P$ has an end in $Y_2$.  Let ${\zR_2}$ be the $Y_1 - V(\sigma_{\rho_2}(c_2))$ linkage contained in $\zR_1$.  Note $|\zR_2| = t(g+1)$.  For $i = 1,2$, let $\zQ_i'$ be a set of $4s(g+1) + 2m + 13$ elements of $\zQ_i$ with an endpoint in $Y_2$.  Such a choice is possible as at most $t(g+1)$ elements of $\zQ_i$ have an end in $Y_1$ and thus not in $Y_2$.  The transactions $\zQ_1'$ and $\zQ_2'$ are the constituent transactions of a handle transaction of thickness $2M + 4s(g+1) + 12$.  Finally, note that $(C_{K_\theta + 1}, \dots, C_{s(g)})$ is a nest in $\rho_2$ of order at least $2s(g+1) + 8$ and both $\zR_2$ and $\zQ_1' \cup \zQ_2'$ are orthogonal to $(C_{K_\theta + 1}, \dots, C_{s(g)})$.

Apply Lemma \ref{lem:handletrans} to $(G_1 - Z, \Omega_1)$, the segments $Y_1, Y_2$ in $V(\Omega_1)$, rendition $\rho_2$, nest $(C_{K_\theta+1}, \dots, C_{s(g)})$, and linkages $\zQ_1' \cup \zQ_2'$ and $\zR_2$ with the parameters $s = s(g+1)$ and $\theta = M$.  Thus, there exist $\zQ^+ \subseteq \zQ_1'\cup \zQ_2'$ such that $\zQ^+$ is a handle transaction of thickness $M$ and rendition $\rho_3$ in $\Delta^+$ for $\Delta^+$ equal to the closure of $c_1$ with a handle added to the interior of $c_1'$.   Let $\rho_3 = (\Gamma_3, \zD_3)$ with the unique vortex labeled $c_3$.   The lemma ensures that $\zQ^+$ is disjoint from $\sigma_{\rho_3}(c_3)$.  The vortex society of $c_3$ has a cylindrical rendition $\rho_3' = (\Gamma_3', \zD_3', c_3')$ with a nest $\zC' = (C_1', \dots, C_s')$ as the outcome of Lemma \ref{lem:handletrans}.

As $\rho_3$ is equal to $\rho_1$ outside the disc bounded by the track of $C_{K_\theta + s(g+1) + 8}$ in $\rho_1$, we can extend the linkages $\zR_3$ and $\zQ^+$ along $\zP_1$ to obtain linkages $\bar\zP_1$ and $\bar\zQ^+$ such $\bar\zP_1$ links $Y_1$ and $\widetilde{c_3'}$ in $\rho_3'$ and every element of $\bar\zQ^+$ has both endpoints in $Y_2$.  Note that the final property in Lemma \ref{lem:xcaptrans} ensures that $\bar\zQ^+$ is a handle transaction in the complementary society to the vortex society of $c_0$ in $\delta$. Let $\bar\zP_2$ be the union of $\zP_2$ along with the $V(\Omega_2) - \tilde\Delta_0$ linkage contained in $\bar\zQ^+$.

Finally, we let $\Sigma^+$ be the surface $(\Sigma \setminus c_1 )\cup \Delta^+$ and we define a $\Sigma^+$-decomposition $\delta^+ = (\Gamma^+, \zD^+)$ as follows.  We define $\Gamma^+$ to be equal to $\Gamma$ in $\Sigma \setminus c_1$ and equal to $\Gamma_3$ in $\Delta^+$.  Define $\zD^+= (\zD \setminus c_1) \cup \zD_3$.
We conclude that $(W_1, W_0, \delta^+, c_0, c_3, \Delta_0, \zC', \zC_2, \bar\zP_1,\allowbreak \bar\zP_2, \gamma \cup (\zQ^+, Y_2))$ is a $\Sigma^*$-configuration of $G - (A \cup Z)$ with parameters $(r, s(g+1), s_2, t(g+1), M, c, h+1)$, contradicting our choice to maximize $g$.  This completes the analysis of the case of outcome (ii) in the application of Theorem \ref{thm:maintrans2}.

The analysis of outcome (i) of Theorem \ref{thm:maintrans2} is analogous to the previous argument, applying Lemma \ref{lem:xcaptrans} instead of Lemma \ref{lem:handletrans}.  We omit the details here.  Outcome (iii) yields a $K_p$ minor grasped by $\zC_1 \cup \zR_1$.  The elements of $\zR_1$ are subpaths of the elements of $\zP_1$ and by the property that $(\zP_1, \zC_1)$ overlays $W_1$, we see that the $K_p$ minor is grasped by $W$ as well, a contradiction.

We conclude that outcome (iv) of Theorem \ref{thm:maintrans2} holds.  Thus, there exists $Z \subseteq V(G_1) \setminus \Omega_1$ with $|Z| \le 5 \cdot 10^{20}p^{102}(\theta + \alpha)$ and a rendition in the disc of $(G_1 - Z, \Omega)$ of depth \begin{align*}
p^{10^7p^{24}}(2M + 4s(g+1) + t(g+1) + 12) &\le s(g) \\
&\le \left( 6 p^{10^7p^{24}}\right)^{\binom{p+1}{2} + 2} \cdot M \\
& \le \frac{1}{10^5p^{26}} \left( p^{10^7p^{24}}\right)^{\binom{p+1}{2} + 4} \cdot M \\
& \le \left \lfloor \frac{1}{10^5p^{26}} p^{10^7p^{26}} M \right \rfloor = \mu,
\end{align*}
 and breadth $2p^2$.
Using the inequality $K_\theta + 2s(g+1) + 8 \le s(g) - M - 1$ which we proved above and the fact that $p \ge 5$, we see that
\begin{align*}
|A \cup Z| &\le 2 \alpha \cdot (g+1)+ 5 \cdot 10^{20}p^{102}\theta + 5 \cdot 10^{20}p^{102} \alpha  \\
&\le K_\theta + 2\alpha \left( 5 \cdot 10^{20}p^{102} + g + 1\right) \\
&\le K_\theta + 2 s(g+1) \\
&\le s(G) - M \le \mu.
\end{align*}
It follows that we can fix $\zC_1'$ to be a set of $M$ elements of $\zC_1$ which are disjoint from $A \cup Z$ and similarly, fix $\zP_1'$ to be a set of $M$ elements of $\zP_1$ which are disjoint from $A \cup Z$.
Define $\delta' = (\Gamma', \zD')$ to be the $\Sigma$-decomposition of $G-(A \cup Z)$ with $\Gamma'$ the restriction of the drawing $\Gamma$ to the subgraph $G-(A \cup Z)$ and $\zD' = \zD$.  We conclude that $(W_1, W_0, \delta', c_0, c_1, \Delta_0, \zC_1', \zC_2, \zP_1', \zP_2, \gamma)$ is the desired $\Sigma$-decomposition of $G-(A \cup Z)$ with parameters $(r, M, s_2, M, M, c, h)$ such that the vortex society of $c_1$ has breadth at most $2p^2$ and depth at most $\mu$.  This completes the proof.
\end{proof}

\section{Bounded depth vortices and the proof of Theorem \ref{thm:main}}

The statements of Theorem \ref{thm:main} and \ref{thm:main3} are quite close - the final tool which we will need in order to deduce Theorem \ref{thm:main} is a further understanding of the structure of bounded depth vortices. Then we finish the proof of Theorem \ref{thm:main}. 

Recall that, as in Section \ref{sec:planarstrip}, by an interval we mean a set of consecutive, positive integers.
\begin{DEF}
Let $(G, \Omega)$ be a society.  A \emph{linear decomposition} of $(G, \Omega)$ is a labeling $v_1, \dots, v_n$ of $V(\Omega)$ such that $v_1, v_2, \dots, v_n$ occur in that order on $\Omega$ and  subsets $(X_1, X_2, \dots, X_n)$ such that
\begin{itemize}
\item $X_i \subseteq V(G)$ with $v_i \in X_i$ for all $i$ and $\bigcup_1^n X_i = V(G)$.  For every edge $uv \in E(G)$, there exists $i$ such that $\{u,v\} \subseteq X_i$.
\item For all $x \in V(G)$, the set $\{i : x \in X_i\}$ is an interval in $\{1, 2, \dots, n\}$.
\end{itemize}
The {\em adhesion} of such a linear decomposition is
$$\max(|X_i\cap X_{i+1}|\,:\,i=1,2,\ldots,n-1).$$
(This is called ``depth" in \cite{RS9}, but we prefer the term adhesion.)
\end{DEF}
It is not hard to see that every society with a linear decomposition of adhesion at most $d$ has depth at most $2d$.
The following partial converse is shown in \cite[Theorem 8.1]{RS9}.  We include the original proof for completeness.
\begin{theorem}[{\cite[Theorem 8.1]{RS9}}]\label{thm:bdeddepthvortex}
Let $d\ge0$ be an integer.
Every society of depth at most $d$ has a linear decomposition of adhesion at most $d$.
\end{theorem}

\begin{proof}
Fix an enumeration $v_1, \dots, v_n$ of $V(\Omega)$ such that $v_1, \dots, v_n$ occur in that order.  For every $i$, $1 \le i \le n-1$, there exists a separation of order at most $d$ separating the sets of vertices $\{x_1, \dots, x_i\}$ and $\{x_{i+1},\dots,x_n\}$.  Let $(A_i, B_i)$ to be such a separation of minimal order, and subject to that, we fix $(A_i, B_i)$ to minimize $|A_i|$ from all such minimal order separations.

It follows that $A_i \subsetneq A_{i+1}$ for $1 \le i \le n-2$.  To see this, fix an index $i$, and assume that $A_i \nsubseteq A_{i+1}$.  Thus, $(A_i \cap A_{i+1}, B_i \cup B_{i+1})$ is a separation separating $\{x_1, \dots, x_{i}\}$ from $\{x_{i+1}, \dots, x_n\}$ with $A_i \cap A_{i+1}$ a proper subset of $A_i$.  By our choice of $(A_i, B_i)$, it follows that the order of $(A_i \cap A_{i+1}, B_i \cup B_{i+1})$ is strictly greater than $|A_i \cap B_i|$.  But in the case, $(A_i \cup A_{i+1}, B_i \cap B_{i+1})$ is a separation of order strictly less than $|A_{i+1} \cap B_{i+1}|$ and separates $\{x_1, \dots, x_{i+1}\}$ and $\{x_{i+2}, \dots, x_n\}$, contrary to our choice of $(A_{i+1}, B_{i+1})$.  We conclude that $A_i \subseteq A_{i+1}$ and as $x_{i+1} \in A_{i+1} \setminus A_i$, $A_i \subsetneq A_{i+1}$.  Similarly, $B_i \supsetneq B_{i+1}$ for all $1 \le i \le n-2$.

For $2 \le i \le n-1$, fix $X_i = A_i \cap B_{i-1}$.  Let $X_1 = A_1$ and $X_n=B_{n-1}$.  Note $v_i \in X_i$ for $1 \le i \le n$.  We claim $(X_1, \dots, X_n)$ form a linear decomposition of $(G, \Omega)$.  First, observe that any isolated vertices (in fact, any components not containing a vertex $v_i$) are contained in the set $X_n$.  Consider an edge $e = uv$ of $G$.  Assume, as a case, that $\{u,v\} \subseteq A_{n}$.  Fix $i$ to be the smallest index such that both $u$ and $v$ are contained in $A_i$.  If $i =1$, then $\{u,v\} \subseteq X_1$ by definition.  Otherwise, as $u$ and $v$ are not both contained in $A_{i-1}$, it must hold that $\{u,v\} \subseteq B_{i-1}$ and thus, $\{u,v\} \subseteq X_i$.  We reduce to the case that both $u$ and $v$ are not contained in $A_n$.  However, as the two endpoints of every edge is either contained in $A_n$ or $B_n$, we may assume that $\{u,v\} \subseteq B_n$.  In this case, $\{u,v\} \subseteq X_n$.
We conclude every vertex and every edge is contained in some $X_i$.

To complete the proof that $(X_1, \dots, X_n)$ is a linear decomposition, consider a vertex $x$ of $G$.  As $A_1 \subseteq A_2 \subseteq \cdots A_n$ and $B_1 \supseteq B_2 \supseteq \cdots B_n$, there exists some index $s$ and an index $t$ such that $x \in A_i$ for all $i \ge s$ and $x \in B_i$ for all $i \le t$.  Moreover, as $x \in A_i$ or $x\in B_i$ for all $i$, we have that $s \le t$.  By the definition of $X_i$, we have that $x \in X_i$ for exactly the indices $i$ such that $s \le i \le t+1$ and specifically, there is an interval of indices $i$ such that $x \in A_i$ as desired.
\end{proof}

Note that from the definition of a linear decomposition $(X_1, \dots, X_n)$ of a society $(G, \Omega)$, it easily follows that for any $i$, the set $Z = (X_i \cap X_{i-1}) \cup (X_i \cap X_{i+1}) \cup \{v_i\}$ separates $X_i \setminus Z$ from the set $V(G) \setminus X_i$.  Moreover, as every edge is contained in some $X_i$ and each vertex is contained in an interval of the $X_i$, we have that for every $j$, $1 \le j \le n-1$, the set $Z = X_j \cap X_{j+1}$ is a cutset separating $\bigcup_1^j X_i \setminus Z$ from $\bigcup_{j+1}^n X_i \setminus Z$.  Thus, for any connected subgraph $H$ of $G$, we have that $\{i : V(H) \cap X_i \neq \emptyset\}$ is also an interval.

We now proceed with the proof of Theorem \ref{thm:main}.

\begin{proof}[Proof. Theorem \ref{thm:main}]
First, we observe that as in the proof of Theorem \ref{thm:main3}, for $p \le 4$, there trivially exists a $K_p$ minor grasped by $W$.  We assume for the remainder of the proof that $p \ge 5$ and that there is no $K_p$ minor grasped by $W$.

Apply Theorem \ref{thm:main3} with the values $M = p, s_2 = 2p$, $r' = r + \left \lfloor \frac{1}{10^5p^{24}}\cdot p^{10^7p^{26}} \right \rfloor$.  Let $\mu =  \left \lfloor \frac{1}{10^5p^{25}}\cdot p^{10^7p^{26}}\right \rfloor$.  We may do so because
\begin{align*}
49152p^{24}(r' + 4p) + \mu &\le 49152p^{24}(r + \frac{1}{10^5p^{24}}\cdot p^{10^7p^{26}} + 4p) + \mu \\
& \le 49152p^{24}r + \left( \frac{49152}{10^5} + \frac{1}{10^5} + \frac{1}{10^5p^{25}}\right) p^{10^7p^{26}}\\
&\le R.
\end{align*}

There exists a set $Z \subseteq V(G)$, $|Z| \le \mu$, and a surface $\Sigma$ formed by adding $c$ crosscaps and $h$ handles to the sphere such that $G-Z$ has a $\Sigma$-configuration $(W_1, S_0, \delta, c_0, c_1, \Delta_0, \zC_1, \zC_2, \zP_1, \zP_2, \gamma)$ with parameters $(r', p, 2p, p, p, c, h)$ such that the vortex society of $c_1$ has depth at most $\mu$.  Moreover, as $c+h \le \binom{p+1}{2}$, we have that $\Sigma$ has genus at most $p(p+1)$.

We extend $\delta$ to a $\Sigma$-decomposition $\delta' = (\Gamma', \zD')$ of $G-Z$ by fixing a vortex free rendition of the vortex society of $c_0$ such that the wall $W_0$ is flat in the rendition and fixing a rendition of breadth at most $2p^2$ and depth at most $\mu$ of the vortex society of $c_1$.  Thus, $\delta'$ has breadth $t \le 2p^2$ and depth at most $\mu$.

Let $c_1, \dots, c_t \in C(\delta')$ be the vortices of the decomposition $\delta'$.
Let $P_1, \dots, P_r$ be the horizontal paths and $Q_1, \dots, Q_r$ be the vertical paths of $W$.  Let $\zH = \{P_i \cup Q_j: 1 \le i \le r, 1 \le j \le r \text{ such that }(V(P_i) \cup V(Q_j)) \cap Z = \emptyset\}$.  As there exist $r' - \mu$ disjoint paths in $W - Z$ linking every element $T \in \zH$ and the wall $W_0$, it follows that the only cells $c \in C(\delta')$ such that $\sigma_{\delta'}(c)$ contains an element of $\zH$ as a subgraph are when $c = c_i$ for some $i\le t$.

\begin{claim}
Let $1 \le i \le t$ and let $(G_i, \Omega_i)$ be the vortex society of $c_i$.  There exists a set $Z_i \subseteq V(G_i)$, $|Z_i| \le 2\mu + 1$ such that $G_i - Z_i$ does not contain any element of $\zH$ as a subgraph.
\end{claim}

\begin{cproof}
Fix $i$.  By Theorem \ref{thm:bdeddepthvortex}, there exists a linear decomposition $v_1, \dots, v_n$, $(X_1, \dots, X_n)$ of $(G_i, \Omega_i)$.  For every $T \in \zH$ which is contained in $G_i$, the set $I_T = \{j: 1 \le j \le n \text{ such that } V(T) \cap X_j \neq \emptyset\}$ is an interval since $T$ is a connected subgraph.  Fix $k$ to be the smallest index such that there exists $T \in \zH$ such that $k$ is the largest value in $I_T$.  Then for every $T \in \zH$ such that $T$ is a subgraph of $G_i$, it follows that $V(T) \cap X_k \neq \emptyset$.

Let $Z_i = (X_k \cap X_{k-1}) \cup (X_k \cap X_{k+1}) \cup \{v_k\}$.  The order of $Z_i$ is at most $2\mu + 1$, and we claim $Z_i$ is the desired set of vertices such that $G-Z_i$ does not contain any element of $\zH$.  Assume, to reach a contradiction, that there exists $T \in \zH$ which is a subgraph of $G_i$ but which is disjoint from $Z_i$.  As $Z_i$ is a cutset separating $X_k \setminus Z_i$ from the rest of the graph and $T$ is connected and intersects $X_k$, we have that $T$ is a subgraph of $G_i[X_k]$.  But there exist $r' - \mu$ disjoint paths in $W-Z$ linking $W_0$ and $T$
and
\begin{align*}
r' - \mu &= r + \left \lfloor \frac{1}{10^5p^{24}}p^{10^7p^{26}} \right \rfloor - \left \lfloor \frac{1}{10^5p^{25}}p^{10^7p^{26}}\right \rfloor \\
&\ge \frac{p -1}{10^5p^{25}}p^{10^7p^{26}}- 1 \\
& > 3 \mu > |Z_i|
\end{align*}
contradicting the fact that $Z_i$ is a cutset separating $W_0$ and $T$ in $G-Z$.  This completes the proof of the claim.
\end{cproof}

Fix $Z^* = Z \cup \bigcup_1^t Z_i$.  We have that
\begin{align*}
|Z^*| &\le 2p^2(2\mu + 1) + \mu \\
& \le \frac{5p^2}{10^5p^{25}} p^{10^7p^{26}}\\
& \le p^{10^7p^{26}}.
\end{align*}
Note as well that $Z^*\setminus Z \subseteq \bigcup_1^t V(\sigma_{\delta'}(c_i))$ and thus $W_0$ is disjoint from $Z^*$.  If we let $\delta^*$ be the restriction of $\delta'$ to the subgraph $G-Z^*$, we have the desired $W$-central $\Sigma$-decomposition of $G-Z^*$ of depth at most $\mu \le p^{10^7p^{24}}$ and breadth at most $2p^2$.  Moreover, the $r$-wall $W_0$ is flat in $\delta^*$, completing the proof of the theorem.
\end{proof} 
\section{Bounding the genus}
\label{sec:bdgenus}

The objective of this section is to prove Theorem~\ref{thm:main2}.

Let $\Sigma$ be a surface of Euler genus $g=2h+c$ with $h$ handles and $c$ cross-caps.
By a D-representation of $\Sigma$, we mean a
representation of $\Sigma$ as a disk $\Delta$ with $h$ handles and $c$ cross-caps attached to the boundary.
More precisely, a \emph{D-representation} of $\Sigma$ is a triple $\delta = (\Delta, I, (X_1, X_2, \dots, X_{c+h}))$ where $\Delta$ is a convex $4h+2c$-gon with a fixed clockwise order on the boundary, $I$ is a subset of  $\{1,2,\ldots,c+h\}$ of order exactly $h$, and $(X_1, X_2, \dots, X_{c+h})$ is a partition of the sides of $\Delta$ with the following properties.  For $i \in I$, $X_i$ consists of four consecutive sides of $\Delta$, say $a_i, b_i, a'_i, b'_i$, in clockwise order.  For $i \notin I$, $X_i$ consists of two consecutive sides $c_i, c_i'$ of $\Delta$.  We will typically assume that the sets of sides $X_1, X_2, \dots, X_{c+h}$ occur on the boundary in that order.  The \emph{realization} of $\delta$ is the surface $\Sigma'$ homeomorphic to $\Sigma$ obtained from $\Delta$ by identifying $a_i$ with $a_i'$ and $b_i$ with $b_i'$ in the same direction for $i\in I$
and identifying $c_j$ with $c_j'$ in opposite direction for $j\notin I$.  We say that $I$ is also the \emph{type} of $\delta$.  

Let $L$ be a graph which embeds in a surface $\Sigma$, and let $\delta = (\Delta, I, (X_1, \dots, X_{c+h})$ be a D-representation of $\Sigma$.  Let $\Sigma'$ be the realization $\delta$ homeomorphic to $\Sigma$.  A \emph{clean drawing} of $L$ in $\Sigma'$ is a drawing such that
\begin{itemize}
\item every vertex is in the interior of $\Delta$, and
\item each edge of $L'$ uses at most one handle or
cross-cap.  That is, every edge intersects at most one of the sides $a_i,b_i,c_j$, and if it does, it does so at most once
and does so transversally, by which we mean that if it intersects say $a_i$, then it crosses from $a_i$ to $a_i'$.
\item The set of edges which use handles or crosscaps forms a matching in $L$.
\end{itemize}
Note that this is sometimes called the ``cut graph" of a graph embedded in a surface.  It is easy to see that for every graph $L$ which embeds in a surface $\Sigma$ and for every D-representation of $\Sigma$ with realization $\Sigma'$, that there exists a graph $L'$ which contains $L$ as a minor and such that $L'$ has a clean drawing.  The graph $L'$ can actually be obtained by repeatedly subdividing edges of $L$

\begin{LE}
\label{lem:Hexists}
Let $L$ be a graph which embeds in a surface $\Sigma$.  Let $\delta = (\Delta, I, (X_1, \dots, X_{c+h}))$ be a D-representation of $\Sigma$ with realization $\Sigma'$.  There exists a graph $L'$ which contains $L$ as a minor such that $L'$ has a clean drawing in $\Sigma'$.  Moreover, if $X \subseteq E(L')$ is the set of edges of $L'$ which use either a cross-cap or handle in the drawing in $\Sigma$, then $L' - X$ has a Hamiltonian cycle.
\end{LE}

\begin{proof}
Fix $L'$ to be a graph containing $L$ as a minor which has a clean drawing in $\Sigma'$.  Let $X$ be the set of edges of $L'$ which use either a crosscap or handle of the drawing in $\Sigma'$.  We may assume, by adding edges, that $L' - X$ is connected and that it has a cycle.  From all such possible $L'$, we pick $L'$ and a cycle $C$ in $L' - X$ to minimize the number of vertices in $V(L') \setminus V(C)$.
We claim that $C$ is a Hamiltonian cycle in $L' - X$.

Suppose to the contrary that $J$ is a component of $(L'-X)-V(C)$. Since $L' - X$ is connected,
it has an edge $uv$ with $u\in V(C)$ and $v\in V(J)$.
Let $u'$ be a neighbor of $u$ in $C$, and let $c$ be the image of a simple curve with endpoints $v$ and $u'$
whose interior intersects $L'$ only in edges of $H-E(C)$ incident with $u$, and it intersects each such edge
at most once. We now subdivide each edge intersected by $c$ and add a path $P$ from $v$ to $u'$
using the new vertices. Let the new graph be $H'$ and let $C'$ be the cycle in $H'$ obtained from $C-uu'$
by adding the path $P$ and edge $uv$.
Now $H',C'$ violate the choice of $L'$ and $C$.
We conclude $V(C) = V(L')$ and that $L'-X$ is Hamiltonian, as required.
\end{proof}

\begin{proof}[Proof of Theorem~\ref{thm:main2}]
Let $L$ and $r$ be given and let $p = |V(L)|$.  Let $G$ be a graph and $W$ an $R$-wall in $G$ where $R$ will be specified shortly.  Assume that $G$ does not contain an $L$ minor grasped by $W$ and thus $W$ does not grasp a $K_p$ minor as well.  We will apply Theorem \ref{thm:main3} with the above values of $r$ and $p$ above and $M$ and $s_2$ chosen below.  If the $\Sigma$-configuration $(W_1, W_0, c_0, c_1, \Delta_0, \zC_1, \zC_2, \zP_1, \zP_2, \gamma)$ returned by the theorem is in a surface $\Sigma$ in which $L$ cannot be embedded, then we have found the $\Sigma$-decomposition desired in the statement of the present theorem.  If instead, the graph $L$ can be drawn in $\Sigma$, we will explicitly find an $L$ minor in $\zC_2 \cup \gamma$ which is grasped by $W$ using the properties of the $\Sigma$-configuration, a contradiction.

For every surface $\Sigma$ of genus at most $p(p+1)$ with $h$ handles and $c$ cross-caps and $I \subseteq \{1, 2, \dots, c+h\}$ such that $L$ can be drawn in $\Sigma$, we define a value $n(\Sigma, c, h, I)$ as follows.  Fix a $D$-representation $\delta$ of $\Sigma$ of type $I$.  Fix $L'$ to be a graph satisfying the outcomes of Lemma \ref{lem:Hexists}, and let $n(\Sigma, c, h, I) = |V(L')|$.

Fix $N$ to be the maximum of $n(\Sigma, c, h, I)$ over all possible choices of $\Sigma, c, h, I$ where $\Sigma$ has genus at most $p(p+1)$ and such that $L$ can be drawn in $\Sigma$.

Let $s_2 = 3N$ and $M = N^3$.  Let $\nu = \left \lfloor \frac{1}{10^5p^{26}}p^{10^7p^{26}}N^3\right \rfloor $.  Assume that $R \ge 49152p^{24}(r+2s_2) + \nu$.  Let $\mu = 49152p^{24}2s_2 + \nu$ and note that $\mu$ depends only on the graph $L$.  Apply Theorem \ref{thm:main3} to $G$ and $W$ with parameters $r, p, s_2, M$.  As there is no $K_p$ minor grasped by $W$, the theorem implies that there exists $A \subseteq V(G)$, $|A| \le \nu$, a surface $\Sigma$ of genus at most $p(p+1)$ and a $\Sigma$-configuration $(W_1, W_0, \delta, c_0, c_1, \Delta_0, \zC_1, \zC_2, \zP_1, \zP_2, \gamma)$ with parameters $(r, M, s_2, M, M, c, h)$.

If the graph $L$ does not embed in $\Sigma$, then by adding a vortex-free rendition of the vortex society of $c_0$ in the disk and a rendition of the vortex society of $c_1$ of breadth at most $2p^2$ and depth at most $\nu$, we see that $G$ has a $\Sigma$-decomposition of breadth at most $2p^2\le \mu$ and depth at most $\nu\le \mu$.  Moreover, the $r$-wall $W_0$ is flat in the new $\Sigma$-decomposition of $G-A$, satisfying the conclusions of the theorem.

For the remainder of the proof, we may assume that $L$ can be drawn in the surface $\Sigma$ with $h$ handles and $c$ cross-caps.  By Lemma \ref{lem:Hexists}, there exists a D-representation $\delta' = (\Delta, I, (X_1, \dots, X_{c+h})$ with realization $\Sigma'$ and a graph $L'$ which contains $L$ as a minor such that $L'$ has a clean drawing in $\Sigma'$.  Let $X$ be the subgraph of $L'$ formed by the edges which use either a cross-cap or handle in the drawing.  Lemma \ref{lem:Hexists} implies that $X$ is a matching and furthermore that $L' - E(X)$ has a Hamiltonian cycle $C$.  Consider the graph plane graph $L' - E(X)$ given by the drawing of $L'-E(X)$ restricted to the disk $\Delta$.  The cycle $C$ defines a closed curve in $\Delta$ and partitions the remaining edges of $L'$ into two subgraphs: let $L_1$ be the subgraph consisting of edges contained in the disk bounded by $C$ and $L_2$ the subgraph formed edges not contained in the disk bounded by $C$.  Thus, $L'$ is the union of the pairwise edge disjoint subgraphs $X \cup C \cup L_1 \cup L_2$.  Let $n_i = |E(L_i)|$ for $i = 1, 2$.  Given the bound of $3n-6$ on the number of edges in a planar graph, we have that $n_1 \le n$ and $n_2 \le n$.

Fix $\zQ_1, \dots, \zQ_{c+h}$ to be the handle and crosscap transactions in $\gamma$.  Let $\zC_2 = (C_1, C_2, \dots, C_{4n})$ and $\zP_2 = P_1, P_2, \dots, P_t$ for some positive integer $t$ such that traversing the vertices $\widetilde{\Delta_0}$ defined in the $\Sigma$-configuration in the cyclic order around the boundary of $\Delta_0$, we encounter the endpoints of $P_1, P_2, \dots, P_t$ in that order.  Let $Y_{i,j}$ be the set of vertices of $V(P_i) \cap V(C_j)$

Fix $n = |V(L')|$; by the definition of $N$, we may assume that $n \le N$.  Let the vertices of $L'$ be labeled $1, 2, \dots, n$ and assume they occur on the cycle $C$ in that cyclic ordering.  Fix indices $a_1 \le a_2 \le \dots \le a_n$ with the following properties.
\begin{itemize}
\item $a_{i+1} - a_i \ge n$ and $t - a_n \ge n$, and
\item if there is an edge of $X$ connecting $i$ and $j$ through the crosscap or handle $X_{i'}$, there exists a path $Q_{i,j} \in \zQ_{i'}$ containing both $P_{a_i}$ and $P_{a_j}$.
\end{itemize}
Such indices exist by the choice of parameters $M$ and the properties of the $\Sigma$-configuration.

It is easy to see that the cycle $C_{n+1}$ along with the paths $\{Q_{i,j}: ij \in E(X)\}$ contains as a minor $X \cup C$.  We define $U^0_1, \dots, U^0_n$ to be the sets of vertices forming such a model of an $X \cup C$ minor, expanding the sets to ensure that the model is also grasped by $W$.  Essentially, to each branch set we will add an $n$-wall contained in $(C_{n+1}, \dots, C_{2n}) \cap \zP_2$.  Explicitly, for $1 \le i \le n$, let $U^0_i$ be a minimal subset of vertices containing:
\begin{itemize}
\item for $1 \le j \le n$ the vertices of a subpath of $C_{n+j}$ containing $Y_{a_i + l, n+j}$ for $0 \le l \le n-1$;
\item for $0 \le j \le n-1$ the vertices of a subpath of $P_{a_i + j}$ with one endpoint in $Y_{a_i + j, n + 1}$ and the other endpoint $Y_{a_i + j, 2n}$;
\item for $1 \le i \le n-1$ the vertices of a subpath of $C_{n+1}$ which contains $Y_{a_i, n+1}$ and a neighbor of $Y_{a_{i+1}, n+1}$ on $C_{n+1}$ but which is disjoint from $Y_{a_{l}, n+1}$ for $l \neq i$; for $i = n$ the vertices of a subpath of $C_{n+1}$ which contains $Y_{a_i, n+1}$ and a neighbor of $Y_{a_1, n+1}$ on $C_{n+1}$ but which is disjoint from $Y_{a_l, n+1}$ for $l \neq n$;
\item for $1 \le i \le n$, the vertices of a subpath of $P_{a_i}$ containing $Y_{a_{i}, n+1}$ and $Y_{a_i, 1}$;
\item for $1 \le i < j \le n$ such that $i$ is adjacent $j$ in $X$, the vertices of a subpath $Q_{i,j}$ which contains both $Y_{a_i, 1}$ and a neighbor of $Y_{a_j, 1}$ on $Q_{i,j}$ but which is disjoint from $Y_{a_j, 1}$.
\end{itemize}
By construction, $U^0_1, \dots, U^0_n$ form a model of $C \cup X$ grasped by $W$.

We define $U^i_1, \dots, U^i_n$ for $1 \le i \le n_1 + n_2$ forming branch sets of a minor adding one more edge of $L_1 \cup L_2$ each time we increment $i$.  Moreover, we do so maintaining the property that $U^i_j \subseteq U^{i+1}_j$ so that the end result is a $L'$ minor grasped by $W$, yielding a contradiction.  The process is the same for adding the edges of $L_1$ as for adding the edges of $L_2$; in the interest of brevity we only present the construction in detail for adding the edges of $L_1$.

Consider the graph $L_1$ which has a plane drawing in the subdisk of $\Delta$ bounded by the drawing of $C$.  Each edge of $L_1$ defines two subpaths of $C$ connecting it's endpoints.  By picking a minimal such subpath (by containment as a subgraph), we see from the plane embedding of $C \cup L_1$ that there exists an edge $e \in E(L_1)$ with the property that there exists a subpath $R$ of $C$ connecting the endpoints of $e$ such that $R$ is internally disjoint from all other edges of $L_1$.  By deleting the edge $e$ and reiterating the same argument, we conclude that we can enumerate the edges of $L_1$ as $e_1, e_2, \dots, e_{n_1}$ so that for all $i$, there exists a subpath $R_i$ of $C$ linking the endpoints of $e_i$ such that $R_i$ is internally disjoint from $e_{i+1}, e_{i+2}, \dots, e_{n_1}$.

Having defined $U_1^i, \dots, U_n^i$, we define $U_1^{i+1}, \dots, U_n^{i+1}$ as follows.  Let the endpoints of $e_{i+1}$ be $l$ and $l'$ with $l < l'$.  For $j$, $1 \le j \le n$, $j \neq l$ such that $j$ is incident an edge in $\{e_{i+1}, \dots, e_{n_1}\}$, define $U^{i+1}_j$ to be the union of $U^i_j$ along with the vertex set of a minimal subpath of $P_{a_j}$ containing $Y_{a_j,2n+i}$ and $Y_{a_j, 2n+i+1}$.  For $j$, $1 \le j \le n$ and $j \neq l$, let $U^{i+1}_j = U^i_j$.  Finally, the $U^{i+1}_l$ is defined to be the union of three sets: $U^i_l$, the vertex set of a minimal subpath of $P_{a_l}$ containing $Y_{a_l,2n+i}$ and $Y_{a_l, 2n+i+1}$, and the vertex set of a minimal subpath of $C_{2n+i}$ containing $Y_{a_l, 2n+1}$ and a neighbor of $Y_{a_{l'}, 2n+1}$ corresponding to the path $R_{i+1}$ linking $l$ and $l'$ in $C$.  The choice of the path $R_{i+1}$ ensures that the sets $U^{i+1}_1, \dots, U^{i+1}_n$ are pairwise disjoint.  Note as well that as $n_2 = |E(L_2)| \le n$ and as $s_2 = 3n$, the cycle $C_{i+1}$ in the definition of $U^{i+1}_j$ is always available.
We conclude that $U^{n_1}_1, \dots, U^{n_1}_n$ form a model of a $C \cup X \cup L_1$ minor which is grasped by $W$.

We now consider the subgraph $L_2$.  Let $Z$ be the subset of $V(L')$ which are incident to an edge of $X$.  The plane drawing of $L_2$ in $\Delta$ minus the interior of the disk bounded by the drawing of $C$ implies that there exists an edge $e \in E(L_2)$ and a path $R$ of $C$ such that neither an endpoint of an edge of $L_2$ nor a vertex of $Z$ is an interior vertex of $R$.  By deleting the edge $e$ and reiterating this argument, we can find an enumeration $e_{n_1+1}, \dots, e_{n_1 + n_2}$ of the edges of $L_2$ such that for every index $i$, $1 \le i \le n_2$, there exists a subpath of $C$ connecting the endpoints of $e_{n_1+i}$ which is internally disjoint from $Z$ and the edges $e_{n_1+i}, \dots, e_{n_1+n_2}$.  The same process as in the previous paragraph yields subsets $U^{n_1+n_2}_1, \dots, U^{n_1+n_2}_n$ forming a model of an $L'$ minor grasped by $W$, a contradiction and completing the proof of the theorem.
\end{proof}

Note that the proof above is not constructive.  Lemma \ref{lem:Hexists} does not yield a concrete bound on $|V(L')|$ and consequently, the proof of Theorem \ref{thm:main2} does not give an explicit bound on the parameters in the statement.  Robertson and Seymour's original proof of the excluded clique minor theorem had a similar existential argument on graphs embedded in surfaces in \cite{RS7}, and as a result, the original proof also does not yield an explicit bound on the parameters.  Recent work by Huynh, Geelen and Richter \cite{HGR} provides an explicit bound for this step.  The results of \cite{HGR} could also be used with our proof of Theorem \ref{thm:main2} to give explicit bounds.  We have chosen not to do so in the interests of simplifying the presentation and following our goal of giving a self-contained proof.

\section{The relationship between tangles and walls}\label{sec:tangles}

As we briefly discussed in the introduction, the original statement of Robertson and Seymour's excluded minor theorem gives the structure with respect to a tangle, whereas our Theorem \ref{thm:main} gives the structure in relation to a wall.  In this section, we give the necessary background on tangles and walls and discuss how we can go back and forth between the two.  Tangles, first introduced by Robertson and Seymour in \cite{RS10}, are now widely used in applications and are discussed in depth in graduate texts \cite{Diestel5}.  Lemmas \ref{lem:free}, \ref{lem:well-linked}, and \ref{lem:tw1} are well known facts on the properties of tangles and their relationships with tree decompositions.  We include the proofs here for completeness.

Let $G$ be a graph and $k \ge 1$ a positive integer.  A \emph{tangle of order $k$} is a set $\zT$ of ordered separations $(A, B)$ such that
\begin{itemize}
\item for every separation $(A, B)$ of $G$ of order strictly less than $k$, exactly one of $(A, B)$ and $(B, A)$ is in $\zT$, and
\item there do not exist $(A_1, B_1), (A_2, B_2), (A_3, B_3) \in \zT$ such that $G = \bigcup_1^3 G[A_i]$.
\end{itemize}
For a separation $(A, B) \in \zT$, we typically refer to $A$ as the \emph{small} half of the separation and $B$ as the \emph{large} half.  A tangle $\zT'$ in $G$ of order at most $k$ is a \emph{restriction} of $\zT$ if $\zT' \subseteq \zT$.

A canonical example of tangles is those defined by wall subgraphs.  Let $W$ be an $r$-wall subgraph in a graph $G$, and assume the horizontal and vertical paths are specified.  Let $\zT_W$ be the set of ordered separations $(A, B)$ of order at most $r-1$ such that $B \setminus A$ contains the vertex set of both a horizontal and vertical path of $W$.  Then $\zT_W$ is a tangle in $G$; we will refer to it as the tangle \emph{induced} by $W$.

A set $X \subseteq V(G)$ is \emph{free} in a tangle $\zT$ of order $k$ if for all separations $(A, B)$ of order at most $|X|-1$ with $X \subseteq A$, we have that $(B, A) \in \zT$.  Intuitively, no small order separation separates the set $X$ from the tangle.  We show that every tangle has a large free set.

\begin{LE}\label{lem:free}
Let $G$ be a graph and $\zT$ a tangle of order $k$ for some $k \ge 1$.  Then there exists as set $X \subseteq V(G)$ of order $k-1$ which is free in $\zT$.
\end{LE}
\begin{proof}
There exists a connected component $G'$ of $G$ such that $(V(G) \setminus V(G'), V(G')) \in \zT$ as otherwise every connected component defines the small side of a separation of order 0 and we get a contradiction to the definition of a tangle.  Thus, if we fix any elment $x \in V(G')$, the set $X =\{x\}$ is free.

Fix $X$ to be a free set of maximal order at least one and less than $k$, and assume that $|X| < k-1$.  Fix a separation $(A, B) \in \zT$ with $X \subseteq A$ of order $|X|$ chosen to minimize $|B|$.  Such a separation exists as the trivial separation $(X, V(G)) \in \zT$.  Fix a vertex $x \in B \setminus A$ and consider the set $X' = X \cup \{x\}$.  By the choice of $X$, the set $X'$ is not free.  Fix $(C,D)$ to be a separation of order less than $|X'|$ such that $X' \subseteq C$ and $(C,D) \in \zT$.  As $X$ is free, it follows that the order of $(C,D)$ is exactly $|X|$.

Consider the separation $(C \cap A, B \cup D)$.  We have that $(C \cap A, B \cup D) \in \zT$, as $G = G[A] \cup G[B \cup D]$.  As $X \subseteq A \cap C$, by the fact that $X$ is free we have that $(C \cap A, B \cup D)$ has order at least $|X|$.  It follows that the separation $(A \cup C, B \cap D)$ has order at most $|X|$.  As $A$ is a proper subset of $A \cup C$, by the choice of $(A, B)$ we have that $(B \cap D, A \cup C) \in \zT$.  But now $G = G[A] \cup G[C] \cup G[B \cap D]$, contradicting the properties of the tangle $\zT$.
\end{proof}

Free sets have useful properties; of interest here is that they are well-linked.  For a graph $G$ and $X \subseteq V(G)$, the set $X$ is \emph{well-linked} if for every partition $X_1, X_2$ of $X$, there does not exist a separation $(A, B)$ with $X_1 \subseteq A$, $X_2 \subseteq B$ and $|A \cap B| < min(|X_1|, |X_2|)$.

\begin{LE}\label{lem:well-linked}
Let $G$ be a graph, $\zT$ be a tangle in $G$ of order $k\ge 1$, and let $X\subseteq V(G)$, $|X| < k$, be free.  Then $X$ is well-linked.
\end{LE}

\begin{proof}
Fix a partition $X_1, X_2$ of $X$, and assume that there exists a separation $(A, B)$ of order $k$ with $X_1 \subseteq A$, $X_2 \subseteq B$ and $k < min(|X_1|, |X_2|)$.  Without loss of generality, we may assume that $(A, B) \in \zT$.  The separation $(A \cup X_2, B)$ has order at most $|X|< k$ and thus either $(A \cup X_2, B)$ or $(B, A \cup X_2)$ is in $\zT$.  It cannot be the case that $(B, A \cup X_2) \in \zT$ lest $G=G[A] \cup G[B]$ and the tangle conditions not hold.  Thus $(A \cup X_2, B) \in \zT$.  However, $|(A \cup X_2) \cap B| = |X_2| + |(A \cap B) \setminus X_2| < |X|$, contradicting the fact that $X$ is free and proving the claim.
\end{proof}

We will see that well-linked sets are a certificate that a graph has large tree-width.  To do so, we will need the definition of tree decompositions.
\begin{DEF}
Let $G$ be a graph.  A \emph{tree decomposition} of $G$ is a pair $(T, \zX)$ such that $T$ is a tree and $\zX = \{ X_v \subseteq V(G): v \in V(T)\}$ with the properties that
\begin{itemize}
\item for every edge $xy \in E(G)$, there exists $v \in V(T)$ such that $x,y \in X_v$, and
\item for every vertex $x \in V(G)$, the set of vertices $\{v \in V(T): x \in X_v\}$ induces a connected subtree of $T$.
\end{itemize}
The \emph{width} of the decomposition is $min_{v \in V(T)} |X_v| -1$.  The \emph{tree-width} of $G$ is the minimum width of a tree decomposition of $G$.
\end{DEF}

\begin{LE}\label{lem:tw1}
Let $G$ be a graph, $k \ge 1$ a positive integer, and $X \subseteq V(G)$ a well-linked set of order at least $3k$.  Then the tree-width of $G$ is at least $k$.
\end{LE}

\begin{proof}
Assume the claim is false and let $(T,\zX)$ be a tree decomposition of $G$ of width at most $k-1$.  We may assume that for all $v_1,v_2 \in V(T)$, $X_{v_1} \neq X_{v_2}$.  For every edge $v_1v_2 \in E(T)$, let $T_1$ and $T_2$ be the two trees of $T - v_1v_2$ such that $v_i \in V(T_i)$.  The trees $T_1$ and $T_2$ define a separation $\left ( \bigcup_{u \in V(T_1)} X_u, \bigcup_{u \in V(T_2)} X_u\right)$ of order at most $k-1$.  As $X$ is well-linked, at most one half of this separation has at least $k$ vertices of $X$.  Thus, we can orient the edge $v_1v_2$ in $T$ towards the endpoint $v_i$ for which $\bigcup_{u \in V(T_i)} X_u$ contains at least $2k$ vertices of $X$.

By this process, we orient every edge of the tree $T$ and fix $v \in V(T)$ to be a sink in the oriented graph with all incident edges oriented towards $v$.  Let $v_1, \dots, v_l$ be the neighbors of $v$ in $T$ and let $T_i$ be the subtree of $T - vv_i$ containing $v_i$.  The vertex set $\bigcup_{u \in V(T_i)} X_u$ contains at most $k-1$ vertices in $X$.  As $X_v$ has at most $k$ vertices of $X$, there exists a subset $I \subseteq \{1,\dots, l\}$ such that $Y = \bigcup_{i \in I}\bigcup_{u \in V(T_i)} X_u$ contains at least $k+1$ and at most $2k-1$ vertices of $X$.  If we let $W = \bigcup_{i \notin I} \bigcup_{u \in V(T_i)} X_u$, then the separation $(Y \cup X_v, W \cup X_v)$ contradicts the fact that $X$ is well-linked.
\end{proof}

As a consequence of these results, the presence of a large order tangle $\zT$ is a certificate that the graph has large treewidth.  We conclude, by Theorem \ref{thm:grid}, that the graph also has a large wall subgraph.  Such a wall itself induces a large tangle, and it can be shown that the tangle induced by the wall is a restriction of $\zT$.  The version of Theorem \ref{thm:grid} from \cite{RST} does just this, and this is a property we will require in the proofs to come.  However, neither the versions of Theorem \ref{thm:grid} from \cite{Diestel} nor the recent polynomial version of \cite{CC} maintain this stronger property.

The remainder of this section will be devoted to showing that by  applying any version of Theorem \ref{thm:grid} to an appropriately chosen subgraph, we can find a wall such that the tangle induced by the wall is a restriction of the larger, original tangle.  This will allow us to apply any variant of Theorem \ref{thm:grid} in the coming proofs.  We will require the following lemma which is closely related to a similar statement for edge disjoint paths linking two sets of vertices in \cite{ChuImproved}.

\begin{LE}\label{lem:linked_linkedsets}
Let $X$ and $Y$ be well-linked sets in a graph $G$.  Let $k\ge 1$ be a positive integer and assume that $|X| \ge k$, $|Y| \ge 3k$.  If there do not exist $k$ pairwise disjoint $X-Y$ paths, then there exists an edge $e \in E(G)$ such that $X$ is well-linked in $G-e$.
\end{LE}

\begin{proof}
Assume $k$ such $X-Y$ paths do not exist.  Fix $(A, B)$ to be a separation of order $k' < k$ with $X \subseteq A$ and $B$ has at least $k$ vertices of $Y$, and from all such separations, pick $(A, B)$ to minimize $|B|$.  Such a separation exists because there is a separation of order strictly less than $k$ separating $X$ and $Y$.  Note that by the well-linked property of $Y$, the set $B$ must have at least $2k+1$ vertices of $Y$.

Fix an edge $e= xy$ of $G[B] - A$.  Note that such an edge exists as otherwise every vertex of $(Y \cap B) \setminus A$ has every neighbor in $A \cap B$ and there exists a separation violating the well-linkedness of $Y$.

By assumption, $X$ is not well-linked in $G-e$.  There exists a separation $(C,D)$ of $G-e$ of order $l$ and a partition $X_1, X_2$ of $X$ with $X_1 \subseteq C$, $X_2 \subseteq D$, and $|X_i| > l$ for $i = 1,2$.  We may assume that $x \in C \setminus D$ and $y \in D \setminus C$ since otherwise $(C,D)$ is a separation of $G$ violating the well-linkedness of $X$.  Note as well that $x$ has at least one neighbor in $C \setminus D$ as otherwise again we get a separation violating the well-linkedness of $X$ in $G$.  Similarly, $y$ has at least one neighbor in $D \setminus C$.

The separation $(A \cap C, B \cup D)$ separates $X_1$ and $X_2$ in $G$ and therefore the order is at least $l+1$.  Thus, the separation $(A \cup C, B \cap D)$ of $G-e$ has order at most $k'-1$.  The pair $(A \cup C, (B \cap D) \cup \{x\})$ is a separation of $G$ of order at most $k' < k$.  As $x$ has a neighbor in $C \setminus D$, we have that $(B \cap D) \cup \{x\}$ is a proper subset of $B$.  Thus, by our choice of $(A, B)$, we have that $(B \cap D) \cup \{x\}$ has at most $k-1$ vertices of $Y$.  Symmetrically, considering the separation $(A \cup D, (B \cap C) \cup \{y\})$, we see that $(B \cap C) \cup \{y\}$ has at most $k-1$ vertices of $Y$ and thus $B$ has at most $2k-2$ vertices of $Y$.  As $|Y| \ge 3k$, we see that $A$ has at least $k+2$ vertices of $Y$.  The separation $(A, B)$ contradicts the well-linkedness of $Y$, completing the proof.
\end{proof}

We are now ready to present the main result of this section using Lemmas \ref{lem:free}, \ref{lem:well-linked}, and \ref{lem:tw1}. 

\begin{LE}\label{lem:wallplus}
Let $f$ be a function such that for every positive integer $r$, every graph with tree-width at least $f(r)$ contains an $r$-wall as a subgraph.  Let $k \ge 1$ be an integer.  Let $G$ be a graph and $\zT$ a tangle of order at least $3f(6k^2) + 1$.  Then there exists a $k$-wall $W$ in $G$ such that $\zT_W$ is a restriction of $\zT$.
\end{LE}

\begin{proof}
Let $\zT$ have order $3f(6k^2) + 1$ and fix $X$ to be a free set of order $3f(6k^2)$.  The set $X$ exists by Lemma \ref{lem:free} and is well-linked by Lemma \ref{lem:well-linked}.  Fix a minimal (by subgraph containment) subgraph $G'$ of $G$ such that $X$ is well-linked in $G'$.

By Lemma \ref{lem:tw1}, $G'$ has tree-width at least $f(6k^2)$ and by assumption, there exists a $6k^2$-wall subgraph $W$.  Fix the horizontal and vertical paths of $W$, and fix $P_1$ to be the first horizontal path in $W$.  Let $Y$ be the endpoints of the vertical paths in $V(P_1)$.  The set $Y$ is well-linked in $G'$ and has order $6k^2$.  By Lemma \ref{lem:linked_linkedsets} and the minimality of $G'$, there exists a linkage of order $2k^2$ from $X$ to $Y$.  Fix $Y'$ to be the endpoints of such a linkage.

Let $Q_1, \dots, Q_{6k^2}$ be the vertical paths of $W$.  Fix pairwise disjoint subpaths $R_1, \dots, R_{k}$ of $P_1$ such that the endpoints of $R_i$ are $s_i$ and $t_i$ with the following properties.
\begin{itemize}
\item The vertex $s_1$ is the endpoint of $Q_1$ and $t_{k}$ is an endpoint of $Q_{6k^2}$.
\item For $2 \le i \le k$, $s_i$ is an endpoint of a vertical path of $W$ and is adjacent on $P_1$ to $t_{i-1}$.
\item Each $R_i$ contains $2k$ vertices of $Y'$.
\end{itemize}

Let $W'$ be the $k$-wall defined by the vertical paths of $W$ with an endpoint in $\{s_1, s_2, \dots, s_k\}$, the horizontal path $P_1$ and another arbitrarily chosen $k-1$ horizontal paths of $W$.

We claim $W'$ is the wall desired by the lemma.  Fix a separation $(A, B)$ of order at most $k-1$ in $G$.  As $X$ is well-linked and has more than $3k$ vertices, one of $A \setminus B$ and $B \setminus A$ has at most $k$ vertices of $X$.  Assume $|(A \setminus B) \cap X| \le k$.  By the fact that $X$ is free, $(A, B \cup X) \in \zT$.  It follows that $(A, B) \in \zT$ as well since $G[A] \cup G[B] = G$.

Let the vertical paths of $W'$ be $Q_1', \dots, Q_{k}'$ so that the path $Q_i'$ has a common endpoint with $R_i$.  Thus we have $k$ pairwise disjoint paths $Q_i' \cup R_i$ for $1 \le i \le k$.  Fix an index $i$ such that $Q_i' \cup R_i$ is disjoint from $A \cap B$.  The path $R_i$ contains $2k$ vertices of $Y'$.  As there exist $2k$ disjoint paths from $X$ to $R$, at most $k-1$ of those paths intersect $A \cap B$, and $A \setminus B$ contains at most $k$ vertices of $X$, we conclude that $Q_i' \cup R_i$ is contained in $B \setminus A$.  Thus $(A, B) \in \zT_{W'}$.  As this holds for an arbitrary separation of order at most $k-1$, we conclude that $\zT_{W'}$ is a restriction of $\zT$, as desired.
\end{proof}


\section{Deriving the Robertson Seymour structure theorem}\label{sec:RSproof}

In this section, we give a proof of the original Robertson and Seymour structure theorem in terms of tangles.

We begin with the necessary notation to state the Robertson and Seymour Theorem 3.1 of \cite{RS16}.  Let $G$ and $H$ be graphs and let $|V(H)| = p$.  Let $\zT$ be a tangle in $G$.  A model of an $H$-minor is \emph{controlled} by $\zT$ if for every separation $(A, B) \in \zT$ of order at most $p-1$, there does not exist a branch set of the model contained in $A \setminus B$.  Note that if $\zT$ is equal to $\zT_W$ for a wall-subgraph $W$ of $G$ and if the $H$-minor model is grasped by $W$, then it is also controlled by $\zT_W$.  For a subset $Z$ of vertices, we denote by $\zT - Z$ the tangle in $G-Z$ formed by separations $\{(A, B): \text{ $(A, B)$ a separation in $G-Z$ of order at most $k-|Z| - 1$ and }(A \cup Z, B \cup Z) \in \zT\}$.  Given a surface $\Sigma$ and a $\Sigma$-decomposition $\delta$ of $G$, we say that the decomposition is \emph{$\zT$-central} if for every separation $(A, B) \in \zT$, there does not exist a cell $c \in C(\delta)$ such $B \subseteq V(\sigma_\delta(c))$.

We now state Robertson and Seymour's structure theorem in terms of our notation.

\begin{theorem}[Theorem 3.1 \cite{RS16}] \label{thm:RS_struct}
For all graphs $H$ there exist values $\alpha = \alpha(H)$ and $\Theta$ which satisfy the following.  Let $G$ and $H$ be a graph and $\zT$ a tangle in $G$ of order $\Theta$.  If $G$ does not have a model of $H$ as a minor controlled by $\zT$, then there exists a set $Z \subseteq V(G)$ of at most $\alpha(H)$ vertices, a surface $\Sigma$ in which $H$ cannot be drawn, and a $\zT-Z$-central $\Sigma$-decomposition of $G-Z$ of breadth and depth at most $\alpha$.
\end{theorem}

Before proceeding with the proof, we need a few additional results.  Let $G$ be a graph and let $\zT$ be a tangle of order $k \ge 1$ in $G$.  Let $S \subseteq V(G)$ be a subset of at most $k-1$ vertices.  There exists a unique subset $Y \subseteq V(G)$ such that $Y$ induces a connected component of $G-S$ and the separation $(V(G) \setminus Y, Y \cup S) \in \zT$.  We refer to the set $Y$ as the \emph{$S$-flap for $\zT$}.  Note that by definition, the $S$-flap induces a connected subgraph of $G$.

\begin{LE}\label{lem:flaps}
Let $G$ be a graph, $\zT$ a tangle in $G$ of order $k$, and $S$ and $S'$ disjoint subsets of $V(G)$ each of size at most $k-1$.  Let $Y$ and $Y'$ be the $S$ and $S'$-flaps of $\zT$, respectively.  Then either $Y \cap Y' \neq \emptyset$ or there exists an edge with one endpoint in $Y$ and one endpoint in $Y'$.
\end{LE}

\begin{proof}
As each of the separations $(V(G) \setminus Y, Y \cup S)$ and $(V(G) \setminus Y', S' \cup Y')$ is in $\zT$, by the definition of a tangle, there must be an edge of $G$ or a vertex not contained in either $G-Y$ or in $G-Y'$.  If there were such a vertex, then $Y \cap Y' \neq \emptyset$ and the claim holds.  Assume $e$ is an edge in neither $G-Y$ nor $G-Y'$, then at least one end of $e$ is in $Y$ and one end of $e$ is in $Y'$.  Let $u$ be the endpoint of $e$ in $Y$; we may assume $u$ is not in $Y'$ and therefore the other endpoint $v$ of $e$ is in $Y'$.  Thus, $e$ has one end in $Y$ and one end in $Y'$, completing the proof.
\end{proof}

The proof of the next lemma closely follows the ideas of the proof  of Theorem \ref{thm:main} with $S$-flaps playing a similar role to that of the crosses formed by horizontal and vertical paths.

\begin{LE}\label{lem:decomp4tangles}
Let $G$ be a graph, and let $k, l, r \ge 0$ such that $r \ge k(2l+1) + 6$.  Let $W$ be an $r$-wall in $G$.  Let $\zT$ be a tangle of order at least $k(2l+1)+6$ such that $\zT_W$ is a restriction of $\zT$.   Assume that there exists a $\Sigma$-decomposition $\delta$ of $G$ in a surface $\Sigma$ of breadth $k$ and depth $l$ such that $W$ is flat in $\delta$.  Then there exists $Z \subseteq V(G)$, $|Z| \le k(2l+1)$ and a decomposition of $G-Z$ in $\Sigma$ of breadth at most $k$ and depth at most $l$ which is $\zT-Z$ central.
\end{LE}

\begin{proof}
Let $\delta = (\Gamma, \zD)$.  Let $\zY$ be the set of $S$-flaps for all sets $S \subseteq V(G)$ of order less than the order of $\zT$.  Let $c_1, \dots, c_k$ be the vortices of $\delta$.  Let $(G_i, \Omega_i)$ be the vortex society of $c_i$.

\begin{claim}
For every $1 \le i \le k$, there exists a set $Z_i \subseteq V(G_i)$, $|Z_i| \le 2l+1$ such that $V(G_i) \setminus Z_i$ does not contain $Y$ for all $Y \in \zY$.
\end{claim}

\begin{cproof}
Fix the index $i$ and a linear decomposition $(X_1, \dots, X_n)$ of $(G_i, \Omega_i)$ with the vertices $V(\Omega_i)$ labeled $v_1, \dots, v_n$.  As each $Y \in \zY$ induces a connected graph, for every $Y \in \zY$ such that $Y \subseteq V(G_i)$, the set $\{j: X_j \cap Y \neq \emptyset\}$ is an interval of $\{1, 2, \dots, n\}$.  Fix $J$ to be the minimum index such that there exists $Y \in \zY$ such that $Y \subseteq V(G_i)$ and $\{j: X_j \cap Y \neq \emptyset\} \subseteq \{1, 2, \dots, J\}$.  Let $Z_i = (X_J \cap X_{J-1}) \cup (X_J \cap X_{J+1}) \cup \{v_J\}$.  Fix $Y \in \zY$ such that $Y \subseteq X_1 \cup \dots \cup X_J$.

Let $Y' \in \zY$ be such that $Y' \subseteq V(G_i)$ and assume, to reach a contradiction, that $Y' \cap Z_i = \emptyset$.  We first observe that $Y'$ must intersect a vertex of $X_j$ for some $j \le J$.  This trivially holds if $Y \cap Y' \neq \emptyset$.  If $Y \cap Y' = \emptyset$, then by Lemma \ref{lem:flaps} there exist vertices $y \in Y$, $y' \in Y'$ such that $yy' \in E(G)$.  As every edge is contained in $X_j$ for some $j$, it again follows that there exists $j$ such that $X_j \cap Y' \neq \emptyset$.  By the choice of $J$, it must be the case that $j \le J$.  This proves that $Y'$ intersects $X_j$ for some $j \le J$; by the minimality of $J$, it must hold that $X_J \cap Y' \neq \emptyset$ as well.

The set $Z_i$ separates $X_J \setminus Z_i$ from the rest of $G$.  It follows that $Y' \subseteq X_J \setminus Z_i$.  Given that $W$ is disjoint from $V(G_i) \setminus V(\Omega_i)$, we have that $(V(G) \setminus (X_J \setminus Z_i), X_J)$ is a separation of $G$ separating $W$ and $Y'$.

As $W$ has order at least $|Z_i| + 1$, the separation $(X_J, V(G) \setminus (X_J \setminus Z_i)) \in \zT_W \subseteq \zT$.

The set $Y'$ is a flap, and thus there exists a separation $(A, B) \in \zT$ with $Y' = B \setminus A$.  The set $Y' \subseteq X_J \setminus Z_i$ implies that $B$ must be a subset of $X_j$.  It now follows that $G = G[A] \cup G[X_j]$, contradicting the properties of the tangle $\zT$ and completing the proof of the claim.
\end{cproof}

Let $Z = \bigcup_1^k Z_i$.  The $\Sigma$-decomposition $\delta$ defines a $\Sigma$-decomposition $\delta' = (\Gamma', \zD')$ where $\Gamma'$ is obtained by restricting $\Gamma$ to $G-Z$ and for every $D' \in \zD'$, there is a unique element $D \in \zD$ such that $D' \subseteq D$ and the boundary of $D'$ intersects the boundary of $D$ exactly in the vertices of $V(G) \setminus Z$ contained in the boundary of $D$.  Note that every vertex of $V(W) \cap Z$ must be contained in the boundary cycle of $W$.  Thus, by deleting the boundary cycle of $W$, we see that there is a wall $W'$ which is a subwall of $W$ of order $k(2l + 1)+4$ which is flat in $\delta'$.

We claim that $\delta'$ is the desired $\zT-Z$ central $\Sigma$-decomposition.  If not, there exists a separation $(A, B) \in \zT - Z$ of $G-Z$ such that $B \subseteq V(\sigma_{\delta'}(c))$ for some cell $c \in C(\delta')$.  As $(A \cup Z, B \cup Z )\in \zT$, there exists an $(A \cap B) \cup Z$-flap $Y$ of $\zT$ contained in $B \setminus A$.  By the choice of $Z$ to intersect every flap of $\zT$ in the vortices, it follows that $c$ is not a vortex of $\delta'$.

Given that $c$ is not a vortex, there exists a separation $(A', B')$ of $G-Z$ of order at most three with $B \subseteq B'$ and $B' = V(\sigma_{\delta'}(c))$.  The wall $W'$ has at least $k(2l+1) + 4$ horizontal and vertical paths and so, the separation $(B', A')$ is in the tangle $\zT_{W'}$ of $G-Z$ and every horizontal and vertical path of $W'$ disjoint from $A' \cap B'$ is contained in $A' \setminus B'$.  Consider now the separation $(A' \cup Z, B' \cup Z)$ of $G$.  Each horizontal and vertical path of $W'$ is contained in a horizontal and vertical path, respectively, of the original wall $W$.  By the bound on the size of $Z$, we see that there is a horizontal and vertical path of $W$, each containing a horizontal and vertical path of $W'$, which are disjoint from $(A' \cap B') \cup Z$.  Thus, there is a horizontal and vertical path of $W$ as well contained in $A' \setminus B'$ and thus $(B' \cup Z, A' \cup Z) \in \zT_W \subseteq \zT$.

We conclude that $(B', A') \in \zT - Z$.  As $B \subseteq B'$, we have that $G-Z = G[A] \cup G[B']$, contradicting the definition of the tangle $\zT-Z$ and completing the proof.
\end{proof}

The final ingredient is the relationship between the tree-width of a graph and the presence of a large wall subgraph.  As we are not concerned here with the exact bounds on the parameters, we can apply the version of Theorem \ref{thm:grid} from \cite{Diestel}.

\begin{theorem}[Theorem 12.4.4 \cite{Diestel}]\label{thm:grid_easy}
For every integer $r$, there is an integer $k$ such that every graph with tree-width at least $k$ contains the $r$-wall as a subgraph.
\end{theorem}

\begin{proof}[Proof. Theorem \ref{thm:RS_struct}]
By Theorem \ref{thm:grid_easy} and Lemma \ref{lem:wallplus}, for every $R$ there exists a $\theta$ such that if $\zT$ is a tangle in $G$ of order $\Theta$, there exists an $R$-wall $W$ in $G$ with $\zT_W$ a restriction of $\zT$.  We will fix the value $R$ in the next paragraph.

Assume that there is no $H$-minor controlled by $\zT$.  Thus, the wall $W$ does not grasp an $H$ minor as well.  Let $\mu(H)$ be the value given by Theorem \ref{thm:main2}, let
$r = \mu(2\mu + 2) + 6$ and let $R = 49152|V(H)|^{24}r + \mu$.  By Theorem \ref{thm:main2}, there exists a set $A$ of at most $\mu$ vertices, a surface $\Sigma$ in which $H$ cannot be drawn, a $\Sigma$-decomposition $\delta$ of $G-A$ of depth and breadth at most $\mu$, and an $r$-subwall $W'$ of $W$ which is disjoint from $A$ and flat in $\delta$.

Let $r' = \mu(2\mu + 1) + 6$ and define an $r'$-subwall $W''$ of $W'$ by taking the first and last horizontal and vertical paths of $W'$ and another arbitrarily chosen $\mu (2\mu + 1) + 4$ horizontal and vertical paths of $W'$.  Note that $W''$ is a subwall of $W$ and is flat in $\delta$.  Moreover, there exist at least $r' + \mu$ distinct vertical paths of $W$ which intersect every horizontal path of $W''$.  The analogous statement holds for the horizontal paths of $W$ intersecting every vertical path of $W''$.

\begin{claim}
The tangle $\zT_{W''}$ of $G-A$ is a restriction of $\zT - Z$.
\end{claim}

\begin{cproof}
Fix a separation $(X, Y)\in \zT_{W''}$ of order at most $r' -1$.  By the construction of $W''$, there are at least $r' + \mu$ distinct vertical paths of $W$ which intersect a vertex of $Y \setminus X$.  If we consider the separation $(X \cup A, Y \cup A)$ of $G$ of order at most $\mu + r' - 1$, there are at least $r' + \mu$ vertical paths of $W$ which intersect a vertex of $(Y \cup A) \setminus (X \cup A)$.  Thus, there is a vertical path of $W$ contained in $(Y \cup A) \setminus (X \cup A)$, implying that $(X\cup A, Y \cup A) \in \zT_W \subseteq \zT$.  Thus, $(X, Y) \in \zT - A$, as desired.
\end{cproof}

We now apply Lemma \ref{lem:decomp4tangles} to $G-A$ with the tangle $\zT - A$, wall $W''$, paramters $k = \mu, l = \mu, r'$ and the $\Sigma$-decomposition $\delta$.  We conclude that there exists $Z \subseteq V(G) \setminus A$ with $|Z| \le \mu(2\mu + 1)$ such that $(G-A)-Z$ has a $\Sigma$-decomposition of width and breadth at most $\mu$ which is $(\zT - A) - Z = \zT - (A \cup Z)$-central.  The theorem now holds with $\alpha = \mu(2\mu + 2)$.
\end{proof}


\section{The global excluded minor decomposition}\label{sec:global}

As we noted in the introduction, a global decomposition theorem for graphs not containing a fixed $H$ as a minor can be quickly proven from Theorem \ref{thm:RS_struct}.  As the original structure theorem of Robertson and Seymour does not include explicit bounds on the parameters, no known bounds previously exist for the global structure theorem as well.

In this section, we use the best known bounds for Theorem \ref{thm:grid}, due to Chuzhoy and Zihan, to give a version of the global decomposition theorem with explicit bounds.

\begin{theorem}[Theorem 1.1 \cite{CZ}]\label{thm:polygrid}
There exist constants $c_1$ and $c_2$ such that for all $r \ge 2$ and every graph $G$ of tree-width at least $c_1r^9 log^{c_2} r$ contains an $r$-wall.
\end{theorem}

The next theorem is the analog of Theorem \ref{thm:RS_struct} obtained by using Theorem \ref{thm:polygrid} and Theorem \ref{thm:main} to obtain the best bounds possible.  The proof is essentially the same as the proof of Theorem \ref{thm:RS_struct}.  We include it for completeness.
\begin{theorem}\label{thm:RS_struct_exact}
There exists an absolute constant $c$ which satisfies the following.  Let $p \ge 1$ be a positive integer.  Let $\zT$ be a tangle in $G$ of order $\Theta$ with
$$\Theta \ge p^{18 \cdot 10^7p^{26} + c}.$$
Then either $\zT$ controls a $K_p$ minor or there exists $Z \subseteq V(G)$, $|Z| \le 5p^2\cdot p^{10^7p^{26}}$, a surface $\Sigma$ of Euler genus at most $p(p+1)$, and a $\Sigma$-decomposition of $G-Z$ of breadth at most $2p^2$ and depth at most $p^{10^7p^{26}}$ which is $\zT - Z$ central.
\end{theorem}

\begin{proof}

The theorem trivially holds if $p \le 2$.  We first fix the following values.
\begin{align*}
R &= 49152p^{24} r_1 + p^{10^7p^{26}} \\
r_1 &= (2p^2)(2p^{10^7p^{26}} + 1) + 6 + p^{10^7p^{26}}\\
\end{align*}
Let $c_1$ and $c_2$ be the constants in Theorem \ref{thm:polygrid}.  By appropriately choosing $c$, we have that $\Theta \ge 3c_1(6R^2)^9log^{c_2}(6R^2)+ 1$.  Thus, by Theorem \ref{thm:polygrid} and Lemma \ref{lem:wallplus}, there exists an $R$-wall $W$ in $G$ such that $\zT_W$ is a restriction of $\zT$.

Assume that $\zT$ does not control a $K_p$ minor.  This implies that the wall $W$ does not grasp a $K_p$ minor as well.  By Theorem \ref{thm:main} and the choice of $R$, there exists an $r_1$-subwall $W_1$, a surface $\Sigma$ of Euler genus at most $p(p+1)$, a subset $A \subseteq V(G)$ with $|A| \le p^{10^7p^{26}}$ disjoint from $W_1$, and a $W$-central $\Sigma$-decomposition $\delta$ of breadth at most $2p^2$ and depth at most $p^{10^7p^{26}}$ such that $W_1$ is flat in $\delta$.

As in the proof of Theorem \ref{thm:RS_struct}, we fix a subwall $W_2$ of $W_1$ containing the first and last horizontal and vertical path and another arbitrarily chosen $(2p^2)(2p^{10^7p^{26}} + 1) + 4$ horizontal and vertical subpaths of $W_1$.  Again, as in the proof of Theorem \ref{thm:RS_struct}, the tangle $\zT_{W_2}$ defined by $W_2$ in $G-A$ is a restriction of $\zT - A$.

Apply Lemma \ref{lem:decomp4tangles} to $G-A$, tangle $\zT - A$, wall $W_2$, $\Sigma$-decomposition $\delta$ and parameters $k = 2p^2$, $l = p^{10^7p^{26}}$, and $r= 2p^2(2p^{20^7p^{26}} + 1) + 6$.  This yields a subset $A' \subseteq V(G) \setminus A$ with $|A'| \le 2p^2(2\cdot p^{10^7p^{26}} + 1)$ and a $\Sigma$-decomposition $\delta'$ of breadth at most $2p^2$ and depth at most $p^{10^7p^{26}}$ which is $(\zT - A) - A' = \zT - (A \cup A')$ central.  Setting $Z = A \cup A'$ completes the proof of the theorem.

\end{proof}

In order to state the global structure theorem, we will require some additional notation.  We have already defined tree decompositions in Section \ref{sec:tangles}.  Let $G$ be a graph and $(T, \zX)$ a tree decomposition of $G$ with $\zX = \{X_v: v \in V(T)\}$.  The \emph{adhesion} of $(T, \zX)$ is $max_{uv \in E(T)} |X_u \cap X_v|$.  For a vertex $v$ of $T$, \emph{torso of $G$ at $v$} is the graph obtained from $G[X_v]$ by adding an edge to every non-adjacent pair of vertices $x,y \in X_v$ for which there exists a $u$ adjacent in $T$ to $v$ such that $\{x,y\} \subseteq X_v \cap X_u$.
Let $(G, \Omega)$ be a vortex, and consider a linear decomposition given by a labeling $v_1, v_2, \dots, v_t$ of $V(\Omega)$ and sets $(X_1, X_2, \dots, X_t)$.  The \emph{width} of the linear decomposition is $max_{1 \le i \le t} |X_i|$.  An \emph{$\alpha$-near embedding} of a graph $G$ in a surface $\Sigma$ of \emph{width} $k$ and \emph{breadth} $t$ is a pair $(\delta, A)$ such that $A \subseteq V(G)$, $|A| \le \alpha$ and $\delta$ is $\Sigma$-decomposition of $G-A$ of breadth $t$ such that
\begin{itemize}
\item for every $c \in C(\delta)$ such that $c$ is not a vortex, $V(\sigma(c))$ induces a clique in $G$ and $V(\sigma(c)) \subseteq \pi(N(\delta))$, and
\item for every $c \in C(\delta)$ which is a vortex, there exists a linear decomposition of the vortex society of $c$ of width at most $k$.
\end{itemize}

We now state the global decomposition theorem.

\begin{theorem} \label{thm:global}
Let $c$ be the constant in Theorem \ref{thm:RS_struct_exact}.  Let $p \ge 1$ be a positive integer and let $G$ be a graph which does not contain $K_p$ as minor.  Let $\alpha = p^{18\cdot 10^7p^{26}+c}.$
Then $G$ has a tree decomposition $(T, \zX)$ of adhesion at most $4\alpha$ such that for all $v \in V(T)$, if $G'$ is the torso of $G$ at $v$ then $G'$ has an $\alpha$-near embedding of breadth at most $2p^2$ and width at most $\alpha$ in a surface of Euler genus at most $p(p+1)$.
\end{theorem}

We follow the proof of Theorem 4 of \cite{DKMW} and actually prove a slightly stronger statement to facilitate the induction.  Theorem \ref{thm:global} immediately follows from the following.

\begin{theorem}
Let $c$ be the constant in Theorem \ref{thm:RS_struct_exact}.  Let $p \ge 1$ be a positive integer and let $G$ be a graph which does not contain $G$ as a minor.  Let $\alpha = p^{18\cdot 10^7p^{26}+c}.$
Let $Z \subseteq V(G)$ such that $|Z| \le 3\alpha$.
Then $G$ has a tree decomposition $(T, \zX)$ with a root $r \in V(T)$ of adhesion at most $\alpha$ such that for every $v \in V(T)$, there exists a surface $\Sigma_v$ of Euler genus at most $p(p+1)$ and the torso $G_v$ of $G$ at $v$ has a $4\alpha$-near embedding $(\delta_v, A_v)$ of width at most $\alpha$ and breadth at most $2p^2$.  Moreover, at the root $r$, the set $Z \subseteq A_r$.
\end{theorem}

\begin{proof}
Assume the theorem is false.  Fix $G$, $p$ and $Z$ to form a counterexample which minimizes $|V(G)| + |V(G) \setminus Z|$.  As the theorem trivially holds for $p \le 3$, we have that $p \ge 4$.  We observe as well that $|V(G)| > 4\alpha$ as otherwise the trivial tree decomposition on a tree with a single vertex $r$ and $V(G) = X_r$ and the trivial decomposition with $A_r = V(G)$ would satisfy the theorem.  It follows then by minimality that $|Z| = 3\alpha$ as otherwise, we could add an arbitrary vertex to $Z$ and contradict our choice of counterexample.

\begin{claim}\label{cl:global1}
For all separations $(U_1, U_2)$ of order at most $\alpha$, $|U_i \cap Z| \le |U_1 \cap U_2|$ for exactly one of  $i = 1$ or $i = 2$.
\end{claim}

\begin{cproof}
Fix a separation $(U_1, U_2)$ of order at most $\alpha$.  By the size of $Z$ and the order of $(U_1, U_2)$, it cannot hold that $|U_i \cap Z| \le |U_1 \cap U_2|$ for both $i = 1$ and $i = 2$.

Assume that $|U_i \cap Z| > |U_1 \cap U_2|$ for both $i = 1,2$. Note that $(U_1, U_2)$ must be a non-trivial separation.  Let $Z_i = (Z \cap U_i) \cup (U_1 \cap U_2)$.  By minimality, the theorem holds for the set $G[U_i]$ with the set $Z_i$.  Thus, for $i = 1,2$, there exists a tree decomposition $(T_i, \zX_i)$ with root $r_i$ such that for every $t \in V(T_i)$, there exists an $4\alpha$-near embedding $(\delta_t, A_t)$.  Let $T$ be the tree formed by the union of $T_1$ and $T_2$ along with a new vertex $r$ adjacent to $r_1$ and $r_2$.  Let $X_r = Z \cup (U_1 \cap U_2)$ and $\zX = \zX_1 \cup \zX_2 \cup \{X_r\}$ and now $(T, \zX)$ is a tree decomposition of $G$ of adhesion at most $4\alpha$.

For every vertex $t \in V(T) \setminus \{ r_1, r_2, r\}$, the torso of $G$ at $t$ is the same as the torso of $G[U_i]$ at $t$ for $i$ equal to one or two and so has the appropriate $4\alpha$-near embedding.  For $i = 1, 2$, the torso of $G$ at $r_i$ may have additional edges in $U_1 \cap U_2$ with respect to the torso of $G[U_i]$ at $r_i$; however, as $U_1 \cap U_2 \subseteq Z_i$ and $Z_i \subseteq A_{r_i}$, we conclude that $(\delta_{r_i}, A_{r_i})$ is a $4\alpha$-near embedding of the torso of $G$ at $r_i$ for $i = 1, 2$.  Finally, as $|X_r| \le 4\alpha$, the torso of $G$ at $r$ has a trivial $4\alpha$-near embedding with $A_r = Z \cup (U_1 \cap U_2)$.  Thus, $G$ does not form a counterexample to the theorem, a contradiction.
\end{cproof}

Let $\zT$ be the set of ordered separations $(U_1, U_2)$ of order at most $\alpha -1$ such that $|U_2 \cap Z| > \alpha$.  By Claim \ref{cl:global1}, for every separation $(U_1, U_2)$ of order at most $\alpha - 1$, exactly one of $(U_1, U_2)$ and $(U_2, U_1)$ is in $\zT$.  Moreover, given three separations $(U^1_1, U^1_2), (U^2_1, U^2_2), (U^3_1, U^3_2) \in \zT$, as $U^j_1$ contains at most $\alpha-1$ vertices of $Z$ for each $j = 1, 2, 3$, we have that at least one vertex of $Z$ is not contained in $U_1^1 \cup U_1^2 \cup U_1^3$ and thus $G \neq G[U_1^1]\cup G[U_1^2] \cup G[U_1^3]$.  We conclude that $\zT$ is a tangle of order $\alpha$.

By Theorem \ref{thm:RS_struct_exact} and the fact that $G$ has no $K_p$ minor, there exists a surface $\Sigma$ of genus at most $p(p+1)$, a subset $A \subseteq V(G)$, $|A| \le 5p^2 \cdot p^{10p^7p^{26}}$ and a $\Sigma$-decomposition $\rho$ of $G-A$ of breadth $t \le 2p^2$ and depth at most $p^{10^7p^{26}}$ which is $\zT - A$ central.  We break $G$ up into subgraphs based on the decomposition $\rho$, and using $\rho$ along with the minimality of $G$, find the desired tree decomposition for each subgraph.  We will do in such a way that the tree decompositions of the subgraphs can be combined to get a tree decomposition of the whole graph $G$, contradicting our choice of counterexample.

We begin by decomposing the vortices of $\rho$.  For every $1 \le i \le t$, let $c_1, \dots, c_t$ be the vortices of $\rho$, let $(G_i, \Omega_i)$ be the vortex society of $c_i$, and let $n(i) = |V(\Omega_i)|$.  Fix a linear order of $V(\Omega_i)$ $v^i_1,\dots, v^i_{n(i)}$ and a linear decomposition $(Y_1^i, \dots, Y_{n(i)}^i)$ of depth at most $p^{10^7p^{26}}$ for each $i$.  For all $1 \le i \le t$, $1 \le j \le n(i)$, fix $\bar{Y}_j^i = (Y^i_j \cap Y^i_{j-1}) \cup  (Y^i_j \cap Y^i_{j+1}) \cup \{v^i_j\}$ where $Y^i_0 = Y^i_{n(i) + 1} = \emptyset$.  Thus, every vertex of $V(G_i)$ which is not contained in $\bar{Y}^i_j$ for all indices $j$ must be a vertex of $V(G_i) \setminus V(\Omega_i)$ which is contained in exactly one set $Y^i_{j'}$ for some index $j'$.

Let $H^i_j$ be the induced subgraph $G_i[Y^i_j \cup A]$.  Let $W^i = \bigcup_{j = 1}^{n(i)} \bar{Y}_j^i$.  Note that $(\bar{Y}_1^i, \dots, \bar{Y}_{n(i)}^i)$ defines a linear decomposition of $(G_i[W^i], \Omega_i)$ of width at most $2p^{10^7p^{26}} + 1$.  By construction, for all $i$, $1 \le i \le t$ and for all $j$, $1 \le j \le n(i)$, $(V(G) \setminus (Y^i_j \setminus \bar{Y}^i_j), Y^i_j \cup A)$ is a separation of $G$ of order at most $|A| + |\bar{Y}^i_j| \le 2p^{10^7p^{26}} + 1 + 5p^2p^{10^7p^{26}} < 6p^2p^{10^7p^{26}} < \alpha$.  Thus, as $\rho$ is $\zT - A$ central, there are at most $6p^2p^{10^7p^{26}}$ vertices of $Z$ contained in $Y^i_j$.

For every $c \in C(\rho)$ such that $c$ is not a vortex, let $H^c$ be the subgraph of $G$ with vertex set $V(\sigma(c)) \cup A$ and edge set $E(\sigma(c))$ along with every edge with one endpoint in $A$ and one endpoint in $V(\sigma(c))$.  As above, $(V(G) \setminus (V(\sigma(c)) \setminus \pi(N)), V(H^c))$ is a separation of $G$ of order at most $|A| + 3 < 3p^2p^{10^7p^{26}}$ and we conclude that $H^c$ contains at most $3p^2p^{10^7p^{26}}$ vertices of $Z$.

Finally, let $H_0$ be the subgraph of $G$ induced by $\pi(N(\rho)) \cup A \cup Z \cup \bigcup_{i=1}^t W^i$.  Note that $\rho$ restricted to $H_0$ yields a $4\alpha$-near embedding $(\rho', A \cup Z)$ of breadth $t$ and depth at most $2p^{10^7p^{26}} + 1$ in a surface $\Sigma$ of genus at most $p(p+1)$.

For every $1 \le i \le t$ and for every $1 \le j \le n(i)$, by minimality, $H^i_j$ has a tree decomposition $(T^i_j, \zX^i_j)$ such that the torso at each vertex $x \in V(T^i_j)$ has a $4\alpha$-near embedding.  Moreover, there is a root $r^i_j$ of $T^i_j$ such that the deleted set of vertices in the near embedding of the torso at $r^i_j$ contains $A \cup \bar{Y}^i_j \cup  (Z\cap V(H^i_j)) = V(H^i_j) \cap V(H_0)$.  Similarly, for every $c \in C(\rho)$ which is not a vortex, the graph $H^c$ has a tree decomposition $(T^c, \zX^c)$ with root $r^c$ such that the deleted set of the near embedding at the torso of $r^c$ contains $A \cup (Z \cap V(H^c)) \cup (V(\sigma(c)) \cap \pi(N)) = V(H^c) \cap V(H_0)$.

We construct a tree decomposition for $G$ as follows.  Add a new vertex $r_0$ with $X_{r_0} = V(H_0)$ and define a new tree $T$ obtained by adding an edge from $r$ to $r^i_j$ for all values of $i$ and $j$ and from $r$ to $r^c$ for all non-vortex cells $c \in C(\rho)$.  Let $\zX = \bigcup_1^t\bigcup_1^{n(i)} \zX^i_j \cup \bigcup_{c \in C(\rho), \text{ $c$ not a vortex}} \zX^c \cup \{X_0\}$.
Note that the near embedding $\rho'$ we found above is also a near embedding for the torso of $G$ at $r$ by the careful choice of $(V(H_0) \cap V(H^i_j)) \setminus (A \cup Z)$ and $V(H_0) \cap V(H^c)) \setminus (A \cup Z)$.  We conclude that $(T, \zX)$ satisfies the conclusions of the theorem, contradicting our choice of counterexample.
\end{proof}

\section{Sufficiency}
\label{sec:suffic}

In this section we prove Theorem~\ref{thm:mainconverse}.
Our objective is to prove that if a graph satisfies the structure of Theorem~\ref{thm:main} even for
$r=0$, then $G$ has no $K_{p'}$ minor for some $p'$.
We need the following lemma.

\begin{LE}
\label{lem:mainconverse}
Let $b,g,p\ge 0$ be integers, let  $\Sigma$  be a surface of Euler genus at most $g$, $\Delta_1,\Delta_2,\ldots,\Delta_b\subseteq\Sigma$
be disjoint closed disks, and let $H_0$ be a graph drawn in $\Sigma$ in such a way that every vertex is
drawn in the boundary of one of the disks $\Delta_i$ and every edge is disjoint from the interiors of the disks.
For $i=1,2,\ldots,b$ let $(H_i,\Omega_i)$ be a society such that $V(H_i)=V(\Omega_i)$ consists of the vertices
of $H_0$ drawn in the boundary of $\Delta_i$ and the order of $\Omega_i$ is determined by the order on the boundary of
$\Delta_i$.  Assume that $(H_i, \Omega_i)$ has depth at most $d$ for all $i$.
Let $H:=H_0\cup H_1\cup\cdots\cup H_b$.
If $H$ has a $K_p$ minor, then
$$p\le 200(2g+b)^3d.$$
\end{LE}

\begin{proof}

We prove a stronger statement to assist with the induction.

\begin{itemize}
\item[($*$)]{\it
Let $b, g, p \ge 0$ and $\Sigma$ a surface of Euler of genus at most $g$ the disks $\Delta_1, \dots, \Delta_b \subseteq \Sigma$, subgraphs $H_0, H_1, \dots, H_b$, and societies $(H_i, \Omega_i)$ be as in the statement of Lemma \ref{lem:mainconverse}.  Let $H = \bigcup_0^b H_i$.  Let $D \ge 1, T \ge 10(2g + b)$ be integers and assume that there exists a partition $X_1, \dots, X_t$ of $V(H)$, $t \le T$, with the following properties.  For all $1 \le i \le t$ there exists an index $j$ such that $X_i$ is a segment of $\Omega_j$ and every transction in $(H_j, \Omega_j)$ with all endpoints in $X_i$ has order at most $D$.  Moreover, for every pair of indices $i$ and $j$, there exist at most $D$ edges of $E(H) \setminus E(H_0)$ with one endpoint in $X_i$ and the other endpoint in $X_j$.  Then if $H$ contains $K_p$ as a minor, we have that
$$p \le 2(2g+b) T^2 D.$$
}
\end{itemize}

Assume that $(*)$ is false, and pick a counterexample $b, g, p, \Sigma$, $\Delta_1, \dots, \Delta_b$, $H_0, \dots, H_b$, $T$, $D$, $X_1, \dots, X_t$ to minimize $|V(H)|$.  Thus, we may fix $p > 2(2g+b)T^2D$ such that $H$ has a $K_p$ minor, violating the desired bound in the claim.

Proposition $(*)$ clearly holds when $H$ is the null graph, and hence $H$ has at least one vertex.  As $H$ is non-empty, we have that $b$ is at least one, and so $T \ge 10$ and $p \ge 200$.  Moreover, if $H$ were not connected, a single connected component would still satisfy the conditions of $(*)$ and contain $K_p$ as a minor, contradicting our choice of counterexample to minimize the number of vertices.  Thus, we have that $H$ is a connected~ graph.

As a first case, assume that $b\ge2$. Since $H$ is connected, there exists an edge $uv\in E(H)$ such that $u,v$ belong
to different graphs $H_i$, say $u\in V(H_{b-1})$ and $v\in V(H_b)$.
Let $H':=H - \{u,v\}$; we wish to define disks and a partition of the vertices on the boundary of the disks in order to apply the minimality of our counterexample to bound the maximum size of a clique minor in $H'$.
To that end we need to merge the disks $\Delta_{b-1}$ and $\Delta_b$ by ``blowing up" the edge $uv$.
 More precisely, let $\Lambda$ be a small closed neighborhood of (the drawing of) the edge $uv$ such that
\begin{itemize}
\item $\Lambda$ is disjoint from all vertices and edges of $H_0\setminus\{u,v\}$,
\item $\Lambda\cap\Delta_i=\emptyset$ for all $i=1,2,\ldots,b-2$, and
\item $\Lambda\cap\Delta_i$ homeomorphic to the closed unit interval for $i=b-1,b$.
\end{itemize}
Then $\Delta_{b+1}:=\Delta_{b-1}\cup\Lambda\cup\Delta_b$ is a closed disk disjoint from $\Delta_1, \dots, \Delta_{b-2}$.  The vertices of $H'$ are drawn on the boundary of $\Delta_1, \dots, \Delta_{b-2}, \Delta_{b+1}$ and no edge has an internal point which intersects the boundary of one of those disks.  Let $H_{b+1}$ be the subgraph of $H'$ drawn in the disc $\Delta_{b+1}$ and let $(H_{b+1}, \Omega_{b+1})$ be the society drawn in $\Delta_{b+1}$ as in the statement.   We define a partition $X_1', \dots, X_{t'}'$ for some $t' \le t + 2$ as follows.  If $X_i$ is disjoint from $\{u,v\}$, we let $X_i' = X_i$.  If $X_i$ contains $u$ or $v$, we replace $X_i$ by either one or two segments of $\Omega_{b+1}$ partitioning $X_i$.  Note that as each $X_i'$ is a subset of some $X_j$, we still have the property that at most $D$ edges of $E(H') \setminus E(H_0)$ have endpoints in distinct $X_i'$ and $X_j'$.  The graph $H'$ contains $K_{p-2}$ as a minor.  By the minimality of our counterexample, we have that
\begin{align*}
p-2 &\le 2(2g + b -1) (T+2)^2 D\\
& = \left[ 2(2g + b) - 2\right] \cdot \left[ T^2 + ( 4T + 4) \right ] D \\
& = 2(2g + b) T^2 D + \left [ (4T + 4) 2(2g+b) - 2T^2 \right]D - (8T + 8)D\\
& \le 2(2g + b) T^2 D + \left [ 10T(2g+b) - 2T^2 \right]D - (8T + 8)D\\
& < 2(2g + b)T^2 D -2,
\end{align*}
contradicting the bounds on $p$.  Note that the last inequality holds by the assumption that $T \ge 10(2g + b)$.
We conclude that $b=1$ and that all the vertices of $H$ lie on the boundary of $\Delta_1$.

As a second case let us assume that $H_0$  has an edge $e=uv$ such that $\Delta_1\cup e$ includes the image of
a non-null-homotopic  simple closed curve, say $\phi$.
The image of $\phi$ divides $\Delta_1$ into two closed disks, say $\Delta'$ and $\Delta''$.
Let $\Sigma'$ be the  manifold obtained by cutting $\Sigma$ open along the curve $\phi$
and attaching disk(s) and  to the resulting boundary component(s).
We will regard $\Delta'$ and $\Delta''$ as subsets of $\Sigma'$.
Let $Z'$ be the set of vertices of $H - \{u,v\}$  that lie on $\Delta'$ and let $Z''$ be the set of those that lie  on $\Delta''$.
Since there exist at most $D$ disjoint paths in $H_1$ between any pair of segments $X_i$ and $X_j$, we have that there are at most  $t^2D$ disjoint paths in $H_1 - \{u,v\}$ from $Z'$
to $Z''$.
By Menger's theorem there exists a subset of $V(H_1) \setminus \{u,v\}$ of size at most $t^2D$ that separates $Z'$ from $Z''$.
 By adding $u$ and $v$ to that set we arrive at a set $Y\subseteq V(H_1)$ of size at most $t^2D+2$ such that $u,v\in Y$
and the graph $H_1 -  Y$
 can be written as $H_1 - Y =H'\cup H''$, where $V(H')\subseteq Z'$ and $V(H'')\subseteq Z''$.
Let $\Omega'$ be the cyclic ordering of $V(H')$ determined by the order of appearance on the boundary  of $\Delta'$,
and let $\Omega''$ be defined analogously.  Let $X_1', \dots, X_{t'}'$ be the partition of $V(H_1) \setminus Y$ into segments of $\Omega'$ and $\Omega''$ obtained by restricting the segments $X_1, \dots, X_t$ to $\Omega'$ and $\Omega''$.  Note that at most two $X_i$ have vertices in both $V(\Omega')$ and $V(\Omega'')$, and so we have that $t' \le t+2$.  As each $X_i'$ is a subset of some $X_j$, we have that there are most $D$ edges of $H'$ or $H''$ with endpoints in distinct segments of $X_1', \dots, X_{t'}'$.
The graph $H\setminus Y$ has a $K_{p-t^2D-2}$ minor.

Assume as a case that $\Sigma'$ is connected; then it is a surface of Euler genus at most $g-1$.  By the minimality of $G$, using the inequalities above we have that
\begin{align*}
p - t^2D - 2 & \le 2(2(g-1) + b+ 1)(T + 2)^2D \\
& \le 2(2g + b) T^2 D + \left [ 10T(2g+b) - 2T^2 \right]D - (8T + 8)D\\
& \le 2(2g + b) T^2 D  - T^2 D - 2, \\
\end{align*}
a contradiction.

We may therefore assume that $\Sigma'$ is not connected. In that case $\Sigma'$ is the union of two surfaces,
each of  genus at most $g-1$, where each of the two surfaces embeds a component of $H\setminus Y$.
One of those components  has a $K_{p-t^2D-2}$ minor, and  the lemma follows by the same argument as above.  This completes the analysis of the case that $H_0$  has an edge $uv$ such that $\Delta_1 \cup uv$ includes the image of a non-null-homotopic simple closed curve,
and so we conclude that such an edge does not exist.

The nonexistence of an edge of the previous paragraph implies that the graph $H$ and disk  $\Delta_1$ are
contained in a disk.
We may therefore assume that $\Sigma$ is the sphere.
We may also assume that the graph $H_0$  is edge-maximal. Thus it has a Hamilton cycle
that traces the vertex-set
of $H_0$ in the same order as $\Omega_1$. The Hamilton cycle  bounds a face of $H_0$, say $f_1$, that includes the interior of $\Delta_1$.
Every face of $H_0$ other than $f_1$ is bounded by a triangle.
Let $f_0$ be a face of $H_0$  other than $f_1$.
Let $Y_1,Y_2,\ldots,Y_p$ be a model of a $K_p$ minor in $H$.
Let $e\in E(H_0)$. The complement of the point set $\Delta_1\cup e$  consists of two connected components,
each homeomorphic to an open disk.
Let $\Lambda_e$ denote the closure of the component which is disjoint from $f_0$,  let $I_e$ be the set of
integers $i\in\{1,2,\ldots,p\}$ such that $\Lambda_e$ includes a vertex of $Y_i$ that is not an endpoint of $e$,
and  let $J_e$ be the set of
integers $i\in\{1,2,\ldots,p\}$ such that the complement of $\Lambda_e$ includes a vertex of $Y_i$.
If $e_1,e_2,e_3$ are the three edges incident with $f_0$, then $I_{e_1}\cup I_{e_2}\cup I_{e_3}$ includes all but possibly
three elements of the set $\{1,2,\ldots,p\}$.
It follows that there exists an edge $e\in E(H_0)$  such that $|I_e|\ge p/3-1$, and so we may choose such an edge
so that $\Lambda_e$ is minimal with respect to inclusion. We claim that $|J_e|\ge p/3-1$.
To prove the claim suppose for a contradiction that $|J_e|< p/3-1$.
Thus $|I_e|\ge p-|J_e|-2>2p/3-1$.  Since $p\ge2$, it  follows that $I_e$ is non-empty and thus $e\not\in E(C)$.
The edge $e$ is incident with a face $f$ of $H_0$ such that $f\ne f_1$ and $f\subseteq\Lambda_e$.
Let $e',e''$ be the other two edges incident with $f$. It follows that $\Lambda_{e'}\cup\Lambda_{e''}\subseteq\Lambda_e$
and $|I_{e'}|+|I_{e''}|\ge|I_e|-1\ge2p/3-2$. Thus one of the edges $e',e''$ contradicts the choice of $e$.
This proves our claim that $|J_e|\ge p/3-1$.

Let $k:=\lfloor p/3\rfloor-1$, let $A$ be the set of vertices of $H$ that belong to $\Lambda_e$, except the endpoints of $e$,
and  let $B$ be the set of vertices of $H$ that do not belong to $\Lambda_e$.
Since $|I_e|,|J_e|\ge p/3-1$, the existence of the model of the $K_p$  minor implies that there exist $k$ disjoint paths
in $H$ from $A$ to $B$.  Of those paths, at least $k-2$  do not use either an endpoint of $e$, and therefore may be chosen
to be of length one. The presence of the edge $e$ implies that those $k-2$ paths are  subgraphs of $H_1$,
and hence they form a transaction $\zP$ in the society $(H_1,\Omega_1)$.   As $g = 0$ and $b = 1$, we conclude that
$$|\zP| \ge \left \lfloor \frac{2T^2D + 1}{3} \right \rfloor -3 \ge \frac{2T^2D}{3} -4.$$

If we consider the two segments of $\Omega_1$ defined by the endpoints of the edge $e$, one intersects at most $T/2 + 1$ distinct $X_i$ because at most two $X_i$ contain an endpoint of $e$ and intersect both segments of $\Omega_1$ defined by $e$.  As there are at most $D$ elements of $\zP$ between each pair of segments $X_i$, we have that
$$|\zP| \le \left ( \frac{T+2}{2} \right) T \cdot D,$$
implying
$$\frac{2T^2D}{3} -4 \le \left ( \frac{T+2}{2} \right) T \cdot D,$$
a contradition given that $T \ge 10$.  This completes the proof of $(*)$.

The lemma holds now applying $(*)$ with $T = 10(2g + b)$, $D = d$, and the partition $X_i = V(\Omega_i)$ for $1 \le i \le b$.
\end{proof}

Lemma \ref{lem:mainconverse} quickly bounds the size of a clique minor in a graph admitting a $\Sigma$-decomposition in terms of the breadth and depth of the decomposition along with the genus of the surface.


\begin{LE}\label{lem:mainconverse2}
Let $g, b, d, p \ge 0$ be integers.  Let $G$ be a graph, and let $\Sigma$ be a surface of Euler genus $g$.  Assume that $G$ has a $\Sigma$-decomposition $\delta = (\Gamma, \zD)$ of breadth $b$ and depth $d$.  Let $X_1, \dots, X_p$ be a model of a $K_p$ minor in $G$ and assume that for every cell $c \in C(\delta)$ and for all $1 \le i \le p$, $X_i \nsubseteq V(\sigma(c)) \setminus \widetilde{c}$.  Then $$p \le 4 + \sqrt{6g}  + 200(2g + b)^3d.$$
\end{LE}

\begin{proof}
We proceed by induction on $|V(G)| + |E(G)|$.  The claim clearly holds when $G$ is the empty graph, so we may assume that $G$ is non-empty and that the statement holds for all smaller graphs.  We may also assume that every vertex of $G$ is contained in one of the subsets $X_i$ by applying induction to the graph after deleting any vertex not contained in the model of the $K_p$ minor.

Fix the $\Sigma$-decomposition $\delta = (\Gamma, \zD)$ of $G$.  If there exists a cell $c \in C(\delta)$ and an edge $e \in E(\sigma(c))$ such that at most one endpoint of $e$ is in $\widetilde{c}$ and both endpoints of $e$ are contained in $X_i$ for some $i$, then there exists a $\Sigma$-decomposition of $G/e$ of breadth $b$ and depth at most $d$ obtained by contracting the edge $e$ in the drawing $\Gamma$.  By induction, the desired bound holds on $p$.

We conclude that for every cell $c$, $V(\sigma(c)) = \widetilde{c}$.  Otherwise, any vertex $v$ of $V(\sigma(c)) \setminus \widetilde{c}$ would be contained in some $X_i$ and as no edge incident $v$ can be contained in both $\sigma(c)$ and $X_i$, we would have that $X_i = \{v\}$, contrary to the assumption that no $X_i$ is contained in the interior of a vortex.

Let $\Delta_1, \dots, \Delta_b$ be the elements of $\zD$ for which the corresponding cell is a vortex.  The drawing $\Gamma$ consists of a drawing of $G$ with every crossing edge contained in union of the interiors of $\Delta_1, \dots, \Delta_b$.

Let $X \subseteq V(G)$ be the set of vertices of $G$ such that the image in $\Gamma$ is disjoint from $\Delta_1 \cup \dots \cup \Delta_b$.  For all $x \in X$, there exists $i$ such that $x \in X_i$.  If $X_i \neq x$, then $x$ has a neighbor $y$ in $X_i$.  The graph $G/xy$ satisfies the conditions of the theorem, and the claim holds by induction.  We conclude that $X_i = \{x\}$, and that this holds for every element of $X$.

It follows that $G[X]$ is a clique subgraph drawn in $\Sigma$.
The graph $G-X$ has every vertex drawn on the boundary of one of the $\Delta_i$, every edge internally disjoint from the boundary of the disks $\Delta_i$ and every crossing contained in the interior of some $\Delta_i$.  Thus, $G-X$ satisfies the conditions of Lemma \ref{lem:mainconverse}.  The sets $\{X_i: X_i \nsubseteq X\}$ form a model of a $K_{p-|X|}$ minor and thus $p - |X| \le 200(2g+b)^3d$.  The set $X$ induces a clique subgraph drawn in $\Sigma$.  It follows immediately (see [\cite{JW}, Lemma 2.1] and \cite{BT}) from Euler's formula that $|X| \le \sqrt{6g} + 4$, completing the proof of the lemma.
\end{proof}
We now present the proof of Theorem \ref{thm:mainconverse}.

\begin{proof}[Proof. (Theorem \ref{thm:mainconverse})]
Let $X_1, \dots, X_p$ be a model of a clique minor grasped by the $R$-wall $W$. Assume that the theorem is false and that $p > \alpha + 2bd + 4 + \sqrt{6g} + 200(2g + b)^3d$.  Note that this implies that $R > \alpha + 2bd + 4 + \sqrt{6g} + 200(2g + b)^3d$ as well.

Let $c \in C(\delta)$ be a cell of the $W$-central $\Sigma$-decomposition $\delta$ of $G-A$.  Assume for the moment that $|\widetilde{c} | \le 3$.  We claim that there does not exist an index $i$ with $X_i \subseteq V(\sigma(c)) \setminus \widetilde{c}$.  To see this, assume that such a set $X_i$ existed.  There exist $p> 3$ distinct horizontal paths and $p$ distinct vertical paths intersecting $X_i$.  As $|\widetilde{c}| \le 3$, at least one horizontal and one vertical path are contained in $\sigma(c)  - \widetilde{c}$.  In this case, $(V(\sigma(c)) \cup A, V(G) \setminus (V(\sigma(c)) \setminus \widetilde{c}))$ is a separation of $G$ order at most $|A| + 3 < R$ with the $W$-majority side contained in a cell of $\delta$, contrary to the fact that $\delta$ is $W$-central.  We conclude that no such $X_i$ is contained in $V(\sigma(c)) \setminus \widetilde{c}$.

Assume now that $c$ is a vortex of the decomposition.  Let $(H, \Omega)$ be the vortex society of $c$.  Let $X_i$ be a branch set such that $X_i \subseteq V(H) \setminus V(\Omega)$.  We claim that there exist $p-|A|$ disjoint paths in $H$ from $X_i$ to $V(\Omega)$.  If not, then there exists a separation $(Y, Z)$ of $G-A$ of order at most $p - |A| - 1$ with $X_i \subseteq Z \subseteq V(H) \setminus V(\Omega)$ and $V(\Omega) \subseteq Y$.  The separation $(Y \cup A, Z \cup A)$ is of order at most $p-1$, and as above, there must exist a horizontal path and a vertical path of $W$ contained in $G[Z]$, contradicting the fact $\delta$ is $W$-central.  Fix $\zP_i$ to be a linkage in $H$ of order $p - |A|$ from $X_i$ to $V(\Omega)$.

Let $I = \{i: X_i \subseteq V(H) \setminus V(\Omega)\}$ and assume, to reach a contradiction, that $|I| \ge 2b+1$.  Fix $S$ to be a minimal segment, by containment, of $\Omega$ such that for at least $b+1$ distinct indices $i \in I$, at least $b+2$ of the elements of $\zP_i$ have an endpoint in $S$.  Let $x$ be an endpoint of $S$.  Then for at least $b+1$ of the indices $i \in I$, at least $b+1$ of the elements of $\zP_i$ have an endpoint in $S -x$.  Moreover, by the minimality of $S$, at most $b$ indices $i$ satisfy the property that at least $b+2$ elements of $\zP_i$ have an endpoint in $S-x$.  Thus, for at least $b+1$ distinct indices $i$ of $I$, at least $b+1$ elements of $\zP_i$ have an endpoint in $V(\Omega) \setminus (S - x)$.  Note that here we are using that $p \ge 2b+3$.

We claim that there exist $b+1$ disjoint paths from $S-x$ to $V(\Omega) \setminus (S-x)$ in $H$.  If not, there exists a set $Z$ of at most $b$ vertices intersecting each such path.  Thus, there exists an index $i \in I$ such that $X_i$ is disjoint from $Z$ and $b+1$ of the elements of $\zP_i$ have an endpoint in $S-x$.  Similarly, there exists an index $i' \in I$ such that $X_{i'}$ is disjoint from $Z$ and $b+1$ elements of $\zP_{i'}$ have an endpoint in $V(\Omega) \setminus (S-x)$.  Thus, there exists a path from $S-x$ to $V(\Omega) \setminus (S-x)$ in $\zP_i \cup X_i \cup X_{i'} \cup \zP_{i'}$ using the fact that $X_i$ and $X_{i'}$ induce connected subgraphs connected by an edge.  But $Z$ was a cut set separating $S-x$ and $V(\Omega) \setminus (S-x)$, a contradiction.  We conclude that the desired $b+1$ paths exist, but this contradicts the fact that the breadth of $(H, \Omega)$ is at most $b$.  This contradiction proves that $|I| \le 2b$.

By restricting to the $X_i$ which are disjoint from $A$ and for which $X_i$ is not drawn in the interior of any vortex, we see that $G-A$ has a $K_{p'}$ minor for some $p' > 4 + \sqrt{6g} + 200(2g + b)^3d$, contradicting Lemma \ref{lem:mainconverse2} and completing the proof of the theorem.
\end{proof}

We conclude the section proving Corollary \ref{cor:mainconverse}.


\begin{proof}[Proof. (Corollary \ref{cor:mainconverse})]
Assume, to reach a contradiction, that $G$ contains a model of a $K_p$ minor $X_1, \dots, X_p$ with $p > 8R^2$.  As $p > 8R^2$, the sets $X_1, \dots, X_{8R^2}$ also form the branch sets of a model of a $4R \times 2R$-grid minor.  We conclude that $G$ contains a $2R$-wall $W$ with $V(W) \subseteq \bigcup_1^{8R^2} X_i$ such that the vertices $x_{i,j}$ of $W$ corresponding to each of the vertices $(i,j)$ in the elementary $2R$-wall have the property that $x_{i,j} \in X_{4R(j-1) + i}$.  Moreover, the horizontal paths $P_1, \dots, P_{2R}$ and vertical paths $Q_1, \dots, Q_{2R}$ have the property that $V(P_i) \cap V(Q_j) \subseteq X_{4R(i-1) + 2j -1} \cup X_{4R(i-1) + 2j}$.

We can group the $X_i$ together to find a smaller clique minor which is grasped by the wall $W$.  To form the first branch set of the smaller minor, we group together the bags $X_l$ containing the union of $P_i \cap Q_i$ for $1 \le i \le R$ down the diagonal of the wall.  Subsequent branch sets shift the diagonal to the right.  Explicitly, define $Y_l$, $1 \le l \le R$ as $Y_l = \{ X_{4r(i-1) + 2(i + l) - 1} \cup X_{4R(i-1) + 2(i + l)} : 1 \le i \le R\}$.  Thus, $Y_1, \dots, Y_R$ are the branch sets of a model of $K_R$ grasped by $W$.

By assumption, there exists a subset $A$ of at most $\alpha$ vertices such that $G-A$ has a $W$-central $\Sigma$-decomposition of breadth at most $b$ and depth at most $d$ for a surface $\Sigma$ of genus at most $g$.  The clique minor formed by $Y_1,\dots, Y_l$ contradicts the conclusion of Theorem \ref{thm:mainconverse}, completing the proof of the corollary.
\end{proof}

\begin{center}{\bf ACKNOWLEDGEMENTS} \end{center}

Robin Thomas passed away in the Spring of 2020 as we were in the final stages of preparing this paper.  He was our mentor, our colleague, and our friend.  We will miss him.


\baselineskip 11pt
\vfill
\smallrm
\noindent
This material is based upon work supported by the National Science Foundation
under Grant No.~DMS-0701077.
Any opinions, findings, and conclusions or recommendations expressed in
this material are those of the authors and do not necessarily reflect
the views of the National Science Foundation.

\end{document}